\documentclass[11pt]{amsart} 
\usepackage{amsmath,amsthm}
\usepackage{xypic}
\usepackage{enumerate}
\usepackage{lscape}

\usepackage{amscd}
\usepackage{amsfonts,amssymb,latexsym}
\usepackage[all]{xy}

\xyoption{arc}

\newcommand{\udots}{\mathinner{\mskip1mu\raise1pt\vbox{\kern7pt\hbox{.}}
\mskip2mu\raise4pt\hbox{.}\mskip2mu\raise7pt\hbox{.}\mskip1mu}}

\newcommand{\SC}{{\mathcal{C}}}
\newcommand{\SD}{{\mathcal{D}}}
\newcommand{\SE}{{\mathcal{E}}}
\newcommand{\SF}{{\mathcal{F}}}
\newcommand{\SG}{{\mathcal{G}}}
\newcommand{\SH}{{\mathcal{H}}}
\newcommand{\SI}{{\mathcal{I}}}

\newcommand{\SK}{{\mathcal{K}}}

\newcommand{\SM}{{\mathcal{M}}}
\newcommand{\SN}{{\mathcal{N}}}
\newcommand{\SO}{{\mathcal{O}}}
\newcommand{\SP}{{\mathcal{P}}}

\newcommand{\SR}{{\mathcal{R}}}
\newcommand{\SSS}{{\mathcal{S}}}
\newcommand{\ST}{{\mathcal{T}}}
\newcommand{\SU}{{\mathcal{U}}}
\newcommand{\SV}{{\mathcal{V}}}
\newcommand{\SW}{{\mathcal{W}}}
\newcommand{\SX}{{\mathcal{X}}}
\newcommand{\SY}{{\mathcal{Y}}}
\newcommand{\SZ}{{\mathcal{Z}}}

\newcommand{\PP}{\mathbb{P}}
\newcommand{\ZZ}{\mathbb{Z}}

\newcommand{\CC}{\mathbb{C}}
\newcommand{\RR}{\mathbb{R}}

\newcommand{\Ext}{\operatorname{Ext}}
\newcommand{\Gr}{\operatorname{Gr}}
\newcommand{\Spec}{\operatorname{Spec}}

\newcommand{\codim}{\operatorname{codim}}

\newcommand{\Hom}{\operatorname{Hom}}

\newcommand{\Pic}{\operatorname{Pic}}

\newcommand{\Sym}{\operatorname{Sym}}
\newcommand{\id}{\operatorname{Id}}
\newcommand{\im}{\operatorname{Im}}
\newcommand{\Inn}{\operatorname{Inn}}
\newcommand{\Aut}{\operatorname{Aut}}
\newcommand{\Bir}{\operatorname{Bir}}
\newcommand{\Out}{\operatorname{Out}}

\newcommand{\rk}{\operatorname{rk}}
\newcommand{\pdeg}{\operatorname{pardeg}}
\newcommand{\SParEnd}{\operatorname{SPEnd}}
\newcommand{\ParEnd}{\operatorname{PEnd}}
\newcommand{\ParHom}{\operatorname{PHom}}
\newcommand{\SParHom}{\operatorname{SPHom}}
\newcommand{\End}{\operatorname{End}}
\newcommand{\owt}{\operatorname{wt}}
\newcommand{\Mat}{\operatorname{Mat}}
\newcommand{\diag}{\operatorname{diag}}

\newcommand{\wt}{\widetilde}

\newcommand{\GL}{\operatorname{GL}}

\newcommand{\SSL}{\operatorname{SL}}
\newcommand{\ssl}{\mathfrak{sl}}
\newcommand{\parsl}{\mathfrak{parsl}}
\newcommand{\pargl}{\mathfrak{pargl}}
\newcommand{\PGL}{\operatorname{PGL}}
\newcommand{\SMP}{\SM^{\op{par}}}

\newcommand{\fm}{\mathfrak{m}}

\newcommand{\op}{\operatorname}

\newtheorem{proposition}{Proposition}[section]
\newtheorem{theorem}[proposition]{Theorem}
\newtheorem{definition}[proposition]{Definition}

\newtheorem{lemma}[proposition]{Lemma}

\newtheorem{corollary}[proposition]{Corollary}
\newtheorem{remark}[proposition]{Remark}

\numberwithin{equation}{section}

\usepackage{hyperref}
\hypersetup{
    colorlinks=false,
    pdfborder={0 0 0}
}

\title[Automorphisms moduli of parabolic bundles]{Automorphism group of the moduli space of parabolic bundles over a curve}
\author[D. Alfaya]{David Alfaya}
\author[T. G\'omez]{Tom\'as L. G\'omez}

\date{}

\address{Instituto de Ciencias Matem\'aticas (CSIC-UAM-UC3M-UCM),
Nicol\'as Cabrera 15, Campus Cantoblanco UAM, 28049 Madrid, Spain}

\email{david.alfaya@icmat.es}

\address[Current address]{Department of Applied Mathematics and Institute for Research in Technology, ICAI School of Engineering, Comillas Pontifical University, C/Alberto Aguilera 25, 28015 Madrid, Spain}
\email[currently]{dalfaya@comillas.edu}

\address{Instituto de Ciencias Matem\'aticas (CSIC-UAM-UC3M-UCM),
Nicol\'as Cabrera 15, Campus Cantoblanco UAM, 28049 Madrid, Spain}

\email{tomas.gomez@icmat.es}

\begin{document}

\begin{abstract}
We find the automorphism group of the moduli space of parabolic bundles on a smooth curve (with fixed determinant and system of weights). This group is generated by: automorphisms of the marked curve, tensoring with a line bundle, taking the dual, and Hecke transforms (using the filtrations given by the parabolic structure). A Torelli theorem for parabolic bundles with arbitrary rank and generic weights is also obtained. These results are extended to the classification of birational equivalences which are defined over "big" open subsets (3-birational maps, i.e. birational maps giving an isomorphism between open subsets with complement of codimension at least 3).

Finally, an analysis of the stability chambers for the parabolic weights is performed in order to determine precisely when two moduli spaces of parabolic vector bundles with different parameters (curve, rank, determinant and weights) can be isomorphic. 
\end{abstract}

\maketitle

\tableofcontents

\section{Introduction}

Let $X$ be an irreducible smooth complex projective curve. Let $D=\sum_{i=1}^n x_i$ be an effective divisor on $X$ consisting on distinct points and let $\xi$ be a line bundle on $X$. Let $\alpha$ be a rank $r$ generic full flag system of weights over $D$. Let $\SM(r,\alpha,\xi)$ be the moduli space of stable parabolic vector bundles $(E,E_\bullet)$ over $(X,D)$ of rank $r$ with system of weights $\alpha$ and determinant $\det(E)\cong \xi$.

Before describing the automorphisms of this moduli space, let us go back to the non-parabolic case and recall the known classification of the automorphisms of the moduli space of vector bundles. The following two transformations generate the automorphism group of the moduli space $\SM(r,\xi)$ of stable vector bundles over $X$ with rank $r$ and determinant $\xi$. Given an automorphism $\sigma:X\to X$
\begin{enumerate}
\item Send $E\to X$ to $L\otimes \sigma^*E$, where $L$ is a line bundle over $X$ with $L^r\otimes \sigma^*\xi\cong \xi$
\item Send $E$ to $L\otimes \sigma^*(E^\vee)$, where $L$ is a line bundle satisfying $L^r\otimes \sigma^* \xi^{-1}\cong\xi$
\end{enumerate}
This result was initially proved by Kouvidakis and Pantev \cite{KP} using an argument on the fibers of the Hitchin map defined on the moduli space of Higgs bundles. Hwang and Ramanan \cite{HR04} gave a different proof based on the study of Hecke curves on the moduli space. They proved that the Hitchin discriminant was isomorphic to the union of the images of all possible Hecke curves.

Later on a simplified proof was given in \cite{BGM13}, in which the study of the Hecke transformation and the minimal rational curves on the moduli space was substituted by the geometric characterization of the nilpotent cone bundle of a generic vector bundle. This lead to the proof that given a generic bundle $E$ whose image under the automorphism $E'$ is itself generic, there exists an isomorphism of Lie algebra bundles
$$\End_0(E)\cong \End_0(E')$$
Then, it is proven that such an automorphism exists if and only if $E'$ is obtained from $E$ by one of the previously described transformations. The argument was further generalized to the moduli space of symplectic bundles in \cite{BGM12}. In this paper, we will generalize this result to the parabolic scenario.

Coming back to the moduli of parabolic vector bundles, first, we develop four ``basic transformations'' that can be applied intrinsically to families of quasi-parabolic vector bundles. The first three types come from adapting the previously mentioned ones (pullback with respect to an automorphism of the curve, tensoring with a line bundle and dualization) to parabolic vector bundles, finding naturally induced filtrations at the parabolic points on the resulting vector bundles. Nevertheless, in the parabolic setup there is a fourth new type of transformation that can be defined using the additional information provided by the parabolic structure. We can use the steps of the filtration to perform a Hecke transformation on the underlying vector bundle at the parabolic points. What is more, the full parabolic structure at each parabolic point can be ``rotated'' in a certain way so that it induces a parabolic structure on the resulting bundle. The possible combinations of these four types of transformations
\begin{itemize}
\item Taking pullback with respect to an automorphism $\sigma:X\to X$ that fixes the set of parabolic points $D$ (but not necessarily fixes every point in $D$) $(E,E_\bullet)\mapsto \sigma^*(E,E_\bullet)$
\item Tensoring with a line bundle $(E,E_\bullet)\mapsto (E,E_\bullet)\otimes L$
\item Dualization $(E,E_\bullet)\mapsto (E,E_\bullet)^\vee$
\item Hecke transformations $(E,E_\bullet) \mapsto \SH_x(E,E_\bullet)$ with respect to the subspace $E_{x,2}\subset E|_x$ for some $x\in D$ 
\end{itemize}
form a group $\ST$ that we call group of basic transformations.

Instead of working with a fixed moduli space $\SM(r,\alpha,\xi)$ and compute its automorphisms, it will come more natural to study the possible isomorphisms between two moduli spaces $\SM(X,r,\alpha,\xi)$ and $\SM(X',r',\alpha', \xi')$, leading to what is usually called an Extended Torelli type theorem. We will prove that basic transformations are the only ones giving rise to isomorphisms between moduli spaces of parabolic vector bundles. More precisely, the main result in this article is the following Theorem (see Theorem \ref{theorem:ExtendedTorelli})

\begin{theorem}
\label{theorem:autoModuliIntro}
Let $(X,D)$ and $(X',D')$ be two smooth projective curves of genus $g\ge 6$ and $g'\ge 6$ respectively with set of marked points $D\subset X$ and $D'\subset X'$. Let $\xi$ and $\xi'$ be line bundles over $X$ and $X'$ respectively, and let $\alpha$ and $\alpha'$ be full flag generic systems of weights over $(X,D)$ and $(X',D')$ respectively. Let
$$\Phi: \SM(X,r,\alpha,\xi)\stackrel{\sim}{\longrightarrow} \SM(X',r',\alpha',\xi')$$
be an isomorphism. Then
\begin{enumerate}
\item $r=r'$
\item $(X,D)$ is isomorphic to $(X',D')$, i.e., there exists an isomorphism $\sigma:X\stackrel{\sim}{\to} X'$ sending $D$ to $D'$.
\item There exists a basic transformation $T$ such that for every $(E,E_\bullet)\in \SM(r,\alpha,\xi)$ $$\sigma^*\Phi(E,E_\bullet) \cong T(E,E_\bullet)$$
\end{enumerate}
\end{theorem}
Moreover, for $r=2$, the dual of a parabolic vector bundle can be rewritten in terms of a tensor product by a certain line bundle, so every isomorphism $\Psi$ comes from a basic transformation that does not involve dualization. 

Apart from acting on parabolic vector bundles, the group $\ST$ acts on line bundles $\xi$ and systems of weights $\alpha$ so that for every $T\in \ST$, if $(E,E_\bullet)$ has determinant $\xi$ and is stable for the weights $\alpha$, then $T(E,E_\bullet)$ has determinant $T(\xi)$ and is stable for the weights $T(\alpha)$. For $T$ to induce an isomorphism $T:\SM(r,\alpha,\xi) \stackrel{\sim}{\longrightarrow} \SM(r,\alpha,\xi')$ it is necessary and sufficient that
\begin{itemize}
\item $T(\xi)\cong \xi'$
\item $T(\alpha)$ is in the same stability chamber as $\alpha'$
\end{itemize}
This will allow us to compute the automorphism group $\Aut(\SM(r,\alpha,\xi))$ in Theorem \ref{theorem:autoModuli}.

In order to prove the theorem, we will generalize the approaches used in \cite{BGM12} and \cite{BGM13} to the particular features of the moduli space of parabolic vector bundles, although a deeper analysis on some invariant subspaces of the Hitchin map and the Hitchin discriminant will be necessary. We will prove that for a generic parabolic vector bundle $(E,E_\bullet)\in \SM(r,\alpha,\xi)$ if $\sigma^*\Phi(E,E_\bullet)=(E',E'_\bullet)$ then there exists an isomorphism of Lie algebra bundles
$$\ParEnd_0(E,E_\bullet)\cong \ParEnd_0(E',E'_\bullet)$$
Using some algebraic methods, we will prove that if such isomorphism exists then $(E',E'_\bullet)$ can be obtained from $(E,E_\bullet)$ through the application of a basic transformation $T\in \ST$. More precisely, for any $\Phi$ and for a generic $(E,E_\bullet)\in \SM(r,\alpha,\xi)$ there exists some $T\in \ST$ such that $\sigma^*\Phi(E,E_\bullet)\cong T(E,E_\bullet)$. We will then show that we can choose $T$ so that it does not vary with $(E,E_\bullet)$, i.e., for any $\Phi$ there exists some $T\in \ST$ such that the formula $\sigma^*\Phi(E,E_\bullet)\cong T(E,E_\bullet)$ holds for an open set of points $(E,E_\bullet)\in \SM(r,\alpha,\xi)$ in the moduli space. Finally we prove that the equality extends to the whole moduli space.

The structure of the paper is the following. In Section \ref{section:ModuliParabolic} we recall the notion of parabolic vector bundle, parabolic stability and some properties of the moduli space of stable parabolic vector bundles. The precise notions of generic and concentrated systems of weights are given and we prove some technical lemmas regarding the behavior of generic parabolic vector bundles.

Parabolic Hitchin pairs and the Hitchin map are analyzed in Section \ref{section:Hitchin}. In Section \ref{section:Torelli} we study the geometry of the fibers of the Hitchin map corresponding to singular spectral curves, usually called the Hitchin discriminant. We prove that the image of the Hitchin discriminant can be intrinsically described from the geometry of $\SM(r,\alpha,\xi)$ as an abstract variety. We use this description to prove a Torelli type theorem for the moduli space of parabolic vector bundles (Theorem \ref{theorem:Torelli}).

\begin{theorem}
If $(X,D)$ and $(X',D')$ are marked curves of genus at least 4 such that $\SM(X,r,\alpha,\xi)\cong \SM(X',r',\alpha',\xi')$, then $(X,D)\cong (X',D')$ and $r=r'$.
\end{theorem}

This theorem has already been proved by Balaji, del Ba\~no and Biswas \cite{TorelliParabolic} for $r=2$ and small parabolic weights, in the sense that parabolic stability is equivalent to stability of the underlying vector bundle. In contrast, our theorem only assumes that the parabolic weights are generic and it is valid for any rank.

Section \ref{section:Hecke} is devoted to describing the four kinds of ``basic transformations'' that can be applied intrinsically to families of quasi-parabolic vector bundles. The parabolic version of the Hecke transformation is described and we analyze the stability of the resulting bundles. A presentation for the group $\ST$ of basic transformations is explicitly described and its abstract structure is computed in Proposition \ref{prop:basicTransPresentation}.
$$\ST \cong \left ((\ZZ^{|D|}\times \Pic(X)) / \SG_D \right) \rtimes \left( \Aut(X,D) \times \ZZ/2\ZZ\right)$$
where $\SG_D<\ZZ^{|D|}\times \Pic(X)$ is a (normal) subgroup isomorphic to $(r\ZZ)^{|D|}$.

Then, in Section \ref{section:Algebra} we study the algebra of parabolic endomorphisms. Several classification and structure theorems are given. The main result of this section is the description of all the possible parabolic vector bundles which share the same Lie algebra bundle of traceless parabolic endomorphisms.

Theorem \ref{theorem:autoModuliIntro} is proved through Section \ref{section:Automorphisms}. As a corollary, in Theorem \ref{theorem:autoModuli} we describe the group of automorphisms of the moduli space $\SM(r,\alpha,\xi)$ as a subgroup of the group of basic transformations $\ST$ described in Section \ref{section:Hecke}, which varies depending on $\alpha$ and $\xi$. The dependence of the group on $\alpha$ and $\xi$ arises from some basic concerns coming from fixing the determinant $\xi$ (arithmetic obstructions involving the rank and degree of the bundles) and an analysis of the stability chamber of $\alpha$.

If we examine closely the results leading to the Extended Torelli (Theorem \ref{theorem:ExtendedTorelli}) and the computation of the automorphism group (Theorem \ref{theorem:autoModuli}) in Section \ref{section:Automorphisms}, we observe a certain common underlying behavior for all moduli spaces of parabolic vector bundles. Basic transformations in $\ST$ induce all possible isomorphisms between moduli spaces, even crossing stability walls. Restricting ourselves to parabolic vector bundles with a fixed determinant $\xi$ naturally imposes a condition on the possible applicable basic transformations, leading to a subgroup
$$\ST_\xi= \{T\in \ST | T(\xi)=\xi\}$$
of transformations which preserve the determinant. Nevertheless, in general this group does not coincide with the group of automorphisms of the moduli space $\SM(r,\alpha,\xi)$, as not all the transformations preserve $\alpha$-stability. Some of them induce a wall crossing. If $g\ge 3$, wall crossings are $3$-birational, in the sense that there are open subsets $\SU\subset\SM(r,\alpha,\xi)$ and $\SU'\subset \SM(r,\alpha',\xi)$ whose respective complements have codimension at least $3$ such that there is an isomorphism $\SU\cong \SU'$. Up to this identification, basic transformations $T\in \ST_\xi$ induce a birational transformation $T:\SM(r,\alpha,\xi)\dashrightarrow \SM(r,\alpha,\xi)$ which is an automorphism of some open subset whose complement has codimension at least $3$. We will call this kind of maps $3$-birational maps.

On the other hand, Boden and Yokogawa \cite{BY} proved that if $\alpha$ is full flag the moduli space $\SM(r,\alpha,\xi)$ is rational, so the birational geometry of $\SM(r,\alpha,\xi)$ is completely independent on the geometry of $(X,D)$ apart from the dimensional level. Then, it seems like the notion of $3$-birational maps (and in general $k$-birational maps) is more natural for the study of the moduli space of parabolic vector bundles than the analysis of the isomorphisms or general birational maps. In Section \ref{section:birational} we give a precise definition for $k$-birational maps and prove $3$-birational versions of the Torelli theorem (\ref{theorem:TorelliBirational}) and the Extended Torelli theorem (\ref{theorem:ExtendedTorelliBirational}). More particularly, for genus at least $4$, we obtain that
\begin{theorem}
If $\Phi: \SM(X,r,\alpha,\xi) \dashrightarrow \SM(X',r',\alpha',\xi')$ is a $3$-birational map then $r=r'$ and $(X,D)\cong (X',D')$.
\end{theorem}

\begin{theorem}
If $\Phi: \SM(X,r,\alpha,\xi) \dashrightarrow \SM(X',r',\alpha',\xi')$ is a $3$-birational map then $r=r'$ and there is an isomorphism $\sigma:(X,D)\longrightarrow (X',D')$ and a basic transformation $T\in \ST$ such that
\begin{enumerate}
\item $T(\xi)=\sigma^*\xi'$
\item For every $(E,E_\bullet)$ for which $\Phi$ is defined, $\sigma^*\Phi(E,E_\bullet)\cong T(E,E_\bullet)$
\end{enumerate}
\end{theorem}
Then we conclude (Corollary \ref{cor:AutoBir}) that for $r>2$ the $3$-birational automorphisms of $\SM(r,\alpha,\xi)$ are
$$\Aut_{3-\Bir}(\SM(r,\alpha,\xi))\cong \ST_\xi <\ST$$
For $r=2$, the correspondence between 3-birational maps and basic transformations is not 1 on 1, as the dualization map can be described alternatively in terms of the tensorization. Instead, we obtain a correspondence with the subgroup $\ST^+_\xi$ of basic transformations that do not involve the dual and preserve the determinant $\xi$,
$$\Aut_{3-\Bir}(\SM(2,\alpha,\xi))\cong \ST_\xi^+ <\ST$$

Finally, we aim to describe explicitly the dependency of $\Aut(\SM(r,\alpha,\xi))$ and the isomorphism class of $\SM(r,\alpha,\xi)$ on the stability parameters $\alpha$. In section \ref{section:concentrated} we analyze this problem for the concentrated chamber, in which $\alpha$-stability is (roughly) equivalent to stability of the underlying vector bundle. We prove that in this chamber the Hecke transformation does never induce an automorphism of $\SM(r,\alpha,\xi)$, even when combined with other basic transformations. The automorphism group is then explicitly described.

In section \ref{section:chamberAnalysis} we analyze the stability space $\Delta$ and the partition in stability chambers. Given two systems of weights $\alpha$ and $\beta$ we consider the problem of determining whether all $\alpha$-stable parabolic vector bundles are also $\beta$-stable or, conversely, there exists some $\alpha$-stable parabolic vector bundle which is not $\beta$-stable. For the latter case to happen there must exist an $\alpha$-stable parabolic vector bundle $(E,E_\bullet)$ admitting a $\beta$-destabilizing subbundle $(F,F_\bullet)\subset (E,E_\bullet)$, in the sense that for all $(F',F'_\bullet)\subset (E,E_\bullet)$
$$\frac{\pdeg_{\alpha}(F',F'_\bullet)}{\rk(F')} < \frac{\pdeg_{\alpha}(E,E_\bullet)}{\rk(E)}$$
but
$$\frac{\pdeg_{\beta}(F,F_\bullet)}{\rk(F)}\ge \frac{\pdeg_{\beta}(E,E_\bullet)}{\rk(E)}$$
Therefore, the map
\begin{eqnarray*}
\xymatrixcolsep{3pc}
\xymatrixrowsep{0.05pc}
\xymatrix{
\Delta \ar[r] & \mathbb{R} \\
{\alpha'} \ar@{|->}[r] & \rk(F) \pdeg_{\alpha'}(E,E_\bullet) - \rk(E) \pdeg_{\alpha'}(F,F_\bullet)
}
\end{eqnarray*}
is negative for $\alpha$ and non-negative for $\beta$. Thus, if we consider the weights $\alpha_t=t\alpha+(1-t)\beta$, there must exist some $t\in (0,1]$ such that
\begin{equation}
\label{eq:introNumericalBarrier}
\rk(F) \pdeg_{\alpha_t}(E,E_\bullet) - \rk(E) \pdeg_{\alpha_t}(F,F_\bullet)=0
\end{equation}
This equation defines a hyperplane in $\Delta$, depending only on the numerical data for $(E,E_\bullet)$ and $(F,F_\bullet)$, namely, the degrees $\deg(E)$, $\deg(F)$, the ranks $\rk(E)$, $\rk(F)$ and the parabolic type of $(F,F_\bullet)$, say $\overline{n}_F$. We call a ``numerical wall'' any hyperplane on $\Delta$ obtained from an equation of the form
\eqref{eq:introNumericalBarrier} when we range over all possible choices for the integers $\deg(F)$, $\rk(F)$ and $\overline{n}_F$ (the rank and degree of $E$ are fixed in our moduli space). We say that a ``numerical wall'' is ``geometrical'' if there actually exists some parabolic vector bundle $(E,E_\bullet)$ and a subbundle $(F,F_\bullet) \subset (E,E_\bullet)$ with the correct invariants $\deg(F)$, $\rk(F)$ and $\overline{n}_F$. Numerical and geometrical stability chambers are defined as the regions of $\Delta$ separated by the numerical or geometrical walls respectively. We will also call the geometrical chambers simply stability chambers. We just proved that any two different stability chambers are separated by a numerical wall, but it is not clear that any numerical wall can be realized into a geometrical one, so a stability chamber can contain several numerical chambers.

We prove that there is a finite number of different chambers in $\Delta$ and we construct an invariant $\overline{M}(r,\alpha,d)$ classifying the ``numerical'' stability chambers in $\Delta$. Theorem \ref{theorem:stabilityChamber} proves that if the genus is big enough then the invariant $\overline{M}(r,\alpha,d)$ is in correspondence with (geometrical) stability chambers in $\Delta$ and use it to obtain a computable version of the Extended Torelli theorem \ref{theorem:ExtendedTorelli}.

The last section of this paper (section \ref{section:examples}) presents some examples, showing that the previous results are sharp in the following sense. As we proved that Hecke does not take part in any automorphism of $\SM(r,\alpha,\xi)$ when $\alpha$ is concentrated, it is natural to  wonder if for any of the presented basic transformations $T$ (pullback, tensorization, dualization and Hecke) there exist a (general enough) marked curve $(X,D)$ and a generic system of weights such that $T$ induces an automorphism of $\SM(r,\alpha,\xi)$. We provide an example of rank $2$, $2$ marked points and arbitrary genus for which the composition of Hecke with taking the pullback by some $\sigma:X\to X$ induce a nontrivial automorphism of the moduli. Moreover, dualization and tensoring induce nontrivial automorphisms, up to the usual constraint $T(\xi)=\xi$.

On the other hand, we can find a system of weights $\alpha$ of rank $r>2$ such that the combination of Hecke and dualization induces a nontrivial involution of $\SM(r,\alpha,\xi)$ which does not come from an involution of the curve $X$.

\noindent\textbf{Acknowledgments.} 
This research was funded by MICINN (grants MTM2016-79400-P, PID2019-108936GB-C21 and ``Severo Ochoa Programme for Centres of Excellence in R\&D'' projects SEV-2015-0554 and CEX2019-000904-S) and the 7th European Union Framework Programme (Marie Curie IRSES grant 612534 project MODULI). The first author was also supported by a predoctoral grant from Fundaci\'on La Caixa -- Severo Ochoa International Ph.D. Program and a postdoctoral position from ICMAT Severo Ochoa project. We would like to thank Indranil Biswas for his helpful comments on this work and Suratno Basu for useful discussions regarding the proof of Proposition \ref{prop:recoverDualVariety}. We would also want to thank the anonymous reviewer for the careful reading and suggestions.

\section{Moduli space of parabolic vector bundles}
\label{section:ModuliParabolic}

Let $X$ be an irreducible smooth complex projective curve. Let $D=\{x_1,\ldots,x_n\}$ be a set of $n\ge 1$ different points of $X$ and let us denote $U=X\backslash D$.

A parabolic vector bundle on $(X,D)$ is a holomorphic vector bundle $E$ of rank $r$ endowed with a weighted flag on the fiber $E|_x$ over each parabolic point $x\in D$ called parabolic structure
$$E|_x=E_{x,1} \supsetneq E_{x,2} \supsetneq \cdots \supsetneq E_{x,l_x} \supsetneq E_{x,l_{x}+1}=0$$
$$0\le \alpha_1(x) <\alpha_2(x) <\ldots <\alpha_{l_x}(x)<1$$
We say that $\alpha_i(x)$ is the weight associated to $E_{x,i}$. We will denote by $\alpha=\left\{(\alpha_1(x),\ldots,\alpha_{l_x}(x))\right\}_{x\in D}$ the system of real weights corresponding to a fixed parabolic structure. A system of weights is called full flag if $l_x=r$ for all parabolic points $x\in D$. We will use the simplified notation $(E,E_\bullet)=(E,\{E_{x,i}\})$ to denote a parabolic vector bundle.

Equivalently \cite{SimpsonNonCompact}, we can describe the parabolic structure as a collection of decreasing left continuous filtrations of sheaves on $X$, one filtration for each parabolic point. More precisely, for each $x\in D$, let $E_{\alpha}^x\subset E$ be a subsheaf on $X$ indexed by a real $\alpha\ge 0$ such that
\begin{enumerate}
\item For every $\alpha\ge \beta$, $E_\alpha^x \subseteq E_\beta^x$
\item For every $\alpha>0$, there exists $\varepsilon>0$ such that $E_{\alpha-\varepsilon}^x=E_\alpha^x$
\item For every $\alpha$, $E_{\alpha+1}^x = E_\alpha^x(-x)$
\item $E_0^x=E$
\end{enumerate}
If $E^x_\alpha$ is a left continuous filtration, let $\alpha_i(x)$ be the $i$-th weight $\alpha \ge 0$ where the filtration jumps, i.e., such that for every $\varepsilon>0$, $E^x_\alpha\ne E^x_{\alpha+\varepsilon}$. Then we can define the parabolic structure $\{E_{x,i}\}$ at the fiber $E|_x$ as the one having parabolic weights $\{\alpha_i(x)\}$ such that
$$E|_x/E_{x,i} \otimes \SO_x=E/E^x_{\alpha_i(x)}$$
Reciprocally, if $\{E_{x,i}\}$ is a filtration of the fiber $E|_x$, endowed with weights $\alpha_i(x)$, define the subsheaves $E^x_{\alpha_i(x)}\subseteq E$ as the ones fitting in the short exact sequence
$$0\longrightarrow E_{x,\alpha_i(x)} \longrightarrow E \longrightarrow E/E_{x,i} \otimes \SO_x\longrightarrow 0$$
Then take $E^x_\alpha=E$ for $\alpha_{l_x}(x)-1\le \alpha\le \alpha_1(x)$ and $E^x_\alpha=E^x_{\alpha_i(x)}$ for $\alpha_{i-1}(x)< \alpha\le \alpha_i(x)$. Then define $E^x_\alpha$ for $\alpha>\alpha_{l_x}(x)$ by the property
$$E^x_{\alpha+1}=E^x_{\alpha}(-x)$$
The resulting filtration $E^x_\alpha$ is a parabolic structure at the point $x$. The relations between these two formalisms will be explored further in Section \ref{section:Hecke}.
Given a parabolic vector bundle $(E,E_\bullet)$, we define its parabolic degree as
$$\pdeg(E,E_\bullet)=\deg(E)+\sum_{x\in D}\sum_{i=1}^{l_x} \alpha_i(x) (\dim(E_{x,i})-\dim(E_{x,i+1}))$$
As we will be working with stability conditions for different systems of weights $\alpha$, it will be useful to denote
$$\owt_\alpha(E,E_\bullet)=\sum_{x\in D} \sum_{i=1}^{l_x} \alpha_i(x) (\dim(E_{x,i})-\dim(E_{x,i+1}))$$
Similarly, let
$$\pdeg_\alpha(E,E_\bullet)=\deg(E)+\owt_\alpha(E,E_\bullet)$$
We say that a parabolic vector bundle $(E,E_\bullet)$ is of type $\overline{n}=(n_i(x))$ if 
$$n_i(x)=\dim(E_{x,i})-\dim(E_{x,i+1})$$
for every $i=1,\ldots,l_x$ and every $x\in D$. Then if $(E,E_\bullet)$ is of type $\overline{n}$, we can write
$$\owt_\alpha(E,E_\bullet)=\sum_{x\in D} \sum_{i=1}^r \alpha_i(x) n_i (x)$$
Notice that the right hand side does only depend on $\overline{n}$ and $\alpha$. We will denote it by $\owt_\alpha(\overline{n})$.

Let $E'\subseteq E$ be a proper subbundle of a parabolic vector bundle $(E,E_\bullet)$. The parabolic structure on $E$ induces a parabolic structure on $E'$ as follows. For each parabolic point $x\in D$, we obtain a filtration by considering the set of subspaces $\{E'_{x,i}\}=\{E'_x\cap E_{x,j}\}$ for $j=1,\ldots, l_x$. The weight $\alpha'_i(x)$ of $E'_{x,i}$ is taken as
$$\alpha'_i(x)=\max_{j}\{\alpha_j(x):E'|_x\cap E_{x,j}=E'_{x,i}\}$$
Then $\alpha'$ is a subset of the weights in $\alpha$. While this would be the ``canonical'' form of the parabolic structure of $E'$, it will be useful to present it in terms of the original system of weights $\alpha$. In particular, if $E'\subsetneq E$, let us take instead $\wt{E'_{x,i}}=E'_x\cap E_{x,i}$ for $i=1,\ldots,l_x$. Notice that while these spaces $\wt{E'_{x,i}}$ form a filtration of $E'_x$, they do not constitute a parabolic structure in the canonical sense, as there exists at least one $j$ such that $\wt{E'_{x,j}}=\wt{E'_{x,j+1}}$. Nevertheless, we can use this other filtration to compute the parabolic degree of $(E',E'_\bullet)$. In particular, let us define $\overline{n'}=(n'_i(x))$ as follows
$$n'_i(x)=\dim(\wt{E'_{x,i}})-\dim(\wt{E'_{x,i+1}}) = \dim(E'_x\cap E_{x,i})-\dim(E'_x\cap E_{x,i+1})$$
Then $\owt_{\alpha'}(E',E'_\bullet)=\owt_\alpha(\overline{n'})$. If $(E,E_\bullet)$ is full flag, then $0\le n_i'(x)\le 1$ for every $i=1,\ldots,r$ and every $x\in D$. We say that a subbundle $E'\subsetneq E$ of a parabolic vector bundle is of type $\overline{n}'$ if the induced filtration $\wt{E'_\bullet}$ is of type $\overline{n}'$.

Given parabolic vector bundles $(E,E_\bullet)$ and $(F,F_\bullet)$ with systems of weights $\alpha$ and $\beta$ respectively, a morphism $\varphi:E\longrightarrow F$ is called parabolic (respectively strongly parabolic) if it preserves the parabolic structure, i.e., if for every $x\in D$ and every $i=1,\ldots,l_{E,x}$ and $j=1,\ldots,l_{F,x}$ such that $\alpha_i(x)>\beta_j(x)$ (respectively $\alpha_i(x)\ge \beta_j(x)$)
$$\varphi(E_{x,i})\subseteq F_{x,j+1}$$
We denote by $\ParHom((E,E_\bullet),(F,F_\bullet))$ the sheaf of local parabolic morphisms from $(E,E_\bullet)$ to $(F,F_\bullet)$ and write $\SParHom((E,E_\bullet),(F,F_\bullet))$ for the subsheaf of strongly parabolic morphisms.

In particular, if $(E,E_\bullet)$ is a parabolic vector bundle, an endomorphism $\varphi:E\to E$ is parabolic if for every $x\in D$ and every $i=1,\ldots,l_x$
$$\varphi(E_{x,i})\subseteq E_{x,i}$$
We denote by $\ParEnd(E,E_\bullet)$ the sheaf of local parabolic endomorphisms of $(E,E_\bullet)$. Similarly, an endomorphism is strongly parabolic if for every $x\in D$ and every $i=1,\ldots,r$
$$\varphi(E_{x,i})\subseteq E_{x,i+1}$$
We denote by $\SParEnd(E,E_\bullet)$ the sheaf of strongly parabolic endomorphisms of $(E,E_\bullet)$.

The sheaves $\ParHom$ and $\SParHom$ are subsheaves of the sheaf of morphisms $\Hom$ and they all coincide away from the parabolic points $D\subset X$. Following the notation in \cite{BB05}, let $T_{(E,E_\bullet),(F,F_\bullet)}$ be the torsion sheaf supported in $D$ that fits in the following short exact sequence
$$0\longrightarrow \ParHom((E,E_\bullet),(F,F_\bullet)) \longrightarrow \Hom(E,F) \longrightarrow T_{(E,E_\bullet),(F,F_\bullet)} \longrightarrow 0$$
Then define $t_{(E,E_\bullet),(F,F_\bullet)}$ as the rational number such that
$$\rk(E)\rk(F)t_{(E,E_\bullet),(F,F_\bullet)}=\dim(T_{(E,E_\bullet),(F,F_\bullet)})$$

If $\alpha=\beta$, then $t_{E,F}$ only depends on the types $\overline{n}'$ and $\overline{n}''$ of $(E,E_\bullet)$ and $(F,F_\bullet)$ respectively. More explicitly, 
$$r'r''t_{\overline{n}'',\overline{n}'}=\sum_{x\in D} \sum_{i>j}n''_i(x)n'_j(x)$$
Observe that if we take $\overline{n}'=\overline{n}''$, then $(r')^2t_{\overline{n}',\overline{n}'}$ is just the dimension of the flag variety of type $\overline{n}'$.
If $L$ is a line bundle over $X$ and $(E,E_\bullet)$ is a parabolic vector bundle over $(X,D)$, we define the parabolic vector bundle $(E,E_\bullet)\otimes L$ as the one having underlying vector bundle $E\otimes L$ and whose filtrations are given by
$$(E\otimes L)_{x,i}=E_{x,i}\otimes L$$
This is a particular simple case of the general concept of tensor product of parabolic bundles. The general definition can be found in \cite{Biswas03}.

\begin{definition}
We say that a quasi-parabolic vector bundle is $\alpha$-(semi)stable if for every proper subbundle $E'\subsetneq E$ with the induced parabolic structure
\begin{equation}
\label{eq:a-stability}
\frac{\pdeg_\alpha(E',E'_\bullet)}{\rk(E')} < \frac{\pdeg_\alpha(E,E_\bullet)}{\rk(E)} \quad \text{(respectively }\le\text{ )}
\end{equation}
We say that $(E,E_\bullet)$ is $\alpha$-unstable if it is not $\alpha$-semistable.
\end{definition}

Let $\xi$ be a line bundle over $X$ and let $\alpha$ be a system of weights of type $\overline{n}$. Let $\SM(X,r,\alpha,\xi)$, or just $\SM(r,\alpha,\xi)$, be the moduli space of semi-stable parabolic vector bundles $(E,E_\bullet)$ on $(X,D)$ of rank $r$ with system of weights $\alpha$ and $\det(E)\cong \xi$. It is a complex projective scheme of dimension 
$$\dim(\SM(r,\alpha,\xi))=(r^2-1)(g-1)+r^2t_{\overline{n},\overline{n}}$$
In particular, observe that if $\alpha$ is full flag, i.e., if $\overline{n}=(1,\ldots,1)$, then
$$\dim(\SM(X,r,\alpha,\xi))=(r^2-1)(g-1)+\frac{n(r^2-r)}{2}$$

Similarly, let $\SM(X,r,\alpha,d)$, or just $\SM(r,\alpha,d)$ be the moduli of semistable parabolic vector bundles $(E,E_\bullet)$ on $(X,D)$ of rank $r$ with system of weights $\alpha$ and $\deg(E)=d$. It has dimension
$$\dim(\SM(r,\alpha,d))=r^2(g-1)+1+r^2t_{\overline{n},\overline{n}}$$
On the other hand, given a subbundle $E'\subsetneq E$, let us denote
$$s(E',E)=\rk(E')\deg(E)-\rk(E)\deg(E')$$
Reordering the inequality \eqref{eq:a-stability}, we obtain that $(E,E_\bullet)$ is $\alpha$-(semi)stable if and only if for every subbundle $E'\subsetneq E$ yields
\begin{multline*}
s(E',E)=\rk(E')\deg(E)-\rk(E)\deg(E')\\
> \rk(E) \owt_{\alpha}(E',E'_\bullet) -\rk(E')\owt_{\alpha}(E,E_\bullet) \quad \text{(resp. }\ge\text{ )}
\end{multline*}
Moreover, if we give $E'\subsetneq E$ the induced parabolic structure from $(E,E_\bullet)$, there exists a unique parabolic vector bundle $(E'',E''_\bullet)$ fitting in the short exact sequence
\begin{equation}
\label{eq:parabolicQuotient}
0\longrightarrow (E',E'_\bullet) \longrightarrow (E,E_\bullet) \longrightarrow (E'',E''_\bullet)\longrightarrow 0
\end{equation}
in the sense that for each $\alpha\in \RR$, the corresponding $\alpha$ step in each sheaf filtration form a short exact sequence and for each $\alpha>\beta$ the following diagram commutes
\begin{equation*}
\xymatrixrowsep{1.5pc}
\xymatrix{
0 \ar[r] & (E')_\alpha^x \ar[r] \ar@{^(->}[d] & E_\alpha^x \ar[r] \ar@{^(->}[d] & (E'')_\alpha^x \ar[r] \ar@{^(->}[d] & 0\\
0 \ar[r] & (E')_\beta^x \ar[r] & E_\beta^x \ar[r] & (E'')_\beta^x \ar[r] & 0
}
\end{equation*}
In particular, we have
\begin{eqnarray*}
\deg(E)=\deg(E')+\deg(E'')\\
\rk(E)=\rk(E')+\rk(E'')\\
\owt_\alpha(E,E_\bullet)=\owt_\alpha(E',E'_\bullet)+\owt_\alpha(E'',E''_\bullet)
\end{eqnarray*}
Therefore, $(E,E_\bullet)$ is $\alpha$-(semi)stable if and only if
\begin{equation}
\label{eq:a-stability2}
s(E',E)>\rk(E'')\owt_\alpha(E',E'_\bullet)-\rk(E')\owt_\alpha(E'',E''_\bullet) \quad \text{(resp. } \ge \text{ )}
\end{equation}

Then if we take $(E'',E''_\bullet)$ fitting in the short exact sequence \eqref{eq:parabolicQuotient} as before, it is of type $\overline{n}''=(n''_i(x))$, where $n''_i(x)=n_i(x)-n'_i(x)$.

Rewriting the stability condition \eqref{eq:a-stability2} in terms of $\overline{n}'$ and $\overline{n}''$, we obtain that $(E,E_\bullet)$ is $\alpha$-(semi)table if and only if for every $\overline{n}'$ and for every subbundle $E'\subseteq E$ of type $\overline{n}'$.
$$s(E',E)>\rk(E'')\owt_\alpha(\overline{n}')-\rk(E')\owt_\alpha(\overline{n}'') \quad \text{(resp. }\ge\text{ )}$$
Observe that, as $\rk(E')=\sum_{i=1}^r n'_i(x)$ for any $x\in D$, then the right hand side does only depend on $\alpha$ and $\overline{n}'$. Let us denote
$$s_{\min}(\alpha,\overline{n}')=r''\owt_\alpha(\overline{n}')-r'\owt_\alpha(\overline{n}'')$$
where $r'=\sum_{i=1}^r n'_i(x)$ for any $x\in D$ and $r''=r-r'=\sum_{i=1}^r n''_i(x)$.

\begin{lemma}
\label{lemma:degreeBound}
Let $l>0$ be an integer. If $g\ge 1+\frac{l}{r-1}$ then for any system of weights $\alpha$ and any admissible $\overline{n}'$,
$$s_{\min}(\alpha,\overline{n}')\le r'r''((g-1)+t_{\overline{n}',\overline{n}''})-l$$
In particular, if $g\ge 3$ or $g=2$ and $r\ge 3$ 
$$s_{\min}(\alpha,\overline{n}')\le r'r''((g-1)+t_{\overline{n}',\overline{n}''})-2$$
\end{lemma}

\begin{proof}
By \cite[Lemma 2.5.2]{BB05} we have
$$s_{\min}(\alpha,\overline{n}')-r'r''t_{\overline{n}',\overline{n}''}=r''\owt_{\alpha}(\overline{n}')-r'\owt_{\alpha}(\overline{n}'') - r'r''t_{\overline{n}',\overline{n}''}\le 0$$
Moreover, for any $1\le r'<r$
$$r'r''(g-1)-l\ge (r-1)(g-1)-l \ge 0$$
so
$$s_{\min}(\alpha,\overline{n}')-r'r''t_{\overline{n}',\overline{n}''} \le 0 \le r'r''(g-1)-l$$
\end{proof}

Moving on with the stability analysis, let
$$s(\overline{n}',E)=\min_{\substack{E'\subsetneq E\\ E' \text{ of type }\overline{n}'}} s(E',E)$$
Then $(E,E_\bullet)$ is $\alpha$-(semi)stable if and only if, for all admissible $\overline{n}'$
$$s(\overline{n}',E)> s_{\min}(\alpha,\overline{n}') \quad \text{(resp. }\ge \text{ )}$$
Let us denote by
$$\SM_{\overline{n}',s}(r,\alpha,d)=\{(E,E_\bullet)\in \SM(r,\alpha,d) | s(\overline{n}',E)=s\}$$

\begin{lemma}
\label{lemma:unstableCodimension}
Let $l>0$ be an integer. Let $X$ be a curve of genus $g\ge 1+\frac{l-1}{r-1}$ and let $D\subset X$ be a set of points in $X$. Let $\alpha$ and $\beta$ be full flag systems of weights of rank $r$ over $(X,D)$. Then the set of parabolic vector bundles $(E,E_\bullet)\in \SM(r,\alpha,d)$ that are $\beta$-unstable has codimension at least $l$ in $\SM(r,\alpha,d)$. In particular, for $g\ge 2$, it has codimension at least $2$ in $\SM(r,\alpha,d)$.
\end{lemma}

\begin{proof}
An $\alpha$-stable quasi-parabolic vector bundle $(E,E_\bullet)$ is  $\beta$-unstable if and only if for some admissible $\overline{n}'$ we have
$$s(\overline{n}',E)<s_{\min}(\beta,\overline{n}')$$
On the other hand, as $(E,E_\bullet)$ is $\alpha$-semistable, then
$$s(\overline{n}',E)\ge s_{\min}(\alpha,\overline{n}')$$
Therefore, $(E,E_\bullet)$ is $\alpha$-stable but $\beta$-unstable if and only if
$$(E,E_\bullet)\in \bigcup_{\overline{n}'}\coprod_{s_{\min}(\alpha,\overline{n}')\le s < s_{\min}(\beta,\overline{n}')} \SM_{\overline{n}',s}(r,\alpha,d)$$
As this is a finite union of subschemes, it is enough to prove that the complement of each component has codimension at least $l$. By Lemma \ref{lemma:degreeBound}, for every $\overline{n}'$ and every $s< s_{\min}(\beta,\overline{n}')$ we have
$$s< s_{\min}(\beta,\overline{n}')\le r'r''((g-1)+t_{\overline{n}',\overline{n}''})-(l-1)\le r'r''((g-1)+t_{\overline{n}',\overline{n}''})$$
Therefore, we can apply \cite[Theorem 1.4.1]{BB05} and we know that either $\SM_{\overline{n}',s}(r,\alpha,d)$ is empty or it has codimension
$$\delta_{\overline{n}',s}= r'r''((g-1)+t_{\overline{n}',\overline{n}''})-s\ge r'r''((g-1)+t_{\overline{n}',\overline{n}''})-s_{\min}(\beta,\overline{n}')+1$$
Applying again Lemma \ref{lemma:degreeBound} we obtain that for $g\ge 1+\frac{l-1}{r-1}$ we have $\delta_{\overline{n}',s}\ge l$.
\end{proof}

\begin{corollary}
\label{cor:stabCodim2}
Under the same hypothesis as the previous lemma, if $g\ge 1+\frac{l-1}{r-1}$ and $\xi$ is any line bundle over $X$ then the set of parabolic vector bundles $(E,E_\bullet)\in \SM(r,\alpha,\xi)$ that are $\beta$-unstable has codimension at least $l$ in $\SM(r,\alpha,\xi)$.
\end{corollary}

\begin{proof}
Let $S^d\subsetneq \SM(r,\alpha,d)$ be  the subset of parabolic vector bundles $(E,E_\bullet)$ that are $\alpha$-stable but $\beta$-unstable. For each line bundle $\xi$ of degree $d$, let
$$S^\xi=S^d\cap \SM(r,\alpha,\xi)$$
Let $\xi,\xi'\in \Pic^d(X)$. Then, there exists a line bundle $L\in J(X)$ such that $L^r=\xi'\otimes \xi^{-1}$. As tensoring with a line bundle preserves stability, it is clear that $(E,E_\bullet)\in S^\xi$ if and only if $(E,E_\bullet)\otimes L\in S^{\xi'}$. Therefore, $S^\xi\cong S^{\xi'}$ for every $\xi$ and $\xi'$. Similarly, for every $\xi$ and $\xi'$, $\SM(r,\alpha,\xi)$ is isomorphic to $\SM(r,\alpha,\xi')$. Therefore, we conclude that the codimension of $S^\xi$ in $\SM(r,\alpha,\xi)$ is the same as the codimension of $S^d$ in $\SM(r,\alpha,d)$, which is at least $l$ by the previous Lemma.
\end{proof}

In particular, applying the previous Corollary to $\beta=0$ yields

\begin{corollary}
\label{cor:underlyingStableGeneric}
Let $g\ge 1+\frac{l-1}{r-1}$. If $\xi$ is any line bundle over $X$, then the set of parabolic vector bundles $(E,E_\bullet)\in \SM(r,\alpha,\xi)$ whose underlying vector bundle $E$ is unstable has codimension at least $l$ in $\SM(r,\alpha,\xi)$. In particular, for $g\ge 2$ it has codimension at least $2$.
\end{corollary}

\begin{corollary}
\label{cor:forgetfulVBdominant}
Let $g\ge 1+\frac{l-1}{r-1}$. Let $\xi$ be any line bundle over $X$ and $\alpha$ any full flag system of weights. Let $\SM^{\op{ss-vb}}(r,\alpha,\xi)\subset \SM(r,\alpha,\xi)$ be the open nonempty subset parameterizing parabolic vector bundles $(E,E_\bullet)$ whose underlying vector bundle $E$ is semistable. Then the forgetful map
$$p:\SM^{\op{ss-vb}}(r,\alpha,\xi) \longrightarrow \SM(r,\xi)$$
is dominant.
\end{corollary}

\begin{proof}
By the previous Corollary, $\SM^{\op{ss-vb}}(r,\alpha,\xi)$ is a open subset of $\SM(r,\alpha,\xi)$, so $\dim(\SM^{\op{ss-vb}}(r,\alpha,\xi))=\dim(\SM(r,\alpha,\xi))=\dim(\SM(r,\xi))+n\frac{r^2-r}{2}$. Let $S$ be the image of $p$. For every $E\in S$, the fiber $p^{-1}(E)$ is contained in the space of flags over $E|_x$ for every $x\in D$, so $\dim(p^{-1}(E))\le n\frac{r^2-r}{2}$ for every $E\in \overline{S}$. Therefore
$$\dim(\SM(r,\xi))+n\frac{r^2-r}{2}=\dim(\SM^{\op{ss-vb}}(r,\alpha,\xi))=\dim(p ^{-1}(\overline{S}))\le \dim(\overline{S})+n\frac{r^2-r}{2}$$
So $\dim(\overline{S})=\dim(\SM(r,\xi))$. As the latter, is irreducible, $\overline{S}=\SM(r,\xi)$.
\end{proof}

Now, we recall the notions of ``generic'' and ``concentrated'' systems of weights as described in \cite{AG18TorelliDH}. Given a set $S$ and an integer $k$, let $\SP^k(S)$ denote the set of subsets of size $k$ of $S$. For each $0<r'<r$, each map $I:D\to \SP^{r'}(\{1,\ldots,r\})$ and each integer $-nr^2 \le m \le nr^2$, let
$$A_{I,m}=\left\{\alpha : r'\sum_{x\in D} \sum_{i=1}^r \alpha_i(x) - r \sum_{x\in D} \sum_{i\in I(x)} \alpha_i(x) =m \right\}$$
If we denote by $\SI_{r'}$ the set of possible maps $I:D\to \SP^{r'}(\{1,\ldots,r\})$, let
$$A=\bigcup_{r'=1}^{r-1} \bigcup_{I\in \SI_{r'}} \bigcup_{m=-nr^2}^{nr^2} A_{I,m}$$
We say that a full flag system of weights $\alpha$ over $(X,D)$ is generic if $\alpha\not\in A$. By \cite[Corollary 2.3]{AG18TorelliDH}, then there are no strictly semistable parabolic vector bundles and $\SM(r,\alpha,\xi)$ is a smooth rational variety \cite[Theorem 6.1]{BY}.

A full flag system of weights $\alpha=\{(\alpha_1(x),\ldots,\alpha_r(x))\}_{x\in D}$ is said to be concentrated if $\alpha_r(x)-\alpha_1(x)<\frac{4}{nr^2}$ for all $x\in D$. By \cite[Proposition 2.6]{AG18TorelliDH}, if $\deg(E)$ and $\rk(E)$ are coprime and $\alpha$ is a full flag concentrated  system of weights then for every parabolic vector bundle $(E,E_\bullet)$ over $(X,D)$ the following are equivalent
\begin{enumerate}
\item $E$ is semistable as a vector bundle
\item $E$ is stable as a vector bundle
\item $(E,E_\bullet)$ is $\alpha$-semistable as a parabolic vector bundle
\item $(E,E_\bullet)$ is $\alpha$-stable as a parabolic vector bundle
\end{enumerate}

We introduce some extension results that will be needed later on.

\begin{lemma}
\label{lemma:ParEndNoSections}
Let $k>0$. Suppose that $g\ge 2k+2$. Let $(E,E_\bullet)$ be a generic stable parabolic vector bundle. Then for any effective divisor $F$ of degree $k$ and any sheaf $\SF \hookrightarrow \End_0(E)(F)$ such that the quotient is supported on a finite set of points we have
$$H^0(\SF)=0$$
\end{lemma}

\begin{proof}
By \cite[Lemma 2.2]{BGM13}, there exists an open subset $\SU\subset \SM(r,\xi)$ such that for every $E\in \SU$ we have $H^0(\End_0(E)(F))=0$. By Corollary \ref{cor:underlyingStableGeneric}, for $g\ge 2k+2> 2$, parabolic vector bundles $(E,E_\bullet)$ whose underlying vector bundle $E$ is semistable form a nonempty open subset of the moduli space $\SM^{\op{ss-vb}}(r,\alpha,\xi) \subset \SM(r,\alpha,\xi)$. Consider the preimage of $\SU$ by the forgetful morphism 
$$p:\SM^{\op{ss-vb}}(r,\alpha,\xi) \longrightarrow \SM(r,\xi)$$
Therefore, for every $(E,E_\bullet)\in p^{-1}(\SU)$, $H^0(\End_0(E)(F))=0$. Let $\SF\subset \End_0(E)(F)$ be any subsheaf whose quotient is supported on a finite set of points. Let $s\in H^0(\SF)$. As $\SF\hookrightarrow \End_0(E)(F)$, taking the image, $s$ induces a section $\overline{s}\in H^0(\End_0(E)(F))$, so we have $\overline{s}=0$. Let $V=X\backslash \op{supp}(\End_0(E)(F)/\SF)$. Then $s|_V=0$. As $\End_0(E)(F)$ is torsion free, $\SF$ is itself torsion free and then $s=0$. Finally, by Corollary \ref{cor:forgetfulVBdominant}, $p$ is dominant, so $p^{-1}(\SU)$ is an open nonempty set of $\SM(r,\alpha,\xi)$.
\end{proof}

\begin{lemma}
\label{lemma:uniqueExtensionCodim2}
Let $M$ be a smooth complex scheme and let $U$ be an open subset whose complement has codimension at least $2$. Let $(\SE,\SE_\bullet)$ be a family of parabolic vector bundles  over $(X,D)$ parameterized by $U$. If $(\SE,\SE_\bullet)$ admits an extension to $M\times X$, then the extension is unique.
\end{lemma}

\begin{proof}
Let $(\SF^1,\SF^1_\bullet)$ and $(\SF^2,\SF^2_\bullet)$ be families of parabolic vector bundles over $(X,D)$ parameterized by $M$ extending $(\SE,\SE_\bullet)$. Then $\SF^i$ are vector bundles over $M\times X$ and $\SF_{x,j}^i$ are vector bundles over $M\times \{x\}$ extending $\SE$ and $\SE_{x,j}$ respectively.

If the codimension of $M\backslash U$ in $M$ is at least $2$, then
$$\begin{array}{l}
\codim(M\times\{x\} \backslash U\times \{x\},M\times \{x\})\ge 2\\
\codim(M\times X \backslash U\times X,M\times X)\ge 2
\end{array}$$
As $M\times \{x\}$ and $M\times X$ are smooth varieties, they are Serre $S_2$ varieties, so given a vector bundle over $U\times \{x\}$, or $U\times X$, if there exists an extension as a vector bundle to $M\times \{x\}$ or $M\times X$ respectively, then the extension is unique. Therefore, $\SF^1=\SF^2$ and $\SF_{x,j}^1=\SF_{x,j}^2$ for every $x\in D$ and every $j=1,\ldots,r$, so the extension of the parabolic vector bundle is unique.
\end{proof}

Finally, we will briefly explain the notion of parabolic projective bundle. The filtration $E_{x,i}$ of $E|_x$ describing a parabolic structure on a vector bundle $E$ defines a filtration by projective subspaces $\PP(E_{x,i})$ of $\PP(E|_x)$. Given a parabolic vector bundle $(E,E_\bullet)$, we define its projectivization as the projective bundle $\PP(E)$ endowed with the following full flag of projective subspaces over each parabolic point $x\in D$
$$\PP(E)|_x =\PP(E_{x,1})\supsetneq \PP(E_{x,2}) \supsetneq \cdots \supsetneq \PP(E_{x,r})$$

In general, we define a parabolic projective bundle as a projective bundle $\PP$ over $X$ endowed with a full flag of affine spaces over each parabolic point $x\in D$
$$\PP|_x =\PP_{x,1} \supsetneq \PP_{x,2} \supsetneq \cdots \supsetneq \PP_{x,r}$$

\begin{lemma}
\label{lemma:projectiveReduction}
Let $X$ be a smooth complex projective curve and let $D$ be a reduced effective divisor over $X$. Then every parabolic projective bundle $(\PP,\PP_\bullet)$ admits a reduction to a parabolic vector bundle $(E,E_\bullet)$
$$(\PP,\PP_\bullet)\cong (\PP(E),\PP(E_\bullet))$$
Moreover, if $(E,E_\bullet)$ and $(E',E'_\bullet)$ are any two reductions, there exists a line bundle $L$ over $X$ such that
$$(E',E'_\bullet) \cong (E,E_\bullet)\otimes L$$
\end{lemma}

\begin{proof}
Let $P$ be the parabolic subgroup of $\GL(r,\CC)$ consisting on upper triangular matrices. Let $\SG$ be the group scheme over $X$ given by the following short exact sequence.
$$0\to \SG \to \GL(r,\CC)\times X \to (\GL(r,\CC)/P)\otimes \SO_D \to 0$$
Let $P\SG=\SG/\CC^*$. A projective parabolic bundle is a $P\SG$-torsor and the reductions are reductions of structure sheaf to $\SG$. From the short exact sequence
\begin{equation*}
1\longrightarrow \SO_X^*\longrightarrow \SG \longrightarrow P\SG \longrightarrow 1
\end{equation*}
we deduce that the obstruction for the existence of $\SG$-reductions of a $P\SG$-torsor is given by $H^2(X,\SO_X^*)=0$. On the other hand, as $\SO_X^*$ belong to the center of $\SG$, the space of reductions of a $P\SG$-torsor to $\SG$ is a torsor for the group $H^1(X,\SO_X^*)$. Every element in $H^1(X,\SO_X^*)$ corresponds to a line bundle over $X$ and it is clear that for every line bundle $L$
$$\PP\left( (E,E_\bullet)\otimes L \right) = \PP(E,E_\bullet)$$
so we conclude that all the reductions are related by tensorization with a line bundle.
\end{proof}

We conclude this section with a digression about $(l,m)$-stability for parabolic vector bundles.

\begin{definition}
A parabolic vector bundle $(E,E_\bullet)$ is $(l,m)$-(semi)stable if for every subbundle $F$ with the induced parabolic structure
$$\frac{\pdeg(F,F_\bullet)+l}{\rk(F)}< \frac{\pdeg(E,E_\bullet)-m}{\rk(E)}  \quad \text{(respectively, }\le\text{ )}$$
\end{definition}

\begin{lemma}
\label{lemma:generic-lm-stable}
Let $k>0$ be an integer. Assume that
$$g\ge m+l+1+\frac{l+k}{r-1}$$
Then the $(l,m)$ stable bundles form a nonempty Zariski open subset of $\SM(r,\alpha,\xi)$ such that its complement has codimension at least $k$. In particular, for $g\ge m+2l+2$, then the locus of $(l,m)$ stable bundles is nonempty for any rank.
\end{lemma}

\begin{proof}
First of all, let us prove that under that genus condition, the $(l,m)$ stable parabolic vector bundles form a nonempty Zariski open subset of $\SM(r,\alpha,d)$ whose complement has codimension at least $k$.

The proof of this first part is completely analogous to the proof of \cite[Proposition 2.7]{BB05}. For the convenience of the reader, we outline the main computations here.

Let $(E,E_\bullet)$ be a stable parabolic vector bundle which fails to be $(l,m)$-stable. Then, there exists a subbundle $E'$ of rank $r'$ and degree $d'$ with the induced parabolic structure and weight multiplicities $0\le n'_i(x) \le 1$ such that
\begin{equation}
\label{eq:generic-lm-stable1}
\frac{d'+\owt(E',E'_\bullet)+l}{r'} \ge \frac{d+\owt(E,E_\bullet)-m}{r}
\end{equation}
Let $(E'',E''_\bullet)$ be the parabolic vector bundle fitting in the sequence
$$0\longrightarrow (E',E'_\bullet)\longrightarrow (E,E_\bullet) \longrightarrow (E'',E''_\bullet)\longrightarrow 0$$
Then $E''$ has rank $r''=r-r'$, degree $d''=d-d'$ and $E''_\bullet$ is the induced parabolic filtration, which has weight multiplicities $n_i''(x)=1-n_i'(x)$. For simplicity in the equations, let us denote by $\owt$, $\owt'$ and $\owt''$ the parabolic weight $\owt(E,E_\bullet)$, $\owt(E',E'_\bullet)$ and $\owt(E'',E''_\bullet)$ respectively. Then $\owt=\owt'+\owt''$. Reordering the factors in equation \eqref{eq:generic-lm-stable1} and substituting $r$, $d$ and $\owt$ in terms of $(E',E'_\bullet)$ and $(E'',E''_\bullet)$ yields
$$(d'+d'')r'-d'(r'+r'')\le(r'+r'')\owt' - r'(\owt'+\owt'')+ mr'+lr$$
which is equivalent to
\begin{equation}
\label{eq:generic-lm-stable2}
d''r'-d'r''\le r''\owt'-r'\owt''+mr'+lr
\end{equation}
Let $h^1=\dim \op{PExt}^1((E'',E''_\bullet),(E',E'_\bullet)) = h^1(\ParHom((E'',E''_\bullet),(E',E'_\bullet)))$. As $(E,E_\bullet)$ is stable, then $h^0=H^0(\ParHom((E'',E''_\bullet),(E',E'_\bullet))=0$, so by Riemann-Roch formula
$$h^1=-\chi(\ParHom((E'',E''_\bullet),(E',E'_\bullet)))=r'd''-r''d'+r'r''(g-1+t_{\overline{n}'',\overline{n}'})$$
Applying inequality \eqref{eq:generic-lm-stable2}, we obtain
$$h^1\le r''\owt'-r'\owt''+mr'+lr+r'r''(g-1+t_{\overline{n}'',\overline{n}'})$$
Finally, let $\delta$ be the dimension of the locus of non-$(l,m)$-stable bundles in $\SM(r,\alpha,d)$. Generically, if $(E,E_\bullet)$ is not $(l,m)$-stable, then parabolic vector bundles $(E',E'_\bullet)$ and $(E'',E''_\bullet)$ constructed before can be found to be stable, so they are elements of the moduli spaces $\SM(r',d',\alpha')$ and $\SM(r'',d'',\alpha'')$ respectively, where $\alpha'$ and $\alpha''$ are the systems of weights induced from $\alpha$ with multiplicities $\overline{n}'$ and $\overline{n}''$ respectively. The possible $(E,E_\bullet)$ fitting in the sequence
$$0\longrightarrow (E',E'_\bullet)\longrightarrow (E,E_\bullet)\longrightarrow (E'',E''_\bullet)\longrightarrow 0$$
are then bounded by the projectivization of the parabolic $\Ext^1$-space, which has dimension $h^1-1$. Then
\begin{multline}
\label{eq:generic-lm-stable3}
\delta\le \max_{\overline{n}'} \left\{\dim \SM(r',d',\alpha')+\dim \SM(r'',d'',\alpha'')+h^1-1 \right\} \\
= \max_{\overline{n}'} \left\{ \begin{array}{l}
(r')^2(g-1)+1+(r')^2 t_{\overline{n}',\overline{n}'} + (r'')^2(g-1)+1+(r'')^2t_{\overline{n}'',\overline{n}''}\\
+r''\owt'-r'\owt''+mr'+lr+r'r''((g-1)+t_{\overline{n}'',\overline{n}'}) 
\end{array} \right\}
\end{multline}
From \cite[Lemma 2.4.1]{BB05}, we know that
$$(r')^2t_{\overline{n}',\overline{n}'}+(r'')^2t_{\overline{n}'',\overline{n}''}+r'r''t_{\overline{n}'',\overline{n}'} = r^2t_{\overline{n},\overline{n}}-r'r''t_{\overline{n}',\overline{n}''}$$
substituting in the previous equation yields
$$\delta \le \max_{\overline{n}'} \left \{ \begin{array}{l}
(r'+r'')^2(g-1)-r'r''(g-1)+1+r^2t_{\overline{n},\overline{n}}\\
-r'r''t_{\overline{n}',\overline{n}''}+r''\owt'-r'\owt''+mr'+lr 
\end{array} \right \}$$
By \cite[Lemma 2.5.2]{BB05}, we have 
$$-r'r''t_{\overline{n}',\overline{n}''}+r''\owt'-r'\owt''\le 0$$
Therefore, taking into account that $\dim\SM(r,\alpha,d)=r^2(g-1)+1+r^2t_{\overline{n},\overline{n}}$ we obtain
$$\delta\le \dim\SM(r,\alpha.,d)-\min_{r'} \left \{r'r''(g-1)-mr'-lr \right\} $$
Then we can guarantee that $\dim \SM(r,\alpha,d) - \delta\ge k>0$ whenever
$$g\ge 1+ \max_{r'} \frac{mr'+lr+k}{r'r''}$$
As $r'+r''=r$ and  $r'\ge 1$, $r''\ge 1$, then $\frac{1}{r'r''}$ attains its maximum value when $r'=1$ and $r''=r-1$ or $r'=r-1$ and $r''=1$. Simultaneously, $\frac{m}{r''}$ attains its maximum for $r''=1$, so the maximum of the above expression is attained at $r'=r-1$ and $r''=1$, leading us to the desired bound for the genus
$$g\ge m+\frac{rl+k}{r-1}+1 = m+l+1+\frac{l+k}{r-1}$$

Now, let $S^d\subsetneq \SM(r,\alpha,d)$ be the subset parameterizing stable parabolic vector bundles $(E,E_\bullet)\in \SM(r,\alpha,d)$ which are not $(l,m)$-stable. Notice that if $(E,E_\bullet)\in S^d$, then for every degree zero line bundle $L$, $(E,E_\bullet)\otimes L\in S^d$. To prove it, observe that if $(E',E'_\bullet)\subsetneq (E,E_\bullet)$ is a subbundle contradicting $(l,m)$-stability for $(E,E_\bullet)$, then $(E',E'_\bullet)\otimes L \subsetneq (E,E_\bullet)\otimes L$ contradicts $(l,m)$-stability for $(E,E_\bullet)\otimes L$. As the latter is always stable for any $L$, then $S^d$ is invariant by tensorization with line bundles of degree $0$. 

For every line bundle $\xi$ of degree $d$, let $S^\xi=S^d\cap \SM(r,\alpha,\xi)$. If $\xi'$ is another line bundle of degree $d$ then there exists a line bundle $L$ such that $L^r=\xi'\otimes \xi^{-1}$. Therefore, tensoring by $L$ gives us an isomorphism between $S^\xi$ and $S^{\xi'}$. Then, the fibers of the determinant map $\det : S^d \longrightarrow \Pic^d(X)$ are all isomorphic and, therefore, equidimensional. As the same happens with $\det : \SM(r,\alpha,d) \longrightarrow \Pic^d(X)$, then we obtain that for every $\xi\in \Pic^d(X)$, the codimension of $S^\xi$ in $\SM(r,\alpha,\xi)$ is the same as the codimension of $S^d$ in $\SM(r,\alpha,d)$ and the Lemma follows.
\end{proof}

\begin{lemma}
\label{lemma:1-0-stability}
Let $(E,E_\bullet)$ be a $(1,0)$-semistable parabolic vector bundle. Let $x\in D$ and let $1<k\le r$ be an integer. Let $E'_{x,k}\subsetneq E|_x$ be any subspace such that
$$E_{x,k-1} \supsetneq E_{x,k}' \supsetneq E_{x,k+1}$$
And let $E_\bullet'$ be the quasi-parabolic structure obtained substituting $E_{x,k}$ by $E_{x,k}'$ in $E_\bullet$. Then $(E,E_\bullet')$ is a stable parabolic vector bundle.
\end{lemma}
 
\begin{proof}
Let $F\subsetneq E$ be a subbundle. Let $F_\bullet$ and $F_\bullet'$ be the parabolic structures induced by $E_\bullet$ and $E_\bullet'$ on $F$ respectively. We have
$$\owt_x(F_\bullet')= \owt_x(F_\bullet) +\left( \dim(F|_x \cap E_{x,k}') - \dim(F|_x \cap E_{x,k}) \right) (\alpha_i(x)-\alpha_{i-1}(x))$$
As $E_{x,k+1}\subseteq E_{x,k}\cap E_{x,k}'$ and $E_{x,k+1}$ has codimension one in both $E_{x,k}$ and $E_{x,k+1}$, clearly $\dim(F|_x\cap E_{x,k}')\le \dim(F|_x \cap E_{x,k})+1$. Therefore,
$$\owt_x(F_\bullet') \le \owt_x(F_\bullet)+(\alpha_k(x)-\alpha_{k-1}(x)) < \owt_x(F_\bullet)+1$$
Finally, by $(1,0)$ semistability yields
\begin{multline*}
\frac{\pdeg(F,F_\bullet')}{\rk(F)}=\frac{\deg(F)+\sum_{x\in D} \owt_x(F_\bullet')}{\rk(F)} < \frac{\deg(F)+\sum_{x\in D} \owt_x(F_\bullet) +1}{\rk(F)} \\
\le \frac{\pdeg(E,E_\bullet)}{\rk(E)}=\frac{\pdeg(E,E_\bullet')}{\rk(E)}
\end{multline*}
as this holds for every subbundle $F$, $(E,E_\bullet')$ is stable.
\end{proof}

\section{Parabolic Hitchin Pairs}
\label{section:Hitchin}

Let $L$ be a line bundle over a complex projective curve $X$. An $L$-twisted Hitchin pair over $X$ is a pair $(E,\varphi)$ consisting on a vector bundle $E$ over $X$ and a traceless morphism  $\varphi \in H^0(\End_0(E)\otimes L)$ called the field.

If $L$ is the canonical bundle $K$, then a $K$-twisted Hitchin pair is usually known as a Higgs bundle and the morphism $\varphi$ is known as the Higgs field.

Given a Hitchin pair $(E,\varphi)$, a subbundle $F\subseteq E$ is said to be $\varphi$-invariant if
$$\varphi(F)\subseteq F\otimes L$$
An $L$-twisted Hitchin pair is called stable (respectively semistable) if and only if for every $\varphi$-invariant proper subbundle $0\ne F\subsetneq E$
$$\mu(F)<\mu(E) \quad\quad \text{(respectively }\le \text{)}$$

We will denote by $\SM_L(r,\xi)$ the moduli space of semistable $L$-twisted Hitchin pairs of rank $r$ and determinant $\det(E)\cong \xi$. Notice that by Serre duality, for $L=K$ the cotangent space of $\SM(r,\xi)$ at a stable vector bundle $E$ is
$$T^*_E\SM(r,\xi)\cong H^1(\End_0(E))^\vee=H^0(\End_0(E)\otimes K)\, ,$$
hence, the cotangent bundle of $\SM(r,\xi)$ lies as a subscheme of the moduli space of semi-stable Higgs bundles. In fact, it is an open subscheme.

Let us recall the definition of the Hitchin map
$$H:\SM_L(r,\xi)\longrightarrow \SH_L=\bigoplus_{k=2}^r H^0(X,L^k)$$
Let $S=\op{Tot}(L)=\underline{\Spec}\Sym^\bullet(L^{-1})$ be the total space of the vector bundle $L$. Let $\pi:S\to X$ be the projection and let $x\in H^0(S,\pi^*L)$ be the tautological section. Let us consider the characteristic polynomial of the field $\varphi$
$$\det(x\cdot \id - \pi^* \varphi)=x^r+\tilde{s_1}x^{r-1}+ \tilde{s_2} x^{r-2}+\cdots +\tilde{s_r}$$
Then there exist unique sections $s_i\in H^0(X,L^i)$ such that $\tilde{s_i}=\pi^*s_i$. Note that $\varphi$ is traceless by hypothesis, so $s_1=0$. The Hitchin map is then built sending each Hitchin pair $(E,\varphi)$ to the coefficients of the characteristic polynomial of $\varphi$
$$(s_i)_{i=1}^r \in \bigoplus_{i=2}^r H^0(X,L^i)$$

The zeros of the characteristic polynomial $\det(x\cdot \id - \pi^*\varphi)$ define a curve $X_s\subset \op{Tot}(L)$ which is an $r$-to-$1$ cover of $X$. We call it the spectral curve at $s\in \SH_L$.

A parabolic $L$-twisted Hitchin pair over a pointed curve $(X,D)$ is a parabolic vector bundle $(E,E_\bullet)$ over $(X,D)$ endowed with an $L$-twisted strongly parabolic endomorphism $\varphi\in H^0(\SParEnd_0\otimes L)$. A $K(D)$-twisted parabolic Hitchin pair is called a parabolic Higgs bundle.

A parabolic $L$-twisted Hitchin pair $(E,E_\bullet,\varphi)$ is called stable (respectively semistable) if for every $\varphi$-invariant proper subbundle $0\ne F\subsetneq E$ with the induced parabolic structure
$$\frac{\pdeg(F,F_\bullet)}{\rk(F)} < \frac{\pdeg(E,E_\bullet)}{\rk(E)} \quad \text{(respectively, }\le\text{ )}$$

We denote by $\SM_L(r,\alpha,\xi)$ the moduli space of semistable $L$-twisted parabolic Hitchin pairs. From this point on, we will assume that $\alpha$ is full flag, although the proof of most of the lemmas in this section can be adapted to other parabolic types. Similarly to the non-parabolic case, by Serre duality if $(E,E_\bullet)$ is a stable parabolic vector bundle
$$T^*_{(E,E_\bullet)}\SM(r,\alpha,\xi) \cong H^1(\ParEnd_0(E,E_\bullet))^\vee=H^0(\SParEnd_0(E,E_\bullet)\otimes K(D))$$
Therefore, the cotangent bundle of the moduli space of stable parabolic vector bundles is a subset of the moduli of parabolic Higgs bundles. In fact, it is an open subvariety.

We can define an analogue of the Hitchin map in the parabolic case by sending each parabolic Hitchin pair $(E,E_\bullet,\varphi)$ to the characteristic polynomial of $\varphi$. Nevertheless, as the field is assumed to be strongly parabolic, it is nilpotent at the parabolic points, so its characteristic polynomial vanishes at $D$. Moreover, we require the field to be traceless, so the independent coefficient of the characteristic polynomial is always zero. Therefore, the image of the Hitchin map lies in
$$H:\SM_L(r,\alpha,\xi)\longrightarrow \SH'_L=\bigoplus_{i=2}^r H^0(X,L^i(-D))$$

We will be interested in computing the image of the Hitchin map both for the parabolic and non-parabolic cases. For non-parabolic Hitchin pairs, the following Lemma holds as a consequence of an argument from Beauville, Narasimhan and Ramanan \cite{BNR}

\begin{lemma}
Let $L$ be a line bundle over $X$ such that $r\deg(L)>2g$. Then the Hitchin map
$$H:\SM_L(r,\xi)\longrightarrow \bigoplus_{k=2}^r H^0(X,L^k)$$
is surjective.
\end{lemma}

\begin{proof}
By hypothesis $\deg(L^r)>2g$, so $L^r$ is very ample. Therefore, it admits a section $\tau\in H^0(X,L^r)$ with at most simple zeros. Let $\overline{\tau}=(0,0,\ldots,\tau)\in \bigoplus_{k=2}^r H^0(X,L^k)$. Then $X_{\overline{\tau}}$ has equation $x^r+\tau=0$. As $\tau$ has at most simple zeroes, $X_{\overline{\tau}}$ is smooth. The smoothness condition for families of curves is open, so there is an open nonempty subset $U\subseteq \bigoplus_{k=2}^r H^0(X,L^k)$ such that for every $s\in U$, the spectral curve $X_{s}$ is smooth.

On the other hand, from \cite[Proposition 3.6]{BNR} there exists a bijection between torsion free sheaves of rank $1$ over $X_s$ (whose pushforward is automatically a stable pure dimension sheaf over $\op{Tot}(L)$) and stable Hitchin pairs $(E,\varphi)$ over $X$ such that $H(E,\varphi)=s$. As there always exist rank 1 torsion free sheaves over $X_s$, for every $s\in U$ there exists at least a stable Hitchin pair whose image by the Hitchin map is $s$, so
$$U\subseteq H(\SM_L^s(r,\xi)) \subseteq \bigoplus_{k=2}^r H^0(X,L^k)$$
The set $U$ is Zariski open and nonempty, so it is dense and $H$ is dominant. By \cite[Theorem 6.1]{Ni}, it is also proper, so it must be surjective.
\end{proof}

Let us prove the parabolic analogue for the Lemma

\begin{lemma}
Suppose that $g\ge 2$ and let $L$ be a line bundle over $X$ such that $r\deg(L)>2g$. Then the parabolic Hitchin map
$$\SM_L(r,\alpha,\xi) \longrightarrow \bigoplus_{k=2}^r H^0(X,L^k(-D))$$
is dominant.
\end{lemma}

\begin{proof}
Let $(\SE,\Phi)$ be a versal family of traceless semistable Hitchin $L$-twisted pairs, where $\SE\longrightarrow \SM\times X$ is a vector bundle and $\Phi:\SE\longrightarrow \SE\otimes p_2^*L$ satisfies that for every $t\in \SM$, $(\SE_t,\Phi_t)$ is a semistable Hitchin pair. By the previous corollary, the induced Hitchin map $h:\SM\longrightarrow \SH_L$ is surjective. Let
$$\SM'=h^{-1}\left ( \bigoplus_{k=2}^r H^0(L^k(-D)) \right)$$
Then it is the closed subset of $\SM$ corresponding to stable pairs whose field is nilpotent at every $x\in D$. As the Hitchin map is surjective, its restriction $h:\SM'\longrightarrow \SH_L'$ is surjective.

For every $x\in D$, let $\SF_x$ be the total space of the flag bundle over $\SE|_{\SM'\times \{x\}}$, i.e.
$$\SF_x=\op{Tot}\left(\op{Fl}(\SE|_{\SM'\times \{x\}})\right )$$
Let $\pi:\SF_x\twoheadrightarrow \SM'$ be the projection. Taking the pullback of the versal family to $\SF_x$, it is a family of triples $(\SE,\{\SE_{x,i}\},\Phi)$ consisting on a vector bundle, a full flag filtration at the point $x$ and a field. Consider the closed subset $\SH_x\subseteq \SF_x$ consisting on triples where the field preserves the filtration. It is closed by \cite[Lemma 4.3]{Yok93} (see \cite[Chapter 4]{Al17} for more details). As the characteristic polynomial of $\Phi_t:E_t\to E_t\otimes L$ annihilates at $x$, it is nilpotent at $x$ and therefore it admits an adapted filtration at $x$, $\{E_{t,x,i}\}$ such that
$$\Phi_t(E_{t,x,i})\subseteq E_{t,x,i+1}$$
Therefore, the map $\SF_x\longrightarrow \SM'$ is surjective. Now, let 
$$\SN=\SF_{x_1} \times_{\SM'} \SF_{x_2}\times_{\SM'} \cdots \times_{\SM'}\SF_{x_n} \twoheadrightarrow \SM'$$
be the fiber product of all $\SF_x$ over $\SM'$ for $x\in D$. Taking the pullback of the families defined over $\SF_{x}$ for $x\in D$, there is a versal family over $\SN$ of triples $(\SE,\SE_\bullet,\Phi)$ such that $(\SE,\SE_\bullet)$ is a vector bundle over $\SN\times X$ with a filtration over $\SN\times D$ and $\Phi\in H^0(\SParEnd_0(E,E_\bullet)\otimes p_2^*L)$ such that for every $t\in \SN$, $(\SE_t,\Phi_t)$ is a stable Hitchin pair.

Let $\SU\subseteq \SN$ be the open subset consisting on points $t\in \SN$ such that $(\SE_t,\SE_{t,\bullet})$ is a stable parabolic vector bundle with respect to the parabolic weights $\alpha$. By Corollary \ref{cor:underlyingStableGeneric}, there exists at least a filtered vector bundle $(E,E_\bullet)$ such that $E$ is stable and $(E,E_\bullet)$ is parabolically stable. Therefore, $\SU$ is nonempty and thus, dense. Therefore $h(\SU)\subseteq H(\SM_L(r,\alpha,\xi))$ is dense in $\SH_L'$.
\end{proof}

\begin{corollary}
\label{cor:HitchinDominant}
Suppose that $g\ge 2$. Let $U$ be any nonempty open subset of $\SM(r,\alpha,\xi)$ and let $L$ be a line bundle over $X$ such that $r\deg(L)>2g$. Then the linear space generated by the images of
$$H(H^0(\SParEnd_0(E,E_\bullet)\otimes L))\subsetneq \bigoplus_{i=2}^r H^0(X,L^k(-D))$$
when $(E,E_\bullet)$ runs over $U$ is $\bigoplus_{i=2}^r H^0(X,L^k(-D))$.
\end{corollary}

\begin{proof}
Let $\SM_L^{ss-vb}(r,\alpha,\xi)\subseteq \SM_L(r,\alpha,\xi)$ be the subset of the moduli space of semi-stable parabolic Hitchin pairs consisting of pairs whose underlying parabolic vector bundle is semi-stable. Let
$$p_{(E,E_\bullet)}:\SM_L^{ss-vb}(r,\alpha,\xi) \longrightarrow \SM(r,\alpha,\xi)$$
be the forgetful map and let $\overline{U}=p_{(E,E_\bullet)}^{-1}(U)$. Then $\overline{U}$ is an open nonempty subset of $\SM_L(r,\alpha,\xi)$, so it is a dense subset. The previous Lemma implies that $H$ is dominant. Therefore, $H(\overline{U})$ is dense in $\bigoplus_{i=2}^r H^0(X,L^k(-D))$, so its linear span is the whole space.
\end{proof}

In the case of Higgs bundles, i.e., when $L$ is the canonical bundle $K$, a classical result by Hitchin shows that the Hitchin map becomes a complete integrable system for the moduli space of Higgs bundles. In the case of parabolic bundles, we will be interested in the following result from Faltings

\begin{lemma}[{\cite[V.(ii)]{Faltings}}]
The parabolic Hitchin map
$$H:\SM_{K(D)}(r,\alpha,\xi) \longrightarrow  \SH'$$
is equidimensional.
\end{lemma}

Then, we can state some additional properties. For simplicity, let us write $\SH'=\SH_{K(D)}'=\bigoplus_{k=2}^r H^0(X,K^kD^{k-1})$. In order to simplify the notation, through this last part of the section let $m=\dim(\SM(r,\alpha,\xi))$. Then, as we are considering full flag parabolic bundles, $\dim (\SH')=m$ and $\dim(\SM_{K(D)}(r,\alpha,\xi))=\dim(T^*\SM(r,\alpha,\xi))=2m$.

\begin{corollary}
\label{cor:HitchinDominant2}
Let $\SU\subseteq \SM(r,\alpha,\xi)$ be any nonempty open subset. Then the restriction of the parabolic Hitchin morphism to the cotangent bundle
$$H_\SU:T^*\SU \longrightarrow \SH'$$
is dominant.
\end{corollary}

\begin{proof}
First, observe that as $\SM(r,\alpha,\xi)$ is irreducible \cite{BY}, then $\SU$ is dense in $\SM(r,\alpha,\xi)$ for any nonempty open subset of the moduli space, so $\dim(T^*\SU)=\dim(T^*\SM(r,\alpha,\xi))=2m$. Suppose that $H_\SU$ is not dominant and let $S=\overline{\im(H_\SU)}\subsetneq \SH'$. $\SH'$ is irreducible, so $\dim(S)<\dim(\SH')=m$. By the previous Lemma, $H:\SM_{K(D)}(r,\alpha,\xi)\to \SH'$ is equidimensional, so
$$\dim(H^{-1}(S))=\dim(S)+m<2m=\dim(T^*\SU)$$
However, this contradicts that, by construction, $T^*\SU\subseteq H^{-1}(S)$.
\end{proof}

\begin{corollary}
\label{cor:HitchinEquidimensional}
Let $\SU\subseteq T^*\SM(r,\alpha,\xi)$ be any open subset. Then the restriction of the parabolic Hitchin morphism to the cotangent bundle
$$H_\SU:T^*\SU \longrightarrow \SH'$$
is equidimensional.
\end{corollary}

\begin{proof}
Let $s\in \SH'$. By the Lemma $\dim(H^{-1}(s))=m$. As $H_\SU^{-1}(s)\subseteq h^{-1}(s)$, then
$$\dim(H_\SU^{-1}(s))\le \dim(H^{-1})(s)=m$$
On the other hand, as $\dim(T^*\SM(r,\alpha,\xi))=2m = \dim(\SH')+m$. By the previous corollary $H_\SU$ is dominant, so by \cite[3.22]{Hartshorne77}, for every $s\in \SH'$,  $\dim(H_\SU^{-1}(s))\ge m$. Therefore, every fiber has dimension $m$.
\end{proof}

In particular, observe that if $s\in \SH'$ corresponds to a smooth spectral curve $X_s$, then by \cite[Lemma 3.2]{GL}, if $\pi:X_s\to X$ is the covering then the fiber $H^{-1}(s)$ is isomorphic to
$$\op{Prym}(X_s/X)=\{L\in \Pic(X_s) | \det(\pi_*L) \cong \xi\}$$
which is an irreducible abelian variety of dimension $m$. Then $H_\SU^{-1}(s)$ is dense in $H^{-1}(s)$.

In the following chapter, we will be interested in understanding how the geometry of $H_\SU^{-1}(s)$ relates to that of $H^{-1}(s)$ when $s$ does not correspond to a smooth spectral curve. We will need the following proposition derived directly from the work of Faltings \cite{Faltings}

\begin{proposition}
\label{prop:CotangentCodimension}
Let $g\ge 4$. Then the complement of $T^*\SM(r,\alpha,\xi)$ in $\SM_{K(D)}(r,\alpha,\xi)$ has codimension at least $3$.
\end{proposition}

\begin{proof}
Combine the remark \cite[V.(iii)]{Faltings} on Theorem  \cite[II.6.(iii)]{Faltings} with the codimension bound computations in \cite[p. 536]{Faltings} and Lemma \cite[II.7.(ii)]{Faltings}. Faltings proves that if $g\ge 3$ (or $g=2$ with some additional constraints) these bounds imply that the codimension is at least $2$, but the same computations prove that if $g\ge 4$ the codimension is at least $3$.
\end{proof}

\section{Hitchin Discriminant and Torelli Theorem}
\label{section:Torelli}

Let $\SD \subsetneq \SH'$ be  the divisor of the Hitchin space consisting of characteristic polynomials whose corresponding spectral curve is singular. We call $H^{-1}(\SD)$ the Hitchin discriminant. In order to simplify the notation, from now on let us write $\SH'=W$ and let us write
$$H:\SM_{K(D)}(r,\alpha,\xi) \to W$$
$$H_0=H_{T^*\SM(r,\alpha,\xi)}:T^*\SM(r,\alpha,\xi)\longrightarrow W$$

\begin{proposition}
\label{prop:HitchinComponents}
Assume that $g\ge 2$. Then the divisor $\SD$ has at most $n+1$ irreducible components, which can be described as follows.
\begin{enumerate}
\item For each parabolic point $x\in D$, let $\SD_x$ be the set of characteristic polynomials whose spectral curve is singular over $x$.
\item Let $\SD_U$ be the set of characteristic polynomials whose spectral curve is singular, but it is smooth over each $x\in D$. And let $\overline{\SD_U}\supsetneq \SD_U$ be the set of characteristic polynomials whose spectral curve is singular over some $y\not\in D$ (but not necessarily smooth over $D$).
\end{enumerate}
Then
$$\SD=\overline{\SD_U} \cup \bigcup_{x\in D} \SD_x$$
\end{proposition}

\begin{proof}
It becomes clear that for every $s\in \SD$, the corresponding singular curve $X_s$ is either singular over some parabolic point $x\in D$ or it is smooth at the parabolic points $x\in D$ and it is singular over some point in $U=X\backslash D$. Therefore, $\SD=\overline{\SD_U} \cup \bigcup_{x\in D} \SD_x$ and it is enough to prove that each element in the decomposition is irreducible.

Let us denote by $X_0\subset \op{Tot}(KD)$ the image of $X$ in the total space of $KD$ given by the zero section of the line bundle. If a spectral curve $X_s$ is singular over $x\in D$, it has a singular point precisely at $(x,0)\in X_0$. A spectral curve  $X_s$ has a singularity over $X_0$ at $(y,0)$ if and only if the characteristic polynomial $p_s(z,t)=t^r + \sum_{k=1}^r s_k(z) t^{r-k}$ satisfies the following properties
\begin{enumerate}
\item $s_r(z) \in H^0(K^rD^r)$ annihilates of order at least 2 at $z=y$, i.e., $s_r\in H^0(K^rD^r(-2y))$
\item $s_{r-1}(z)\in H^0(K^{r-1}D^{r-1})$ annihilates at $z=y$, i.e., $s_{r-1}\in H^0(K^{r-1}D^{r-1}(-y))$
\end{enumerate}
As $s=(s_2,\ldots,s_r) \in W$, we already know that
$$s_{r-1}\in H^0(K^{r-1}D^{r-2})\subseteq H^0(K^{r-1}D^{r-1}(-x))$$ and $s_r\in H^0(K^rD^{r-1})$, so the points in $\SD_x$ are precisely those with $s_r\in H^0(K^rD^{r-1}(-x))$. Therefore
$$\SD_x=\bigoplus_{k=2}^{r-1} H^0(K^kD^{k-1}) \oplus H^0(K^rD^{r-1}(-x))$$
is irreducible for every $x\in D$.

On the other hand, let $X_s$ be a spectral curve with a singularity over some $y\not\in D$. Suppose that $(y,t_0^*)\in X_s\subset \op{Tot}(KD)$ is the singular point. Assume that $r>2$. As $y\not\in D$ and $K$ is base point free ($g\ge 2$), then by identifying $H^0(K)$ with the space of sections of $KD$ which annihilate at $D$ we can find a section $t_0\in H^0(K)$ such that $t_0(y)=t_o^*$. Then the curve defined by the polynomial $p_s^{t_0}(z,t)=p_s(z,t-t_0)$ is singular at the point $(y,0)$, but smooth over every point $x\in D$. Set $s^{t_0}=(s_i^{t_0})$ as
$$p_s^{t_0}(z,t)=(t-t_0(z))^r+\sum_{k>0}s_k(z)(t-t_0(z))^{r-k}=t^n+\sum_{k=1}^r s_k^{t_0}(z) t^{r-k}$$
More precisely, we have
$$s_k^{t_0}=\binom{r}{k}(-t_0)^{\otimes k} + \sum_{1\le j<k} \binom{r-j}{r-k}s_j\otimes (-t_0)^{\otimes k-j}\in H^0(K^kD^{k-1})$$
As $X_{s^{t_0}}$ is singular at $y\not\in D$, but smooth at every $x\in D$, then applying the previous criterion yields
\begin{multline*}
s^{t_0}\in \bigoplus_{k=1}^{r-2} H^0(K^kD^{k-1})\oplus H^0(K^{r-1}D^{r-2}(-y)) \\
\oplus \left( H^0(K^rD^{r-1}(-2y) \backslash \bigcup_{x\in D} H^0(K^rD^{r-1}(-2y-x)) \right) := \SR_y
\end{multline*}
Observe that if $g\ge 2$ and $r>2$ then $\deg(K^{r-1}D^{r-1})=(r-1)(2g-2+|D|)>3$. Therefore, for any divisor $N$ with $0\le \deg(N)\le 3$ we have
$$\deg(K^{1-r}D^{1-r}(N))=-\deg(K^{r-1}D^{r-1})+\deg(N)<-3 +\deg(N)\le 0$$
Therefore, $h^0(K^{1-r}D^{1-r}(N))=0$ and, using Riemann-Roch theorem
\begin{multline*}
h^0(K^rD^{r-1}(-N)) = \deg(K^rD^{r-1})-\deg(N)+1-g+h^0(K^{1-r}D^{1-r}(N))\\
=\deg(K^rD^{r-1})+1-g-\deg(N)
\end{multline*}
Then $h^0(K^rD^{r-1}(-N))=h^0(K^rD^{r-1})-\deg(N)$. Therefore, the last summand in the expression of $\SR_y$ is the complement of an hyperplane in $H^0(K^rD^{r-1}(-2y))$ so, in particular, $\SR_y$ is irreducible and nonempty.

Observe that as the polynomials in $W$ have $s_1=0$, then $s_1^{t_0}=-rt_0$. Therefore, given any point $s'\in \SR_y$, we can obtain a point in $\SD_U$ taking
$$s=(s')^{s'_1/r}$$
Therefore, for every $y\not\in D$ we obtain a map from $\SR_y$ to $\SD_U$ and for every element in $\SD_U$ there exists an element in $\SR_y$ mapping to it for some $y\not\in D$. Then, we can build the following variety
$$\SR:= \coprod_{x\not\in D} \SR_y \subsetneq (X\backslash D) \times W$$
Let us prove that $\SR$ is irreducible. Let $\SR'$ be the subbundle of $(X\backslash D)\times W$ whose fiber over $y\not\in D$ is
$$\bigoplus_{k=1}^{r-2} H^0(K^kD^{k-1})\oplus H^0(K^{r-1}D^{r-2}(-y))\oplus H^0(K^rD^{r-1}(-2y) \subsetneq W$$
Moreover, for every $x\in D$, let $\SR^x\subsetneq \SR'$ be the subbundle of $\SR'$ whose fiber over $y\not\in D$ is
$$\bigoplus_{k=1}^{r-2} H^0(K^kD^{k-1})\oplus H^0(K^{r-1}D^{r-2}(-y))\oplus H^0(K^rD^{r-1}(-2y-x) \subsetneq \SR'_y$$
Then we can write
$$\SR=\SR'\backslash \bigcup_{x\in D} \SR^x$$
As $\SR^x$ are subbundles of $\SR$, then $\SR$ is irreducible. Finally, the maps $\SR_y \longrightarrow \SD_U$ induce a well defined surjective map
$$\SR \longrightarrow \SD_U$$
Therefore $\SD_U$ is irreducible. Moreover, from construction we obtain that $\SR'\longrightarrow \overline{\SD_U}$ is also surjective, so $\overline{\SD_U}$ is also irreducible.

It remains to consider the case $r=2$, but in that case we have simply $W=H^0(K^2D)$. Then the spectral curve corresponding to a point $s=s_2\in W$ has equation $t^2+s_2(z)=0$. Therefore, it has a singularity over $y\not\in D$ if and only if $s_2$ annihilates of order at least 2 in $y$, i.e., for  $s_2\in H^0(K^2D(-2y))$. Then
$$\SD=\bigcup_{y\not\in D} H^0(K^2D(-2y)) \cup \bigcup_{x\in D} H^0(K^2D(-x))$$
As $g\ge 2$, the first union is the image of the subbundle $\SR' \hookrightarrow U\times W$ whose fiber over $y\in U$ is $H^0(K^2D(-2y))$ under the projection map
$$\SR' \hookrightarrow U\times W \stackrel{p_{W}}{\twoheadrightarrow} W$$
so it is irreducible and corresponds to $\overline{\SD_U}$.
\end{proof}

\begin{proposition}
\label{prop:GenericFiberHitchin}
Suppose that $g\ge 4$. Then for $s\in W\backslash \SD$ the fiber $H_0^{-1}(s)$ is an open subset of an abelian variety. For a generic $s$	in each irreducible component of $\SD$ the fiber $H_0^{-1}(s)$ contains a complete rational curve.
\end{proposition}

\begin{proof}
By \cite[Lemma 3.2]{GL}, if $X_s$ is smooth and $\pi:X_s\to X$ is the covering then the fiber $H^{-1}(s)$ is isomorphic to
$$\op{Prym}(X_s/X)=\{L\in \Pic(X_s) | \det(\pi_*L) \cong \xi\}$$
which is an abelian variety.

On the other hand, if $s\in \SD_U$ is generic then $X_s$ has a unique singularity which is a node not lying over a parabolic point. By \cite[Theorem 4]{Bh96} the fiber $H^{-1}(s)$ is an uniruled variety. More precisely, it is birational to a $\PP^1$-fibration over the Jacobian $J(\tilde{X}_s)$, where, $\tilde{X}_s$ is the normalization of $X_s$.

Let $Z=\left(\SM_{K(D)}(r,\alpha,\xi)\backslash T^*\SM(r,\alpha,\xi) \right) \cap H^{-1}(\SD_U)$. By Proposition \ref{prop:CotangentCodimension}, for $g\ge 4$, the complement $\SM_{K(D)}(r,\alpha,\xi)\backslash T^*\SM(r,\alpha,\xi)$ has codimension at least $3$. Therefore, $Z$ has codimension at least $2$ in $H^{-1}(\SD_U)$. Let $S=H\left(\SM_{K(D)}\backslash T^*\SM(r,\alpha,\xi) \right)$. Let $m=\dim \SH'$ and assume that $\dim(S)<m-1$. Then for any $s\in \SD_U\backslash S$ we have $H_0^{-1}(s)=H^{-1}(s)$ and the fiber contains a complete rational curve.

Now, let us suppose that $\dim(S)=m-1$. Then $Z\longrightarrow \SD_U$ is dominant and, therefore, the generic fiber has dimension $\dim(Z)-\dim(\SD_U)\le m-2$. In other words, for a generic $s\in \SD_U$, $Z\cap H^{-1}(s)$ has codimension at least 2 in $H^{-1}(s)$. As the latter is uniruled and we are only taking away a codimension 2 set, then there exists at least a complete rational curve in $H_0^{-1}(s)$.

It is only left to prove that a generic fiber over $\SD_x$ contains a complete rational curve. As $g\ge 4$, let $\SU\subsetneq \SM(r,\alpha,\xi)$ be the intersection of the open nonempty subsets defined by Lemma \ref{lemma:ParEndNoSections} and Lemma \ref{lemma:generic-lm-stable} for $(l,m)=(1,0)$. It parameterizes $(1,0)$ stable parabolic vector bundles $(E,E_\bullet)$ such that $H^0(\ParEnd_0(E,E_\bullet)(x_0))=0$.

Then for every $(E,E_\bullet)\in \SU$  and every $x\in D$ we have
$$H^1(\SParEnd_0(E,E_\bullet)\otimes K(D-x)) = H^0(\SParEnd_0(E,E_\bullet)(x))^\vee = 0$$
so the evaluation morphism
$$\op{ev}:H^0(\SParEnd_0(E,E_\bullet)\otimes K(D)) \longrightarrow \SParEnd_0(E,E_\bullet)\otimes K(D)|_x$$
is surjective.

For $1<k\le r$, let $N_k(E,E_\bullet)\subsetneq \SParEnd_0(E,E_\bullet) \otimes K(D)|_x$ be the subspace of matrices with a zero in position $(k-1,k)$. For $k=1$, let $N_i(E,E_\bullet)$ be the subspace of matrices with a zero in position $(r,1)$. Let $\tilde{N}_k(E,E_\bullet)$ be the preimage of $N_i$ under the evaluation map. For $k\ne 1$, we can describe $\tilde{N}_k(E,E_\bullet)$ as follows. Let $E_\bullet^k$ be the subfiltration of $E$ obtained removing the element $E_{x,k}$. Then
$$\tilde{N}_k(E,E_\bullet)=H^0(\SParEnd_0(E,E_\bullet^k)\otimes K(D))$$

Let $(E,E_\bullet,\varphi)\in H^{-1}(\SD_x)\cap T^*\SU$. Let $z$ be a coordinate on $X$ around the parabolic point $x\in D$. Then, locally, $\varphi$ can be written as
$$\varphi(z)=\left ( \begin{array}{cccc}
za_{11} & a_{12}  & \cdots & a_{1r}\\
za_{21} & za_{22} & \cdots & a_{2r}\\
\vdots & \vdots & \ddots & \vdots \\
za_{r1} & za_{r2} & \cdots & za_{rr}
\end{array} \right)$$
Where $a_{ij}$ are local sections of $K(D)$ and $\varphi$ is expressed in a basis which is adapted to the parabolic filtration. Then $(E,E_\bullet,\varphi)\in H^{-1}(\SD_x)$ if and only if $z^2 | \det(\varphi(z))$. Nevertheless, if we express the determinant as a sum of products of elements of the matrix above it becomes clear that the only summand that is not a multiple of $z^2$ is precisely $za_{r1}a_{12}a_{23}\cdots a_{r-1,r}$. Therefore, the determinant is a multiple of $z^2$ if and only if at least one of the elements $a_{r1}$, or $a_{k-1,k}$ annihilates at $z=0$ for some $k>1$. This is equivalent to ask $\op{ev}(\varphi)\in N_k$ for some $1\le k\le r$. As the evaluation map is surjective for every parabolic vector bundle in $\SU$, we conclude that for every $(E,E_\bullet)\in \SU$
$$H^{-1}(\SD_x)\cap T^*_{(E,E_\bullet)}\SM(r,\alpha,\xi) = \bigcup_{k=1}^r \tilde{N}_k(E,E_\bullet)$$
Let $\tilde{N}_k=\bigcup_{(E,E_\bullet)\in \SU} \tilde{N}_k(E,E_\bullet)$. By construction $\dim(\tilde{N}_k)=\dim(2\SM(r,\alpha,\xi)-1)$. Assume that $(E,E_\bullet,\varphi)\in \tilde{N}_k$ for some $k>1$. By Lemma \ref{lemma:1-0-stability}, for every $(E,E_\bullet)\in \SU$, every $x\in D$, every $1<k\le r$ and every $E_{x,k}'$ such that $E_{x,k-1}\supsetneq E_{x,k}'\supsetneq E_{x,k+1}$ then $(E,E_\bullet')$ is a stable parabolic vector bundle. Moreover, as $\varphi$ sends $E_{x,k-1}$ to $E_{x,k+1}$, then $\varphi\in H^0(\SParEnd_0(E,E_\bullet')\otimes K(D))$ for every choice of $E_{x,k}'$. Therefore, for every $E_{x,k}'$, $(E,E_\bullet',\varphi)\in H_ 0^{-1}(\SD_x)$. As $E$ and $\varphi$ do not change, all those Higgs bundles lie over the same point of the Hitchin map. The space of possible compatible steps in the filtration is parameterized by $\PP^1$, so they form a complete rational curve in $T^*\SM(r,\alpha,\xi)$.

Therefore, the image of the complete rational curves contains $H(\tilde{N}_k)\subseteq \SD_x$ for every $k>1$. Then it is enough to prove that the image is dense for some $k>1$. Assume that $H(\tilde{N}_k)$ is not dense. Let $S=H(\tilde{N}_k)$ and $m=\dim(\SM(r,\alpha,\xi))$.  Then $\dim(S)<\dim(\SD_x)=m-1$. By equidimensionality of $H_0$, $\dim H_0^{-1}(S)=m+\dim(S)<2m-1=\dim(\tilde{N}_k)$, but $\tilde{N}_k\subseteq H_0^{-1}(S)$. 

\end{proof}

\begin{lemma}
\label{lemma:DiscriminantCharacterization}
Let $\SR\subset T^*\SM(r,\alpha,\xi)$ be the union of the complete rational curves in $T^*\SM(r,\alpha,\xi)$. Then $\SD$ is the closure of $H_0(\SR)$ in $W$.
\end{lemma}

\begin{proof}
Let $\PP^1 \hookrightarrow T^*\SM(r,\alpha,\xi)$ be a complete rational curve. Composing  with the Hitchin map, we obtain a morphism 
$$\PP^1 \hookrightarrow T^*\SM(r,\alpha,\xi) \longrightarrow W$$
from $\PP^1$ to an affine space, so it is constant. Therefore, each complete rational curve must be contained in a fiber of the Hitchin morphism.

Let $s\in W\backslash \SD$. By the previous Proposition \ref{prop:GenericFiberHitchin}, $H_0^{-1}(s)$ is an open subset of an abelian variety, so it does not admit any nonconstant morphism from $\PP^1$. Therefore, there is no complete rational curve over $W\backslash \SD$. Then, by the second part of the Proposition \ref{prop:GenericFiberHitchin}, we know that for every irreducible component of $\SD$, a generic fiber contains a rational curve. Therefore, $H_0(\SR)$ is dense in $\SD$ and, as $\SD$ is closed in $W$, $\SD=\overline{\SR}$.
\end{proof}

\begin{proposition}
\label{prop:HitchinGlobalFunctions}
The global algebraic functions $\Gamma(T^*\SM(r,\alpha,\xi))$ produce a map
$$\tilde{h}:T^*\SM(r,\alpha,\xi) \longrightarrow \Spec(\Gamma(T^*\SM(r,\alpha,\xi)))\cong W\cong \CC^m$$
which is the parabolic Hitchin map up to an isomorphism of $\CC^m$, where $m=\dim W$. Moreover, consider the action of $\CC^*$ on $T^*\SM(r,\alpha,\xi)$ given by dilatation on the fibers. Then there is a unique $\CC^*$ action on $W$ such that $\tilde{h}$ is $\CC^*$-equivariant,i.e., such that
$$\tilde{h}(E,E_\bullet,\lambda\varphi)=\lambda\cdot \tilde{h}(E,E_\bullet,\varphi)$$
\end{proposition}

\begin{proof}
The Hitchin map
$$H:\SM_{K(D)}(r,,\alpha,\xi) \longrightarrow W$$
is projective and has connected fibers (see, for example, \cite[Lemma 3.1 and Lemma 3.2]{AG18TorelliDH}). Then each holomorphic function $f:\SM_{K(D)}(r,\alpha,\xi)\longrightarrow \CC$ factors through $W$ and, as $W$ is affine, we obtain that 
$$\Spec(\Gamma(\SM_{K(D)}(r,\alpha,\xi))\cong \Spec(\Gamma(W))\cong W$$
Let $f:T^*\SM(r,\alpha,\xi) \longrightarrow \CC$. By \cite[V.(iii)]{Faltings}, we know that the codimension of the complement of $T^*\SM(r,\alpha,\xi)$ in $\SM_{K(D)}(r,\alpha,\xi)$ is at least $2$. As $\alpha$ is generic, $\SM_{K(D)}(r,\alpha,\xi)$ is smooth. Therefore, by Hartog's theorem $f$ extends to a holomorphic function $\overline{f}:\SM_{K(D)}(r,\alpha,\xi)\longrightarrow \CC$, so we conclude that $\Gamma(T^*\SM(r,\alpha,\xi))=\Gamma(\SM_{K(D)}(r,\alpha,\xi))$. Therefore, we obtain a map
$$\tilde{h}:T^*\SM(r,\alpha,\xi) \longrightarrow \Spec(\Gamma(T^*\SM(r,\alpha,\xi))) \cong W$$
The $\CC^*$ action on the cotangent bundle then descends to a unique action on $\Spec(\Gamma(T^*\SM(r,\alpha,\xi)))$ making $\tilde{h}$ a $\CC^*$-equivariant map.
\end{proof}

The previous Lemma allow us to recover the Hitchin map canonically with the corresponding $\CC^*$ action up to an automorphism of the Hitchin space. The $\CC^*$ action stratifies the space $W$ in subspaces corresponding to the elements whose rate of decay is at least $|\lambda|^k$ for each $k=2,\ldots, r$. Observe that, at first, the $\CC^*$ action only allows us to recover a filtration of $W$, but, in particular, we can recover uniquely the subspace of maximal decay $|\lambda|^r$, which corresponds to
$$W_r=H^0(K^rD^{r-1})\subsetneq W$$
In general, for $k>1$, let $W_k=H^0(K^kD^{k-1})$. Let
$$h_k:H^0(\SParEnd_0(E)\otimes K_X(D))\to W_k$$
be the composition of the Hitchin map $H:H^0(\SParEnd_0(E)\otimes K_X(D))\to W$ with the projection $W\twoheadrightarrow W_k$.

\begin{proposition}
\label{prop:recoverDualVariety}
The intersection $\SC:=\SD\cap W_r \subsetneq W_r$ has $n+1$ irreducible components
$$\SC=\SC_X \cup \bigcup_{x\in D} \SC_x$$
As $r\ge 2$, the linear series $|K^rD^{r-1}|$ is very ample and induces an embedding $X\subset \PP(W_r^*)$. Then $\PP(\SC_X)\subset \PP(W_r)$ is the dual variety of $X\subset \PP(W_r^*)$ and for each $x\in D$, $\PP(\SC_x)\subset \PP(W_r)$ is the dual variety of $x\hookrightarrow X \subset \PP(W_r^*)$.
\end{proposition}

\begin{proof}
A spectral curve $X_s$ corresponding to a point $s=s_r\in H^0(K^rD^{r-1})$ has equation $t^r+s_r(z)=0$. Therefore, it is singular precisely at the points $(z,t)=(x,0)$ where $x$ is a zero of order at least 2 of $s_r$. Observe that the equation is an equation on the points of the total space of $KD$ so, as in previous lemmata, here we are considering $s_r$ as a section of $K^rD^r$. Therefore, $s\in \SC$ if and only if $s_r\in H^0(K^rD^r(-2x)$ for some $x\in X$. As we already know that $s_r\in H^0(K^rD^{r-1})$ we have two possible cases
\begin{enumerate}
\item $s_r\in H^0(K^rD^{r-1}(-2x))$ for some $x\not\in D$
\item $s_r\in H^0(K^rD^{r-1}(-x))$ for some $x\in D$
\end{enumerate}
Therefore, we can write
\begin{multline*}
\SC=\bigcup_{x\in U} H^0(K^rD^{r-1}(-2x)) \cup \bigcup_{x\in D} H^0(K^rD^{r-1}(-x))\\
= \bigcup_{x\in X} H^0(K^rD^{r-1}(-2x)) \cup \bigcup_{x\in D} H^0(K^rD^{r-1}(-x))
\end{multline*}
Let us denote
\begin{eqnarray*}
\SC_X&=&\bigcup_{x\in X}H^0(K^rD^{r-1}(-2x))\\
\SC_x&=&H^0(K^rD^{r-1}(-x)) \quad \quad x\in D
\end{eqnarray*}
In the proof of Proposition \ref{prop:HitchinComponents} we already proved that $\SC_X$ and, obviously, $\SC_x$ are irreducible for every $x\in D$. Moreover, for $g\ge 2$, the Riemann-Roch computations carried out in the proof of Proposition \ref{prop:HitchinComponents} imply that $\SC_X$ and $\SC_x$ have both codimension $1$ in $W_r$ and they are distinct, so they are precisely the irreducible components of $\SD$.

For the second part of the Proposition, observe that if $X\subseteq \PP(W_r^*)$ is the embedding given by the linear system $|K^rD^{r-1}|$, then the set of hyperplanes in $\PP(W_r^*)$ which are tangent to $X$ at $x\in X$ is precisely $\PP(H^0(K^rD^{r-1}(-2x)))$. Therefore, the dual variety of $X$ is $\PP(\SC_X) \subset \PP(W_r)$. Furthermore, for every $x\in D$, the dual variety of $x\subset \PP(W_r^*)$ identifies with the set of hyperplanes passing through $x$, which is precisely $\PP(H^0(K^rD^{r-1}(-x)))$. Therefore, we conclude that the dual of $x\subset\PP(W_r^*)$ is $\PP(\SC_x)\subset \PP(W_r)$.
\end{proof}

Notice that for every $x\in D$, $\PP(\SC_x)$ is the dual variety of a point and $\PP(\SC_X)$ is the dual variety of a compact Riemann surface so, $\PP(\SC_X)\not\cong \PP(\SC_x)$ for all $x\in X$. For every $x\in D$, $\SC_x\subset W_r$ is an hyperplane. In particular, this allows us to identify canonically $\SC_X$ inside $\SC$ as the only irreducible component that is not an an hyperplane in $W_r$.

\begin{theorem}[Torelli theorem]
\label{theorem:Torelli}
Let $(X,D)$ and $(X',D')$ be two smooth projective curves of genus $g\ge 4$ and $g'\ge 4$ respectively with set of marked points $D\subset X$ and $D'\subset X'$. Let $\xi$ and $\xi'$ be line bundles over $X$ and $X'$ respectively, and let $\alpha$ and $\alpha'$ be full flag generic systems of weights over $(X,D)$ and $(X',D')$ respectively. Then if $\SM(X,r,\alpha,\xi)$ is isomorphic to $\SM(X',r',\alpha',\xi')$ then $r=r'$ and $(X,D)$ is isomorphic to $(X',D')$, i.e., there exists an isomorphism $X\cong X'$ sending $D$ to $D'$.
\end{theorem}

\begin{proof}
In order to simplify the notation, let $\SM=\SM(X,r,\alpha,\xi)$ and $\SM'=\SM(X',r',\alpha',\xi')$. First, let us prove that $r=r'$. If $\Phi:\SM\stackrel{\sim}{\longrightarrow}\SM'$ is an isomorphism, then there is an isomorphism $d(\Phi^{-1}):T^*\SM \stackrel{\sim}{\longrightarrow} T^*\SM'$ which is $\CC^*$ equivariant for the standard dilation action. By Proposition \ref{prop:HitchinGlobalFunctions}, there exist unique $\CC^*$ actions $\cdot$ and $\cdot'$ on $\Gamma(T^*\SM)$ and $\Gamma(T^*\SM')$ respectively that are compatible with the dilation on the fibers. Therefore, there must exist an algebraic $\CC^*$-equivariant isomorphism $f:\Gamma(T^*\SM) \stackrel{\sim}{\longrightarrow} \Gamma(T^*\SM')$ such that the following diagram commutes
\begin{eqnarray*}
\xymatrixcolsep{3pc}
\xymatrix{
T^*\SM \ar[r]^{d(\Phi^{-1})} \ar[d]_{\tilde{h}} & T^*\SM' \ar[d]^{\tilde{h}}\\
\Gamma(T^*\SM) \ar[r]^{f} & \Gamma(T^*\SM')
}
\end{eqnarray*}
As $f$ is $\CC^*$-equivariant, it must preserve the filtration by subspaces in terms of the decay and it must send the subspace of maximum decay $|\lambda|^r$ of $\Gamma(T^*\SM))$ to the subspace of maximum decay $|\lambda|^{r'}$ of $\Gamma(T^*\SM')$. In particular, the number of steps of the filtration must be the same. As the filtrations of $\Gamma(T^*\SM)$ and $\Gamma(T^*\SM')$ have $r-1$ and $r'-1$ steps respectively, then $r=r'$. 

By Proposition \ref{prop:HitchinGlobalFunctions}, there is an isomorphism $W\cong W'$, that we will denote by a slight abuse of notation as $f:W\to W'$, such that the following diagram commutes
\begin{eqnarray*}
\xymatrixcolsep{3pc}
\xymatrix{
T^*\SM \ar[r]^{d(\Phi^{-1})} \ar[d]_{H_0} & T^*\SM' \ar[d]^{H_0'}\\
W \ar[r]^{f} & W'
}
\end{eqnarray*}
and there exist unique $\CC^*$ actions on $W$ and $W'$ such that $H_0$ and $H_0'$ are $\CC^*$-equivariant. As $d(\Phi^{-1})$ is an isomorphism, it maps complete rational curves on $T^*\SM$ to complete rational curves on $T^*\SM'$. By Lemma \ref{lemma:DiscriminantCharacterization}, $f$ sends the locus of singular spectral curves $\SD\subset W$ to the locus of singular spectral curves $\SD'\subset W'$. On the other hand, the differential map $d(\Phi^{-1})$ is $\CC^*$-equivariant, so $f$ must be a $\CC^*$-equivariant map. Therefore, it must send the subspace of $W$ of elements with maximum decay $W_r$ to the subspace of $W'$ of elements with maximum decay $W_r'$. Let $f_r:W_r\to W_r'$ be the restriction of $f$ to $W_r$.

By definition of $W_r$, we know that $f_r$ is $\CC^*$-equivariant and homogeneous of degree $r$, so it must be linear, and it maps $\SC=\SD\cap W_r$ to $\SC'=\SD'\cap W_r'$. Let $\SC_X$ and $\SC_X'$ be the only components of $\SC$ and $\SC'$ respectively that are not hyperplanes. By Proposition \ref{prop:recoverDualVariety}, the dual variety of $\PP(\SC_X)$ in $\PP(W_r)$ is $X\subset \PP(W_r^*)$ and, similarly, the dual variety of $\PP(\SC_X')$ in $\PP(W_r')$ is $X'\subset \PP((W_r')^*)$, so $f$ induces an isomorphism $f^\vee:\PP(W_r^*)\to \PP((W_r')^*)$ that sends $X$ to $X'$. Moreover, the dual of the rest of the components $\PP(\SC_x)$ of $\PP(\SC)$ correspond to the divisor $D\subset X \subset \PP(W_r^*)$ and the dual of the components $\PP(\SC_x')$ of $\PP(\SC')$ correspond to the divisor $D'\subset X'\subset \PP((W_r')^*)$, so $f^\vee$ must send $D$ to $D'$. Therefore, $f^\vee$ induces an isomorphism $f^\vee:(X,D)\stackrel{\sim}{\longrightarrow}(X',D')$.

\end{proof}

\section{Basic transformations for quasi-parabolic vector bundles}
\label{section:Hecke}

Let $x\in X$ be a point. Given a vector bundle $E$ over $X$ and a subspace on the fiber $H\subseteq E|_x$, the Hecke transformation of $E$ at $x$ with respect to the subspace $H$ is defined as the subsheaf $\SH_{x}^H(E)\subseteq E$ fitting in the short exact sequence
$$0\longrightarrow \SH_{x}^H(E) \longrightarrow E \longrightarrow (E|_x/H)\otimes \SO_x \longrightarrow 0$$
this kind of transformations were first studied in \cite{NR78, HR04} and have been used broadly to study the geometry of the moduli spaces of vector bundles. Let $x\in D$ be a parabolic point. For each parabolic vector bundle $(E,E_\bullet)$ on $(X,D)$, each term in the parabolic filtration $E_{x,i}\subseteq E|_x$ for $1\le i\le l_x+1$ gives us a canonical choice for a linear subspace in the fiber $E|_x$, so we might define subsheaves $E_x^i\subseteq E$ through the Hecke transformation as
\begin{equation}
\label{eq:HeckeParabolicDef}
0\longrightarrow E_x^i \longrightarrow E \longrightarrow (E|_x/E_{x,i}) \otimes \SO_x \longrightarrow 0
\end{equation}
Note that for each $x\in D$ and each $i=1,\ldots, l_x+1$, these  subsheaves $E_x^i$ coincide with the jumps at the continuous parabolic filtration associated to $(E,E_\bullet)$
$$E_x^i = E^x_{\alpha_i(x)}$$
In fact, the Hecke transformation gives us a one to one correspondence between parabolic structures $\{E_{x,i}\}$ on $E$ and collections of  decreasing  sequences of subsheaves
$$E=E_x^1 \supsetneq E_x^2 \supsetneq \cdots \supsetneq E_x^{l_x} \supsetneq E_x^{l_x+1}=E(-x)$$
for every $x\in D$.

Let us restrict the short exact sequence \eqref{eq:HeckeParabolicDef} to the point $x$. If $f:E_x^i|_x \to E|_x$ is the induced map at the fiber, we get
\begin{eqnarray*}
\xymatrixrowsep{0.4pc}
\xymatrixcolsep{1pc}
\xymatrix{
0 \ar[r] & E|_x/E_{x,i} \otimes \SO_X(-x)|_x \ar@{=}[d] \ar[r] & E_x^i|_x  \ar[dr] \ar[rr]^f & &E|_x \ar[r] &E|_x/E_{x,i} \ar[r] & 0\\
& \op{Tor}(E|_x/E_{x,i},\SO_x) &&  E_{x,i} \ar[ur] \ar[dr] & &\\
& & 0 \ar[ur]  & & 0 & & 
}
\end{eqnarray*}
Observe that the tail of the filtration $E_{x,i}\supsetneq E_{x,i+1} \supset E_{x,l_x+1}=0$ of $E|_x$, induce a filtration of $E_x ^i|_x$
$$E_x^i|_x = f^{-1}(E_{x,i}) \supsetneq f^{-1}(E_{x,i+1}) \supsetneq \cdots \supsetneq f^{-1}(E_{x,l_x}) \supsetneq f^{-1}(0)=\frac{E|_x}{E_{x,i}} \otimes \SO_X(-x)|_x$$
on the other hand, the head of the filtration $E|_x=E_{x,1}\supsetneq E_{x,2} \supsetneq \cdots \supsetneq E_{x,i}$ induce canonically a filtration on $\frac{E|_x}{E_{x,i}}$
$$\frac{E|_x}{E_{x,i}}=\frac{E_{x,1}}{E_{x,i}} \supsetneq \frac{E_{x,2}}{E_{x,i}} \supsetneq \cdots \supsetneq \frac{E_{x.i}}{E_{x,i}}=0$$
thus, $E_x^i|_x$ gets an induced filtration at $x$ of the same length as that of $E|_x$
$$E_x^i|_x = f^{-1}(E_{x,i})  \supsetneq \! \cdots \!\supsetneq f^{-1}(E_{x,l_x}) \supsetneq \frac{E|_x}{E_{x,i}} \otimes \SO_X(-x)|_x \supsetneq \!\cdots \! \supsetneq \frac{E_{x,i-1}}{E_{x,i}}\otimes \SO_X(-x)|_x \supsetneq 0$$
On the other hand, $E_x^i|_y$ is canonically isomorphic to $E|_y$ for each $y\in D\backslash \{x\}$, thus inheriting its filtration. Therefore, for each $x\in D$ and each $1\le i\le l_x+1$, we can provide $E_x^i$ a canonical quasi-parabolic structure with the same number of steps as $(E,E_\bullet)$. In particular, if $(E,E_\bullet)$ is full flag, then the induced quasi-parabolic structure on $E_x^i$ is full flag. This ``rotation'' procedure -- also called by some authors elementary transformation of the parabolic bundle -- has been used in the literature as a fruitful way to induce correspondences between moduli spaces of parabolic vector bundles \cite{BY, IIS2, Inaba13}.

We call $E_x^2$ with the induced parabolic structure the Hecke transformation of $(E,E_\bullet)$ at $x$, and we will denote it by
$$(E_x^2,(E_x^2)_\bullet)=\SH_x(E,E_\bullet)$$
More generally, we will write
$$\SH_x^k(E,E_\bullet)=\overbrace{\SH_x\circ \cdots \circ \SH_x}^{k}(E,E_\bullet)$$
It is straightforward to check that for each $1 \le k\le l_x$ the quasi-parabolic bundle $\SH_x^k(E,E_\bullet)$ coincides with the vector bundle $E_x^{k+1}$ with the induced parabolic structure previously described. Also, by construction, for every quasi-parabolic vector bundle and every $x\in D$ the following relation holds
$$\SH_x^{l_x}(E,E_\bullet) = (E,E_\bullet)\otimes \SO_X(-x)$$
Moreover, it is clear that two Hecke transformations at two different parabolic points commute with each other. Let $H$ denote an effective divisor on $X$ supported on $D=\{x_1,\ldots,x_n\}$. If we take $H=\sum_{x\in D} h_x x$, then we define $\SH_H$ as the composition
$$\SH_H=\SH_{x_1}^{h_{x_1}} \circ \SH_{x_2}^{h_{x_2}} \circ \ldots \circ \SH_{x_n}^{h_{x_n}}$$

We can understand the Hecke transformation of a quasi-parabolic vector bundle in another equivalent way working directly over the filtration by subsheaves. Let $(E,E_\bullet)$ be a full flag parabolic vector bundle and let $x\in D$ be a parabolic point. We define the Hecke transformation of $(E,E_\bullet)$ at $x$ to be the parabolic vector bundle $\SH_x(E,E_\bullet)=(H,H_\bullet)$ obtained by taking the Hecke transformation of $E$ with respect to $E_{x,2}$ and ``rotating'' the parabolic structure at $x$ in the following way. We take
$$\begin{array}{rl}
\forall i=1,\ldots,r & H_x^i=E_x^{i+1}\\
\forall y \in D \, \backslash \{x\} \forall i=1,\ldots, r & H_y^i=\SH_x^{E_{x,2}}(E_y^i)
\end{array}$$
In particular, $H=E_x^2=\SH_x^{E_{x,2}}(E)$. For example, for $r=3$, $D=x+y$, we send the parabolic vector bundle
$$(E,E_\bullet)=\left \{\begin{array}{ccccccc}
E= E_x^1 & \supsetneq & E_x^2 & \supsetneq & E_x^3 & \supsetneq & E(-x)\\
E= E_y^1 & \supsetneq &  E_y^2 & \supsetneq & E_y^3 & \supsetneq & E(-y)
\end{array} \right \}$$
to
$$\SH_x(E,E_\bullet)=\left \{\begin{array}{ccccccc}
E_x^2 & \supsetneq & E_x^3 & \supsetneq & E(-x)=E_x^1(-x) & \supsetneq & E_x^2(-x)\\
E_x^2=\SH_x^{E_{x,2}}(E_y^1) & \supsetneq & \SH_x^{E_{x,2}}(E_y^2) & \supsetneq & \SH_x^{E_{x,3}}(E_y^3) & \supsetneq & E_x^2(-y)
\end{array} \right\}$$
Observe that if we choose weights $\alpha$ on a full flag quasi parabolic vector bundle $(E,E_\bullet)$, then we might maintain the same system of weights $\alpha$ on $\SH_x(E,E_\bullet)$ and, in that case
$$\pdeg_\alpha\left( \SH_x(E,E_\bullet) \right) =\pdeg_\alpha(E,E_\bullet)-1$$
Nevertheless, we will see that this is not a natural choice of weights, as it does not preserve stability.

\begin{lemma}
\label{lemma:unstableHecke}
Suppose that $g\ge 3$. Suppose that $d=\deg(\xi)$ and $r$ are coprime and let $\alpha$ be a generic concentrated system of weights. For every divisor $H$ with $0<H\le (r-1)D$ such that $d< |H|$, there exists at least a stable parabolic vector bundle $(E,E_\bullet)\in \SM(r,\alpha,\xi)$ over $(X,D)$ such that $\SH_H(E,E_\bullet)$ is $\alpha$-unstable.
\end{lemma}

\begin{proof}
Let $d=\deg(\xi)$. By tensoring with an appropriate line bundle, we can assume that $0\le d<r$. Brambila-Paz, Grzegorczyk and Newstead \cite[Theorem B (Theorem 6.3)]{BGN97} proved that for every genus $g\ge 2$ smooth projective curve and every $0\le d<r$, the space of stable vector bundles $E$ of rank $r$ and degree $d$ such that $H^0(E)\ge k$ (called the Brill-Nether locus and usually denoted by $B(r,d,k)$) is nonempty if $d>0$ and 
$$r\le d+(r-k)g$$
with $(r,d,k)\ne (r,r,r)$. As we are assuming that $d$ and $r$ are coprime, then $0<d$ and for $k=1$ and $g\ge 3$
$$d+(r-k)g-r\ge d+3(r-1)-r=d+2r-3\ge 2(r-1)>0$$
Then, there exists a stable vector bundle $E$ with rank $r$ and degree $d$ such that $H^0(E)>0$. 
As $H^0(E)>0$, $\SO_X$ is a subsheaf of $E$ and, saturating, there is a line bundle $L\subsetneq E$ with $0\le \deg(L)<\mu(E)$. Tensoring with a suitable degree zero line bundle, we might assume that $\det(E)\cong \xi$. The weights $\alpha$ are concentrated and rank and degree are coprime, so the stability of any parabolic vector bundle is equivalent to the stability of its underlying vector bundle. Therefore, for every choice of filtrations over $E|_x$, for $x\in D$, the parabolic vector bundle $(E,E_\bullet)$ is stable. In particular, we can choose a parabolic structure $(E,E_\bullet)$ such that $E_{x,r}=L|_x$ for every $x\in D$.

Then, $(E,E_\bullet)$ is stable and $L$ is a subsheaf of $E_x^k=\SH_x^{E_{x,k}}$ for every $k<r$. Therefore, $L$ is a subsheaf of $\SH_H(E,E_\bullet)$. Let $\overline{L}$ be its saturation. Then
$$\deg(\overline{L})\ge \deg(L)\ge 0$$
On the other hand, as $d< \left|H\right|$, 
$$\mu(\SH_x(E,E_\bullet))=\frac{d-|H|}{r} \le\frac{d-d}{r}=0\le \deg(\overline{L})$$
so the underlying vector bundle of $\SH_x(E,E_\bullet)$ is unstable. Therefore, as the parabolic weights are concentrated, $\SH_x(E,E_\bullet)$ is $\alpha$-unstable as a parabolic vector bundle.
\end{proof}

\begin{lemma}
\label{lemma:StrataModuliVB}
Let $X$ be a smooth complex projective curve of genus $g$. Let $r,s,k,d$ be integers such that $0<k<r$. Then if
$$g>\frac{r-1-s}{k}+1$$
then there exist a stable vector bundle $E$ of degree $d$ and rank $r$ and a subbundle $F\subsetneq E$ of rank $k$ such that
$$kd-r\deg(F)= s$$
\end{lemma}

\begin{proof}
By \cite[Remark 3.3]{BL98}, there exists a stable vector bundle $E$ such that
$$s=s_k(E)=k\deg(E)-r\max\{\deg(F) \, |\, F\subsetneq E \text{ subbundle with } \rk(F)=k\}$$
if for every $1\le i< k$
$$0<i(r-i)(g-1)-\frac{i}{k}(k(r-k)(g-1)-s+r-1)$$
As $k>0$, multiplying by $k/i>0$ yields that this is equivalent to proving that
$$0<k(r-i)(g-1)-k(r-k)(g-1)+s-r+1=k(k-i)(g-1)+s-r+1$$
But, as $1\le i< k$ we obtain
$$g>\frac{r-1-s}{k}+1\ge \frac{r-1-s}{k(k-i)}+1$$
for all $1\le i<k$ and the lemma follows.
\end{proof}

\begin{lemma}
\label{lemma:unstableHecke2}
Suppose that $g>3$. Suppose that $d=\deg(\xi)$ and $r$ are coprime with $0<d<r$ and let $\alpha$ be a generic concentrated system of weights. If $H$ is a divisor with $|H|=2d-r>0$, there exists at least a stable parabolic vector bundle $(E,E_\bullet)\in \SM(r,\alpha,\xi)$ over $(X,D)$ such that $\SH_H(E,E_\bullet)$ is $\alpha$-unstable.
\end{lemma}

\begin{proof}
As $0<|H|=2d-r$, we have $d>r/2$. In particular, as we assumed $r>d$, then $r\ge 3$. For every vector bundle $E$
$$\mu(\SH_H(E))=\frac{d-|H|}{r}=\frac{d-(2d-r)}{r}=\frac{r-d}{r}=1-\mu(E)$$
Let $k=r-|H|=r-(2d-r)=2(r-d)$ and let
$$d'=\left \lceil \frac{2(r-d)^2}{r} \right \rceil$$
It is easy to check that if $r\ge 3$ and $0<r-d<d<r$ then
\begin{equation}
\label{eq:unstableHecke2}
\frac{r-d}{r} < \frac{d'}{k} \le \frac{d}{r}
\end{equation}
Assume that there exists a stable vector bundle $E$ of rank $r$ and degree $d$ with a subbundle $F\subseteq E$ of rank $k=r-|H|=r-(2d-r)=2(r-d)$ and degree $d'$. As the parabolic weights are concentrated and rank and degree are coprime, parabolic stability is equivalent to stability of the underlying bundle for any choice of the filtrations $E_\bullet$. We have $\rk(F)=r-|H|$, so we can choose a parabolic structure $(E,E_\bullet)$ on $E$ such that $F|_x = E_{x,|H|}$ for every $x\in D$. Therefore, $F$ is a subsheaf of $\SH_H(E,E_\bullet)$. By inequality \eqref{eq:unstableHecke2}, the saturation of $F$ in $\SH_H(E,E_\bullet)$ is a destabilizing subsheaf of the underlying vector bundle of $\SH_H(E,E_\bullet)$. As the weights are concentrated, then $\SH_H(E,E_\bullet)$ is $\alpha$-unstable as a parabolic vector bundle.

In order to find the desired $E$ and $F$ we can apply Lemma \ref{lemma:StrataModuliVB} for $k=2(r-d)$ and $s=kd-rd'$. To guarantee the genus hypothesis of the Lemma, it is enough to show that
$$g> 3 \ge \frac{r-1-s}{2(r-d)}+1$$
Using  the bound $\lceil x \rceil < x+1$ on the $d'$ formula yields
$$\frac{r-1-s}{2(r-d)}+1 <\frac{r}{r-d}+r-2d+1$$
Multiplying by $r-d$ and reordering the factors, the desired inequality is then equivalent to
$$r+(r-d)(r-2d)-2(r-d) = r-(r-d)(2d-r+2)\le 0$$
Let
$$\left \{ \begin{array}{cc}
r=2\overline{r}\\
d=\overline{r}+\varepsilon
\end{array}\right.$$
Substituting in the above expression and reordering yields
$$r-(r-d)(2d-r+2)=-2\varepsilon(\overline{r}-\varepsilon-1)$$
which is clearly less or equal to $0$, as $\varepsilon>0$ and $\overline{r}-\varepsilon=r-d\ge 1$.
\end{proof}

Notice that preserving the same system of weights on the Hecke transformation is not a natural choice, but rather an imposition if we want to restrict ourselves to analyzing the stability with respect to a fixed set of parameters $\alpha$. In fact, if the ``rotation'' operation on the parabolic structure is held at the continuous filtration level, the following parabolic weights arise as the natural ones on $\SH_x(E,E_\bullet)$.

Given a system of weights $\alpha$ over $(X,D)$ and a divisor $H=\sum_{x\in D} h_x x$ with $0\le H\le (r-1)D$ we define $\SH_H(\alpha)$ to be the set of parameters satisfying
$$\SH_H(\alpha)_i(x)=\left \{ \begin{array}{ll}
\alpha_{i+h_x}(x)-\alpha_{1+h_x}(x) & i+h_x\le r\\
\alpha_{i+h_x-r}(x)-\alpha_{1+h_x}(x)+1 & i+h_x>r
\end{array} \right.$$
Let us prove that if a parabolic vector bundle $(E,E_\bullet)$ is $\alpha$-stable, then its Hecke transformation $\SH_H(E,E_\bullet)$ is $\SH_H(\alpha)$-stable. In order to do so, we will give yet another interpretation of the Hecke transformation in terms of the parabolic tensor product.

Let $0<H\le (r-1)D$ be an effective divisor, and let $\varepsilon(x)\in [0,1)$ be real numbers indexed by $x\in D$ such that
$$\alpha_{h_x}(x)<\varepsilon(x) \le \alpha_{1+h_x}(x)$$
Let $(\SO_X(-D),\SO_{X,\bullet}(-D)^{1-\varepsilon})$ be the parabolic line bundle obtained by giving $\SO_X(-D)$ the trivial filtration with weight $1-\varepsilon(x)$ at $x\in D$. Consider the parabolic vector bundle $(H,H_\bullet)=(E,E_\bullet)\otimes (\SO_X(-D),\SO_{X,\bullet}(-D)^{1-\varepsilon})$. By construction, for every $x\in D$ and every $\alpha\in \RR$
$$H_\alpha^x = E_{\alpha+\varepsilon(x)}^x$$
In particular, for $\alpha=0$, one gets
$$H_0^x=E_{\varepsilon(x)}^x =E_{\alpha_{1+h_x}(x)}^x=E_x^{1+h_x}$$
Therefore $H$ is the underlying vector bundle of $\SH_H(E,E_\bullet)$. A similar computation shows that, in fact
$$\SH_H(E,E_\bullet)=(E,E_\bullet)\otimes (\SO_X(-D),\SO_{X,\bullet}(-D)^{1-\varepsilon})$$
as quasi-parabolic vector bundles. Let us prove that for each admissible choice of $\varepsilon$, the right hand side is a stable parabolic vector bundle.

\begin{proposition}
Let $(E,E_\bullet)$ be a (semi)-stable parabolic vector bundle with system of weights $\alpha$, and let $(L,L_\bullet^\varepsilon)$ be a parabolic line bundle with system of weights $\varepsilon$. Then $(E,E_\bullet)\otimes (L,L_\bullet^\varepsilon)$ is stable for the induced parabolic structure.
\end{proposition}

\begin{proof}
We have that $(F,F_\bullet)\subset (E,E_\bullet)$ if and only if $(F,F_\bullet)\otimes (L,L_\bullet^\varepsilon)\subset (E,E_\bullet)\otimes (L,L_\bullet^\varepsilon)$, and it is straightforward to check that
$$\pdeg\left ( (F,F_\bullet)\otimes (L,L_\bullet^\varepsilon) \right) = \pdeg(F,F_\bullet)+\rk(F)\pdeg(L,L_\bullet^\varepsilon)$$
$$\pdeg\left ( (E,E_\bullet)\otimes (L,L_\bullet^\varepsilon) \right) = \pdeg(E,E_\bullet)+\rk(E)\pdeg(L,L_\bullet^\varepsilon)$$
Therefore,
$$\frac{\pdeg((E,E_\bullet)\otimes (L,L_\bullet^\varepsilon))}{\rk(E)}-\frac{\pdeg((F,F_\bullet)\otimes (L,L_\bullet^\varepsilon))}{\rk(F)}=\frac{\pdeg(E,E_\bullet)}{\rk(E)}-\frac{\pdeg(F,F_\bullet)}{\rk(F)}$$
so $(E,E_\bullet)\otimes (L,L_\bullet^\varepsilon)$ is (semi)stable if and only if $(E,E_\bullet)$ is (semi)stable.
\end{proof}

\begin{corollary}
A full flag parabolic vector bundle $(E,E_\bullet)$ is $\alpha$-(semi)stable if and only if $\SH_H(E,E_\bullet)$ is $\SH_H(\alpha)$-(semi)stable.
\end{corollary}

Thus, Hecke transformations preserve stability with respect to the natural induced system of weights, but Lemmas \ref{lemma:unstableHecke} and \ref{lemma:unstableHecke2} show that the induced system $\SH_H(\alpha)$ might not belong to the same stability chamber as the original one $\alpha$.

We can also describe an analogue of dualization in the quasi-parabolic context. Given a quasi-parabolic vector bundle $(E,E_\bullet)$ described as a set of decreasing filtrations
$$E|_x=E_{x,1} \supsetneq E_{x,2} \supsetneq \cdots \supsetneq E_{x,l_r} \supsetneq 0$$
for each $x\in D$, observe that if we take the dual of the corresponding spaces then we obtain
$$E^\vee|_x = E_{x,1}^\vee \twoheadrightarrow  E_{x,2}^\vee(-x) \twoheadrightarrow \cdots \twoheadrightarrow  (E_{x,l_x})^\vee(-x) \twoheadrightarrow 0$$
taking the kernels of the successive quotients (i.e., taking the corresponding annihilators in $E^\vee|_x$) we obtain
$$E^\vee|_x = \op{ann}(0) \supsetneq \op{ann}(E_{x,l_r}) \supsetneq \ldots \supsetneq \op{ann}(E_{x,2}) \supsetneq \op{ann}(E_{x,1})=0$$
which clearly provides us a quasi-parabolic structure over $E^\vee$ with the same number of steps. We will denote the vector bundle $E^\vee$ with this induced quasi-parabolic structure as $(E,E_\bullet)^\vee$ and we will call it its quasi-parabolic dual. Observe that if $(E,E_\bullet)$ is full flag, then $(E,E_\bullet)^\vee$ is also full flag. Notice that this definition of dual is different to the usual notion of parabolic dual of a parabolic vector bundle, described, for example in \cite{Biswas03}. Let us fix a system of weights $\alpha$ for $(E,E_\bullet)$. Biswas defines the parabolic dual of the parabolic vector bundle $(E,E_\bullet)$ in terms of the left continuous decreasing filtrations $E_\alpha$ as the parabolic vector bundle $(E,E_\bullet)^*$ obtained by
$$\left((E,E_\bullet)^*\right)_\alpha=\lim_{t\to \alpha^+} (E_{-1-t})^*$$
It is clear that the underlying vector bundle of $(E,E_\bullet)^*$ does not, in general, coincide with $E^\vee$. In fact, the underlying vector bundle $\left((E,E_\bullet)^*\right)_0$ depends on the choice of the parabolic weights $\alpha$. More precisely, they depend on whether $\alpha_1(x)=0$ for the points $x\in D$. If $\alpha_1(x)>0$ for each $x\in D$, we have 
$$\left((E,E_\bullet)^*\right)_0=(E_{-1})^\vee=E^\vee(-D)$$
In this case, it can be checked that the induced filtration on $E^\vee(-D)$ is precisely the one obtained by tensoring $(E,E_\bullet)^\vee$ with $\SO_X(-D)$.
One of the main advantages of the latter approach in conjunction to the definition of parabolic tensor product is that it allows us to work with sheaves of parabolic morphism in a way similar to the one used for regular vector bundles, as the sheaf of parabolic morphism (morphisms preserving the parabolic structure) from $(E,E_\bullet)$ to $(F,F_\bullet)$ simply becomes
$$\ParHom((E,E_\bullet),(F,F_\bullet))=(E,E_\bullet)^* \otimes (F,F_\bullet)$$

Suppose that $\alpha$ is a full flag system of weights with $\alpha_1(x)>0$ for all $x\in D$. If $(E,E_\bullet)$ is a stable (respectively semi-stable) parabolic vector bundle, then $(E,E_\bullet)^*$ is stable (respectively semi-stable) with respect to the following system of weights $\alpha^\vee$
$$\alpha^\vee_i(x)=1-\alpha_i(x)$$
Under these hypothesis on $\alpha$, $(E,E_\bullet)^*=(E,E_\bullet)^\vee\otimes \SO_X(-D)$ as quasi-parabolic vector bundles, so we just saw that if $\alpha_1(x)>0$ for all $x\in D$, then $(E,E_\bullet)$ is $\alpha$-stable if and only if $(E,E_\bullet)^\vee$ is $\alpha^\vee$-stable.

Notice that, in particular, if the system of weights $\alpha$ is concentrated, then $\alpha^\vee$ is also concentrated, so in the concentrated chamber $\alpha$-stability is equivalent to $\alpha^\vee$-stability.

Up to this point, we have studied three types of operations that can be performed on quasi-parabolic vector bundles $(E,E_\bullet)$ and the corresponding transformations on the systems of weights that must be done to ensure stability of the resulting parabolic vector bundle
\begin{itemize}
\item Tensor with a line bundle $(E,E_\bullet) \mapsto (E,E_\bullet)\otimes L$
\item Dualization $(E,E_\bullet)\mapsto (E,E_\bullet)^\vee$
\item Hecke transformations $(E,E_\bullet)\mapsto \SH_H(E,E_\bullet)$
\end{itemize}
Moreover, if $(E,E_\bullet)$ is a parabolic $\alpha$-(semi)stable vector bundle and $\sigma:X\to X$ is an automorphism of $X$ that sends $D$ to itself (not necessarily fixing each parabolic point), then the pullback $\sigma^*(E,E_\bullet)$ is a $\sigma^*\alpha$-(semi)stable parabolic vector bundle, where
$$\sigma^*\alpha_i(x)=\alpha_i(\sigma^{-1}(x))$$

These four transformations can be clearly extended canonically to families of $\alpha$-(semi)stable parabolic vector bundles, so we will denote the combinations of them as ``basic'' transformations of quasi-parabolic vector bundles.

\begin{definition}
Let $(X,D)$ be a Riemann surface with a set of marked points $D\subset X$. A basic transformation of a quasi-parabolic vector bundle is a tuple $T=(\sigma,s,L,H)$ consisting on
\begin{itemize}
\item An automorphism $\sigma:X\stackrel{\sim}{\longrightarrow} X$ that sends $D$ to itself (but does not necessarily fix any point of $D$)
\item A sign $s\in\{1,-1\}$.
\item A line bundle $L$ on $X$.
\item A divisor $H$ on $X$ such that $0\le H \le (r-1)D$.
\end{itemize}
\end{definition}

Given a quasi-parabolic vector bundle $(E,E_\bullet)$ and a basic transformation $T=(\sigma,s,L,H)$, let
$$T(E,E_\bullet)=\left\{ \begin{array}{ll}
\sigma^*\left(L\otimes \SH_H(E,E_\bullet)\right) & s=1\\
\sigma^*\left(L\otimes \SH_H(E,E_\bullet)\right)^\vee & s=-1
\end{array}\right. $$

If $\xi$ is a line bundle, we define
$$T(\xi)=\left\{ \begin{array}{ll}
\sigma^*\left(L^r\otimes \xi(-H)\right) & s=1\\
\sigma^*\left(L^r\otimes \xi(-H)\right)^\vee & s=-1
\end{array}\right. $$

Finally, if $\alpha$ is a rank $r$ system of weights over $(X,D)$, we define
$$T(\alpha)_i(x)=\left\{ \begin{array}{ll}
\SH_H(\alpha)_i(\sigma^{-1}(x)) & s=1\\
1-\SH_H(\alpha)_{r-i+1}(\sigma^{-1}(x))& s=-1
\end{array}\right. $$

Observe that the action of $T$ on the space of admissible systems of weights is stable under translations of the system in the following sense. Let $\varepsilon=(\varepsilon(x))_{x\in D} \in \RR^{|D|}$ such that for every $x\in D$ $-\alpha_1(x)\le \varepsilon(x)<1-\alpha_r(x)$. Consider the system of weights $\alpha[\varepsilon]$ defined as
$$\alpha[\varepsilon]_i(x)=\alpha_i(x)+\varepsilon(x)$$
we call $\alpha[\varepsilon]$ the translation of $\alpha$ by $\varepsilon$.

Then for any admissible type of a subbundle $\overline{n}'$
\begin{multline*}
s_{\min}(\alpha[\varepsilon],\overline{n})=r''\sum_{x\in D} \sum_{i=1}^r n_i'(x)(\alpha_i(x)+\varepsilon(x)) - r' \sum_{i=1}^r n_i''(x)(\alpha_i(x)+\varepsilon(x))\\
= r''\sum_{x\in D} \sum_{i=1}^r n_i'(x)\alpha_i(x) - r' \sum_{i=1}^r n_i''(x)\alpha_i(x) = s_{\min}(\alpha,\overline{n})
\end{multline*}
Therefore, $\alpha$-stability is completely equivalent to $\alpha[\varepsilon]$-stability. Let
$$\Delta=\{\alpha=(\alpha_i(x))\in [0,1)^{r|D|} : \forall x\in D\textup{ and } \forall i=1,\ldots,r-1 \, ,  \,\,\alpha_i(x)<\alpha_{i+1}(x)\}$$
be the space of systems of weights over $(X,D)$, and let $\Delta^+=\Delta \cap (0,1)^{r|D|}$. Let us define an equivalence relation $\sim$ on $\Delta$ as follows. $\alpha\sim \beta$ if and only if there exists some $\varepsilon=(\varepsilon(x))_{x\in D}$ such that for every $x\in D$ we have
$$-\alpha_1(x) \le \varepsilon(x)<1-\alpha_r(x)$$
and such that $\beta=\alpha[\varepsilon]$. Define $\tilde{\Delta}$ as the quotient $\Delta/\sim$. Clearly $\Delta/\sim = \Delta^+/\sim$.

Let $\alpha,\beta\in \Delta^0$ such that $\alpha\sim\beta$. Then for every basic transformation $T$ we have $T(\alpha)\sim T(\beta)$. Therefore, basic transformations act on $\tilde{\Delta}$. In particular, in $\tilde{\Delta}$ for every $x\in D$ and every $\alpha\in \Delta$
$$\SH_x^r(\alpha)\sim\alpha$$

By construction $(E,E_\bullet)$ is an $\alpha$-(semi)stable parabolic vector bundle with determinant $\xi$ if and only if $T(E,E_\bullet)$ is an $T(\alpha)$-(semi)stable parabolic vector bundle with determinant $T(\xi)$.

Basic transformations form a group $\ST$, where the product rule is the composition. We can give an explicit natural presentation, which is independent on whether we are making $\ST$ act on quasi-parabolic vector bundles, line bundles or weight systems. 

\begin{lemma}
\label{lem:compositionRules}  
The group of basic transformations $\ST$ is generated by
\begin{itemize}
\item $\Sigma_\sigma=(\sigma,1,\SO_X,0)$
\item $\SD^+=(\id,1,\SO_X,0)=\id_\ST$
\item $\SD^-=(\id,-1,\SO_X,0)$
\item $\ST_L=(\id,1,L,0)$
\item $\SH_H=(\id,1,\SO_X,H)$
\end{itemize}
And we have the following composition rules
\begin{enumerate}

\item $\Sigma_\sigma \circ \Sigma_{\tau}=\Sigma_{\sigma\circ \tau}$
\item $\SD^s\circ \SD^t=\SD^{st}$
\item $\ST_L\circ \ST_M=\ST_{L\otimes M}$
\item If $0\le H_i\le (r-1)D$ for $i=1,2$ then
$$\SH_{H_1}\circ \SH_{H_2}=\ST_{L_{H_1+H_2}}\circ\SH_{H_1+H_2-L_{H_1+H_2}}$$
where, given a divisor $F=\sum_{x\in D} f_x x$, we define
$$L_F=\sum_{x\in D} \left \lfloor \frac{f_x}{r} \right\rfloor x$$
\item $\Sigma_\sigma \circ \SD^s = \SD^s \circ \Sigma_\sigma$
\item $\Sigma_\sigma \circ \ST_L=\ST_{\sigma^*L}\circ \Sigma_\sigma$
\item $\Sigma_\sigma \circ \SH_H=\SH_{\sigma^*H} \circ \Sigma_\sigma$
\item $\SD^-\circ \ST_L = \ST_{L^{-1}}\circ \SD^-$
\item $\SD^-\circ \SH_H= \ST_{\SO_X(D)}\circ \SH_{rD-H} \circ \SD^-$, for $H>0$
\item $\ST_L\circ \SH_H = \SH_H\circ \ST_L$ 
\end{enumerate}
\end{lemma}

\begin{proof}
  Straightforward computation.
\end{proof}

From these composition rules, it is straightforward to compute the inverses of each generator
\begin{itemize}
\item $\Sigma_{\sigma}^{-1}=\Sigma_{\sigma^{-1}}$
\item $(\SD^s)^{-1}=\SD^s$
\item $\ST_L^{-1}=\ST_{L^{-1}}$
\item $\SH_H^{-1}= \ST_{\SO_X(D)} \circ \SH_{rD-H}$ for $H>0$.
\end{itemize}
Then, using the composition rules it is easy to check that the inverse of a basic transformation $T=(\sigma,s,L,H)$ for $H>0$ is
$$(\sigma,s,L,H)^{-1}=\left \{\begin{array}{ll}
(\sigma^{-1},1,\sigma^*L^{-1}(D),rD-\sigma^*H) & s=1\\
(\sigma^{-1},-1,\sigma^*L,\sigma^*H) & s=-1
\end{array}\right.$$
And the inverse for $H=0$ is
$$(\sigma,s,L,0)^{-1}=\left \{ \begin{array}{cc}
(\sigma^{-1},1,\sigma^*L^{-1},0) & s=1\\
(\sigma^{-1},-1,\sigma^*L,0) & s=-1
\end{array}\right\} = (\sigma^{-1},s,\sigma^*L^{-s},0)$$

With this presentation we can describe the abstract group structure of $\ST$.
\begin{proposition}
\label{prop:basicTransPresentation}
The group of basic transformations is isomorphic to the following semidirect product
$$\ST \cong \left ((\ZZ^{|D|}\times \Pic(X)) / \SG_D \right) \rtimes \left( \Aut(X,D) \times \ZZ/2\ZZ\right)$$
where
$$\SG_D=\{(rH,\SO_X(H)) | H \text{ supported on } D\}<\ZZ^{|D|}\times \Pic(X)$$
\end{proposition}

\begin{proof}
Let us consider the surjective map $\pi:\ST \twoheadrightarrow \langle \SD^-,\ST_L\rangle $ which sends a basic transformation $(\sigma,s,L,H)$ to $\Sigma_\sigma\circ \SD^s$. Let us prove that it is a group homomorphism. Let $(\sigma,s,L,H)$ and $(\sigma',s',L',H')$ be basic transformations. Then
$$(\sigma,s,L,H)\circ (\sigma',s',L',H')=(\sigma,s,\SO_X,0) \circ \ST_L \circ \SH_H \circ \Sigma_{\sigma'}\circ \SD^{s'} \circ (\id,1,L',H')$$
On the other hand, by properties (6) and (7), there exist $L_1$ and $H_1$ such that $\ST_L\circ \SH_H \circ \Sigma_{\sigma'} = \Sigma_{\sigma'}\circ \ST_{L_1}\circ \SH_{H_1}$. Similarly, by properties (8), (9) and (10) there exist $L_2$ and $H_2$ such that $\ST_{L_1}\circ \SH_{H_1} \circ \SD^{s'}=\SD^{s'}\circ \ST_{L_2} \circ \SH_{H_2}$, so we obtain that
$$(\sigma,s,L,H)\circ (\sigma',s',L',H')=(\sigma,s,\SO_X,0)\circ (\sigma',s',\SO_X,0) \circ (\id,1,L_2,H_2)\circ (\id,1,L',H')$$
Finally, applying (1)-(5) and property (10) we have that there exist $L_3$ and $H_3$ such that
$$(\sigma,s,\SO_X,0)\circ (\sigma',s',\SO_X,0) \circ (\id,1,L_2,H_2)\circ (\id,1,L',H')=(\sigma\sigma',ss',L_3,H_3)$$
Therefore
$$\pi((\sigma,s,L,H)\circ (\sigma',s',L',H'))=(\sigma\sigma',ss',\SO_X,0)=\pi(\sigma,s,L,H) \circ \pi(\sigma',s',L',H')$$
The kernel of this map coincides clearly with the subgroup $\langle \ST_L,\SH_H\rangle < \ST$ generated by $\ST_L$ and $\SH_H$, so it is normal and we have that
$$\ST \cong \langle \ST_L,\ST_H\rangle \rtimes \langle \Sigma_\sigma, \SD^- \rangle$$
On the other hand, by property (5) we know that $\Sigma_\sigma$ and $\SD^-$ commute, so
$$\langle \Sigma_\sigma, \SD^-\rangle \cong \Aut(X,x) \times \ZZ/2\ZZ$$
Therefore, we conclude that
\begin{equation}
\label{eq:basicTransPresentation1}
\ST\cong \langle \ST_L,\ST_H \rangle \rtimes \left( \Aut(X,x)\times \ZZ/2\ZZ\right )
\end{equation}
Finally, let us consider the following group
$$\SG_D=\{(rH,\SO_X(H)) | H \text{ supported on } D\}<\ZZ^{|D|}\times \Pic(X)$$
As generators  $\SH_x$ for $x\in D$ and $\ST_L$ commute and $\SH_x^r = \ST_{\SO_X(-x)}$ then we have
$$\langle \ST_L,\SH_H\rangle \cong (\ZZ^{|D|}\times \Pic(X)) / \SG_D$$
Combining this with equation \eqref{eq:basicTransPresentation1} the Proposition follows.
\end{proof}

It will be also useful to consider the subgroup $\ST^+<\ST$ consisting on basic transformations of the form $T=(\sigma,1,L,H)$, i.e., basic transformations that do not involve the dual. In particular, later on we will prove that every basic transformation $\ST$ acting on moduli spaces of rank $2$ is equivalent to a transformation in $\ST^+$ (see Lemma \ref{lemma:dualrk2}).

Finally, we briefly describe the analogues of these constructions for projective parabolic bundles. Given a parabolic projective bundle $(\PP,\PP_\bullet)$, let $(E,E_\bullet)$ be a reduction (it always exists by Lemma \ref{lemma:projectiveReduction})
$$(\PP,\PP_\bullet)\cong (\PP(E),\PP(E_\bullet))$$
we define
$$(\PP,\PP_\bullet)^\vee=\PP\left( (E,E_\bullet)^\vee \right)$$
$$\SH_H(\PP,\PP_\bullet) = \PP\left( \SH_H(E,E_\bullet) \right)$$
Any two reductions are related by tensorization with a line bundle. If $L$ is a line bundle, then
$$\left( (E,E_\bullet)\otimes L \right)^\vee=(E,E_\bullet)^\vee \otimes L^\vee$$
$$\SH_H\left((E,E_\bullet)\otimes L\right)=\SH_H(E,E_\bullet)\otimes L$$
Therefore,
$$\PP\left(\left( (E,E_\bullet)\otimes L \right)^\vee \right)=\PP\left((E,E_\bullet)^\vee \right)$$
$$\PP\left(\SH_H\left((E,E_\bullet)\otimes L\right) \right)=\PP\left(\SH_H(E,E_\bullet)\right)$$
So the definition of the dual or Hecke transformations are independent of the choice of the reduction.

\section{The algebra of parabolic endomorphisms}
\label{section:Algebra}

Let $P$ be the parabolic subgroup of $\GL(r,\CC)$ consisting on upper triangular matrices. Let $\SSS$ and $\SG$ be the group schemes over $X$ given by the following short exact sequences.
$$0\to \SSS \to \SSL(r,\CC)\times X \to \left (\SSL(r,\CC)/(P\cap \SSL(r,\CC)) \right)\otimes \SO_D \to 0$$
$$0\to \SG \to \GL(r,\CC)\times X \to (\GL(r,\CC)/P)\otimes \SO_D \to 0$$

Let $\parsl=\op{Lie}(\SSS)$ and $\pargl=\op{Lie}(\SG)$ denote the sheaves of Lie algebras of $\SSS$ and $\SG$ respectively.  Let $\Aut(\parsl)$ be the sheaf of groups of local algebra automorphisms of $\parsl$. Let $\Inn(\parsl)$ be the subsheaf of inner automorphisms, i.e., the image of the adjoint action $\op{Ad}: \SSS \to \Aut(\parsl)$. Let $\GL(\parsl)$ be the sheaf of local linear automorphisms of $\parsl$ as a vector bundle. Analogous notations will be used for $\pargl$.

As $\SSS$ is a group scheme over $X$, $\Aut(\parsl)$ is a group scheme over $X$ and $\Inn(\parsl)$ is a sub-group scheme over $X$.

Before engaging the main classification Lemma (Lemma \ref{lemma:isoParEnd}), let us prove some necessary results about linear maps of algebras of matrices. Through this section, given a ring $R$, let $\Mat_{n\times m}(R)$ be the $R$-module of $n\times m$ matrices with entries in $R$.

\begin{lemma}
\label{lemma:rk1Matrix}
Let $R$ be a commutative unique factorization domain (UFD). Let $M=(m_{ij})\in \op{Mat}_{n\times m}(R)$ be a matrix with entries in $R$. Then all the $2 \times 2$ minors of $M$ have null determinant in $R$ if and only if there exist matrices $A=(a_i)\in \op{Mat}_{n\times 1}(R)$ and $B=(b_i)\in \op{Mat}_{1\times m}(R)$ such that $M=AB$.
\end{lemma}

\begin{proof}
If $M=AB$, then for every pair $(i,j)$, $m_{ij}=a_ib_j$. Therefore, for every $i,k\in [1,n]$ and $j,l\in [1,m]$ with $i<k$ and $j<l$
$$
\left | \begin{matrix}
m_{ij} & m_{il}\\
m_{kj} & m_{kl}
\end{matrix} \right| = \left | \begin{matrix}
a_ib_j & a_ib_l\\
a_kb_j & a_kb_l
\end{matrix} \right|=a_ib_ja_kb_l-a_ib_la_kb_j=0
$$
On the other hand, suppose that every $2\times 2$ minor in $M$ has zero determinant. If $M$ is the zero matrix, it is the product of two zero vectors. Otherwise, let $m_{ij}$ be a nonzero element of $M$. By reordering rows and columns (i.e., permuting the elements of $A$ and $B$), we can assume without loss of generality that $m_{11}\ne 0$. Then for every $i,j>1$
$$\left | \begin{matrix}
m_{11} & m_{1j}\\
m_{i1} & m_{ij}
\end{matrix} \right|=0$$
Therefore $m_{11}m_{ij}=m_{i1}m_{1j}$. $R$ is a GCD domain, so great common divisors exist and are unique up to product by units. Then $m_{11} | \op{GCD}_{j>1} (m_{i1}m_{1j})=m_{i1} \op{GCD}_{j>1}(m_{1j})$ for every $i>1$. We conclude that
$$m_{11} | \op{GCD}_{i>1} \left ( m_{i1} \op{GCD}_{j>1}(m_{1j}) \right) = \op{GCD}_{i>1}(m_{i1}) \op{GCD}_{j>1}(m_{1j})$$
As $R$ is a UFD, there exists a decomposition $m_{11}=a_1b_1$ such that $a_1|\op{GCD}_{j>1}(m_{1j})$ and $b_1|\op{GCD}_{i>1}(m_{i1})$. As $a_1|m_{1j}$ for every $j>1$, there must exist an element $b_j\in R$ such that $m_{1j}=a_1b_j$. Similarly, for every $i>1$, $b_1|m_{i1}$, so there must exist an element $a_i\in R$ such that $m_{i1}=a_ib_1$. Finally, for every $i,j>1$, as $m_{11}m_{ij}=m_{i1}m_{1j}$ yields
$$a_1b_1 m_{ij} = a_i b_1 a_1 b_j$$
As $a_1,b_1\ne 0$ and $R$ is a commutative UFD (and, in particular, it is integral), $m_{ij}=a_ib_j$ for every $i,j>1$. As the latter holds also for $i=1$ or $j=1$ by construction, then letting $A=(a_i)$ and $B=(b_j)$ yields $M=AB$ as desired.
\end{proof}

\begin{lemma}
\label{lemma:uniquenessDecomposition}
If $R$ is a field and $M=(m_{ij})\in \op{Mat}_{n\times m}(R)$ is a nonzero matrix such that all the $2\times 2$ minors have zero determinant, then the decomposition $M=AB$ stated by the previous lemma is unique in the sense that if $M=AB=A'B'$ for some matrices $A=(a_i), A'=(a'_i) \in \op{Mat}_{n\times 1}(R)$ and $B=(b_i), B'=(b'_i) \in \op{Mat}_{1\times m}(R)$ then there exists a nonzero $\rho\in R$ such that $A'=\rho A$ and $B'=\rho^{-1}B$.
\end{lemma}

\begin{proof}
Let $m_{ij}$ be a nonzero element of $M$. Then the $i$-th row of $M$ is nonzero and we have
$$a_i B = a'_i B'$$
with $a_i\ne 0$ and $a'_i\ne 0$. Then $a'_i$ is invertible and we get that $B'=\frac{a_i}{a'_i} B$. Similarly, as the $j$-th column of $M$ is nonzero we get
$$Ab_j=A'b'_j$$
with $b_j\ne 0$ and $b'_j\ne 0$. Then $b'_j$ is invertible and we get that $A'=\frac{b_j}{b'_j} A$. Finally, let $\rho=\frac{b_j}{b'_j}$ and note that as $m_{ij}=a_ib_j=a'_ib'_j \ne 0$, one gets
$$\frac{a_i}{a'_i}=\frac{m_{ij}/b_j}{m_{ij}/b'_j} = \frac{b'_j}{b_j}=\rho^{-1}$$
\end{proof}

\begin{remark}
If $n=m$, then we can rewrite the nullity condition for the minors of $M$ in a more compact way. For any matrix $M\in \Mat_{n\times n}(R)$, all the $2 \times 2$ minors of $M$ have null determinant in $R$ if and only if
$$\wedge^2 M =0$$
\end{remark}

We will introduce some notations that will be useful in order to work with linear morphisms between algebras of matrices.

Let us consider a bijection $\sigma:[1,n]\times [1,m]\to [1,n']\times [1,m']$. Abusing the notation, let
$$\sigma:\Mat_{n\times m}(R) \longrightarrow \Mat_{n'\times m'}(R)$$
be the isomorphism that sends a matrix $M=(m_{ij})\in \Mat_{n\times m}(R)$ to the $n' \times m'$ matrix whose entry $(i,j)$ is
$$\left(\sigma(M)\right)_{ij}=m_{\sigma^{-1}(i,j)}$$

In particular, given a bijection $\tau:[1,n]\times [1,m] \to [1,nm]\times \{1\}=[1,nm]$ and a matrix $M\in \Mat_{n\times m}(R)$, $\tau(M)\in \Mat_{nm\times 1}(R)\cong R^{nm}$ is the column vector obtained by placing all entries of $M$ in a column using the bijection $\tau$. Reciprocally, given such vector $V\in \Mat_{nm\times 1}(R)\cong R^{nm}$, then $\tau^{-1}(V)$ is the corresponding matrix.

In order to simplify the notation, from now on, let us fix once and for all the bijection $\tau:[1,n]^2\to [1,n^2]$ that places the entries of the matrix in row order, i.e.
$$\tau(i,j)=(i-1)n+j$$

We will also fix the bijection $\iota:[1,n]^2 \to [1,n]^2$ sending $\iota(i,j)=(j,i)$, so that for every matrix $M\in \Mat_{n\times n}(R)$
$$\iota(M)=M^t$$

\begin{lemma}
\label{lemma:conjCharacterization}
Let $R$ be a UFD. For every $n>0$ there exists a bijection
$$\sigma:[1,n^2]^2\times [1,n^2]^2$$
such that given any matrix $M\in \GL(\Mat_{n\times n}(R)) \stackrel{\tau} \cong \GL_{n^2}(R)$, $M$ is the matrix associated to a linear transformation of the form
$$X \mapsto AXB$$
for some $A,B\in \Mat_{n\times n} (R)$ if and only if
$$\wedge^2\left( \sigma(M)\right)=0$$
In that case, we will denote $M=\SM_{A,B}$
\end{lemma}

\begin{proof}
The matrix $M\in \GL_{n^2}(R)$ induced by the given linear transformation is given by
\begin{eqnarray*}
\xymatrixcolsep{3pc}
\xymatrixrowsep{0.05pc}
\xymatrix{
R^{n^2} \ar[r] & R^{n^2} \\
V \ar@{|->}[r] & \tau(A \tau^{-1}(V) B )
}
\end{eqnarray*}
For the bijection $\tau$ chosen above, it is straightforward to see that
$$M=A\otimes B^t$$
One just has to check that the morphisms
$$\End(R^n)\cong (R^n)^* \otimes R^n \longrightarrow \End(R^n)\cong (R^n)^*\otimes R^n$$
obtained by composing on the left with $A\in \End(R^n)$ or on the right with $B\in \End(R^n)$ correspond to
$$\id \otimes A:(R^n)^* \otimes R^n \longrightarrow (R^n)^* \otimes R^n$$
and
$$B^t\otimes \id:(R^n)^* \otimes R^n \longrightarrow (R^n)^* \otimes R^n$$
respectively, so the morphism represented by $M$ is just $B^t\otimes A$. In order to write the matrix for the morphism, we need to select a basis for $(R^n)^*\otimes R^n$. The choice of $\tau$ corresponds to selecting the basis of $(R^n)^* \otimes R$ in row order, so the matrix $M$ in the basis induced by the isomorphism $\tau$ is $A\otimes B^t$.

By definition of tensor product, the entries of the matrix $A\otimes B^t$ are all the possible products $a_{ij}b_{kl}$ of an entry $a_{ij}$ of $A$ and an entry $b_{kl}$ of $B$ in a fixed order depending only on the dimension $n$. Therefore, there exists a fixed bijection $\sigma:[1,n^2]^2\times [1,n^2]^2$ such that
$$\sigma \left(A\otimes B^t \right) = \tau(A) \cdot (\tau(B))^t$$

Therefore, the set of matrices $M\in \GL_{n^2}(R)$ for which there exist $A,B\in \Mat_{n\times n}(R)$ such that
$$M(V) = \tau \left( A \tau^{-1}(V) B \right)$$
is the set of matrices $M$ such that there exist vectors $\tau(A),\tau(B) \in R^{n^2}$ such that
$$\sigma(M)=\tau(A)\cdot (\tau(B))^t $$
By Lemma \ref{lemma:rk1Matrix}, such vectors exist if and only if
$$\wedge^2(\sigma(M))=0$$
\end{proof}

\begin{corollary}
\label{cor:innerCharacterization}
Let $R$ be a UFD and let $\sigma$ be the bijection given by the previous lemma. Then $M=(m_{\alpha,\beta})\in \GL_{n^2}(R)$ is the matrix of an inner transformation
$$X\mapsto A X A^{-1}$$
for some $A\in \GL_n(R)$ if and only if $\wedge^2(\sigma(M))=0$ and for every $i,j\in [1,n]$
$$\sum_{k=1}^n m_{\sigma^{-1}(\tau(i,k),\tau(k,j))} = \sum_{k=1}^n m_{\sigma^{-1}(\tau(j,k),\tau(k,i))} = \delta_{ij}$$
\end{corollary}

\begin{proof}
By the lemma, if $\wedge^2(\sigma(M))=0$ then there exist matrices $A,B\in \Mat_{n\times n}(R)$ such that $M$ is the map induced by
$$X\mapsto AXB$$
then $M$ is an inner transformation if and only if $A$ and $B$ are inverses, i.e., if and only if $AB=BA=I$, where $I$ is the identity matrix. This holds if and only if for every $i,j=1,\ldots,n$
\begin{equation}
\label{eq:innerCharacterization1}
\sum_{k=1}^n a_{ik} b_{kj} = \sum_{k=1}^n b_{ik} a_{kj}= \delta_{ij}
\end{equation}
On the other hand, as
$$\sigma(M)=\tau(A) \cdot (\tau(B))^t $$
then for every $i,j,k,l=1,\ldots,n$
$$a_{ij}b_{kl}=m_{\sigma^{-1}(\tau(i,j), \tau(k,l))}$$
so equality \eqref{eq:innerCharacterization1} holds if and only if
$$\sum_{k=1}^n m_{\sigma^{-1}(\tau(i,k),\tau(k,j))} = \sum_{k=1}^n m_{\sigma^{-1}(\tau(j,k),\tau(k,i))} = \delta_{ij}$$
Reciprocally, if $M$ is an inner transformation, $\wedge^2(\sigma(M))=0$ and
$$\sigma(M)=\tau(A) \cdot (\tau(A^{-1}))^t $$
so for every $i,j,k,l=1,\ldots,n$
$$a_{ij}(A^{-1})_{kl}=m_{\sigma^{-1}(\tau(i,j), \tau(k,l))}$$
Then the corollary follows from
$$\sum_{k=1}^n a_{ik} (A^{-1})_{kj} = \sum_{k=1}^n (A^{-1})_{ik} a_{kj}= \delta_{ij}$$
\end{proof}

Note that if $R$ is a field, Lemma \ref{lemma:uniquenessDecomposition} implies that if $M$ is the matrix of an inner transformation, then the matrix $A$ is uniquely determined up to product by nonzero elements of $R$.

Now let $R$ be a local principal ideal domain which is not a field ( i.e., a discrete valuation ring). For example, within the scope of this article, the following Lemmas will be applied to the local ring of a smooth complex projective curve $R=\SO_{X,x}$. Let $\fm$ be the maximal ideal in $R$ and let $K=\op{Frac}(R)$ be the field of fractions. As $R$ is a principal domain, $\fm=(z)$ for some $z\in R$. We will denote by
$$\nu_z:K\to \ZZ$$
the single discrete valuation on $K$ extending the canonical $z$-valuation of the elements in $R$, i.e., the only possible discrete valuation for which $R=\{a\in K : \nu_z(a)\ge 0\}$. Let $\ParEnd_n(R)\subset \Mat_{n\times n}(R)$ be the $R$-module of $n\times n$ matrices whose elements below the diagonal are multiples of $z$, i.e., the $R$-module consisting of matrices of the form
$$\left ( \begin{array}{cccc}
a_{11} & a_{12} & \cdots & a_{1n} \\
za_{21} & a_{22} & \cdots & a_{2n} \\
\vdots & \vdots & \ddots & \vdots \\
za_{n1} & za_{n2} & \cdots & a_{nn}
\end{array} \right )$$
where $a_{ij}\in R$. It is clear that $\ParEnd_n(R)$ forms a sub $R$-algebra of $\Mat_{n\times n}(R)$. If we suppose that $z\ne 0$ (i.e., that $R$ is not a field), then as an $R$-module, $\ParEnd_n(R)$ is isomorphic to $\Mat_{n\times n}(R)$, but they are not isomorphic as $R$-algebras.

Later on we will have to work with this kind of isomorphisms with a little more generality, so it is convenient to fix some general notation. Let us consider a formal sum of indexes in $[1,n]\times [1,m]$
$$\Xi=\sum_{(i,j)\in [1,n]\times [1,m]} \Xi_{ij} \cdot (i,j)\in \ZZ\left ([1,n]\times [1,m] \right)$$
Then we denote by $\SZ_\Xi:\Mat_{n\times m}(K) \cong \Mat_{n\times m}(K)$ the isomorphism of $K$-modules that sends a matrix $M=(m_{ij})$ to the matrix $\SZ_\Xi(M)$ whose element $(i,j)$ is
$$\SZ_\Xi(M)_{ij}=z^{\Xi_{ij}}m_{ij}$$

From the definition, it is clear that
\begin{eqnarray*}
\xymatrixrowsep{0.05pc}
\xymatrixcolsep{0.3pc}
\xymatrix{
\SZ&:&\ZZ\left ( [1,n]\times [1,m] \right) \ar[rrrr] &&&& \GL(\Mat_{n\times m}(K))\\
&& \Xi \ar@{|->}[rrrr] &&&& \SZ_{\Xi}
}
\end{eqnarray*}
is a group homomorphism.

Let $\Xi_{T}=\sum_{1\le j<i\le n} (i,j)$ be the sum of indexes below the diagonal. Then it is clear that the restriction of $\SZ_{\Xi_T}:\Mat_{n\times n}(K) \to \Mat_{n\times n}(K)$ to $\Mat_{n\times n}(R)$ is precisely the isomorphism
$$\SZ_{\Xi_T}:\Mat_{n\times n}(R) \cong \ParEnd_n(R)$$

Using the isomorphism $\tau:\Mat_{n\times n}(K)\cong K^{n^2}$ we can compute the matrix $Z_\Xi$ for the isomorphism $\tau \circ \SZ_{\Xi} \circ \tau^{-1}$. For every $V\in K^{n^2}$ let
$$V_\Xi=Z_\Xi V=  \tau(\SZ_\Xi(\tau^{-1}(V)))$$
Then, by definition of $\SZ_\Xi$, if $V_\Xi=(v_{\Xi,i})_i$ then
$$v_{\Xi,i}=z^{\Xi_{\tau^{-1}(i)}} v_i$$
Given a bijection $\sigma:[1,n]\times [1,m] \to [1,n']\times [1,m']$, let us denote
$$\sigma(\Xi)= \sum_{i,j} \Xi_{ij} \sigma(i,j) \in \ZZ [1,n']\times [1,m']$$
Then, $\tau \circ \SZ_\Xi \circ \tau^{-1}=\SZ_{\tau(\Xi)}$ and the matrix $Z_\Xi$ is the diagonal matrix
$$Z_\Xi=\diag\left (z^{\Xi_{\tau^{-1}(i)}} \right )$$

\begin{lemma}
\label{lemma:parConjCharacterization}
Let $R$ be a local principal ideal domain which is not a field. Let $n>0$ and let $\sigma$ be the bijection given in Lemma \ref{lemma:conjCharacterization}. There exists a formal sum of indexes
$$\Xi=\sum \Xi_{ij}(i,j) \in \ZZ[1,n^2]^2$$
with $-1\le \Xi_{ij}\le 1$ such that given any matrix $M\in \GL_{n^2}(R) \cong \GL(\ParEnd_n(R))$, $M$ is the matrix associated to a linear transformation of the form
$$X\mapsto AXB$$
for some $A,B\in \Mat_{n\times n}(K)$ if and only if
$$\wedge^2(\sigma(\SZ_{-\Xi}(M)))=0$$
Moreover, if $\SZ_{-\Xi}(M)\in \Mat_{n^2 \times n^2}(R)$, then $A$ and $B$ can be chosen in $\Mat_{n\times n}(R)$.
\end{lemma}

\begin{proof}
Let $M\in \GL_{n^2}(R)$ be the matrix associated to a map $X\mapsto AXB$. Then it sends a vector $V\in R^{n^2}$ to
$$MV=\tau\left( \SZ_{-\Xi_T}(A \SZ_{\Xi_T}(\tau^{-1}(V))B) \right)$$
Then we can view $M$ as the restriction to $R^{n^2}$ of the composition of the following morphisms
\begin{eqnarray*}
\xymatrixrowsep{3pc}
\xymatrixcolsep{6pc}
\xymatrix{
K^{n^2} \ar[r]^{M} \ar[d]_{\tau \circ \SZ_{\Xi_T} \circ \tau^{-1}} & K^{n^2} \\
K^{n^2} \ar[r]^{\SM_{A,B}} & K^{n^2} \ar[u]_{\tau \circ \SZ_{-\Xi_T} \circ \tau^{-1}}
}
\end{eqnarray*}
By the computations carried away in the previous lemmata, the matrix $M$ is the product
$$M=Z_{-\Xi_T} \left( A \otimes B^t \right) Z_{\Xi_T}$$
We will see that then there exists a formal sum of indexes $\Xi \in \ZZ[1,n^2]^2$ such that
$$M=\SZ_{\Xi}(A\otimes B^t)$$

For any $\Xi\in \ZZ[1,n]^2$, taking the product on the left by $Z_{\Xi}=\diag(z^{\Xi_{\tau^{-1}(i)}})$ is equivalent to multiplying the $i$-th row of the matrix by $z^{\Xi_{\tau^{-1}(i)}}$ for each $i=1,\ldots,n^2$, so if we set
$$\Xi_l=\sum_{i,j=1}^{n^2} \Xi_{\tau^{-1}(i)} (i,j)$$
for every matrix $N\in \Mat_{n^2\times n^2}(K)$
$$Z_{\Xi}N=\SZ_{\Xi_l}(N)$$
Similarly, product on the right by $Z_{\Xi}$ is equivalent to multiplying the $i$-th column of the matrix by $z^{\Xi_{\tau^{-1}(i)}}$ for each $i=1,\ldots,n^2$, so if we set
$$\Xi_r=\sum_{i,j=1}^{n^2} \Xi_{\tau^{-1}(j)}(i,j)$$
for every matrix $N\in \Mat_{n^2\times n^2}(K)$ yields
$$NZ_{\Xi}=\SZ_{\Xi_r}(N)$$
Therefore, setting
$$\Xi=-(\Xi_T)_l+(\Xi_T)_r$$
we conclude that
$$M=\SZ_{\Xi}(A\otimes B^t)$$
Let us check that $-1\le \Xi_{\alpha,\beta} \le 1$. For each $(\alpha,\beta)=(\tau(i,j),\tau(k,l))$ yields
$$-((\Xi_T)_l)_{\alpha,\beta} = -(\Xi_T)_{i,j} = \left \{ \begin{array}{cc}
-1 & j<i\\
0 & j\ge i
\end{array} \right. $$
$$((\Xi_T)_r)_{\alpha,\beta} = (\Xi_T)_{k,l} = \left \{ \begin{array}{cc}
1 & l<k\\
0 & l\ge k
\end{array} \right. $$
So it yields $-1\le \Xi_{\alpha,\beta} \le 1$. As an example, we show the matrix representing $\Xi$ for $n=4$
$$\left ( \begin{array}{cccc|cccc|cccc|cccc}
0 & 0 & 0 & 0 & 1 & 0 & 0 & 0 & 1 & 1 & 0 & 0 & 1 & 1 & 1 & 0\\
0 & 0 & 0 & 0 & 1 & 0 & 0 & 0 & 1 & 1 & 0 & 0 & 1 & 1 & 1 & 0\\
0 & 0 & 0 & 0 & 1 & 0 & 0 & 0 & 1 & 1 & 0 & 0 & 1 & 1 & 1 & 0\\
0 & 0 & 0 & 0 & 1 & 0 & 0 & 0 & 1 & 1 & 0 & 0 & 1 & 1 & 1 & 0\\
\hline
-1 & -1 & -1 & -1 & 0 & -1 & -1 & -1 & 0 & 0 & -1 & -1 & 0 & 0 & 0 & -1\\
0 & 0 & 0 & 0 & 1 & 0 & 0 & 0 & 1 & 1 & 0 & 0 & 1 & 1 & 1 & 0\\
0 & 0 & 0 & 0 & 1 & 0 & 0 & 0 & 1 & 1 & 0 & 0 & 1 & 1 & 1 & 0\\
0 & 0 & 0 & 0 & 1 & 0 & 0 & 0 & 1 & 1 & 0 & 0 & 1 & 1 & 1 & 0\\
\hline
-1 & -1 & -1 & -1 & 0 & -1 & -1 & -1 & 0 & 0 & -1 & -1 & 0 & 0 & 0 & -1\\
-1 & -1 & -1 & -1 & 0 & -1 & -1 & -1 & 0 & 0 & -1 & -1 & 0 & 0 & 0 & -1\\
0 & 0 & 0 & 0 & 1 & 0 & 0 & 0 & 1 & 1 & 0 & 0 & 1 & 1 & 1 & 0\\
0 & 0 & 0 & 0 & 1 & 0 & 0 & 0 & 1 & 1 & 0 & 0 & 1 & 1 & 1 & 0\\
\hline
-1 & -1 & -1 & -1 & 0 & -1 & -1 & -1 & 0 & 0 & -1 & -1 & 0 & 0 & 0 & -1\\
-1 & -1 & -1 & -1 & 0 & -1 & -1 & -1 & 0 & 0 & -1 & -1 & 0 & 0 & 0 & -1\\
-1 & -1 & -1 & -1 & 0 & -1 & -1 & -1 & 0 & 0 & -1 & -1 & 0 & 0 & 0 & -1\\
0 & 0 & 0 & 0 & 1 & 0 & 0 & 0 & 1 & 1 & 0 & 0 & 1 & 1 & 1 & 0\\
\end{array} \right)$$

Given a matrix $M\in \GL_{n^2}(R)$, in general $\SZ_{-\Xi}(M)\in \Mat_{n^2 \times n^2}(K)$. Following the proof of Lemma \ref{lemma:conjCharacterization}, there exist matrices $A,B\in \Mat_{n\times n}(K)$ such that $\SZ_{-\Xi}(M)=A\otimes B^t$ if and only if
$$\wedge^2(\sigma(\SZ_{-\Xi}(M)))=0$$
moreover, if $\SZ_{-\Xi}(M)\in \Mat_{n^2\times n^2}(R)$, then as $R$ is a principal ideal domain then if
$$\wedge^2(\sigma(\SZ_{-\Xi}(M)))=0$$
there exist $A,B\in \Mat_{n\times n}(R)$ such that $\SZ_{-\Xi}(M)=A\otimes B^t$.
\end{proof}

Similarly to the non-parabolic case, we will denote by
$$\SMP_{A,B}\in \Mat_{n^2\times n^2}(K)$$
the matrix associated to a map $X\mapsto A X B$ for $A,B\in \Mat_{n\times n}(K)$. More explicitly,  for every $V\in K^{n^2}$, let
$$\SMP_{A,B}V=\tau\left( \SZ_{-\Xi_T}(A \SZ_{\Xi_T}(\tau^{-1}(V))B) \right)$$
Note that, in general, if $A,B\in \Mat_{n\times n}(K)$, $\SMP_{A,B}V \in K^{n^2}$ even if $V\in R^{n^2}$. If $\SMP_{A,B} \in \GL_{n^2}^2(R)$, then this imposes some conditions on the structure of $A$ and $B$.

\begin{lemma}
\label{lemma:uniquenessParDecomposition}
If $M=\SMP_{A,B}=\SMP_{A',B'}$ is a nonzero matrix for some $A,A',B,B'\in \Mat_{n^2 \times n^2}(K)$, then there exists a nonzero $\rho\in K$ such that $A'=\rho A$ and $B'=\rho^{-1}B$.
\end{lemma}

\begin{proof}
From the previous lemma, yields
$$\sigma(\SZ_{-\Xi'}(M))=\tau(A) \cdot \tau(B)^t = \tau(A') \cdot \tau(B)^t $$
and now we apply Lemma \ref{lemma:uniquenessDecomposition}.
\end{proof}

\begin{lemma}
Suppose that there exist matrices $A,B\in \Mat_{n\times n}(K)$ such that $M=\SMP_{A,B}\in \GL_{n^2}(R)$. Then there exist $A',B'\in \Mat_{n\times n}(R)$ such that
$$M=\SMP_{A,B/z}=\SMP_{A',B'}/z$$
\end{lemma}

\begin{proof}
By the Lemma \ref{lemma:parConjCharacterization}, $\wedge^2(\sigma(\SZ_{-\Xi}(M)))=0$. Then $\wedge^2(\sigma(\SZ_{-\Xi}(zM)))=0$.
As $-1\le \Xi_{\alpha\beta} \le 1$ for all $\alpha,\beta=1,\ldots,n^2$, then $\SZ_{-\Xi}(zM) \in \Mat_{n^2\times n^2}(R)$. Therefore, there exist $A',B'\in \Mat_{n\times n}(R)$ such that $zM$ is the matrix $\SMP_{A',B'}$. The result yields dividing the matrix by $z$.
\end{proof}

\begin{corollary}
\label{cor:innerOrder1Pole}
Let $A\in \GL_{n}(K)$ be a matrix such that $\SMP_{A,A^{-1}}\in \GL_{n^2}(R)$. Then, there exist nonzero matrices $A',B'\in \Mat_{n\times n}(R)$ such that $B'/z$ is the inverse of $A'$ in $\GL_{n^2}(K)$
$$\SMP_{A,A^{-1}}=\SMP_{A',B'/z}$$
\end{corollary}

\begin{proof}
By the previous lemma, there exist nonzero $A',B'\in \GL_{n^2}(R)$ such that
$$\SMP_{A,A^{-1}}=\SMP_{A',B'/z}$$
Now, we apply the corollary \ref{cor:innerCharacterization}, to
$$\SM_{A,A^{-1}}=\SZ_{-\Xi}(\SMP_{A,A^{-1}})=\SZ_{-\Xi}(\SMP_{A',B'/z})=\SM_{A',B'/z}$$
\end{proof}

\begin{lemma}
\label{lemma:matrixIsParabolic}
Suppose that there exists a matrix $A\in \GL_n(R)$ such that $\SMP_{A,A^{-1}}\in \GL_{n^2}(R)$. Then $A\in \ParEnd_n(R)\cap \GL_n(R)$.
\end{lemma}

\begin{proof}
As $\det(A)$ is invertible
$$A^{-1}=\det(A)^{-1} \op{ad}(A)^t$$
Let us denote by $A_{ij}$ the $(i,j)$ adjoint of matrix $A$, i,e., the determinant of the complement minor to the element $(i,j)$ taken with the corresponding sign. Let $M=\SMP_{A,A^{-1}}$. Then
$$M=\SZ_{\Xi}\left (A\otimes (A^{-1})^t \right) =\det(A)^{-1} \SZ_{\Xi}\left( A\otimes \op{ad}(A) \right)\in \GL_{n^2}(R)$$
Looking at the blocks of $A\otimes \op{ad}(A)$ below the diagonal, $\SZ_{\Xi}(A\otimes \op{ad}(A))$ being a matrix in $R$ implies that
$$z | a_{ij} A_{kl}$$
for $j<i$, $k\le l$, $k<i$. In particular, this implies that $z|a_{ij} A_{kl}$ for $k\le i-1\le l$ and every $j<i$. Let us prove that this implies that $z|a_{ij}$ for $j<i$, so that $A\in \ParEnd_n(R)$. Suppose that $z \nmid a_{ij}$ for some $j<i$. Then $z|A_{kl}$ for all $k\le i-1\le l$. Then we will prove that
$$z | \det\left((A_{kl})_{k,l=1}^n\right)=\det\left(\op{ad}(A)\right) = \det(A)^n \det(A^{-1})$$
which would lead to contradiction, as $A\in \GL_n(R)$ and $z$ is not invertible in $R$. More precisely, we will prove by induction on $i\ge 2$ that if $M=(m_{kl})_{k,l=1}^n\in \op{Mat}_{n\times n}(R)$ with $n\ge i$ satisfies that $z|m_{kl}$ for all $k\le i-1\le l$ then $z|\det(M)$. Then we will take $M=\op{ad}(A)$ to obtain the desired contradiction. If $i=2$, then for all $n\ge i$ and all $M=(m_{kl})_{k,l=1}^n\in \op{Mat}_{n\times n}(R)$ we have $z|m_{1l}$ for every $l$, so $M$ has a row full of multiples of $z$ and, therefore, its determinant is a multiple of $z$. Suppose that $i>2$, that $n\ge i$ and that the statement is true for all $i'<i$ and all $n'\ge i'$. Let us develop the determinant of $M$ through the first row
$$\det\left((m_{kl})_{k,l=1}^n\right)= \sum_{l=1}^n (-1)^{l+1} m_{1l} \det(D^{1l})$$
where $D^{kl}$ is the complement minor of $M$ for the element $(k,l)$. For $l\ge i-1$, $z|m_{1l}$, so it is enough to prove that $z |\det(D^{1l})$ for $l<i-1$. $D^{1l}$ is obtained by removing the first row and the $l$-th column of $(m_{k'l'})_{k',l'=1}^n$. As $l<i-1$, $D^{1l}$ contains all the elements $m_{k'l'}$ for $1<k' \le i-1\le l'$ in the positions $k''=k'-1$, $l''=l'-1$, so we know that $z| (D^{1l})_{k''l''}=m_{k''+1,l''+1}$ for $k''\le i-2 \le l''$. Now, we apply the induction hypothesis to the $(n-1)$-dimensional matrix $D^{1l}$ for $i'=i-1$.
\end{proof}

\begin{lemma}
\label{lemma:parMatrixMainLemma}
Let $M\in \GL_{n^2}(R)$ be a matrix such that there exists $A\in \GL_n(K)$ satisfying $M=\SMP_{A,A^{-1}}$. Then, there exists a matrix $A'\in \ParEnd_n(R)\cap \GL_n(R)$ and an integer $0\le k<n$ such that 
$$M=\SMP_{(A'H^k),(A'H^k)^{-1}}$$
where $H\in \GL_n(K)$ is the matrix
$$H=\left( \begin{array}{c|c}
0 & I_{n-1} \\
\hline
z & 0
\end{array} \right) = \left ( \begin{array}{ccccc}
0 & 1 & 0 & \cdots & 0\\
0 & 0 & 1 & \cdots & 0\\
\vdots & \vdots & \ddots & \ddots & \vdots\\
0 & 0 & 0 & \ddots & 1\\
z & 0 & 0 & \cdots & 0
\end{array} \right)$$
\end{lemma}

\begin{proof}
First, let us prove that $\SMP_{H,H^{-1}}\in \GL_{n^2}(R)$. As
$$\det(\SMP_{H,H^{-1}})=\det(Z_{\Xi_T})\det(H)\det(H^{-1})\det(Z_{\Xi_T}^{-1})=1$$
it is enough to prove that $\SMP_{H,H^{-1}}\in \Mat_{n^2\times n^2}(R)$. We can easily compute that
$$H':=(H^{-1})^t=\left( \begin{array}{c|c}
0 & I_{n-1} \\
\hline
z^{-1} & 0
\end{array} \right) = \left ( \begin{array}{ccccc}
0 & 1 & 0 & \cdots & 0\\
0 & 0 & 1 & \cdots & 0\\
\vdots & \vdots & \ddots & \ddots & \vdots\\
0 & 0 & 0 & \ddots & 1\\
z^{-1} & 0 & 0 & \cdots & 0
\end{array} \right)$$
Now it is enough to prove that $\SM_{H,H^{-1}}=\SZ_{-(\Xi_T)_l+(\Xi_T)_r}(H\otimes H')\in \Mat_{n^2\times n^2}(R)$, but it is straightforward to check that
$$\SZ_{-(\Xi_T)_l+(\Xi_T)_r}(H\otimes H')=\left( \begin{array}{c|c}
0 & I_{n-1} \\
\hline
1 & 0
\end{array} \right) \otimes \left( \begin{array}{c|c}
0 & I_{n-1} \\
\hline
1 & 0
\end{array} \right)$$
Observe that, as $H^n=zI$, yields $H^{-k}=H^{n-k}/z$, so
$$\SMP_{H^{-k},H^k}=\SMP_{H^{n-k}/z,zH^{k-n}}=\SMP_{H^{n-k},H^{k-n}}=\left(\SMP_{H,H^{-1}} \right)^{n-k}\in \GL_{n^2}(R)$$
Corollary \ref{cor:innerOrder1Pole} allows us to find matrices $A',B'\in \Mat_{n\times n}(R)$ such that $M=\SMP_{A',B'}/z$ and $A'B'/z=B'A'/z=I$. First, let us prove that we can assume that $z^n\nmid \det(A')$. As $A'B'/z=I$, we get that
$$\det(A')\det(B')=z^n$$
As $z$ is not invertible in $R$ and $\det(B')\in R$, then $\det(A')|z^n$. Suppose that $z^n|\det(A')$. Then $z\nmid \det(B')$, so $\det(B')\not\in \fm$ and therefore, $\det(B')$ is invertible in $R$. As the inverse of $B'$ is $A'/z$, then
$$\frac{1}{z} A' = (B')^{-1} = \frac{1}{\det(B')} \op{ad}(B')^t$$
where, $\op{ad}(B')$ is the adjoint matrix of $B'$. As the adjoint belongs to $\Mat_{n\times n}(R)$ and $\det(B')^{-1}\in R$, then $\frac{A'}{z}\in \Mat_{n\times n}(R)$. Then, $M=\SMP_{A'/z,(A'/z)^{-1}}$ and $A'/z \in \GL_n(R)$.

If $z\nmid \det(A')$, then $\det(A')$ is invertible and, thus, $A'\in \GL_n(R)$, so $M=\SMP_{A',(A')^{-1}}$. Now suppose that $z^k | \det(A')$ but $z^{k+1} \nmid\det(A')$ for some $0<k<n$. Then
$$M'=\SMP_{A',B'/z}\SMP_{H^{-k},H^k}=\SMP_{A'H^{-k},H^kB'/z} \in \GL_{n^2}(R)$$
so there exist matrices $A'',B''\in \Mat_{n\times n}(R)$ with $z^n \nmid \det(A'')$ and $(A'')^{-1}=B''/z$ such that
$$\SMP_{A'H^{-k},H^kB'/z}=M'=\SMP_{A'',B''/z}$$
but then, by Lemma \ref{lemma:uniquenessParDecomposition} there exists a nonzero $\rho\in K$ such that
$$A''=\rho A'H^{-k}$$
Taking determinants
$$\det(A'')=\rho^n \frac{\det(A')}{z^k}$$
We have $z^{-k}\det(A') \not\in \fm$ by hypothesis and $z^n\nmid \det(A'') \in R$. Taking the $z$-valuation $\nu_z$ at both sides yields
$$\nu_z(\det(A''))= n \nu_z(\rho)$$
As $0\le \nu_z(\det(A''))<n$, yields $\nu_z(\rho)=0$, so $\rho$ is invertible in $R$ and we get
$$A'=\rho^{-1}A''H^k$$
Moreover, $\nu_z(\det(A''))=0$, so $\det(A'')$ is invertible, and therefore, $\rho^{-1}A''\in \GL_n(R)$. From Lemma \ref{lemma:matrixIsParabolic}, $\rho^{-1}A''\in \ParEnd_n(R)\cap \GL_n(R)$ and the Lemma follows.
\end{proof}

\begin{lemma}
Let $\{U_{\alpha}\}_{\alpha\in I}$ be a good cover of $(X,D)$ such that for every $x\in D$ there exists a unique $\alpha_x\in I$ such that $x\in U_{\alpha_x}$ and if $x,y\in D$ are distinct then $\alpha_x\ne \alpha_y$. Let $(E,E_\bullet)$ be a parabolic vector bundle of rank $r$ described by a cocycle $\varphi_{\alpha\beta}:U_{\alpha\beta} \to \SG|_{U_{\alpha\beta}}$. Then $\SH_x(E,E_\bullet)$ is described by the following cocycle $\psi_{\alpha\beta}:U_{\alpha\beta} \to \SG|_{U_{\alpha\beta}}$
$$\psi_{\alpha\beta}=\left \{ \begin{array}{ll}
\varphi_{\alpha\beta} & x\not\in U_\alpha\cup U_\beta\\
H\varphi_{\alpha\beta} &  x\in U_{\beta}\\
\varphi_{\alpha\beta}H^{-1} & x\in U_\alpha
\end{array} \right .$$
where $H=\left( \begin{array}{c|c}
0 & I_{r-1} \\
\hline
z & 0
\end{array} \right)$ and $z$ is a local coordinate in $U_{\alpha_x}$ centered in $x$.
\end{lemma}

\begin{proof}
First, let us prove that $\psi$ is a cocycle. Let $\alpha,\beta,\gamma \in I$ with $U_\alpha\cap U_\beta \cap U_\gamma\ne \emptyset$ and let us compute $\psi_{\gamma\alpha}\psi_{\beta\gamma}\psi_{\alpha\beta}$. If $x$ does not belong to any of the open sets, $\psi$ coincides with $\varphi$ and
$$\psi_{\gamma\alpha}\psi_{\beta\gamma}\psi_{\alpha\beta}=\varphi_{\gamma\alpha}\varphi_{\beta\gamma}\varphi_{\alpha\beta}=1$$
If $x\in U_\alpha$
$$\psi_{\gamma\alpha}\psi_{\beta\gamma}\psi_{\alpha\beta}=H\varphi_{\gamma\alpha}\varphi_{\beta\gamma} \varphi_{\alpha\beta}H^{-1}=HH^{-1}=1$$
If $x\in U_\beta$
$$\psi_{\gamma\alpha}\psi_{\beta\gamma}\psi_{\alpha\beta}=\varphi_{\gamma\alpha}\varphi_{\beta\gamma}H^{-1} H \varphi_{\alpha\beta}=\varphi_{\gamma\alpha}\varphi_{\beta\gamma}\varphi_{\alpha\beta}=1$$
and if $x\in U_\gamma$
$$\psi_{\gamma\alpha}\psi_{\beta\gamma}\psi_{\alpha\beta}=\varphi_{\gamma\alpha}H^{-1} H\varphi_{\beta\gamma} \varphi_{\alpha\beta}=\varphi_{\gamma\alpha}\varphi_{\beta\gamma}\varphi_{\alpha\beta}=1$$
Recall that $E_x^2\subset E$ denotes the second step of the filtration by subsheaves defining the parabolic structure of $(E,E_\bullet)$ at $x$ and it is precisely the underlying vector bundle of $\SH_x(E,E_\bullet)$ (see section \ref{section:Hecke}).The trivialization induced by $\varphi_{\alpha\beta}$ at the stalk $E_x$ is precisely
$$(E_x^2)_x \cong \fm\oplus \SO_{X,x}^{r-1} \subset \SO_{X,x}^r\stackrel{\varphi}{\cong} E_x$$
A trivialization of $E_x^2\cong \SH_x^{E_{x,2}}(E)$  compatible with the induced parabolic structure would be the one obtained by ``rotating'' the given one through the procedure described in the previous chapter
$$(E_x^2)_x\stackrel{\psi}{\cong} \SO_{X,x}^r \stackrel{\Xi_{r}}\cong \SO_{X,x}^{r-1} \oplus \fm \stackrel{\pi}{\cong} \fm \oplus \SO_{X,x}^{r-1}$$
where $\pi$ is the permutation sending $\pi(i)=i-1$ for $i>1$ and $\pi(1)=r$. Therefore, we get
$$(E_x^2)_x \stackrel{H}{\cong} \fm\oplus \SO_{X,x}^{r-1} \subset \SO_{X,x}^r\stackrel{\varphi}{\cong} E_x$$
and we are done.
\end{proof}

\begin{corollary}
\label{cor:HeckeOuter}
Let $(E,E_\bullet)$ be a parabolic vector bundle. Then $\ParEnd_0(E,E_\bullet)$ is isomorphic to $\ParEnd_0\left( \SH_x(E,E_\bullet) \right)$ as a Lie algebra bundle and at the stalk at the parabolic point $x\in D$ the isomorphism coincides with
$$\SMP_{H,H^{-1}} :\ParEnd_0(E,E_\bullet)_x \cong \ParEnd_0\left( \SH_x(E,E_\bullet) \right)_x$$
\end{corollary}

\begin{lemma}
\label{lemma:isoParEnd}
Let $(E,E_\bullet)$ and $(E',E'_\bullet)$ be parabolic vector bundles of rank $r$ such that $\ParEnd_0(E,E_\bullet)$ and $\ParEnd_0(E',E'_\bullet)$ are isomorphic as Lie algebra bundles. Then $(E',E'_\bullet)$ can be obtained from $(E,E_\bullet)$ through a combination of the following transformations
\begin{enumerate}
\item Tensorization with a line bundle over $X$, $(E,E_\bullet) \mapsto (E\otimes L, E_\bullet \otimes L)$
\item Parabolic dualization $(E,E_\bullet)\mapsto (E, E_\bullet)^\vee$
\item Hecke transformation at a parabolic point $x\in D$, $(E,E_\bullet) \mapsto \SH_x(E,E_\bullet)$.
\end{enumerate}
\end{lemma}

\begin{proof}
Giving a vector bundle $\ParEnd_0(E)$ with its Lie algebra structure is equivalent to giving an $\Aut(\parsl)$-torsor $P_{\Aut(\parsl)}$ which admits a reduction to a $\SG$-torsor (which corresponds to the parabolic vector bundle $(E,E_\bullet)$). We will analyze the possible reductions from a given $\Aut(\parsl)$-torsor in two steps.
$$\SG\twoheadrightarrow \Inn(\parsl) \hookrightarrow \Aut(\parsl)$$

First, note that there is an exact sequence of sheaves of groups
$$1\longrightarrow \Inn(\parsl) \longrightarrow \Aut(\parsl) \longrightarrow \Out(\parsl) \longrightarrow 1$$
Our first step is to compute the outer automorphisms of $\parsl$. Over a non-parabolic point $x\not\in D$, taking stalks the previous short exact sequence simply reduces to
$$1 \longrightarrow \Inn(\ssl) \longrightarrow \Aut(\ssl)\longrightarrow\Out(\ssl) \longrightarrow 1$$
We know that $\Inn(\ssl)=\PGL_r$ and $\Out(\ssl)=\ZZ/2\ZZ$ if $r>2$ and $\Out(\ssl)=1$ if $r=2$ (c.f. \cite[Proposition D.40]{FH91}). For the sake of the exposition, let us assume that $r>2$. The proof for $r=2$ is completely analogous and it will be outlined at the end. Then
$$1 \longrightarrow \Inn(\ssl)=\PGL_r \longrightarrow \Aut(\ssl)\longrightarrow\Out(\ssl)=\ZZ/2\ZZ \longrightarrow 1$$
Therefore, in order to determine $\Out(\parsl)$, we only need to determine the stalk of $\Out(\parsl)$ at a parabolic point.
The single nontrivial outer automorphism of $\ssl$ is the one induced by duality of the underlying vector space. Given a parabolic full flag vector bundle $(E,E_\bullet)$, parabolic duality induces an outer isomorphism of the algebra $\parsl$ extending the previous one over non-parabolic points. Let $x\in D$. Let $o_1,o_2\in \Out(\parsl)_x$ be two germs of sections at the parabolic point. Composing with the dualization action if necessary, we may assume that $o_1$ and $o_2$ coincide generically. Then there exist germs of sections $\overline{o}_1,\overline{o}_2\in \Aut(\parsl)_x$ such that $s:=\overline{o}_1\circ\overline{o}_2^{-1}\in \Aut(\parsl)_x$ is a germ whose restriction to the open set correspond to an inner automorphism.

Let $\SO_{X,x}$ be the stalk of the structure sheaf at $x\in D$. Let $\fm$ be the maximal ideal in $\SO_{X,x}$ and let $\SK$ be the field of fractions, i.e., $\SK=\fm^{-1}\SO_{X,x}$. As $X$ is a smooth curve, $\SO_{X,x}$ is a principal ideal domain, so $\fm=(z)$ for some germ $z\in \SO_{X,x}$. Therefore, an element of $\parsl_x$ is represented by an $r\times r$  matrix of elements of $\SO_{X,x}$ whose elements below the diagonal are multiples of $z$, i.e., it is a matrix of the form
$$\left ( \begin{array}{cccc}
a_{11} & a_{12} & \cdots & a_{1r} \\
za_{21} & a_{22} & \cdots & a_{2r} \\
\vdots & \vdots & \ddots & \vdots \\
za_{r1} & za_{r2} & \cdots & a_{rr}
\end{array} \right )$$
where $a_{ij}\in \SO_{X,x}$ and $\sum_{i=1}^ra_{ii}=0$. The germ $s$ is, in particular, a germ of $\GL(\parsl)$. Trace 0 matrices form a linear codimension 1 subspace of $\pargl$ whose complement is generated by the identity matrix. Therefore, any element of $\GL(\parsl)$ extends to an element of $\GL(\pargl)$ sending the identity matrix to itself. Moreover, if an element of $\GL(\parsl)$ belongs to $\Aut(\parsl)$, the extension belongs to $\Aut(\pargl)$, as the identity matrix belongs to the kernel of the Lie bracket in $\pargl$.

Then, any germ $s\in \Aut(\parsl)_x$ can be described by an invertible $r^2 \times r^2$ matrix of elements of $\SO_{X,x}$ by embedding $\Aut(\parsl)_x \hookrightarrow \GL(\Mat_{r\times r}(\SO_{X,x}))\stackrel{\tau}{\cong}\GL_{r^2}(\SO_{X,x})$. Let $S=(s_{ij})\in \GL_{r^2}(\SO_{X,x})$ be such matrix. As $s$ corresponds generically to an inner automorphism, there exists a matrix $G\in \GL_{r}(\SK)$ such that $S=\SMP_{G,G^{-1}}$. By Lemma \ref{lemma:parMatrixMainLemma}, there exists a matrix $G'\in \ParEnd_r(\SO_{X,x}) \cap \GL_r(\SO_{X,x}) \cong \SG_x$ and an integer $0\le k <r$ such that
$$S=\SMP_{G'H^k,(G'H^k)^{-1}}=\SMP_{G',(G')^{-1}}\circ \left(\SMP_{H,H^{-1}}\right)^k$$

Moreover, as $\SMP_{H,H^{-1}}$ is a conjugation operation in $\Mat_{r\times r}(\SK)$, it clearly preserves the $0$-trace and it is a Lie algebra isomorphism. Therefore $\frac{\Aut(\parsl)_x}{\ZZ/2\ZZ \times \Inn(\parsl)_x}$ is generated by the order $r$ automorphism induced from conjugation by the matrix $H$. One trivially checks that taking the dual and conjugating by $H$ is the same as conjugating by $H^{-1}$ and then taking the dual, so the outer automorphism group is
$$\Out(\parsl)_x \cong \langle s,h \rangle/\{s^2=1,h^r=1,sh=h^{-1}s\}=\mathbb{D}_r$$
where $\mathbb{D}_r$ is the dihedral group of order $r$. Therefore, $\Out(\parsl)_x$ fits in a sequence
$$1\longrightarrow \ZZ/2\ZZ \times X \longrightarrow \Out(\parsl) \longrightarrow \ZZ/r\ZZ \otimes \SO_D \longrightarrow 0$$

The space of reductions of structure sheaf of $P_{\Aut(\parsl)}$ to $\Inn(\parsl)$ correspond to sections of the associated $\Out(\parsl)$-torsor, $P_{\Aut(\parsl)}(\Out(\parsl))$. The associated bundle is a $2$-to-$1$ cover of $U$ glued to a $(2r)$-to-$1$ cover of $D$ through the canonical inclusion $\ZZ/2\ZZ<\mathbb{D}_r$. Since we know that there are reductions of the torsor, the bundle must be the disjoint union of a trivial $2$-to-$1$ cover of $X$ and a trivial $2(r-1)$ cover of $D$.

We will prove that $\Inn(\parsl)$ coincides with $\SG/\CC^*:=P\SG$. Then, a reduction of $P_{\Aut(\parsl)}$ to an $\Inn(\parsl)$ is a parabolic projective bundle $(\PP,\PP_\bullet)=(\PP(E),\PP(E_\bullet))$ together with an isomorphism 
$$P_{\Aut(\parsl)}\cong \ParEnd_0(\PP,\PP_\bullet)$$

Let $(\PP,\PP_\bullet) \to X$ be a reduction of $P_{\Aut(\parsl)}$ to $\Inn(\parsl)$. Then the generator of the $\ZZ/2\ZZ$ component of $\Out(\parsl)$ corresponds to its dual parabolic projective bundle
$$(\PP,\PP_\bullet)^\vee=\PP\left((E,E_\bullet)^\vee \right)$$
On the other hand, by Corollary \ref{cor:HeckeOuter}, for each $x\in D$, the generator of the $\ZZ/r\ZZ<\mathbb{D}_r$ outer automorphism corresponds to the Hecke transformation of $(\PP,\PP_\bullet)$ at the parabolic point $x\in D$. As these outer automorphisms generate $\Out(\parsl)$, every reduction can be found as a composition of Hecke transformations and dualization of $(\PP,\PP_\bullet)$.

Now consider the exact sequence of groups
$$1\longrightarrow Z \longrightarrow \SG \longrightarrow \Inn(\parsl) \longrightarrow 1$$
Let us compute the group scheme $Z$. As before, over $x\in U$, $\Inn(\parsl)_x=\PGL_r$ and $\SG_x=\GL_r$, so $Z_x=\CC^*$. Therefore, it is only necessary to compute $Z_x$ for $x\in D$. By definition, $Z$ is the kernel of the adjoint representation. Let $X\in \SG_x \hookrightarrow \Mat_{r\times r}(\SO_{X,x})$ be in the kernel of the representation. Then, for every $Y\in \parsl_x \hookrightarrow \Mat_{r\times r}(\SO_{X,x})$
$$XY-YX=0$$
In particular, as given any $G\in \End_0(\SO_{X,x})$, $zG\in \parsl_x$,
$$0=X(zG)-(zG)X=z(XG-GX)$$
As $\SO_{X,x}$ does not have any zero divisors, $XG-GX=0$ and, therefore, $X$ belongs to the center of $\End_0(\SO_{X,x})$, which consists on $\SO_{X,x}$-multiples of the identity. Clearly, all invertible multiples of the identity belong to $\SG_x$ and they are in the kernel of the adjoint, so
$$Z_x=\SO_{X,x}^*$$
Therefore, we conclude that $Z=\CC^*\times X = \SO_X^*$ and, taking the quotient, $\Inn(\parsl)=\SG/\CC^*:=P\SG$. As $\CC^*$ belongs to the center of $\SG$, the isomorphism classes of reductions of an $\Inn(\parsl)$-torsor to a $\SG$-torsor form a torsor for the group $H^1(X,\SO_X^*)$.

Let $(E,E_\bullet)$ be the parabolic vector bundle corresponding to a reduction of the $P\SG$-torsor $(\PP,\PP_\bullet) \to X$, i.e., $(\PP(E),\PP(E_\bullet)) \cong (\PP,\PP_\bullet)$. Then the other reductions correspond to parabolic vector bundles of the form $(E,E_\bullet)\otimes L$ for any line bundle $L$. Similarly, $(E,E_\bullet)^\vee\otimes L$ and $\SH_x(E,E_\bullet) \otimes L$ are all the possible reductions of $(\PP,\PP_\bullet)^\vee$ and $\SH_x(\PP,\PP_\bullet)$ respectively, so all possible reductions can be computed from $(E,E_\bullet)$ by a repeated combination of dualization, tensoring with a line bundle and application of Hecke transformations at parabolic points.

As we mentioned before, the proof for $r=2$ is completely analogous. The only difference is that we do not need duality anymore, as in this case $\Out(\ssl)=1$ and $\Aut(\ssl)=\Inn(\ssl)$ \cite[Proposition D.40]{FH91}, so the computation gets simplified. Geometrically, this reflects on the fact that the dual of every rank two bundle (respectively rank two parabolic bundle) is isomorphic to a tensor product of a line bundle with the original bundle (see Lemma \ref{lemma:dualrk2} for more details). Following the same argument as before we get that $\frac{\Aut(\parsl)_x}{\Inn(\parsl)_x}$ is generated by the order $r$ automorphism induced from conjugation by the matrix $H$, so
$$\Out(\parsl)_x \cong \ZZ/r\ZZ$$
and thus
$$\Out(\parsl) \cong \ZZ/r\ZZ \otimes \SO_D$$
As before, the space of reductions of $P_{\Aut(\parsl)}$ to $\Inn(\parsl)= P\SG$ correspond to sections of $P_{\Aut(\parsl)}(\Out(\parsl))$. As we know that there are reductions, in rank $2$ this bundle must be a disjoint union of $X$ and a trivial $(r-1)$-cover of $D$. Then if $(\PP,\PP_\bullet) \to X$ is a reduction of $P_{\Aut(\parsl)}$ to $\Inn(\parsl)= P\SG$ the other reductions come from Hecke transformations of $(\PP,\PP_\bullet)$ at the parabolic points and, therefore, if $(E,E_\bullet)$ is a reduction of the $P\SG$ torsor $(\PP,\PP_\bullet)$ to $\SG$ then the rest of the reductions of $P_{\Aut(\parsl)}$ to $\SG$ are obtained as a combination of tensoring with a line bundle and application of Hecke transformations at parabolic points to $(E,E_\bullet)$.
\end{proof}

\section[Isomorphisms between parabolic moduli]{Isomorphisms between moduli spaces of parabolic vector bundles}
\label{section:Automorphisms}

Let $\Phi:\SM(X,r,\alpha,\xi) \to \SM(X',r',\alpha',\xi')$ be an isomorphism between the moduli space of parabolic vector bundles of rank $r$, determinant $\xi$ and weight system $\alpha$ over $(X,D)$ and the moduli space of parabolic vector bundles of rank $r'$, determinant $\xi'$ and weight system $\alpha'$ over $(X',D')$. 

By Torelli Theorem \ref{theorem:Torelli}, we know that $r=r'$ and that $\Phi$ induces an isomorphism between the marked curves $\sigma:(X,D)\stackrel{\sim}{\longrightarrow}(X',D')$. We know that the map of quasi-parabolic vector bundles $(E,E_\bullet) \mapsto \sigma^*(E,E_\bullet)$ induces an isomorphism $\Sigma_\sigma : \SM(X',r,\alpha',\xi') \longrightarrow \SM(X,r,\sigma^*\alpha',\sigma^*\xi')$. Therefore, $\Sigma_\sigma\circ \Phi:\SM(r,\alpha,\xi) \longrightarrow \SM(r,\sigma^*\alpha',\sigma^*\xi')$ is an isomorphism between moduli spaces of parabolic vector bundles on $(X,D)$ such that the induced automorphism on the marked curve is the identity. As we can do this for every automorphism of the marked curve, we can assume without loss of generality that $\Phi$ induced the identity map on $(X,D)$.

For $k>1$, let $W_k=H^0(K^kD^{k-1})$. Recall that we defined
$$h_k:H^0(\SParEnd_0(E)\otimes K_X(D))\to W_k$$
as the composition of the Hitchin map $h:H^0(\SParEnd_0(E)\otimes K_X(D))\to W$ with the projection $W\twoheadrightarrow W_k$.
As we have assumed that $\Phi$ induces the identity map on $(X,D)$, then the Hitchin space for both moduli spaces is the same and by Proposition \ref{prop:HitchinGlobalFunctions}, there exist a $\CC^*$-equivariant automorphism $f:W\stackrel{\sim}{\longrightarrow} W$ such that the following diagram commutes
\begin{eqnarray}
\label{eq:fDefinition}
\xymatrixcolsep{4pc}
\xymatrix{
T^*\SM(r,\alpha,\xi) \ar[r]^{d(\Phi^{-1})} \ar[d]_{h} & T^*\SM(r,\alpha',\xi') \ar[d]^{h} \\
W \ar[r]^{f} & W
}
\end{eqnarray}

Moreover we know that $f$ preserves the block $W_r\subset W$. Our next goal will be to prove that, in fact, there exists linear maps $f_k:W_k\to W_k$ such that the following diagram commutes for every $k>1$ (Corollary \ref{cor:preservehr})
\begin{eqnarray}
\label{eq:frDefinition}
\xymatrixcolsep{4pc}
\xymatrix{
T^*\SM(r,\alpha,\xi) \ar[r]^{d(\Phi^{-1})} \ar[d]_{h_r} & T^*\SM(r,\alpha',\xi') \ar[d]^{h_r} \\
W_k \ar[r]^{f_k} & W_k
}
\end{eqnarray}
in other words, we will prove that $f:W\to W$ is linear and preserves the decomposition $W=\bigoplus_{k=2}^r W_r$. In order to do so, we will analyze how the geometry of the discriminant $\SD\subset W$ and the $\CC^*$-action impose restrictions on the structure of the map $f:W\to W$.

For every $k>1$, let us denote
$$W_{\le k}=\bigoplus_{j=2}^k W_k$$
In particular $W=W_{\le r}$ and, in order to simplify the notation, we consider $W_{\le 1}=0$.

\begin{lemma}
\label{lemma:fBlockStructure}
Let $f:W\to W$ be a $\CC^*$-equivariant isomorphism. If $r=2$ then $f$ is linear isomorphism. Otherwise, if $r\ge 2$, then there exist
\begin{itemize}
\item An algebraic isomorphism $\overline{g}:W_{\le (r-2)}\to W_{\le (r-2)}$,
\item linear isomorphisms $A_{j}:W_{j}\to W_{j}$, $j=r-1,r$ and
\item algebraic maps $g_j:W_{\le (r-2)} \to W_j$, $j=r-1,r$
\end{itemize}
such that for every $s=(\overline{s},s_{r-1},s_r)\in W=W_{\le (r-2)}\oplus W_{r-1}\oplus W_r$
$$f(s_2,\ldots,s_r)=(g(\overline{s}),A_{r-1}(s_{r-1})+g_{r-1}(\overline{s}),A_r(s_r)+g_r(\overline{s}))$$
\end{lemma}

\begin{proof}
Assume that $r\ge 3$ and let $f=(f_2,\ldots,f_r)$. Let us fix coordinates $\overline{x_j}=(x_{j,1},\ldots,x_{j,d_j})$ in $W_j$ for each $j=2,\ldots,r$, where
$$d_j=\dim(W_j)=h^0(K^jD^{j-1})=j(2g-2)+(j-1)n-g+1$$
In these coordinates, each component of the map $f_j:W\to W_j$ is written as a weighted-homogeneous polynomial for the weights induced by the $\CC^*$-action. This means that it must be the sum of monomials of the form $C\prod_{i=1}^n x_{j_i,k_i}^{t_i}$ where
$$\sum_{i=1}^l t_ij_i=j \quad \quad t_i>0, \quad 2\le j_i\le r$$
In particular, the previous equation implies that for every $j$ and every $i=1,\ldots,n$, $j_i\le j$ and, therefore, the map $f_j:W\to W_j$ can only depend on variables coming from $W_l$ for $l\le j$. Moreover, if $j_i=j$ for some $i$, then there cannot be any other factor in the monomial, i.e., it is a linear monomial. Therefore, each $f_j:W\to W_j$ decomposes as a sum
$$f_j(s_2,\ldots,s_r)=g_j(s_2,\ldots,s_{j-1})+A_j(s_j)$$
for some $\CC^*$-equivariant map $g_j:\bigoplus_{i=2}^{j-1} W_i \to W_j$ and some linear map $A_j:W_j\to W_j$.
In the particular case $j=r$ we observe that the monomials composing $f_r$ cannot contain the variables $\{x_{j-1,i}\}_{i=1}^{d_{j-1}}$ either, because they have order $r-1$ for the $\CC^*$ action and there does not exist any variable of order $1$. Then
$$f_r(s_2,\ldots,s_r)=g_r(s_2,\ldots,s_{r-2})+A_r(s_r)$$
Finally, as the inverse $f^{-1}$ must have an analogous decomposition, we conclude that the maps $A_j:W_j\to W_j$ and the maps
$$\overline{g}_j=(A_2,g_3+A_3,\ldots,g_j+A_j):\bigoplus_{i=2}^j W_i \longrightarrow \bigoplus_{i=2}^j W_i$$
must be all invertible. The case $r=2$ is proved in a completely analogous way.
\end{proof}

Let $\op{sing}:\SD \dashrightarrow X$
For each $x\in X$, let $\SD_x\subset \SD$ be closure of the subset of singular curves which are singular over the point $x\in X$. By definition of the map $\op{sing}$
$$\SD_x= \overline{\op{sing}^{-1}(x)}$$

\begin{lemma}
For every $x\in X$, $\SD_x$ is a nonempty connected rational variety.
\end{lemma}

\begin{proof}
For $x\in D$ the subset $\SD_x$ was computed in Proposition \ref{prop:HitchinComponents}
$$\SD_x=\bigoplus_{k=2}^{r-1} H^0(K^kD^{k-1}) \oplus H^0(K^rD^{r-1}(-x))$$
and it clearly satisfies the desired properties. Suppose that $x\in X\backslash D$. Let us consider the image of $\SD_x$ under the evaluation map
$$0\longrightarrow \bigoplus_{j=2}^r H^0(K^jD^{j-1}(-2x)) \longrightarrow W\longrightarrow \bigoplus_{j=2}^r K^jD^{j-1}\otimes \SO_x/\SI_x^2 \longrightarrow 0$$
Then $s\in \SD_x$ is the preimage of
$$\SY=\left \{(s,s')\in \CC^{2(r-1)} \middle | \exists t\in \CC \quad \begin{array}{c}
t^r + \sum_{j=2}^r s_jt^{r-j}=0\\
rt^{r-1}+\sum_{j=2}^{r-1} (r-j)s_j t^{r-j-1}=0\\
\sum_{j=2}^r s_j't^{r-j}=0
\end{array} \right \}$$
Clearly, if we prove that $\SY$ is rational connected and nonempty, then $\SD_x$ is rational connected and nonempty, as it is a vector bundle over $\SY$. For $r=2$, $\SY=\{(0,0)\}$ and we are done, so from now on assume that $r>2$. Let us consider the following diagram
\begin{eqnarray*}
\xymatrix{
&\CC_t \times \CC^{r-1}_s \times \CC^{r-1}_{s'} \ar@{->>}[rr] \ar@{->>}[dd] & & \CC_t \times \CC^{r-1}_s \ar@{->>}[dd] \\
\tilde{\SY} \ar@{->>}[rr] \ar@{^(->}[ur] \ar@{->>}[dd] & & \tilde{\SY}_s \ar@{^(->}[ur] \ar@{->>}[dd] & \\
&\CC^{r-1}_s \times \CC^{r-1}_{s'} \ar@{->>}[rr] && \CC^{r-1}_s\\
\SY \ar@{->>}[rr] \ar@{^(->}[ur]  && \SY_s \ar@{^(->}[ur] &
}
\end{eqnarray*}
where
$$\tilde{\SY}=\left \{(t,s,s')\in \CC^{2r-1} \middle | \begin{array}{c}
t^r + \sum_{j=2}^r s_jt^{r-j}=0\\
rt^{r-1}+\sum_{j=2}^{r-1} (r-j)s_j t^{r-j-1}=0\\
\sum_{j=2}^r s_j't^{r-j}=0
\end{array} \right \}$$
$$\tilde{\SY}_s=\left \{(t,s)\in \CC^{r} \middle | \begin{array}{c}
t^r + \sum_{j=2}^r s_jt^{r-j}=0\\
rt^{r-1}+\sum_{j=2}^{r-1} (r-j)s_j t^{r-j-1}=0\\
\end{array} \right \}$$
$$\SY_s=\left \{s\in \CC^{r-1} \middle | \exists t\in \CC \quad \begin{array}{c}
t^r + \sum_{j=2}^r s_jt^{r-j}=0\\
rt^{r-1}+\sum_{j=2}^{r-1} (r-j)s_j t^{r-j-1}=0\\
\end{array} \right \}$$
and, clearly, all horizontal and vertical arrows are surjective. Let us consider the open subset $U_s \subset \SY_s$ corresponding to polynomials of the form $p(x)=(x-t_1)^2(x-t_2)\cdots (x-t_{r-1})$ with $2t_1+\sum_{i=2}^{r-1} t_i=0$, all $t_i$ different and $t_1\ne 0$. Let $U\subset \SY$, $\tilde{U}\subset \tilde{\SY}$ and $\tilde{U}_s\subset \tilde{\SY}_s$ be the preimages of $U$ under the corresponding projection maps. By definition of $\SY_s$, it is straightforward to compute that a vector $(\tau,\sigma)\in T_{(t,s)}\CC_t\times \CC_s^{r-1}$ lines in $T_{(t,s)}\tilde{\SY}_s$ if and only if
$$\left \{ \begin{array}{r}
t^r+\sum_{j=2}^r \sigma_j t^{r-j} =0\\
rt^{r-1}+\sum_{j=2}^{r-1} (r-k) \sigma_j t^{r-j-1} + \left( r(r-1)t^{r-2} + \sum_{j=2}^{r-2}(r-j)(r-j-1)s_j a^{r-j-2}\right) \tau=0
\end{array} \right. $$
Therefore, the differential of the map $\tilde{\SY}_s \twoheadrightarrow \SY_s$ fails to be injective at $(t,s)$ if and only if
$$r(r-1)t^{r-2} + \sum_{j=2}^{r-2}(r-j)(r-j-1)s_j a^{r-j-2}=0$$
i.e., if the polynomial $p_s(x)=x^r+\sum_{j=2}^r s_j x^{r-j}$ is divisible by $(x-t)^3$. In particular, if $s\in U_s$, the differential of the map $\tilde{U}_s \twoheadrightarrow U_s$ is injective. Moreover, $\tilde{\SY}_s\twoheadrightarrow \SY_s$ is clearly finite, so the map $\tilde{U}_s \twoheadrightarrow U_s$ is a finite and bijective with injective differential. By \cite[Theorem 14.9 and Corollary 14.10]{Harris13}, it is an isomorphism. As points $(t,s)\in \tilde{U}_s$ all have $t\ne 0$, the fiber of the projection $\tilde{U}\twoheadrightarrow \tilde{U}_s$ is a vector space of dimension $r-2$ and it is straightforward to check that $\tilde{U}$ is a vector bundle over $\tilde{U}_s$. Similarly, $U$ is a vector bundle over $U_s$ and it is isomorphic to $\tilde{U}$ through the isomorphism $\tilde{U}_s\cong U_s$. This proves that $\SY$ is birational to a vector bundle over $\tilde{U}_s$. The latter is isomorphic to $\CC^*\times \CC^{r-3}$ in the following way. Consider $\CC^k$ as the space of traceless polynomials $q(x)$ of degree $k+1$. Then $\tilde{U}_s$ is the image of the map
\begin{eqnarray*}
\xymatrixrowsep{0.05pc}
\xymatrix{
\CC^*\times \CC^{r-3} \ar[r] & \tilde{U}_s\\
(t,q(x)) \ar@{|->}[r] & \left ( t,(x-t)^2(q(x)+2tx^{r-3}) \right)
}
\end{eqnarray*}
The inverse can be computed through Ruffini's rule, thus inducing an algebraic isomorphism. Therefore $\SY$ is birational to a vector bundle over $\CC^* \times \CC^{r-3}$, so it is a nonempty connected rational variety.
\end{proof}

\begin{lemma}
\label{lemma:recoverDx}
The map $\op{sing}: \SD \dashrightarrow X$ commutes with $f:\SD\to \SD$.
\end{lemma}

\begin{proof}
We will proceed as in \cite[Remark 4.2]{BGM12}. As $\op{sing}:\SD \dashrightarrow X$ has connected rational fibers and is surjective, there exists a unique such map up to an automorphism of $X$ because of the following. Suppose that $\gamma :\SD \dashrightarrow X$ is another surjective map with connected rational fibers. For a generic $x\in X$, the map $\gamma$ restricts to a rational map $\SD_x \dashrightarrow X$. As $\SD_x$ is a connected rational variety this map must be constant. Therefore, the map $\gamma:\SD \dashrightarrow X$ descends to a nonconstant rational map $\rho_{\gamma}:X\dashrightarrow X$. As $X$ is a smooth projective curve, this map extends to an automorphism of $X$ such that $\gamma = \rho_{\gamma} \circ \op{sing}$. Let $\rho:X\to X$ be the only map such that $f(\SD_x)=\SD_{\rho(x)}$ for all $x\in X$. The map $f:W\to W$ preserves $W_r$ and $\SD$, so it preserves $\SD\cap W_r= \SC_X\cup \bigcup_{x\in D} \SC_x$. Moreover, we know that $\PP(\SC_X)$ is not isomorphic to $\PP(\SC_x)$ for any $x\in D$, so $f$ must induce an automorphism of $\SC_X$. By construction we assumed that the induced automorphism $\sigma:X\to X$ on the dual variety is the identity, and it clearly coincides with $\rho:X\to X$, as for each $x_0\in X\backslash D$ we have
\begin{multline*}
H^0(K^rD^{r-1}(-2x_0))=f(H^0(K^rD^{r-1}(-2x_0)))=f(\SD_x\cap W_r)\\
=\SD_{\rho(x)} \cap W_r = H^0(K^rD^{r-1}(-2\rho(x_0)))
\end{multline*}
\end{proof}

As a consequence, for each $x\in X$, $f(\SD_x)=\SD_x$. Then, in particular, their intersection is preserved by $f$. Let
$$\SN = \bigcap_{x\in X} \SD_x$$
be the subset of spectral curves which are singular over each $x\in X$. The only way this can happen is if the spectral curve is non-reduced, so $\SN$ is precisely the set of non-reduced spectral curves. Clearly, it decomposes in irreducible components depending on the degree of the non-reduced factor.
$$\SN = \bigcup_{d=1}^{\lfloor r/2 \rfloor} \SN^d$$
where
$$\SN^d = \left \{(x^d+a_1 x^{d-1}+ \ldots + a_d)^2(x^{r-2d}-2a_1 x^{r-2d-1}+b_2 x^{r-2d-2}+ \ldots + b_{r-2d}) \right \}$$
for $d<r/2$ and, if $r$ is even,
$$\SN^{r/2}=\left \{(x^{r/2}+a_2x^{r/2-2}+\ldots+a_{r/2})^2 \right \}$$
with $a_j,b_j \in H^0(K^jD^{j-1})$.

\begin{lemma}
\label{lemma:preserveR1}
Suppose that $r\ge 3$ and let $f:W\to W$ be a map such that $f(\SN)=\SN$. Then it preserves $\SN^1\subset \SN$.
\end{lemma}

\begin{proof}
We will show that the irreducible component $\SN^1\subset \SN$ can be identified as the unique irreducible component with the highest dimension. Generically the polynomials $p(x)\in \SN^d$ admit a single decomposition as a product $p(x)=p_1(x)^2p_2(x)$ as above. Therefore, using Riemann-Roch theorem, the dimension of $\SN^d$ equals
$$\dim(\SN^d)=\sum_{j=1}^d h^0(K^jD^{j-1}) + \sum_{j=2}^{r-2d} h^0(K^jD^{j-1})$$
for $d<r/2$ and, if $r>2$ is even, then
$$\dim(\SN^{r/2})=\sum_{j=2}^{r/2} h^0(K^jD^{j-1})$$
Observe that for every $d$ with $1<d< r/2$
\begin{multline*}
\dim(\SN^d)=\sum_{j=1}^d h^0(K^jD^{j-1}) + \sum_{j=2}^{r-2d} h^0(K^jD^{j-1})\le\\
h^0(K)+\sum_{j=r-d}^{r-2} h^0(K^jD^{j-1})+\sum_{j=2}^{r-2d} h^0(K^jD^{j-1})< \sum_{j=1}^{r-2} h^0(K^jD^{j-1}) = \dim(\SN^1)
\end{multline*} 
and, if $r>2$ is even then $r/2\le r-2$, so clearly
$$\dim(\SN^{r/2})=\sum_{j=2}^{r/2} h^0(K^jD^{j-1}) \le \sum_{j=2}^{r-2} h^0(K^jD^{j-1})<\sum_{j=1}^{r-2}h^0(K^jD^{j-1})=\dim(\SN^1)$$
so $\SN^1$ is the irreducible component of $\SN$ of maximum dimension, and it is the only component with such dimension. Therefore, $f(\SN^1)=\SN^1$.
\end{proof}

For each $a\in H^0(K)$, let
$$\SN^1(a)=\left \{(x-a)^2(x^{r-2}+2ax^{r-3}+b_2x^{r-4}+\ldots +b_{r-2})\right \}$$
where $b_j\in H^0(K^jD^{j-1})$. Then, by definition
$$\SN^1=\bigcup_{a\in H^0(K)} \SN^1(a)$$
We can analyze the geometry of $\SN^1$ in terms of $\SN^1(a)$. Let us consider the following intersection variety constructed as a disjoint union of the slices $\SN^1(a)$
$$\SI=\coprod_{a\in H^0(K)} \{a\}\times \SN^1(a) \subsetneq H^0(K)\times W$$
It is clearly the preimage of $0$ under the map
\begin{eqnarray*}
\xymatrixrowsep{0.05pc}
\xymatrixcolsep{0.3pc}
\xymatrix{
F&:&H^0(K)\times W \ar[rrrr] &&&& H^0(K^{r-1}D^{r-2})\times H^0(K^rD^{r-1})\\
&& (a,s_2,\ldots,s_r) \ar@{|->}[rrrr] &&&& {\left ( \begin{array}{c}
ra^{r-1} +\sum_{k=2}^{r-1} (r-k) s_ka^{r-k-1}\\
a^r+\sum_{k=2}^r s_ka^{r-k}
\end{array} \right)}
}
\end{eqnarray*}

\begin{lemma}
\label{lemma:sectionsPowerIndependent}
There exist a basis $\{w_i\}$ of $H^0(K)$ such that the sections $w_i^r \in H^0(K^rD^{r-1})$ are linearly independent.
\end{lemma}

\begin{proof}
Assume that the lemma is false. Then let us prove that the image of $H^0(K)$ under the algebraic map
$$H^0(K) \stackrel{(\cdot)^r}{\longrightarrow} H^0(K^rD^{r-1})$$
is contained in some linear subspace $V\subset H^0(K^rD^{r-1})$ of dimension at most $g-1$. Let $m<g$ be the maximum rank of the images of a basis $\{w_1,\ldots,w_g\}\subset H^0(K)$. Then there is some basis $\{w_1,\ldots,w_g\}$ such that for each $i>m$, $w_i^r$ belongs to the $m$-dimensional linear space
$$V=\op{Span}(\{w_j^r\}_{j\le m})\subset H^0(K^rD^{r-1})$$
In particular, as $\{w_j^r\}_{j\le m}$ generate a subspace of the maximum dimension, the images of the vectors of any other basis containing $\{w_j\}_{j\le m}$, must be contained in $V$. In particular, if we pick any $w_g'\in U=H^0(K)\backslash \op{Span}(\{w_j\}_{j<g})$, then $\{w_1,\ldots,w_{g-1},w_g'\}$ is a basis of $H^0(K)$ and we get that $(w_g')^r\in V$. Therefore, the image of the open subset $U=H^0(K)\backslash \op{Span}(\{w_j\}_{j<g}) \subset H^0(K)$ is contained in $V$. As $U$ is dense and the map $H^0(K)\to H^0(K^rD^{r-1})$ is continuous, the whole image of the map must be contained in $V$. Then, by upper semicontinuity of the dimension of the fibers, all the fibers of the algebraic map $H^0(K) \to V$ must have dimension at least $1$. In particular, there must exist a nonzero $w\in H^0(K)$ such that $w^r=0$, but this is impossible.
\end{proof}

Let $\pi_1:\SI\to H^0(K)$ and $\pi_2:\SI\to \SN^1$ be the canonical projections.

\begin{lemma}
\label{lemma:projC1finite}
The map $\pi_2:\SI\twoheadrightarrow \SN^1$ sending $(a,s)\mapsto s$ is a finite map.
\end{lemma}

\begin{proof}
The fibers of the map are clearly finite, so it is only necessary to prove that $\CC[\SI]$ is a finite algebra over $\CC[\SN^1]$. We know that $\SI\subset H^0(K)\times W$ is defined by the equations $F(a,\overline{s})=0$. Let $I_{\SN^1}$ be the ideal defining $\SN^1\subset W$ and let $\{w_i\}_{i=1}^g$ be a basis of $H^0(K)$ as in the Lemma \ref{lemma:sectionsPowerIndependent}. Then it is straightforward to check that
$$\CC[\SI]\cong \frac{\CC[W][t_1,\ldots,t_g]}{I_{\SN^1}+I}$$
where $I\subset \CC[W][t_1,\ldots,t_g]$ is the ideal generated by each of the components of the vector
$$\left ( r \left (\sum_{i=1}^g t_iw_i \right)^{r-1}-\sum_{k=2}^{r-1} (r-k) s_k \left (\sum_{i=1}^g t_iw_i \right)^{r-k-1}, \\
\left (\sum_{i=1}^g t_iw_i \right)^r-\sum_{k=2}^r s_k \left (\sum_{i=1}^g t_iw_i \right)^{r-k}\right)$$
in any basis of $H^0(K^{r-1}D^{r-2}) \oplus H^0(K^rD^{r-1})$ extending $\{w_i^r\}$.

Therefore, in order to prove that $\CC[\SI]$ is a finitely generated $\CC[\SN^1]=\CC[W]/I_{\SN^1}$-module, it is enough to find a relation in $I$ between $t_i^r$ and lower order terms $t_j^k$ with $k<r$ and coefficients in $\CC[\SN^1]$. Observe that
$$\left( \sum_{i=1}^g t_iw_i \right)^r-\sum_{k=2}^r s_k \left (\sum_{i=1}^g t_iw_i \right)^{r-k}=\sum_{i=1}^g t_i^r w_i^r + O(\{t_j^{r-1}\})$$
As $w_i^r$ are linearly independent, taking the $w_i^r$ coordinate of this vector we obtain an expression of the form $t_i^r+O(\{t_j^{r-1}\})$ which, by construction, has coefficients in $\CC[\SN^1]$ and belongs to $I$ for every $i=1,\ldots,g$. Therefore, $\CC[\SI]$ is generated as a $\CC[\SN^1]$-module by $\{t_1^{j_1}\cdots t_g^{j_g} | j_i<r\}$.
\end{proof}

\begin{lemma}
\label{lemma:projC1smooth}
There is an open nonempty set $\SU^{\op{sm}}\subset \SN^1$ such that the differential of the map $\pi_2: \pi_2^{-1}(\SU^{\op{sm}}) \to \SN^1$ is invertible at every point.
\end{lemma}

\begin{proof}
The differential of the map $\pi_2$ is invertible over the points $s\in \SN^1$ such that $H^0(K)$ is transverse to $T_{(a,s)}\SI\subset H^0(K)\oplus W$. Let us compute the tangent space to $\SI$. By construction it is the kernel of the differential of $F$
$$dF:H^0(K)\oplus W \longrightarrow H^0(K^{r-1}D^{r-2}) \oplus H^0(K^rD^{r-1})$$
It is straightforward to compute the differential at a point $(a,s)$ from the equations of $F$. If $(\alpha,\sigma_2,\ldots,\sigma_r)\in T_{(a,s)}H^0(K)\times W \cong H^0(K)\oplus W$, then
\begin{multline*}
dF(\alpha,\sigma_2,\ldots,\sigma_r)=\\
\left ( \begin{array}{c}
( ra^{r-1} +\sum_{k=2}^{r-1} (r-k) \sigma_ka^{r-k-1} ) + ( r(r-1)a^{r-2} +\sum_{k=2}^{r-2} (r-k)(r-k-1) s_ka^{r-k-2}) \alpha\\
( a^r+\sum_{k=2}^r \sigma_ka^{r-k} )+( ra^{r-1}+\sum_{k=2}^{r-1} (r-k)s_ka^{r-k-1}) \alpha
\end{array} \right)
\end{multline*}
As $(a,s)\in \SI=F^{-1}(0)$ the last summand in the second component is zero, so the equations of $T_{(a,s)}\SI$ become
$$\left\{ \begin{array}{c}
ra^{r-1} +\sum_{k=2}^{r-1} (r-k) \sigma_ka^{r-k-1} ) + ( r(r-1)a^{r-2} +\sum_{k=2}^{r-2} (r-k)(r-k-1) s_ka^{r-k-2}) \alpha=0\\
a^r+\sum_{k=2}^r \sigma_ka^{r-k} =0
\end{array} \right.$$
Therefore, $H^0(K)$ fails to be transverse to $T_{(a,s)}\SI$ if and only if
$$ r(r-1)a^{r-2} +\sum_{k=2}^{r-2} (r-k)(r-k-1) s_ka^{r-k-2}=0$$
This, together with the assumption that $F(a,s)=0$, implies that the polynomial corresponding to $s$ admits a decomposition
$$p_s(x)=(x-a)^3q(x)$$
for some $q$. Repeating the dimension counting argument in Lemma \ref{lemma:preserveR1} we obtain that the set of point admitting such decomposition has positive codimension in $\SN^1$, so its complement $\SU^{\op{sm}}\subset \SN^1$ is an open nonempty set. For $r=2$, $\SU^{\op{sm}}=\SN^1$, for $r=3$, $\SU^{\op{sm}}=\SN^1\backslash \{0\}$ and for $r>3$
$$\dim(\SN^1\backslash \SU^{\op{sm}})= h^0(K)+\sum_{j=2}^{r-3} h^0(K^jD^{j-1})<\sum_{j=1}^{r-2} h^0(K^jD^{j-1})=\dim(\SN^1)$$
\end{proof}

\begin{lemma}
\label{lemma:R1andIbirational}
The projection $\pi_2: \SI \to \SN^1$ is a $\CC^*$-equivariant birational map.
\end{lemma}

\begin{proof}
The space of points in $\SN^1$ admitting at least a decomposition of the form
$$p(x)=(x-a)^2q(x)$$
for at least two different sections $a\in H^0$ corresponds to the points in $\SN^1$ admitting a decomposition of the form
$$p(x)=(x-a)^2(x-b)^2 q(x)$$
For some $a,b$. Again, repeating the dimension argument used in Lemma \ref{lemma:preserveR1}, we obtain that the dimension of this subset is less than the dimension of $\SN^1$. Let $\SU^{\op{bi}}$ denote its complement in $\SN^1$. Then for $r<4$, $\SU^{\op{bi}}=\SN^1$. For $r=4$
$$\dim(\SN^1\backslash \SU^{\op{bi}})= h^0(K)< h^0(K)+h^0(K^2D^1)=\dim(\SN^1)$$
and for $r>4$
\begin{multline*}
\dim(\SN^1\backslash \SU^{\op{bi}})= 2h^0(K)+\sum_{j=2}^{r-4} h^0(K^jD^{j-1})\\
=\sum_{j=1}^{r-3}h^0(K^jD^{j-1}) - (h^0(K^{r-3}D^{r-4})-h^0(K)) <\sum_{j=1}^{r-2} h^0(K^jD^{j-1})=\dim(\SN^1)
\end{multline*}
Therefore, there exists an open nonempty subset $\SU^{\op{bi}}\subset \SN^1$ consisting on points $s$ whose preimage $\pi_2^{-1}(s)$ is a single point. On the other hand, by Lemma \ref{lemma:projC1smooth}, there exist a subset $\SU^{\op{sm}}$ such that the differential of the map $\pi_2|_{\SU^{\op{sm}}}$ is invertible. By Lemma \ref{lemma:projC1finite}, we know that $\pi_2$ is a finite map, so restricting it to $\SU=\SU^{\op{bi}}\cap \SU^{\op{sm}}$, we obtain a finite bijective map with invertible differential. By \cite[Theorem 14.9 and Corollary 14.10]{Harris13}, $\pi_2|_{\pi_2^{-1}(\SU)} : \pi_2^{-1}(\SU) \to \SU$ is an isomorphism and, therefore, it induces a birational map between $\SI$ and $\SN^1$. 
\end{proof}

\begin{lemma}
\label{lemma:recoverLastProjections}
Suppose that $r\ge 3$. Let $f:W\to W$ be a $\CC^*$-equivariant isomorphism such that $f(\SN)=\SN$. Then there is a $\CC^*$-equivariant isomorphism $g:W_{\le (r-2)}\to W_{\le (r-2)}$ and linear maps $f_j:W_j\to W_j$ for $j=r-1,r$ such that the following diagrams commute
\begin{eqnarray*}
\xymatrixcolsep{3.2pc}
\xymatrix{
W \ar[r]^{f} \ar[d]_{\pi_{\le (r-2)}} & W \ar[d]^{\pi_{\le (r-2)}} & W \ar[r]^{f} \ar[d]_{\pi_{r-1}} & W \ar[d]^{\pi_{r-1}} & W \ar[r]^{f} \ar[d]_{\pi_r} & W \ar[d]^{\pi_r} \\
W_{\le (r-2)} \ar[r]^{g} & W_{\le (r-2)} & W_{r-1} \ar[r]^{f_{r-1}} & W_{r-1} & W_r \ar[r]^{f_r} & W_r
}
\end{eqnarray*}
\end{lemma}

\begin{proof}
Taking into account the block decomposition in Lemma \ref{lemma:fBlockStructure}, it is enough to prove that the map $(g_{r-1},g_r): W_{\le (r-2)}\to W_{r-1}\oplus W_r$ is zero. By definition, $W_{\le (r-2)}=\SN^1(0)\subset \SN^1$, so Lemma \ref{lemma:preserveR1} implies that $f(W_{\le (r-2)})\subset \SN^1$. The dimension of $W_{\le (r-2)}$ is
$$\dim(W_{\le (r-2)})=\sum_{j=2}^{r-2} h^0(K^jD^{j-1})$$
so comparing it with the dimensions of $\SN^1\backslash \SU^{\op{sm}}$ and $\SN^1\backslash \SU^{\op{bi}}$ computed in Lemmas \ref{lemma:projC1smooth} and \ref{lemma:R1andIbirational}, we obtain that $\dim(\SN^1\backslash \SU^{\op{sm}})<\dim(W_{\le (r-2)})$ and $\dim(\SN^1\backslash \SU^{\op{bi}})<\dim(W_{\le (r-2)})$. Therefore, $\SU\cap f(W_{\le (r-2)})= \SU^{\op{sm}} \cap \SU^{\op{bi}} \cap f(W_{\le (r-2)})$ is an open dense subset of $f(W_{\le (r-2)})$.

Thus, applying Lemma \ref{lemma:R1andIbirational}, we obtain a $\CC^*$-equivariant rational map $f(W_{\le (r-2)})\dashrightarrow \SI$. On the other hand, let us consider the isomorphism $\overline{g}:W_{\le (r-2)}\to W_{\le (r-2)}$ given by the decomposition in blocks of $f:W\to W$ described in Lemma \ref{lemma:fBlockStructure}. Composing the rational map $f(W_{\le (r-2)})\dashrightarrow \SI$ with the canonical projection, $\pi_1:\SI \to H^0(K)$, and the map $\tilde{f}=f\circ \overline{g}^{-1}:W_{\le (r-2)}\to f(W_{\le (r-2)})$, we obtain a rational map $t: W_{\le (r-2)} \dashrightarrow H^0(K)$, satisfying the following property. Let $(\overline{s},s_{r-1},s_r)=(s_2,\ldots,s_r)\in f(W_{\le (r-2)})$ be a generic point. Then $t(\overline{s})\in H^0(K)$ is the only section such that
$$\left \{ \begin{array}{c}
t(\overline{s})^r+\sum_{j=2}^r s_j t(\overline{s})^{r-j}=0\\
rt(\overline{s})^{r-1}+\sum_{j=2}^{r-1} (r-j)s_j t(\overline{s})^{r-j-1} =0
\end{array} \right .$$
In particular, solving for $s_{r-1}$ and $s_r$ we obtain that
$$\left \{ \begin{array}{c}
s_{r-1}=-rt(\overline{s})^{r-1}-\sum_{j=2}^{r-2} (r-j)s_j t(\overline{s})^{r-j-1}\\
s_r=(r-1)t(\overline{s})^r+\sum_{j=2}^{r-2} (r-j-1) s_j t(\overline{s})^{r-j}
\end{array} \right .$$
On the other hand, as $(\overline{s},s_{r-1},s_r)\in f(W_{\le (r-2)})$, by the block decomposition we know that
$$\left \{ \begin{array}{c}
s_{r-1}=g_{r-1}\circ \overline{g}^{-1}(\overline{s}):=\tilde{g}_{r-1}(\overline{s})\\
s_r=g_r\circ \overline{g}^{-1}(\overline{s}):=\tilde{g_r}(\overline{s})
\end{array} \right .$$
so
\begin{equation}
\label{eq:ruffiniSections}
\left \{ \begin{array}{c}
\tilde{g}_{r-1}(\overline{s})=-rt(\overline{s})^{r-1}-\sum_{j=2}^{r-2} (r-j)s_j t(\overline{s})^{r-j-1}\\
\tilde{g_r}(\overline{s})=(r-1)t(\overline{s})^r+\sum_{j=2}^{r-2} (r-j-1) s_j t(\overline{s})^{r-j}
\end{array} \right .
\end{equation}
As $t:W_{\le (r-2)}\dashrightarrow H^0(K)$ is a $\CC^*$-equivariant rational map between vector spaces there are three possibilities for its structure
\begin{enumerate}
\item $t=0$, in which case we would get $g_{r-1}=0$ and $g_r=0$ leading to the desired result.
\item $t:W_{\le (r-2)}\to H^0(K)$ is an homogeneous polynomial. This is impossible because the action of $\CC^*$ in $W_{\le (r-2)}$ is of order at least 2 and the action of $\CC^*$ in $H^0(K)$ has order $1$. 
\item $t(\overline{s})=\frac{\alpha(\overline{s})}{\beta(\overline{s})}$ for some homogeneous polynomials $\alpha$ and $\beta$ with no common factors.
\end{enumerate}
Then it is only left to prove that (3) is also impossible. Substituting $t=\alpha/\beta$ in \eqref{eq:ruffiniSections} we obtain the following equality
\begin{equation*}
\left \{ \begin{array}{c}
\beta(\overline{s})^{r-1}\tilde{g}_{r-1}(\overline{s})=-r\alpha(\overline{s})^{r-1}-\sum_{j=2}^{r-2} (r-j) s_j \alpha(\overline{s})^{r-j-1} \beta(\overline{s})^j\\
\beta(\overline{s})^r\tilde{g_r}(\overline{s})=(r-1)\alpha(\overline{s})^r+\sum_{j=2}^{r-2} (r-j-1) s_j \alpha(\overline{s})^{r-j}\beta(\overline{s})^j
\end{array} \right .
\end{equation*}
Nevertheless, looking at the last equation modulo $\beta$ we get that $\alpha^r$ is a multiple of $\beta$, thus contradicting that $\alpha$ and $\beta$ do not share a common factor.
\end{proof}

In order to prove that $f$ is linear and decomposes diagonally, we will apply the previous lemma inductively. For each $k>1$, let $ \SN_k\subset \bigoplus_{j=2}^k H^0(K^jD^{j-1})$ be the set of non-reduced ``rank $k$'' spectral curves, i.e., the set of spectral curves defined by degree $k$ polynomials of the form
$$x^k+\sum_{j=2}^k s_j x^{k-j}=0$$
for $s_j\in H^0(K^jD^{j-1})$ which have at least a non-reduced component. With a slight abuse of notation let us also denote by $\SN_k\subset W_{\le k}$ the image of the set of rank $k$ non-reduced spectral curves under the inclusion
$$\SN_k\subseteq \bigoplus_{j=2}^k H^0(K^jD^{j-1}) \subseteq \bigoplus_{j=2}^r H^0(K^jD^{j-1})$$
In other words,
$$\SN_k=\{x^{r-k}q(x) \, \,  |  \, \, q(x)=p_1(x)^2p_2(x) \text{ for some } p_1(x) \text{ and } p_2(x)\}$$

\begin{lemma}
\label{lemma:inductiveRecoverProjection}
Let $k\ge 3$ and let $f_{\le k}:W_{\le k} \to W_{\le k}$ be a $\CC^*$-equivariant isomorphism such that $f_{\le k}(\SN_k)=\SN_k$. Then there is a $\CC^*$-equivariant isomorphism $f_{\le (k-2)}:W_{\le (k-2)}\to W_{\le (k-2)}$ and linear maps $f_j:W_j\to W_j$ for $j=k-1,k$ such that $f_{\le(k-2)}(\SN_{k-2})=\SN_{k-2}$ and the following diagrams commute
\begin{eqnarray*}
\xymatrixcolsep{3.2pc}
\xymatrix{
W_{\le k} \ar[r]^{f_{\le k}} \ar[d]_{\pi_{\le (k-2)}} & W_{\le k} \ar[d]^{\pi_{\le (k-2)}} & W_{\le k} \ar[r]^{f_{\le k}} \ar[d]_{\pi_{k-1}} & W_{\le k} \ar[d]^{\pi_{k-1}} & W_{\le k} \ar[r]^{f_{\le k}} \ar[d]_{\pi_k} & W_{\le k} \ar[d]^{\pi_k} \\
W_{\le (k-2)} \ar[r]^{f_{\le (k-2)}} & W_{\le (k-2)} & W_{k-1} \ar[r]^{f_{k-1}} & W_{r-1} & W_k \ar[r]^{f_k} & W_k
}
\end{eqnarray*}
\end{lemma}

\begin{proof}
Applying the Lemma \ref{lemma:recoverLastProjections} to $r=k$, we obtain the desired diagonal decomposition $f_{\le k}=(f_{\le (k-2)},f_{k-1},f_k):W_{\le (k-2)} \oplus W_{k-1} \oplus W_k \to W_{\le (k-2)} \oplus W_{k-1} \oplus W_k$. Therefore, it is enough to prove that $f_{\le (k-2)}$ preserves $\SN_{k-2}$. We know that $\SN_k$ decomposes in irreducible components as
$$\SN_k=\bigcup_{d=1}^{\lfloor k/2 \rfloor} \SN_k^d$$
where
$$\SN_k^d=\left \{(x^d+a_1 x^{d-1}+ \ldots + a_d)^2(x^{k-2d}-2a_1 x^{k-2d-1}+b_2 x^{k-2d-2}+ \ldots + b_{k-2d}) \right \}$$
for $d<k/2$ and, if $k$ is even,
$$\SN_k^{k/2}=\left\{ (x^{k/2}+a_2x^{k/2-2}+\ldots+a_{k/2})^2\right\}$$
By hypothesis we known that $f_{\le k}(\SN_k)=\SN_k$ and, by Lemma  \ref{lemma:preserveR1}, $f_{\le k}(\SN_k^1)=\SN_k^1$, so, $f_{\le k}$ must preserve the union of the rest of the components.
$$f_{\le k} \left ( \bigcup_{d=2}^{\lfloor k/2 \rfloor} \SN_k^d \right)=\bigcup_{d=2}^{\lfloor k/2 \rfloor} \SN_k^d$$
On the other hand, for each $d>1$ consider the intersection $W_{\le (k-2)} \cap \SN_k^d \subset W_{\le k}$. The elements in $W_{\le (k-2)}$ correspond to polynomials $p(x)\in W_{\le k}$ which have at least a factor $x^2$, i.e.
$$W_{\le (k-2)}=\{x^2q(x) \in W_{\le k}\}$$
On the other hand, the elements in $\SN_k^d$ are polynomials with at least a double factor of order $d$
$$\SN_k^d=\{p_1(x)^2p_2(x) \, \, | \, \, \deg(p_1)=d\}$$
Then when we get the intersection, for each polynomial of the form $p(x)=p_1(x)^2p_2(x)\in W_{\le (k-2)} \cap \SN_k^d$ there are two possibilities
\begin{enumerate}
\item Either the $x^2$ factor is included in $p_1(x)$, so $p_1(x)=xq_1(x)$ for some $q$ of degree $d-1$ and then
$$p(x)=x^2q_1(x)^2p_2(x)\in \SN_{k-2}^{d-1}$$
\item or the $x^2$ factor is included in $p_2(x)$, so $p_2(x)=x^2q_2(x)$ and then
$$p(x)=x^2p_1(x)^2q_2(x) \in \SN_{k-2}^d$$
\end{enumerate}
and the latter can only happen if $d\le (k-2)/2$. Therefore, we conclude that
$$W_{\le (k-2)}\cap \SN_k^d=\left \{ \begin{array}{ll}
\SN_{k-2}^{d-1} \cup \SN_{k-2}^d & d\le (k-2)/2\\
\SN_{k-2}^{d-1} & d>(k-2)/2
\end{array} \right.$$
In particular, taking the full union for $d>1$ yields
$$W_{\le (k-2)} \cap \bigcup_{d=2}^{\lfloor k/2 \rfloor} \SN_k^d = \bigcup_{d=1}^{\lfloor (k-2)/2 \rfloor} \SN_{k-2}^d=\SN_{k-2}$$
As $f_{\le k}$ preserves both $W_{\le (k-2)}$ and the union of the components $\SN_k^d$ for $d>1$, we obtain that $f_{\le k}(\SN_{k-2})=\SN_{k-2}$. Finally, as $\SN_{k-2}\subset W_{\le (k-2)}$ and we already know that $f$ decomposes diagonally with respect to the last two factors $W_{k-1}$ and $W_k$, then $f_{\le (k-2)}(\SN_{k-2})=\SN_{k-2}$.
\end{proof}

Now we can apply the previous lemma inductively and combine it with the previous results to recover the diagonal decomposition.

\begin{lemma}
\label{lemma:recoverHitchin}
Let $f:W\to W$ be a $\CC^*$-equivariant isomorphism such that $f(\SD)=\SD$. Then for every $k>1$, there exist a linear automorphism $f_k:W_k\to W_k$ such that the following diagram commutes
\begin{eqnarray}
\xymatrixcolsep{4pc}
\xymatrix{
W \ar[r]^{f} \ar[d]_{\pi_k} & W \ar[d]^{\pi_k} \\
W_k \ar[r]^{f_k} & W_k
}
\end{eqnarray}
\end{lemma}

\begin{proof}
By Lemma \ref{lemma:recoverDx}, the map $\op{sing}: \SD \dashrightarrow X$ commutes with $f:\SD\to \SD$, so $f$ preserves the closure of the fibers $\SD_x=\overline{\op{sing}^{-1}(x)}$. Then, it preserves its intersection, but we know by construction that $\SN_r=\bigcap_{x\in X} \SD_x$, so $f(\SN_r)=\SN_r$. Moreover, $f:W\to W$ is $\CC^*$-equivariant by hypothesis, so we can apply Lemma \ref{lemma:inductiveRecoverProjection} and we obtain that $f=f_{\le r}$ commutes with the projections into $W_{\le {k-2}}$, $W_{r-1}$ and $W_r$, decomposing diagonally as
$$f_{\le r}=(f_{\le (r-2)},f_{r-1},f_r):W_{\le (r-2)} \oplus W_{r-1} \oplus W_r \longrightarrow W_{\le (r-2)} \oplus W_{r-1} \oplus W_r$$
with $f_{r-1}$ and $f_r$ linear maps. Moreover $f_{\le (r-2)}(\SN_{r-2})=\SN_{r-2}$. Now we can restrict ourselves to $W_{\le (r-2)}$. We proved that we have a $\CC^*$-equivariant isomorphism $f_{\le(r-2)}:W_{\le (r-2)} \to W_{\le (r-2)}$ such that $f_{\le(r-2)}(\SN_{r-2})=\SN_{r-2}$, so we can apply Lemma \ref{lemma:inductiveRecoverProjection} again and find that $f_{\le (r-2)}$ decomposes as
$$f_{\le (r-2)}=(f_{\le (r-4)},f_{r-3},f_{r-2}):W_{\le (r-4)} \oplus W_{r-3} \oplus W_{r-2} \longrightarrow W_{\le (r-4)} \oplus W_{r-3} \oplus W_{r-2}$$
and, moreover $f_{\le (r-4)}(\SN_{r-4})=\SN_{r-4}$. This together with the previous part proves that $f:W\to W$ decomposes as
$$f=(f_{\le (r-4)},f_{r-3},\ldots,f_{r}):W_{\le (r-4)} \oplus W_{r-3} \oplus \cdots \oplus W_{r} \longrightarrow W_{\le (r-4)} \oplus W_{r-3} \oplus \cdots \oplus W_{r}$$
Where $f_j$ are linear for $j\ge r-3$. Repeating this argument successively, we arrive to two different situations depending on the parity of $r$.

If $r$ is even, we arrive to a diagonal decomposition $f=(f_2,\ldots,f_r)$ with $f_j:W_j\to W_j$ linear, so we are done. If $r$ is even, we obtain a diagonal decomposition $f=(f_{\le 2}, f_3,\ldots,f_r)$ with $f_j:W_j\to W_j$ linear for each $j>2$ and $f_{\le 2}:W_2 \to W_2$ a $\CC^*$-equivariant isomorphism. Then, simply apply the $r=2$ case of Lemma \ref{lemma:fBlockStructure} to $f_{\le 2}$ to prove that it is a linear isomorphism. 
\end{proof}

In particular, combining the previous lemma with diagram \ref{eq:fDefinition}, we obtain
\begin{corollary}
\label{cor:preservehr}
For every $k>1$, there exist a linear automorphism $f_k:W_k\to W_k$ such that the following diagram commutes
\begin{eqnarray}
\xymatrixcolsep{4pc}
\xymatrix{
T^*\SM(r,\alpha,\xi) \ar[r]^{d(\Phi^{-1})} \ar[d]_{h_k} & T^*\SM(r,\alpha',\xi') \ar[d]^{h_k} \\
W_k \ar[r]^{f_k} & W_k
}
\end{eqnarray}
\end{corollary}

Once we have characterized $f_r:W_r\to W_r$, we can further state the following Lemma.

\begin{lemma}
\label{lemma:preserveWr}
Let $f_r:W_r\to W_r$ be the linear automorphism constructed in Corollary \ref{cor:preservehr}. Then for every $k>0$ and every $x_0\in X$
$$f_r\left (H^0(K^rD^{r-1}(-kx_0))\right)=H^0(K^rD^{r-1}(-kx_0))$$
\end{lemma}

\begin{proof}
As $d(\Phi^{-1})$ is an isomorphism, it maps complete rational curves on the cotangent bundle to complete rational curves. By Lemma \ref{lemma:DiscriminantCharacterization}, the morphism $f$ must preserve $\SC=\SD\cap W_r$. Applying \ref{prop:recoverDualVariety}, we can decompose $\SC=\SC_X\cup\bigcup_{x\in D} \SC_x$, where $\PP(\SC_X)$ is the dual variety of $X\subset \PP(W_r^*)$ and $\PP(\SC_x)$ is the set of hyperplanes going through $x\in X\subset \PP(W_r^*)$. Moreover, we know that $\PP(\SC_X)$ is not isomorphic to $\PP(\SC_x)$ for any $x$, so $f$ must preserve $\SC_X$. As we assumed that the induced automorphism of the dual variety $\sigma:X\to X$ is the identity, then for each $x_0\in X$, the projectivization of $f$ must preserve all the osculating spaces at $x_0$. The osculating $k$ space at $x_0\in X\subset \PP(W_r^*)$ is precisely $\PP(H^0(K^rD^{r-1}(-kx_0)))$. As $f:W_r\to W_r$ is linear, we conclude that it preserves $H^0(K^rD^{r-1}(-kx_0))$.
\end{proof}

\begin{lemma}
\label{lemma:recoverSParEndsections}
Suppose that $g\ge 4$. Let $(E,E_\bullet)\in \SM(r,\alpha,\xi)$ and $(E',E'_\bullet) \in \SM(r,\alpha',\xi')$ be generic stable parabolic vector bundles such that $\Phi(E,E_\bullet)=(E',E'_\bullet)$. Consider the isomorphism of vector spaces
$$d(\Phi^{-1}): H^0(\SParEnd_0(E)\otimes K_X(D)) \longrightarrow H^0(\SParEnd_0(E')\otimes K_X(D))$$
Then for every $x\in U=X\backslash D$, the image of $H^0(\SParEnd_0(E)\otimes K_X(D-x))$ under $d(\Phi^{-1})$ is $H^0(\SParEnd_0(E')\otimes K_X(D-x))$.
\end{lemma}

\begin{proof}
Let $x_0\in U$. Let $(E,E_\bullet)$ be a generic stable parabolic vector bundle in the sense of Lemma \ref{lemma:ParEndNoSections}. We will prove that
$$H^0(\SParEnd_0(E)\otimes K_X(D-x))=\{\psi \in H^0(\SParEnd_0(E)\otimes K_X(D)) : \forall \varphi\in h_r^{-1}(H_{x_0}) \,\,\, h_r(\psi+\varphi)\in H_{x_0}\}$$
where $H_{x_0}=H^0(K^rD^{r-1}(-x_0)) \subseteq W_r=H^0(K^rD^{r-1})$. By Lemma \ref{lemma:preserveWr}, $H_{x_0}$ is preserved by $f_r:W_r\to W_r$, so the Lemma follows from commutativity of diagram \eqref{eq:frDefinition}.

As we assumed $g\ge 4$, by Lemma \ref{lemma:ParEndNoSections}, for a generic $(E,E_\bullet)$
$$H^1(\SParEnd_0(E,E_\bullet)\otimes K(D-x_0))=H^0(\ParEnd_0(E,E_\bullet)(x_0))^\vee=0$$
Therefore, for a generic parabolic vector bundle the following sequence is exact
\begin{multline*}
0 \longrightarrow H^0(\SParEnd_0(E,E_\bullet)\otimes K(D-x_0)) \longrightarrow H^0(\SParEnd_0(E,E_\bullet)\otimes K(D))\\
 \longrightarrow \SParEnd_0(E,E_\bullet)\otimes K(D)|_{x_0} \longrightarrow 0
\end{multline*}
Therefore, the evaluation map
\begin{equation}
\label{eq:EvMapSurjective}
H^0(\SParEnd_0(E,E_\bullet)\otimes K(D)) \twoheadrightarrow\SParEnd_0(E,E_\bullet)\otimes K(D)|_{x_0}\cong \End(E)|_{x_0}\otimes K(D)|_{x_0}
\end{equation}
is surjective. By definition of the Hitchin map $h_r(\psi)\in H_{x_0}$ if and only if $\det(\psi(x_0))=0$. On the other hand, $\psi\in H^0(\SParEnd_0(E,E_\bullet)\otimes K(D-x_0))$ if and only if $\psi(x_0)=0$.

We will use the following algebra fact,  $\psi(x_0)\in \End(E)|_{x_0}\otimes K(D)|_{x_0}$ is zero if and only if for every other matrix $M\in \End(E)|_{x_0}\otimes K(D)|_{x_0}$ such that $\det(M)=0$, $\det(\psi(x_0)+M)=0$. Finally, as the evaluation map \eqref{eq:EvMapSurjective} is surjective, the latter is equivalent to $\det(\psi(x_0)+\varphi(x_0))=0$ for every $\varphi\in h_r^{-1}(H_{x_0})$.
\end{proof}

\begin{lemma}
\label{lemma:recoverSParEnd}
Suppose that $g\ge 4$. Let $(E,E_\bullet)$ and $(E',E'_\bullet)$ be generic parabolic vector bundles such that $\Phi(E,E_\bullet)=(E',E'_\bullet)$. Then $\Phi$ induces an isomorphism of vector bundles
$$\Phi_{\SParEnd_0}: \SParEnd_0(E,E_\bullet) \cong \SParEnd_0(E',E'_\bullet)$$
\end{lemma}

\begin{proof}
Let $\SE$ be the sub-bundle of the trivial vector bundle
$$H^0(\SParEnd_0(E,E_\bullet))\otimes K(D))\otimes_{\CC}\SO_X \longrightarrow X$$
whose fiber over each $x\in X$ is $H^0(\SParEnd_0(E,E_\bullet))\otimes K(D-x))$. From Lemma \ref{lemma:ParEndNoSections}, the following sequence is exact
\begin{multline*}
0\to \SE \to H^0(\SParEnd_0(E,E_\bullet))\otimes K(D))\otimes_{\CC}\SO_X  \stackrel{\pi}{\longrightarrow} \SParEnd_0(E,E_\bullet) \otimes K(D) \to 0
\end{multline*}
where the last morphism is the evaluation map. Analogously, we define a vector bundle $\SE'$ such that
\begin{multline*}
0\to \SE' \to H^0(\SParEnd_0(E',E'_\bullet))\otimes K(D))\otimes_{\CC}\SO_X  \stackrel{\pi}{\longrightarrow} \SParEnd_0(E',E'_\bullet) \otimes K(D) \to 0
\end{multline*}

By Lemma \ref{lemma:recoverSParEndsections}, over $U=X\backslash D$, the image of $\SE|_U$ under $d(\Phi^{-1})\otimes \id_{\SO_U}$ is $\SE'|_U$. As $\SE$ and $\SE'$ are the saturations of $\SE|_U$ and $\SE'|_U$ in $H^0(\SParEnd_0(E,E_\bullet) \otimes K(D)) \otimes_\CC \SO_U$ and $H^0(\SParEnd_0(E,E_\bullet) \otimes K(D)) \otimes_\CC \SO_U$ respectively, the image of $\SE$ under $d(\Phi^{-1})\otimes \id_{\SO_X}$ must be $\SE'$. Therefore, passing to the quotient, there must exist an isomorphism of vector bundles
$$\Phi_{\SParEnd_0}: \SParEnd_0(E,E_\bullet) \cong \SParEnd_0(E',E'_\bullet)$$
such that the following diagram commutes
\begin{eqnarray*}
\xymatrixcolsep{1pc}
\xymatrix{
0 \ar[r] & \SE \ar[r] \ar[d] & H^0(\SParEnd_0(E,E_\bullet) \otimes K(D)) \otimes_\CC \SO_X \ar[d]^{d(\Phi^{-1}) \otimes \id_{\SO_X}} \ar[r]^-{\pi} & \SParEnd_0(E,E_\bullet) \otimes K(D) \ar[d]_{\Phi_{\SParEnd_0} \otimes \id_{K(D)}} \ar[r] & 0\\
0 \ar[r] & \SE' \ar[r] & H^0(\SParEnd_0(E',E'_\bullet) \otimes K(D)) \otimes_\CC \SO_X \ar[r]^-{\pi} & \SParEnd_0(E',E'_\bullet) \otimes K(D) \ar[r] & 0
}
\end{eqnarray*}
\end{proof}

\begin{lemma}
\label{lemma:recoverParEnd}
Suppose that $g\ge 6$. Let $(E,E_\bullet)$ and $(E',E'_\bullet)$ be generic parabolic vector bundles such that $\Phi(E,E_\bullet)=(E',E'_\bullet)$. Then $\Phi$ induces an isomorphism of vector bundles
$$\Phi_{\ParEnd_0}: \ParEnd_0(E,E_\bullet) \cong \ParEnd_0(E',E'_\bullet)$$
such that the following diagram commutes
\begin{eqnarray*}
\xymatrixcolsep{4pc}
\xymatrix{
\SParEnd_0(E,E_\bullet) \ar@{^(->}[d] \ar[r]^{\Phi_{\SParEnd_0}} & \SParEnd_0(E',E'_\bullet) \ar@{^(->}[d]\\
\ParEnd_0(E,E_\bullet) \ar[r]_{\Phi_{\ParEnd_0}} & \ParEnd_0(E',E'_\bullet)
}
\end{eqnarray*}
\end{lemma}

\begin{proof}
Given a parabolic point $x\in D$ and a parabolic vector bundle $(E,E_\bullet)$, let $\SParEnd^{(x)}_0(E,E_\bullet)$ be the subsheaf of $\SParEnd_0(E,E_\bullet)$ whose stalk over $y\in X\backslash \{x\}$ is $\SParEnd_0(E,E_\bullet)_y$ and whose stalk over $x$ is $\left(\ParEnd_0(E,E_\bullet)(-x)\right)_x$. It fits into a short exact sequence
$$0\longrightarrow \SParEnd^{(x)}_0(E,E_\bullet) \longrightarrow \SParEnd_0(E,E_\bullet) \longrightarrow \bigoplus_{k=1}^{r^2-r} \CC_x \longrightarrow 0$$
where the last morphism is the evaluation map at $x$ of the elements of $\SParEnd_0(E,E_\bullet)$ out of the diagonal, once a basis compatible with the parabolic filtration is chosen. More explicitly, if $(E,\{E_{i,y}\})$ is the parabolic vector bundle obtained by restricting the parabolic filtration to $y\in D$, then we define
$$\SParEnd_0^{(x)}(E,E_\bullet) = \bigcap_{y\in D\backslash \{x\}} \SParEnd_0(E,\{E_{i,y}\}) \cap \ParEnd_0(E,\{E_{i,x}\})(-x)$$

From the definition, it becomes clear that
\begin{equation}
\label{eq:SParEndx}
\SParEnd_0(E,E_\bullet)(-x) \hookrightarrow \SParEnd_0^{(x)}(E,E_\bullet) \hookrightarrow \SParEnd_0(E,E_\bullet)
\end{equation}
and these sheaves are related with $\ParEnd_0(E,E_\bullet)$ by the following relation
\begin{multline}
\label{eq:ParEndx}
\SParEnd_0(E,E_\bullet)(-D) \hookrightarrow \ParEnd_0(E,E_\bullet)(-D) =\\
 = \bigcap_{x\in D}\SParEnd_0^{(x)}(E,E_\bullet) \hookrightarrow \SParEnd_0(E,E_\bullet)
\end{multline}
We will prove that for every $x\in D$, $\Phi$ induces a morphism
$$\Phi_{\SParEnd_0}^{(x)}:\SParEnd_0^{(x)}(E,E_\bullet) \longrightarrow \SParEnd_0^{(x)}(E',E'_\bullet)$$
such that the following diagram commutes
\begin{eqnarray*}
\xymatrix{
\SParEnd_0(E,E_\bullet)(-x) \ar[d]^{\Phi_{\SParEnd_0} \otimes \id_{\SO_X(-x)}} \ar@{^(->}[r] & \SParEnd_0^{(x)}(E,E_\bullet) \ar[d]^{\Phi_{\SParEnd_0}^{(x)}} \ar@{^(->}[r] & \SParEnd_0(E,E_\bullet) \ar[d]^{\Phi_{\SParEnd_0}}\\
\SParEnd_0(E',E'_\bullet)(-x) \ar@{^(->}[r] & \SParEnd_0^{(x)}(E',E'_\bullet) \ar@{^(->}[r] & \SParEnd_0(E',E'_\bullet)
}
\end{eqnarray*}
Then $\Phi_{\SParEnd_0}$ preserves the subsheaf $\SParEnd_0^{(x)}(E,E_\bullet)$ and $\Phi_{\SParEnd_0}^{(x)}$ is simply the restriction of the morphism. Using the relation \eqref{eq:ParEndx}, we conclude that $\Phi_{\SParEnd_0}$ preserves $\ParEnd_0(E,E_\bullet)(-D)$, in the sense that it induces by restriction to the intersection a morphism
$$\overline{\Phi_{\ParEnd_0}}:\ParEnd_0(E,E_\bullet)(-D) \longrightarrow \ParEnd_0(E',E'_\bullet)(-D)$$
such that the following diagram commutes
\begin{eqnarray*}
\xymatrix{
\SParEnd_0(E,E_\bullet)(-D) \ar[d]^{\Phi_{\SParEnd_0} \otimes \id_{\SO_X(-D)}} \ar@{^(->}[r] & \ParEnd_0(E,E_\bullet)(-D) \ar[d]^{\overline{\Phi_{\ParEnd_0}}} \ar@{^(->}[r] & \SParEnd_0(E,E_\bullet) \ar[d]^{\Phi_{\SParEnd_0}}\\
\SParEnd_0(E',E'_\bullet)(-D) \ar@{^(->}[r] & \ParEnd_0(E',E'_\bullet)(-D) \ar@{^(->}[r] & \SParEnd_0(E',E'_\bullet)
}
\end{eqnarray*}
finally, tensoring the previous diagram by $\SO_X(D)$ and taking $\Phi_{\ParEnd_0}= \overline{\Phi_{\ParEnd_0}} \otimes \id_{\SO_X(D)}$ yields the desired vector bundle isomorphism.

Now let us build the morphism $\Phi_{\SParEnd_0}^{(x)}$. Let $(E,E_\bullet)$ be a generic parabolic vector bundle. Let us define the following subsets of $H^0(\SParEnd_0(E,E_\bullet)\otimes K(D))=T^*_{(E,E_\bullet)}\SM(r,\alpha,\xi)$ recursively.
\begin{equation}
\left \{\begin{array}{ll}
&F_{(E,E_\bullet)}^0=H^0(\SParEnd_0(E,E_\bullet)\otimes K(D))\\
\forall k>0 & G_{(E,E_\bullet)}^k=\{\psi\in F_{(E,E_\bullet)}^{k-1} : h_r(\psi)\in H^0(K^rD^{r-1}(-kx))\} \\
\forall k>0 & F_{(E,E_\bullet)}^k=\{\psi\in G_{(E,E_\bullet)}^k : \forall \varphi \in G_{(E,E_\bullet)}^k \, , \,\, \varphi+\psi\in G_{(E,E_\bullet)}^k\}\\
\forall k>0, y\in X \backslash D & G_{(E,E_\bullet),y}^k = \{\psi\in F_{(E,E_\bullet)}^k : h_r(\psi) \in H^0(K^rD^{r-1}(-kx-y))\}\\
\forall k>0, y\in X \backslash D & F_{(E,E_\bullet),y}^k=\{\psi\in G_{(E,E_\bullet),y}^k : \forall \varphi \in G_{(E,E_\bullet),y}^k \, , \,\, \varphi+\psi\in G_{(E,E_\bullet),y}^k\}\\
\end{array} \right .
\end{equation}
By Lemma \ref{lemma:preserveWr}, $f_r$ preserves $H^0(K^rD^{r-1}(-kx))$ for every $k$, so, by construction, for every $k>0$
$$\begin{array}{l}
d(\Phi^{-1}) \left(F_{(E,E_\bullet)}^k \right) = F_{(E',E'_\bullet)}^k \\
d(\Phi^{-1}) \left(G_{(E,E_\bullet)}^k \right) = G_{(E',E'_\bullet)}^k \\
d(\Phi^{-1}) \left(F_{(E,E_\bullet),y}^k \right) = F_{(E',E'_\bullet),y}^k \\
d(\Phi^{-1}) \left(G_{(E,E_\bullet),y}^k \right) = G_{(E',E'_\bullet),y}^k
\end{array}$$
We will prove the following equalities for $x\in D$ and $y\in X\backslash D$
\begin{equation}
\label{eq:recursiveAlgResult}
\begin{array}{l}
F_{(E,E_\bullet)}^{r-1}=H^0(\SParEnd_0^{(x)}(E,E_\bullet)\otimes K(D))\\
F_{(E,E_\bullet)}^r=H^0(\SParEnd_0(E,E_\bullet)\otimes K(D-x))\\
F_{(E,E_\bullet),y}^{r-1}=H^0(\SParEnd_0^{(x)}(E,E_\bullet)\otimes K(D-y))\\
F_{(E,E_\bullet),y}^r=H^0(\SParEnd_0(E,E_\bullet)\otimes K(D-x-y))
\end{array}
\end{equation}
Assuming that \eqref{eq:recursiveAlgResult} has been proven, we can build the map $\Phi_{\SParEnd_0}^{(x)}$ (and thus complete the proof of the Lemma) as follows. It is straightforward to test that
$$\left( \SParEnd_0^{(x)}(E,E_\bullet) \otimes \SO_X(D) \right)^\vee \cong \SParEnd_0(E,\{E_{i,x}\})(x) \cap \ParEnd_0(E,E_\bullet)(x) \hookrightarrow \End_0(E)(x)$$
As $g\ge 6$, Lemma \ref{lemma:ParEndNoSections} implies that for every $x\in D$ and every $y\in X$
$$H^1( \SParEnd_0^{(x)}(E,E_\bullet) \otimes K(D-y))=H^0\left(\left(\SParEnd_0^{(x)}(E,E_\bullet) \otimes \SO_X(D) \right)^\vee\otimes \SO_X(y)\right)^\vee= 0$$
$$H^1(\SParEnd_0(E,E_\bullet) \otimes K(D-x-y))=H^0\left(\ParEnd_0(E,E_\bullet) (x+y)\right)^\vee=0$$
Therefore, we have the following short exact sequences
\\ \hspace*{-2.3cm} \vbox{
\begin{eqnarray*}
\xymatrixcolsep{0.6pc}
\xymatrix{
0  \ar[r] & H^0(\SParEnd_0(E,E_\bullet)\otimes K(D-x-y)) \ar@{^(->}[d] \ar[r] &  H^0(\SParEnd_0(E,E_\bullet) \otimes K(D-x)) \ar@{^(->}[d] \ar[r] & \SParEnd_0(E,E_\bullet)\otimes K(D-x)|_y \ar@{^(->}[d] \ar[r] & 0  \\
0  \ar[r] & H^0(\SParEnd_0^{(x)}(E,E_\bullet)\otimes K(D-y)) \ar@{^(->}[d] \ar[r] &  H^0(\SParEnd_0^{(x)}(E,E_\bullet) \otimes K(D)) \ar@{^(->}[d] \ar[r] & \SParEnd_0^{(x)}(E,E_\bullet)\otimes K(D)|_y \ar@{^(->}[d] \ar[r] & 0 \\
0 \ar[r] & H^0(\SParEnd_0(E,E_\bullet)\otimes K(D-y)) \ar[r] &  H^0(\SParEnd_0(E,E_\bullet) \otimes K(D)) \ar[r] & \SParEnd_0(E,E_\bullet)\otimes K(D)|_y \ar[r] & 0  \\
}
\end{eqnarray*}
}\\
which are reduced to the following diagram if $y\in X\backslash D$ using \eqref{eq:recursiveAlgResult}
\begin{eqnarray*}
\xymatrixcolsep{1pc}
\xymatrix{
0 \ar[r] & F_{(E,E_\bullet),y}^{r} \ar@{^(->}[d] \ar[r] & F_{(E,E_\bullet)}^{r} \ar@{^(->}[d] \ar[r] & \SParEnd_0(E,E_\bullet)\otimes K(D-x)|_y \ar@{^(->}[d] \ar[r] &0 \\
0 \ar[r] & F_{(E,E_\bullet),y}^{r-1} \ar@{^(->}[d] \ar[r] & F_{(E,E_\bullet)}^{r-1} \ar@{^(->}[d] \ar[r] & \SParEnd_0^{(x)}(E,E_\bullet)\otimes K(D)|_y \ar@{^(->}[d] \ar[r] &0  \\
0 \ar[r] & F_{(E,E_\bullet),y}^0 \ar[r] & F_{(E,E_\bullet)}^0 \ar[r] & \SParEnd_0(E,E_\bullet)\otimes K(D)|_y \ar[r] &0
}
\end{eqnarray*}
Let $\SF$ and $\SG$ be the sub-vector bundles
$$\SF \hookrightarrow  H^0(\SParEnd_0^{(x)}(E,E_\bullet) \otimes K(D-x))\otimes \SO_X$$
and
$$\SG \hookrightarrow  H^0(\SParEnd_0(E,E_\bullet) \otimes K(D-x))\otimes \SO_X$$
whose fiber over $y\in X$ is $H^0(\SParEnd_0^{(x)}(E,E_\bullet)\otimes K(D-y))$ and $H^0(\SParEnd_0(E,E_\bullet)\otimes K(D-x-y))$ respectively. We define the vector bundles $\SF'$ and $\SG'$ analogously for $(E',E'_\bullet)$. Then Lemma \ref{lemma:ParEndNoSections} implies that the rows of the following commutative diagram are exact
\begin{eqnarray}
\label{eq:sParEndcommutative1}
\xymatrixcolsep{1pc}
\xymatrix{
0  \ar[r] & \SG \ar@{^(->}[d] \ar[r] & F_{(E,E_\bullet)}^{r} \otimes \SO_X \ar@{^(->}[d] \ar[r] & \SParEnd_0(E,E_\bullet)\otimes K(D-x) \ar@{^(->}[d] \ar[r] &0 \\
0  \ar[r] & \SF \ar@{^(->}[d] \ar[r] & F_{(E,E_\bullet)}^{r-1} \otimes \SO_X \ar@{^(->}[d] \ar[r] & \SParEnd_0^{(x)}(E,E_\bullet)\otimes K(D) \ar@{^(->}[d] \ar[r] &0  \\
0 \ar[r] & \SE \ar[r] & F_{(E,E_\bullet)}^0  \otimes \SO_X \ar[r] & \SParEnd_0(E,E_\bullet)\otimes K(D) \ar[r] &0
}
\end{eqnarray}
Over $U$, we have proven that
$$\left(d(\Phi^{-1})\otimes \id_{\SO_U}\right)(\SG|_U)=\SG'|_U$$
$$\left(d(\Phi^{-1})\otimes \id_{\SO_U}\right)(\SF|_U)=\SF'|_U$$
$$\left(d(\Phi^{-1})\otimes \id_{\SO_U}\right)(\SE|_U)=\SE'|_U$$
As before, $\SG$, $\SF$ and $\SE$ are the saturations of $\SG|_U$, $\SF|_U$ and $\SE|_U$ and the same holds for $\SG'$, $\SF'$ and $\SE'$, so
$$\left(d(\Phi^{-1})\otimes \id_{\SO_X}\right)(\SG)=\SG'$$
$$\left(d(\Phi^{-1})\otimes \id_{\SO_X}\right)(\SF)=\SF'$$
$$\left(d(\Phi^{-1})\otimes \id_{\SO_X}\right)(\SE)=\SE'$$
By commutativity of diagram \eqref{eq:sParEndcommutative1}, the morphisms between $\SG$, $\SF$ and $\SE$ coincide with the restriction of the morphism $\left(d(\Phi^{-1})\otimes \id_{\SO_X}\right):F^0_{(E,E_\bullet)} \otimes \SO_X \to F^0_{(E,E_\bullet)} \otimes \SO_X$ to the corresponding subsheaves. Taking quotients, we obtain the following commutative diagram proving the desired result
\begin{landscape}
\begin{eqnarray*}
\xymatrixrowsep{4pc}
\xymatrixcolsep{1pc}
\xymatrix{
0 \ar[rr] && \SG \ar@{^(->}[dd] \ar[rr] \ar[rd] && F_{(E,E_\bullet)}^{r}\otimes \SO_X \ar@{^(->}[dd] \ar[rr] \ar[rd]^{d(\Phi^{-1})\otimes \id_{\SO_X}}&& \SParEnd_0(E,E_\bullet)\otimes K(D-x) \ar@{^(->}[dd] \ar[rr] \ar[rd]^{\Phi_{\SParEnd_0}\otimes \id_{K(D-x)}} &&0&\\
&0 \ar[rr] && \SG' \ar@{^(->}[dd] \ar[rr] && F_{(E',E'_\bullet)}^{r}\otimes \SO_X \ar@{^(->}[dd] \ar[rr] && \SParEnd_0(E',E'_\bullet)\otimes K(D-x) \ar@{^(->}[dd] \ar[rr] &&0\\
0 \ar[rr] && \SF \ar@{^(->}[dd] \ar[rr] \ar[rd] && F_{(E,E_\bullet)}^{r-1}\otimes \SO_X \ar@{^(->}[dd] \ar[rr] \ar[rd]^{d(\Phi^{-1})\otimes \id_{\SO_X}} && \SParEnd_0^{(x)}(E,E_\bullet)\otimes K(D) \ar@{^(->}[dd] \ar[rr] \ar[rd]^{\Phi_{\SParEnd_0}^{(x)}\otimes \id_{K(D)}} &&0&\\
&0 \ar[rr] && \SF' \ar@{^(->}[dd] \ar[rr] && F_{(E',E'_\bullet)}^{r-1}\otimes \SO_X \ar@{^(->}[dd] \ar[rr] && \SParEnd_0^{(x)}(E',E'_\bullet)\otimes K(D) \ar@{^(->}[dd] \ar[rr] &&0 \\
0 \ar[rr] && \SE \ar[rr] \ar[rd] && F_{(E,E_\bullet)}^0\otimes \SO_X \ar[rr] \ar[rd]^{d(\Phi^{-1})\otimes \id_{\SO_X}} && \SParEnd_0(E,E_\bullet)\otimes K(D) \ar[rr] \ar[rd]^{\Phi_{\SParEnd_0}\otimes \id_{K(D)}} &&0 &\\
&0 \ar[rr] && \SE' \ar[rr] && F_{(E,E_\bullet)}^0\otimes \SO_X \ar[rr] && \SParEnd_0(E',E'_\bullet)\otimes K(D) \ar[rr] &&0
}
\end{eqnarray*}
\end{landscape}
Finally, we have to prove the equalities in \eqref{eq:recursiveAlgResult}. Let us take the image of $F_{(E,E_\bullet)}^k$ and $G_{(E,E_\bullet)}^k$ by the evaluation map
$$\pi:H^0(\SParEnd_0(E,E_\bullet)\otimes K(D)) \twoheadrightarrow \SParEnd_0(E,E_\bullet)\otimes K(D)|_x$$
Let us identify the right hand side fiber with the vector space of traceless $r\times r$ complex matrices and let us define for $0<k\le r$
$$ \begin{array}{l}
\overline{F_{(E,E_\bullet)}^k} =\left \{\psi=(\psi_{ij})\in \SParEnd_0(E,E_\bullet)\otimes K(D)|_x : \begin{array}{c} \forall l \,\, 0<l\le k \,\, \forall (i,j)\in I^l\\
\psi_{ij}=0 \end{array} \right \}\\
\overline{G_{(E,E_\bullet)}^k}=\overline{F_{(E,E_\bullet)}^{k-1}} \cap \left \{ \psi=(\psi_{ij}) : \prod_{(i,j)\in I^k}\psi_{ij}=0 \right \}
\end{array}$$
where we take $\overline{F_{(E,E_\bullet)}^0}=\SParEnd_0(E,E_\bullet)\otimes K(D)|_x$ and
$$I^k=\{(i,j)\in [1,r]^2 : j-i\cong k \mod{r}\}$$
By definition of $\SParEnd_0^{(x)}(E,E_\bullet)$, it is clear that the following identities hold
\begin{equation*}
\begin{array}{l}
\pi^{-1}\left(\overline{F_{(E,E_\bullet)}^{r-1}}\right)=H^0(\SParEnd_0^{(x)}(E,E_\bullet)\otimes K(D))\\
\pi^{-1}\left(\overline{F_{(E,E_\bullet)}^r}\right)=H^0(\SParEnd_0(E,E_\bullet)\otimes K(D-x))
\end{array}
\end{equation*}
Let us prove by induction that for every $0<k\le r$
$$ \begin{array}{l}
\pi^{-1}\left(\overline{G_{(E,E_\bullet)}^k}\right)=G_{(E,E_\bullet)}^k\\
\pi^{-1}\left(\overline{F_{(E,E_\bullet)}^k}\right)=F_{(E,E_\bullet)}^k
\end{array}$$
For $k=0$ the statement is trivial by construction. Suppose that $$\pi^{-1}\left(\overline{F_{(E,E_\bullet)}^{k-1}}\right)=F_{(E,E_\bullet)}^{k-1}$$
Let $s\in H^0(\SParEnd_0(E,E_\bullet)\otimes K(D))$ be a section in $F_{(E,E_\bullet)}^{k-1}$. Let $s_x$ be its germ at $x$. Looking at the image of the germ in $(\End_0(E)\otimes K(D))_x$, it can be identified with a matrix $S=(S_{ij})\in \Mat_{r\times r}(\SO_{X,x})$. As $\SO_{X,x}$ is a local principal ideal domain, there exists an element $z\in \SO_{X,x}$ such that $(z)\subset \SO_{X,x}$ is the maximal ideal. For $-r<k<r$, let us denote by
$$D^k=\{(i,j)\in [1,r]^2: j-i=k\}$$
the set of indexes corresponding to the $k$-th secondary diagonal of an $r\times r$ matrix. Note that for $0<k<r$
$$I^k=D^k\cup D^{k-r}$$
and for $k=r$, $I^r=D^0$. By induction hypothesis, as $s\in F_{(E,E_\bullet)}^{k-1}$, then $z | S_{ij}$ for each $(i,j)\in D^l$ for $0\le l<k$ and, moreover, $z^2 | S_{ij}$ for each $(i,j)\in D^{l-r}$ for $0<l<k$.
We have that $h_r(s)\in H^0(K^rD^{r-1}(-kx))$ if and only if $z^{k+1}|\det(S)$. Developing the determinant
$$\det(S)= \sum_{\sigma\in \Sigma_r} (-1)^{|\sigma|} \prod_{i=1}^r S_{i\sigma(i)}$$
the only summand with less than $k+1$ factors $z$ is the product of the elements in $I^k$. To check this, observe that the only factors not already divisible by $z$ are those with $j\ge i+k$. Moreover, note that for $i>r-k$, all the elements $S_{ij}$ with $j<i+k-r$ are divisible by $z^2$. Therefore, a permutation $\sigma:[1,r]\longrightarrow[1,r]$ for which $\prod_{i=1}^r S_{i\sigma(i)}$ is not already divisible by $z^{k+1}$ must have
\begin{enumerate}
\item $\sigma(i)\ge i+k$ for every $i\le r-k$
\item $\sigma(i)\ge i+k-r$ for every $i>r-k$
\end{enumerate}
Now the result follows from Lemma \ref{lemma:permutation}.  Therefore, $z^{k+1} | \det(S)$ if and only if
$$z^{k+1} | \prod_{(i,j)\in I^k} S_{ij}$$
As the $k$ elements below the diagonal are already multiple of $z$, the product is a multiple of $z^{k+1}$ if and only if there is at least an extra $z$ factor in some of the $S_{ij}$, i.e., if and only if $s_{ij}$ annihilates for some $(i,j)\in I^k$. Therefore, taking into account that $\pi$ is surjective, we obtain that $s\in G_{(E,E_\bullet)}^k$ if and only if $\pi(x) \in \overline{G_{(E,E_\bullet)} ^k}$.

Now, let us prove that
$$\overline{F_{(E,E_\bullet)}^k} = \left \{\psi \in \overline{G_{(E,E_\bullet)}^k} : \forall \varphi \in \overline{G_{(E,E_\bullet)}^k}, \, \psi+\varphi \in \overline{G_{(E,E_\bullet)}^k} \right\}$$
Suppose that an element $\psi\in\overline{G_{(E,E_\bullet)}^k}$ has some $(i,j)\in I^k$ with $\psi_{ij}\ne 0$. Let $\emptyset \ne \SI\subsetneq I^k$ be the subset of indexes in $I^k$ such that $\psi_{ij}\ne 0$. Then, let us define $\varphi\in \overline{G_{E,E_\bullet}^k}$ in the following way
$$ \varphi_{ij}=\left\{ \begin{array}{ll}
0 & (i,j)\in \SI \\
1 & (i,j) \in I^k\backslash \SI\\
\psi_{ij} & (i,j) \in [1,n]^2 \backslash I^k
\end{array} \right.$$
We can test that as $\psi\in \overline{F_{(E,E_\bullet)}^{k-1}}$, $\varphi\in \overline{F_{(E,E_\bullet)}^{k-1}}$ and as $\SI\ne \emptyset$, then $\prod_{(i,j)\in I^k} \varphi_{ij}=0$. On the other hand, for every $(i,j)\in I^k$
$$(\psi+\varphi)_{ij}\ne 0$$
so $\varphi+\psi \not\in \overline{G_{(E,E_\bullet)}^k}$. Now the equality
\begin{multline*}
\pi\left(F_{(E,E_\bullet)}^k\right)=\pi \left( \left \{\psi \in G_{(E,E_\bullet)}^k : \forall \varphi \in G_{(E,E_\bullet)}^k, \, \psi+\varphi \in G_{(E,E_\bullet)}^k \right \}\right)\\
=\left \{\psi \in \overline{G_{(E,E_\bullet)}^k} : \forall \varphi \in \overline{G_{(E,E_\bullet)}^k}, \, \psi+\varphi \in \overline{G_{(E,E_\bullet)}^k} \right \}=\overline{F_{(E,E_\bullet)}^k}
\end{multline*}
follows from surjectivity of $\pi:G^k_{(E,E_\bullet)} \twoheadrightarrow \overline{G_{(E,E_\bullet)}^k}$. The remaining equalities of \eqref{eq:recursiveAlgResult}
$$\begin{array}{l}
F_{(E,E_\bullet),y}^{r-1}=H^0(\SParEnd_0^{(x)}(E,E_\bullet)\otimes K(D-y))\\
F_{(E,E_\bullet),y}^r=H^0(\SParEnd_0(E,E_\bullet)\otimes K(D-x-y))
\end{array}$$
follow from the given ones using the same argument as the one used for Lemma \ref{lemma:recoverSParEndsections}, taking into account that, as we have already proven, the assumption $g\ge 6$ implies that the following morphisms are surjective for every  $y\in X\backslash D$
$$ \begin{array}{l}
H^0(\SParEnd_0(E,E_\bullet)\otimes K(D-x)) \twoheadrightarrow \SParEnd_0(E,E_\bullet)\otimes K(D-x)|_y\\
H^0(\SParEnd_0^{(x)}(E,E_\bullet) \otimes K(D)) \twoheadrightarrow \SParEnd_0^{(x)}(E,E_\bullet) \otimes K(D)|_y
\end{array}$$
\end{proof}

\begin{lemma}
\label{lemma:permutation}
Let $\sigma:[1,r]\longrightarrow [1,r]$ be a permutation such that
\begin{enumerate}
\item $\sigma(i)\ge i+k$ for every $i\le r-k$
\item $\sigma(i)\ge i+k-r$ for every $i>r-k$
\end{enumerate}
Then
$$\sigma(i)=\left \{\begin{array}{ll}
i+k & i\le r-k\\
i+k-r & i>r-k
\end{array}\right.$$
\end{lemma}

\begin{proof}
Let us prove by induction that $\sigma(i)\le i+k$ for $i\le r-k$. For $i=r-k$, we have that
$$\sigma(r-k) \le r=r-k+k$$
Let $j<r-k$ and let us assume that it is true for all $i$ with $r-k\ge i> j$. Then $\sigma(i)=i+k$ for $r-k\ge i >j$. Therefore, the elements $[j+k+1,r]$ have been selected  by the permutation, so $\sigma(j)\not\in [j+k+1,r]$, so $\sigma(j)\le j+k$.
Once we know that $\sigma(i)=i+k$ for $i\le r-k$, let us prove by induction that $\sigma(i)\le i+k-r$ for every $i>r-k$. As the elements $[k+1,r]$ have already been selected by the permutation, we know that $\sigma(i)\in [i+k-r,k]$ for every $i>r-k$. For $i=r$, we have $\sigma(r)\le k=r+k-r$.
Let $j>r-k$ and suppose that it is true for every $r\ge i>j$. Then $\sigma(i)=i+k-r$ for every $i>j$. Therefore, the elements $[j+k-r+1,k]$ have already been selected by the permutation and we get $\sigma(j)\le j+k-r$.
\end{proof}

\begin{lemma}
\label{lemma:recoverNilpotent}
Suppose that $g\ge 4$. For every $x\in X$, and every $k>1$, the linear subspace
$$H^0(K^kD^{k-1}(-kx))\subseteq W_k$$
is preserved by the linear map $f_k:W_k\longrightarrow W_k$.
\end{lemma}

\begin{proof}
Let $\SU\subset\SM(r,\alpha,\xi)$ and $\SU'\subset \SM(r,\alpha',\xi')$ be the open nonempty subsets of generic parabolic vector bundles in the sense of Lemma \ref{lemma:ParEndNoSections}. Let $\SV=\Phi^{-1}(\SU)\cap \SU'$ and let $\SV'=\Phi(\SV)\subseteq \SU'$. They are also nonempty open subsets of $\SM(r,\alpha,\xi)$ and $\SM(r,\alpha',\xi')$ respectively. As we assumed $g\ge 3$, we have
$$r\deg(K(D-x))=r(2g-3+n)\ge r(2g-2) \ge 2(2g-2)>2g$$
Therefore, we can apply Corollary \ref{cor:HitchinDominant} to $L=K(D-x)$ and the open subsets $U'$ and $U''$. Then we obtain that the linear subspace
$$\bigoplus_{k=2}^r H^0(K^kD^{k-1}(-kx))\subseteq W$$
is the space generated by the images $h(H^0(\SParEnd_0(E,E_\bullet)\otimes K(D-x)))$ both when $(E,E_\bullet)$ runs over $\SV$ and when $(E,E_\bullet)$ runs over $\SV'$.

By Lemma \ref{lemma:recoverSParEnd}, for every $(E,E_\bullet)\in \SV$, if $(E',E'_\bullet)=\Phi(E,E_\bullet)\in \SV'$, then the image of $H^0(\SParEnd_0(E,E_\bullet)\otimes K(D-x))$ by $d(\Phi^{-1})$ is $H^0(\SParEnd_0(E',E'_\bullet)\otimes K(D-x))$. As $\Phi(\SV)=\SV'$, then union of the images $h(H^0(\SParEnd_0(E,E_\bullet)\otimes K(D-x)))$ for $(E,E_\bullet)\in \SV$ is the same as the union of the images $h(H^0(\SParEnd_0(E',E'_\bullet)\otimes K(D-x)))$ for $(E',E'_\bullet)\in \SV'$, so $f:W\to W$ preserves the subspace $\bigoplus_{k=2}^r H^0(X,K^kD^{k-1}(-kx))\subseteq W$.

Finally, the result follows as a consequence of Lemma \ref{lemma:recoverHitchin}, as the map $f:W\to W$ is diagonal with respect to the decomposition $W=\bigoplus_{k=2}^r W_k$.
\end{proof}

For $k>1$, the curve $X$ is embedded in $\PP(W_k)$ via the linear system $|K^kD^{k-1}|$ and the osculating $k$-space at each point $x\in X$ is given by
$$\op{Osc}_k(x)=\PP\left( \ker \left(H^0(K^kD^{k-1})^\vee \to H^0(K^kD^{k-1}(-kx))^\vee\right) \backslash \{0\} \right)$$

The previous corollary, together with Lemma \ref{lemma:recoverSParEnd} proves that the morphism
$$\PP(f_k):\PP\left(H^0(K^kD^{k-1})^\vee\backslash \{0\}\right) \longrightarrow \PP\left(H^0(K^kD^{k-1})^\vee\backslash \{0\}\right)$$
preserves $\op{Osc}_k(x)$ for all $x\in X$. Now, we use the following Lemma

\begin{lemma}
\label{lemma:preserveOsculating}
Let $X\hookrightarrow \PP^N$ be an irreducible smooth complex projective curve embedded in the projective space. If $\varphi\in \PGL(N+1)$ is an isomorphism preserving $\op{Osc}_k:X\to \Gr(k+1,N+1)$ for some $k$, then it preserves $\op{Osc}_k:X\to \Gr(k+1,N+1)$ for every $k$.
\end{lemma}

\begin{proof}
This is a direct consequence of the following fact proved in \cite[p. 1250052-23]{BGM12}. Let $X\hookrightarrow \PP^N$ be an embedding of an irreducible smooth complex projective curve $X$ in a projective space and let $\op{Osc}_k:X\to \Gr(k+1,N+1)$ be the map sending each $x\in X$ to the osculating $k$-space of $X$ in $\PP^N$. Then $\op{Osc}_k$ uniquely determines the embedding $X\hookrightarrow \PP^N$. Therefore the following diagram commutes
\begin{eqnarray*}
\xymatrix{
\PP^N  \ar[r]^{\varphi} & \PP^N\\
X \ar@{^(->}[u] \ar@{=}[r] & X \ar@{^(->}[u]
}
\end{eqnarray*}
For each $k'>0$ and each $x\in X$ we can identify $\op{Osc}_{k'}(x)$ with the intersection of all hyperplanes in $\PP^N$ which intersect with $X$ at $x$ with multiplicity at least $k'$. As the embedded curve $X$ is preserved by $\varphi$ and $\varphi\in \PGL(N+1)$ is a linear isomorphism it is clear that $\varphi$ preserves the set of hyperplanes with this property and, therefore, it preserves its intersection $\op{Osc}_{k'}(x)$ for each $k'$ and each $x\in X$.
\end{proof}

As $\PP(f_k)$ preserves $\op{Osc}_k$, it preserves $\op{Osc}_1$, so $f_k$ must preserve the hyperplanes
$$H^0(K^kD^{k-1}(-x)) \subset H^0(K^kD^{k-1})$$
for every $x\in X$.

In particular, this implies that for every $x\in X$ and generic $(E,E_\bullet)$ the image of the set
$$N_{E,x}=\{\psi\in H^0(\SParEnd_0(E,E_\bullet)\otimes K(D)) : \forall k>1 \,\, h_k(\psi)\in H^0(K^kD^{k-1}(-x))\}$$
by $d(\Phi^{-1})=H^0(\Phi_{\SParEnd_0}\otimes \id)$ is 
$$N_{E',x}=\{\psi\in H^0(\SParEnd_0(E',E'_\bullet)\otimes K(D)) : \forall k>1 \,\, h_k(\psi)\in H^0(K^kD^{k-1}(-x))\}$$
For $x\in U$, the set $N_{E,x}$ coincides with the preimage of the nilpotent cone under the surjective map
$$H^0(\SParEnd_0(E,E_\bullet)\otimes K(D))  \twoheadrightarrow \SParEnd_0(E,E_\bullet)\otimes K(D)|_{x}$$
Taking the image of $N_{E,x}$ under the evaluation map we get a subset $\SN_{E,x}\subset \SParEnd_0(E,E_\bullet)\otimes K(D)|_{x}$. Varying $x$ over $U$, we get a subscheme
$$\SN_E|_U \hookrightarrow \SParEnd_0(E,E_\bullet)|_U$$
such that $\Phi_{\SParEnd_0}|_U(\SN_E|_U)=\SN_{E'}|_U$.

Therefore, if $g\ge 6$, $\Phi_{\ParEnd_0}|_U: \ParEnd_0(E,E_\bullet)|_U \to \ParEnd_0(E',E'_\bullet)|_U$ is an isomorphism of vector bundles that preserves the nilpotent cone. Therefore, it is an isomorphism of $\GL(\parsl)|_U \cong \GL(\ssl)\times U$ torsors that preserves the nilpotent cone. Let $N<\ssl$ denote the subset of nilpotent matrices. Then, let us denote by
$$G_N=\{g\in \GL(\ssl) : g(N)=N\}<\GL(\ssl)$$
the subgroup of invertible linear transformations of $\ssl$ which preserve the nilpotent matrices. As $\Phi_{\ParEnd_0}|_U$ preserves the nilpotent cone, it is an isomorphism of $G_N$-torsors. Now, we can use the following theorem from Botta, Pierce and Watkins \cite{BPW83},

\begin{theorem}
The group $G_N$ is generated by
\begin{enumerate}
\item Inner automorphisms $X\mapsto S^{-1} X S$
\item The maps $X\mapsto aX$ for some $a\ne 0$
\item The map $X\mapsto X^t$ that sends a matrix $X$ to its transpose
\end{enumerate}
\end{theorem}

Using the computation in \cite[Lemma 5.4]{BGM13}, we know that $\Aut(\ssl)$ is generated by inner automorphisms and the map $X\mapsto -X^t$. Therefore, we conclude that $G_N\cong \Aut(\ssl)\times \CC^*$. Thus, up to product by a morphism $U\longrightarrow \CC^*$, $\Phi_{\ParEnd_0}|_U$ is an isomorphism of $\Aut(\ssl)$-torsors, i.e., it is an automorphism of Lie algebra bundles.

\begin{lemma}
\label{lemma:recoverParEndalgebra}
Suppose that $g\ge 6$. Let $(E,E_\bullet)$ and $(E',E'_\bullet)$ be generic parabolic vector bundles such that $\Phi(E,E_\bullet)=(E',E'_\bullet)$. Then there exists a constant $\lambda\in \CC^*$ such that  the vector bundle isomorphism $\lambda\cdot \Phi_{\ParEnd_0}$ defined in Lemma \ref{lemma:recoverParEnd} is an isomorphism of Lie algebras bundles.
\end{lemma}

\begin{proof}
As $\ParEnd_0(E,E_\bullet)$ and $\ParEnd_0(E',E'_\bullet)$ have the same degree, $\det(\Phi_{\ParEnd_0})\in H^0(X,\SO_X)$. $X$ is projective and connected, so $\det(\Phi_{\ParEnd_0})$ is constant. The previous discussion shows that $\Phi_{\ParEnd_0}|_U$ is an isomorphism of $\left(\Aut(\ssl) \times \CC\right)$-torsors. As its determinant is constant, there exists a nonzero $\lambda\in \CC^*$ such that $\lambda \cdot \Phi_{\ParEnd_0}|_U$ is an isomorphism of $\Aut(\ssl)$-torsors, i.e., it is an isomorphism of Lie algebra bundles.

A Lie algebra structure on $\ParEnd_0(E,E_\bullet)$ is in particular a bilinear morphism
$$[\cdot,\cdot]:\ParEnd_0(E,E_\bullet) \otimes \ParEnd_0(E,E_\bullet) \longrightarrow \ParEnd_0(E,E_\bullet)$$
Therefore, the Lie algebra structure induced by endomorphism composition on $(E,E_\bullet)$ is represented by a section $$p_{(E,E_\bullet)}\in H^0(\ParEnd_0(E,E_\bullet)^\vee \otimes \ParEnd_0(E,E_\bullet)^\vee \otimes \ParEnd_0(E,E_\bullet))$$
Similarly, the Lie algebra structure on $(E',E'_\bullet)$ is represented by a section
$$p_{(E',E'_\bullet)}\in H^0(\ParEnd_0(E',E'_\bullet)^\vee \otimes \ParEnd_0(E',E'_\bullet)^\vee \otimes \ParEnd_0(E',E'_\bullet))$$
Through the isomorphism $\lambda\cdot\Phi_{\ParEnd_0}$, the section $p_{(E',E'_\bullet)}$ induces another section
$$(\lambda\cdot\Phi_{\ParEnd_0})^*p_{(E',E'_\bullet)}\in H^0(\ParEnd_0(E,E_\bullet)^\vee \otimes \ParEnd_0(E,E_\bullet)^\vee \otimes \ParEnd_0(E,E_\bullet))$$
Therefore, we obtain a section $p_{(E,E_\bullet)}-(\lambda\cdot\Phi_{\ParEnd_0})^*p_{(E',E'_\bullet)}$. As $\lambda\cdot\Phi_{\ParEnd_0}|_U$ is an isomorphism of Lie algebra sheaves, we obtain that $\left (p_{(E,E_\bullet)}-(\lambda\cdot\Phi_{\ParEnd_0})^*p_{(E',E'_\bullet)}\right)|_U=0$, so $p_{(E,E_\bullet)}-(\lambda\cdot\Phi_{\ParEnd_0})^*p_{(E',E'_\bullet)}=0$ and $\lambda\cdot\Phi_{\ParEnd_0}$ must be an isomorphism of Lie algebras.
\end{proof}

\begin{theorem}
\label{theorem:ExtendedTorelli}
Let $(X,D)$ and $(X',D')$ be two smooth projective curves of genus $g\ge 6$ and $g'\ge 6$ respectively with set of marked points $D\subset X$ and $D'\subset X'$. Let $\xi$ and $\xi'$ be line bundles over $X$ and $X'$ respectively, and let $\alpha$ and $\alpha'$ be full flag generic systems of weights over $(X,D)$ and $(X',D')$ respectively. Let
$$\Phi: \SM(X,r,\alpha,\xi)\stackrel{\sim}{\longrightarrow} \SM(X',r',\alpha',\xi')$$
be an isomorphism. Then
\begin{enumerate}
\item $r=r'$
\item $(X,D)$ is isomorphic to $(X',D')$, i.e., there exists an isomorphism $\sigma:X\stackrel{\sim}{\to} X'$ sending $D$ to $D'$.
\item There exists a basic transformation $T$ such that
\begin{itemize}
\item $\sigma^*\xi'\cong T(\xi)$
\item $\sigma^*\alpha'$ is in the same stability chamber as $T(\alpha)$.
\item For every $(E,E_\bullet)\in \SM(r,\alpha,\xi)$, $\sigma^*\Phi(E,E_\bullet) \cong T(E,E_\bullet)$
\end{itemize}
\end{enumerate}
\end{theorem}

\begin{proof}
Let $\Phi:\SM(X,r,\alpha,\xi) \longrightarrow \SM(X',r',\alpha',\xi')$ be an isomorphism. By Torelli Theorem \ref{theorem:Torelli}, we obtain that $r=r'$ and there exists an isomorphism $\sigma:(X,D)\stackrel{\sim}{\longrightarrow}(X',D')$. Pulling back by that isomorphism, we obtain an isomorphism
$$\Phi':\SM(X,r,\alpha,\xi) \longrightarrow \SM(X,r,\sigma^*\alpha',\sigma^*\xi)$$
From this point, all the moduli spaces will be constructed over the same curve $(X,D)$, so, in order to simplify the notation, from now on, we will denote $\SM(r,\alpha,\xi)=\SM(X,r,\alpha,\xi)$. Let $\xi''=\sigma^*\xi'$ and $\alpha''=\sigma^*\alpha'$. The differential of $\Phi'$ induces an isomorphism of the cotangent bundles $d(\Phi')^{-1}:T^*\SM(r,\alpha,\xi)\longrightarrow T^*\SM(r,\alpha'',\xi'')$. Let $h:T^*\SM(r,\alpha,\xi)\to W$ and $h'':T^*\SM(r,\alpha'',\xi'')\to W$ denote the Hitchin morphisms corresponding to each choice of the system of weights and determinant. Since both moduli spaces are built over the same marked curve $(X,D)$ for the same rank, the Hitchin space is the same for both moduli spaces. By Proposition \ref{prop:HitchinGlobalFunctions}, there exists a $\CC^*$-equivariant automorphism $f:W\longrightarrow W$ such that the following diagram commutes
\begin{eqnarray*}
\xymatrixcolsep{4pc}
\xymatrix{
T^*\SM(r,\alpha,\xi) \ar[r]^{d(\Phi')^{-1}} \ar[d]_{h} & T^*\SM(r,\alpha'',\xi'') \ar[d]^{h''} \\
W \ar[r]^{f} & W
}
\end{eqnarray*}
As $f$ is $\CC^*$-equivariant, it preserves the subspace of maximum decay $W_r\subset W$. Let $h_r:T^*\SM(r,\alpha,\xi)\to W_r$ (respectively $h_r'':T^*\SM(r,\alpha'',\xi'')\to W_r$) be the composition of $h$ with the projection to $W_r$. Let $f_r:W_r\to W_r$ be the restriction of $f$ to $W_r$. Then $f_r$ is linear and, by Corollary \ref{cor:preservehr} we have a diagram
\begin{eqnarray*}
\xymatrixcolsep{4pc}
\xymatrix{
T^*\SM(r,\alpha,\xi) \ar[r]^{d(\Phi')^{-1}} \ar[d]_{h_r} & T^*\SM(r,\alpha'',\xi'') \ar[d]^{h_r} \\
W_r \ar[r]^{f_r} & W_r
}
\end{eqnarray*}
By Lemma \ref{lemma:preserveWr} for every $k>0$ and every $x_0\in X$
$$f_r\left (H^0(K^rD^{r-1}(-kx_0))\right)=H^0(K^rD^{r-1}(-kx_0))$$

By Corollary \ref{cor:stabCodim2} and Lemma \ref{lemma:ParEndNoSections}, there exists an open nonempty subset $\SU\subseteq \SM(r,\alpha,\xi)$ (respectively $\SU'' \subseteq \SM(r,\alpha'',\xi'')$) parameterizing $\alpha''$-stable (respectively $\alpha$-stable) parabolic vector bundles $(E,E_\bullet)$ such that
$$H^1(\SParEnd_0(E,E_\bullet)\otimes K(D-x-y))=0$$
for every $x,y\in X$.  Let $\SV=\SU \cap (\Phi')^{-1}(\SU'')$ and $\SV''=\Phi'(\SV)$. By definition of $\SV''$, there is a natural identification between $\SV''$ and an open nonempty subset in $\SM(r,\alpha,\xi'')$. Let $(E,E_\bullet)\in \SM(r,\alpha,\xi)$ and let $\Phi'(E,E_\bullet)=(E'',E''_\bullet)\in \SM(r,\alpha,\xi'')$ be its image.
Therefore, we can apply Lemma \ref{lemma:recoverParEndalgebra} and we obtain that $\ParEnd_0(E,E_\bullet)$ is isomorphic to $\ParEnd_0(E'',E''_\bullet)$ as Lie algebra bundles. Then Lemma \ref{lemma:isoParEnd} proves that $(E',E'_\bullet)$ can be obtained from $(E,E_\bullet)$ as a combination of the following transformations
\begin{enumerate}
\item Tensorization with a line bundle over $X$, $(E,E_\bullet) \mapsto (E\otimes L, E_\bullet \otimes L)$
\item Dualization $(E,E_\bullet)\mapsto (E, E_\bullet)^\vee$
\item Hecke transformation at a parabolic point $x\in D$, $(E,E_\bullet) \mapsto \SH_x(E,E_\bullet)$.
\end{enumerate}
This means that $(E'',E''_\bullet)=T(E,E_\bullet)$ for some basic transformation $T=(\id,s,L,H)$. In particular, we obtain that $\xi''=T(\xi)$. As the set of possible values for $H$ in the choice of $T$ is finite and the $r$-torsion of the Jacobian $J(X)$ is finite, the space of basic transformations
$$\ST_{\xi,\xi''} = \{T=(\id,s,L,H)\in \ST | T(\xi)\cong \xi''\}$$
is finite. For every $T\in \ST_{\xi,\xi''}$, let us consider the composition of isomorphisms $T\circ (\Phi')^{-1}:\SM(r,\alpha'',\xi'') \to \SM(r,T(\alpha),\xi'')$. By construction of $\SV$ and $\SV'$, it sends $\SV''$ to $T(\SV)$
\begin{eqnarray*}
\xymatrixcolsep{4pc}
\xymatrix{
\SM(r,\alpha'',\xi'') \ar[r]^{(\Phi')^{-1}} & \SM(r,\alpha,\xi) \ar[r]^{T} & \SM(r,T(\alpha),\xi'')\\
\SV'' \ar[r] \ar@{^(->}[u] & \SV \ar[r] \ar@{^(->}[u] & T(\SV) \ar@{^(->}[u]
}
\end{eqnarray*}

Both $\SV$ and $T(\SV)$ parameterize parabolic vector bundles of rank $r$ and determinant $\xi$ which are both $\alpha$-semistable and $\alpha''$-semistable and are generic in the sense of Lemma \ref{lemma:ParEndNoSections}, so they can be canonically identified. Choose once and for all an identification $\SV''\cong T(\SV)$. Let $\Psi_T:\SV''\to \SV''$ be the automorphism of $\SV''$ induced composing $T\circ (\Phi')^{-1}$ with the identification $\SV''\cong T(\SV)$.

For every $(E,E_\bullet)\in \SV$ there exists some $T\in \ST_{\xi,\xi''}$ such that $\Phi'(E,E_\bullet)=T(E,E_\bullet)$. Therefore, for every $(E'',E''_\bullet)\in \SV''$ there exists some $T\in \ST_{\xi,\xi''}$ such that $\Psi_T(E'',E''_\bullet)=(E'',E''_\bullet)$ and we obtain that
$$\SV''=\bigcup_{T\in \ST_{\xi,\xi''}} \op{Fix}(\Psi_T)$$
As the set of fixed points of an automorphism is closed and $\ST_{\xi,\xi''}$ is finite, $\SV''$ is a finite union of closed subsets. $\SM(r,\alpha'',\xi'')$ is irreducible and $\SV''$ is open, so $\SV''$ is irreducible. Then there exists some $T\in \ST_{\xi,\xi''}$ such that $\SV''= \op{Fix}(\Psi_T)$.Therefore, we conclude that there exist $T\in \ST_{\xi,\xi''}$ and an open subset $\SV\subseteq \SM(r,\alpha,\xi)$ such that $\Phi'|_{\SV}=T|_{\SV}$.

Let us prove that, in fact, we can find an open subset $\SW\subseteq \SM(r,\alpha,\xi)$ whose complement has codimension at least $2$ and such that $\Phi'|_{\tilde{\SV}}=T|_{\tilde{\SV}}$. Let $\SW\subset \SM(r,\alpha,\xi)$ be the space of parabolic vector bundles which are both $\alpha$-stable and $T^{-1}(\alpha'')$-stable. By Corollary \ref{cor:stabCodim2}, the complement of $\SW$ has codimension at least $2$. Clearly, $T$ is well defined over $\SW$ and it gives us a map $T:\SW \to \SM(r,\alpha'',\xi'')$. Moreover, as $\SM(r,\alpha,\xi)$ is irreducible, $\SW \cap \SV$ is dense in $\SW$, so every map $\psi:\SW\cap \SV \to \SM(r,\alpha'',\xi'')$ admits a unique extension to $\SW$ by continuity. We know that $\Phi'|_{\SV\cap \SW}=T|_{\SV\cap \SW}$, and $\Phi'|_{\SW}$ and $T|_{\SW}$ are two possible extensions, so they must coincide.

As $\alpha''$ is a full flag system of weights, $\SM(r,\alpha'',\xi'')$ is a fine moduli space for every $\xi''$. Therefore, $\Phi'$ is represented by a parabolic vector bundle $(\SE'',\SE''_\bullet)$ over $\SM(r,\alpha,\xi)\times X$ whose fibers are $\alpha''$-stable as parabolic vector bundles over $X$. We have the following commutative diagram
\begin{eqnarray*}
\xymatrixrowsep{0.5pc}
\xymatrix{
\Phi' \ar@{|->}[rrr] \ar@{|->}[dddd] & & & (\SE'',\SE''_\bullet) \ar@{|->}[dddd]\\
& \Hom(\SM(r,\alpha,\xi),\SM(r,\alpha'',\xi'')) \ar[r]^-{\sim} \ar[dd]^{i^\sharp} & {\underline{\SM(r,\alpha'',\xi'')}}(\SM(r,\alpha,\xi)) \ar[dd]^{i^*} &\\
&&&\\
& \Hom(\SW,\SM(r,\alpha'',\xi'')) \ar[r]^-{\sim} & {\underline{\SM(r,\alpha'',\xi'')}} (\SW) &\\
T \ar@{|->}[rrr] & & & T(\SE,\SE_\bullet)|_{\SW}\\
}
\end{eqnarray*}
where $(\SE,\SE_\bullet)$ is the universal family of the moduli space $\SM(r,\alpha,\xi)$. Therefore, $(\SE'',\SE''_\bullet)$ is the extension of the basic transformation $T(\SE,\SE_\bullet)|_{\SW}$ from $\SW$ to all the moduli space. Note that, $T(\SE,\SE_\bullet)$ is a possible extension as a family of  parabolic vector bundles over $\SM(r,\alpha,\xi)$. By construction, we know that the codimension of the complement of $\SW$ in $\SM(r,\alpha,\xi))$ is at least $2$ and $\SM(r,\alpha,\xi)$ is a smooth complex projective scheme, so by Lemma \ref{lemma:uniqueExtensionCodim2}
$$(\SE'',\SE''_\bullet) \cong T(\SE,\SE_\bullet)$$

As $(\SE'',\SE''_\bullet)$ is a family of $\alpha''$-stable vector bundles, we conclude that $T(\SE,\SE_\bullet)$ is a family of $\alpha''$-stable vector bundles. Nevertheless, it is also a universal family of $T(\alpha)$-stable vector bundles. We know that $(\SE'',\SE''_\bullet)$ is a universal family, so this implies that every $\alpha''$-stable vector bundle is $T(\alpha)$-stable and vice versa, so $\alpha''$ belongs to the same stability chamber as $T(\alpha)$ and $\Phi'=T$.
\end{proof}

If $E$ is a vector bundle of rank $2$, then there is a canonical isomorphism $E^\vee=E \otimes \det(E^\vee)$, i.e, for rank $2$, taking dual does not give new isomorphisms of the moduli space, because taking dual can be rewritten as tensoring with a line bundle. The same holds for parabolic bundles. More preciselly:

\begin{lemma}
\label{lemma:dualrk2}
Let $r=2$. Then for each basic transformation $T=(\sigma,-1,L,H)$ defining an isomorphism $\SM(X,r,\alpha,\xi) \to \SM(X,r,T(\alpha),T(\xi))$ there exists a line bundle $L'$ such that the basic transformation $T'=(\sigma,1,L',H)$ satisfies the following
\begin{itemize}
\item $T(\xi)\cong T'(\xi)$
\item $T(\alpha)\sim T'(\alpha)$
\item For each $(E,E_\bullet)\in \SM(X,r,\alpha,\xi)$, $T(E,E_\bullet)\cong T'(E,E_\bullet)$
\end{itemize}
\end{lemma}

\begin{proof}
Observe that, as $r=2$ and we assume that $0\le H\le (r-1)D$, then $H$ is a simple divisor and, applying the composition rule (9) described in the presentation of the group of basic transformations $\ST$ (Lemma \ref{lem:compositionRules}) 
and taking into account that for every divisor $F$, $\SH_{rF}=\ST_{\SO_X(-F)}$ yields
$$\SD^-\circ \SH_H = \ST_{\SO_X(D)} \circ \SH_{2D-H} \circ \SD^- = \ST_{\SO_X(H)} \circ \SH_H \circ \SD^-$$
Therefore, due to the composition rules (8), (9) and (10) described in the presentation of the group of basic transformations $\ST$ (Lemma \ref{lem:compositionRules}), we can write $T$ as
\begin{multline*}
T=(\sigma,-1,L,H)=\Sigma_{\sigma} \circ \SD^- \circ \ST_L \circ \SH_H = \Sigma_{\sigma} \circ \ST_{L^{-1}} \circ \ST_{\SO_X(H)} \circ \SH_H \circ \SD^{-1}
\\= (\sigma,1,L^{-1}(H),H) \circ \SD^-
\end{multline*}
As $\SH_H$ and $\ST_{L'}$ commute for each $H$ and $L'$, then it is enough to prove that there exists a line bundle $L''$ such that $\ST_{L''}(\xi)=\xi^{-1}$, $\ST_{L''}(\alpha)=\alpha \sim \alpha^\vee$ and for each $(E,E_\bullet)\in \SM(r,\alpha,\xi)$, $\SD^-(E,E_\bullet)\cong \ST_{L''}(E,E_\bullet)$, i.e., that
$$(E,E_\bullet)^\vee \cong (E,E_\bullet)\otimes L''$$
First of all, let us prove that if $r=2$ then $\alpha$ and $\alpha^\vee=\SD^-(\alpha)$ belong to the same chamber. We can assume without loss of generality that $\alpha_1\ne 0$. Let $\varepsilon=(\varepsilon(x))_{x\in D}$, where $\varepsilon(x)=1-\alpha_1(x)-\alpha_2(x)$. Clearly, for each $x\in D$, $-\alpha_1(x)< \varepsilon(x)<1-\alpha_2(x)$, so the shifted weights $\alpha[\varepsilon]$ form a suitable system of weights that belongs to the same stability chamber as $\alpha$. Moreover
\begin{eqnarray*}
\alpha[\varepsilon]_1(x)=\alpha_1(x)+\varepsilon(x)=1-\alpha_2(x)\\
\alpha[\varepsilon]_2(x)=\alpha_2(x)+\varepsilon(x)=1-\alpha_1(x)
\end{eqnarray*}
so $\alpha[\varepsilon]=\alpha^\vee$.
On the other hand, for each $E\in \SM(r,\alpha,\xi)$, there is an isomorphism $\bigwedge^2(E)\cong \xi$. Therefore, there exists an isomorphism
$$E^\vee \cong E\otimes \xi^{-1}$$
Let us prove that under this isomorphism the filtration $E_\bullet^\vee$ is sent to $E_\bullet$. Let us describe this isomorphism explicitly. Let $U_i$ be a covering of $X$ by open subsets such that $E|_{U_i}$ is trivial, and let $g_{ij}:U_i\cap U_j \to \GL(2,\CC)$ be transition functions for $E$. Let
$$g_{ij}=\left( \begin{matrix}
a & b\\
c & d
\end{matrix}\right)$$
Then the transition functions for $E^\vee$ are given by
$$g_{ij}^\vee = (g_{ij}^{-1})^t = \frac{1}{\det(g_{ij})}\left(\begin{matrix}
d& -c\\
-b& a
\end{matrix} \right)$$
Let $M:\CC^2 \to \CC^2$ given by $M=\left(\begin{matrix}
0& 1\\
-1& 0
\end{matrix} \right)$. Then
$$g_{ij}^\vee=
M^{-1}\frac{1}{\det(g_{ij})}g_{ij}M$$
As the transition functions for $\xi=\deg(E)$ are $\det(g_{ij})$, $M$ describes locally the desired isomorphism $E^\vee \cong E\otimes \xi^{-1}$. On the other hand, the dual of any quasi-parabolic vector bundle $(E,E_\bullet)$ corresponds to the vector bundle $E^\vee$, endowed with the parabolic structure given by
$$E^\vee_{2,x}=E_{2,x}^\perp$$
where, given $V\subset E|_x$, $V^\perp\subset E^\vee|_x$ denotes the annihilator of $V$, i.e.,
$$V^\perp = \left \{w\in E^\vee|_x  \, \middle | \, w(v)=0 \, \forall v\in V \right\}$$
Observe that, in rank $2$, for every $v\in E|_x$
$$\langle M(v) \rangle = \langle v \rangle^\perp$$
Therefore, as tensoring by $\xi^{-1}$ acts trivially on the parabolic structure, we observe that the isomorphism $M:E^\vee \cong E \otimes \xi^{-1}$ sends $E_{2,x}$ to $E_{2,x}^\perp$ for every $x\in D$, so it is an isomorphism of quasi-parabolic vector bundles $(E,E_\bullet)^\vee \cong (E,E_\bullet) \otimes \xi^{-1}$.
\end{proof}

\begin{lemma}
\label{lemma:freeTransforms}
Suppose that $g\ge 4$. Let $T\in \ST$ be a basic transformation such that $T\ne \id_\ST=(\id,1,\SO_X,0)$ and such that $T\in \ST^+$. Then for a generic $\alpha$-stable parabolic vector bundle $(E,E_\bullet)$ of rank $r$ we have $T(E,E_\bullet)\not\cong (E,E_\bullet)$.
\end{lemma}

\begin{proof}
Assume that $T\ne \id_\ST$ but $T=(\sigma,s,L,H)$ acts as the identity on $\SM(r,\alpha,\xi)$. First, let us prove that $H=0$. Assume that $H\ne 0$. Let $x\in D$ such that $H\ge kx$, but $H\not\ge(k+1)x$. Then for every $(E,E_\bullet)\in \SM(r,\alpha,\xi)$
$$(\sigma,s,L,0)\circ \SH_H(E,E_\bullet)\cong(E,E_\bullet)$$
By Lemma \ref{lemma:generic-lm-stable} and Lemma \ref{lemma:1-0-stability}, for a generic $(E,E_\bullet)\in \SM(r,\alpha,\xi)$ if $E'_\bullet$ is the filtration obtained by changing the step $E_{x,k}$ on $x\in D$ to $E_{x,k}'$ for some
$$E_{x,k-1}\subsetneq E_{x,k}''\subsetneq E_{x,k+1}$$
then $(E,E'_\bullet)$ is $\alpha$-stable. Then there is a short exact sequence
$$0\longrightarrow \SH_H(E,E'_\bullet) \longrightarrow \SH_{H-kx}(E,E'_\bullet) \longrightarrow E|_x/E'_{x,k}\longrightarrow 0$$
Therefore, as $E'_{x,k}$ changes through all possible steps in the filtration, then the underlying vector bundle of $\SH_H(E,E'_\bullet)$ varies. Nevertheless, as
$$\SH_H(E,E'_\bullet)=(\sigma,s,L,0)^{-1}(E,E'_\bullet)\cong(\sigma^{-1},s,\sigma^*L^{-s},0)(E,E_\bullet')$$
then the underlying vector bundle of $\SH_H(E,E_\bullet')$ must be isomorphic to $E$ for every $E'_{x,k}$, so we obtain a contradiction and $H=0$.

Similarly, $\sigma$ must fix every parabolic point. Otherwise, if $\sigma(x)\ne x$ for some $x\in D$, then taking any variation $E_{x,k}'$ of the parabolic structure at $x$ we would obtain that
$$(\sigma,s,L,0)(E,E'_\bullet)\cong(E,E'_\bullet)$$
Nevertheless, the left hand side of the equation has constant parabolic structure over $x$, while the parabolic structure on the right hand side varies over $x$.

Now, let us prove that $s=1$. If $s=-1$, for every parabolic vector bundle $(E,E_\bullet)$ and every $x\in D$ the isomorphism $\sigma^*(L\otimes E)^\vee \stackrel{\alpha}{\cong} E$ induce a nondegenerate bilinear map
\begin{eqnarray*}
\xymatrixrowsep{0.05pc}
\xymatrixcolsep{0.3pc}
\xymatrix{
\omega&:&E|_x\otimes E|_x \ar[rrrr] &&&& L^{-1}|_x\\
&& (u,v) \ar@{|->}[rrrr] &&&& \langle \alpha^{-1}(u), v\rangle
}
\end{eqnarray*}
Under the isomorphism $T(E,E_\bullet)\cong E$ the $k$-th step of the parabolic structure $E_{x,k}\subset E|_x$ is sent to
$$E_{x,k}^\omega=\left \{v\in E|_x \, : \, \forall u\in E_{x,k},  \, \omega(u,v)=0 \right\}$$
Observe that this transformation inverts the filtration, i.e., $E_{x,k}$ is sent to $E_{x,r-k+1}$. We know that $T$ preserves the parabolic structure, so $r=2$ and, in that case, $T\in \ST^+$ by hypothesis, so $s=1$.

Now, let $S\in \ST$ be any basic transformation such that $S(\SO_X)=\xi$. Then $S\circ T\circ S^{-1}\ne \id_{\ST}$, but $S\circ T \circ S^{-1}:\SM(r,S^{-1}(\alpha),\SO_X) \longrightarrow \SM(r,S^{-1}(\alpha),\SO_X)$ is the identity on $\SM(r,S^{-1}(\alpha),\SO_X)$. Therefore, we can assume without loss of generality that $\xi\cong \SO_X$. In this case, taking determinants yields
$$\SO_X\cong\det(E)\cong\det(\sigma^*L\otimes E)\cong\sigma^*L^{r}$$
Therefore, $L^r\cong \SO_X$.

Then $T=(\sigma,1,L,0)$ preserves $\alpha$ and $\xi$ for any system of weights and every line bundle. Moreover, by Corollary \ref{cor:stabCodim2} for any system of weights $\alpha'$, there exists an open subset $\SU\subset \SM(r,\alpha',\xi)$ whose complement has codimension at least $2$ and such that all the parabolic vector bundles in $\SU$ are $\alpha$ stable. Consider the morphism $T:\SM(r,\alpha',\xi)\longrightarrow \SM(r,\alpha',\xi)$. Over $\SU$ this morphism is the identity, so $T=\id_{\SM(r,\alpha',\xi)}$. Therefore, we can assume that $\alpha$ is any system of weights and $\xi$ is any fixed line bundle. In particular, we can assume without loss of generality that $\alpha$ is concentrated and $\deg(\xi)$ is coprime with $r$.

Under this choices, for each stable vector bundle $E\in \SM(r,\xi)$ and every parabolic structure $E_\bullet$ on $E$ we have $(E,E_\bullet)\in \SM(r,\alpha,\xi)$. Therefore, in order to prove the Lemma it is enough to show that there exists $E\in \SM(r,\xi)$ such that $\sigma^*(E\otimes L) \not\cong E$. This fact is known (c.f. \cite[Remark 0.1]{KP}) but a full proof of this precise statement is not available in the literature, so we will proceed to provide one.

Assume the contrary. Suppose that for each $E\in \SM(r,\xi)$ we have $\sigma^*(E\otimes L) \cong E$. Then the map $\Sigma_\sigma \circ \ST_L: \SM(r,\xi) \longrightarrow \SM(r,\xi)$ sending $E \mapsto \sigma^*(E\otimes L)$ is the identity map on $\SM(r,\xi)$. We will prove that in this case we have $\sigma=\id_X$ and $L=\SO_X$.

Let $\PP\to X\times \SM(r,\xi)$ be the projectivization of the universal bundle on $X\times \SM(r,\xi)$, i.e., the bundle whose fiber over $X\times \{E\}$ is $\PP(E)$. If $\Sigma_\sigma \circ \ST_L=\id_{\SM(r,\xi)}$, then there exists an isomorphism $\PP \cong \sigma^*\PP$. Let $x\in X$ be any point. Then restricting to the fiber over $\{x\}\times \SM(r,\xi)$ we have an isomorphism
$$\PP|_{\{x\}\times \SM(r,\xi)} \cong \PP_{\{\sigma(x)\}\times \SM(r,\xi)}$$
Then \cite[Lemma 4.1]{BGM10} implies that $\sigma(x)=x$, so $\sigma=\id_X$.

Finally, suppose that $L\otimes E \cong E$ for each $E\in \SM(r,\xi)$. Then by tensoring with the appropriate bundle we have $L\otimes E\cong E$ for each $E\in \SM(r,d)$, where $d=\deg(\xi)$. Let $m$ be the smaller positive integer such that $L^m\cong \SO_X$. Then we can associate to $L$ an $m$-to-$1$ smooth unramified cyclic cover $X_L\subset \op{Tot}(L) \to X$ in the following way
$$X_L=\left \{t\in L : t^{\otimes m} = 1\in \SO_X \cong L^m \right \}$$
Let $\pi:X_L\to X$ denote the projection. Then, by \cite[Lemma 2.6 and Proposition 3.1]{NR75} or \cite[Proposition 3.46]{Na05} a stable vector bundle $E\in \SM(X, r,d)$ satisfies $L \otimes E \cong E$ if and only if $E=\pi_* E'$ for some stable vector bundle $E'\in \SM(X_L, r/m, d)$. Therefore, the points of $\SM(X,r,d)$ fixed by $\ST_L$ form a closed subscheme of dimension at most $\dim \SM(X_L,r/m,d)$. Let us compute this dimension. $\pi:X_L\to X$ is unramified of degree $m$ so Riemann-Hurwitz implies that
$$g(X_L)-1=m(g-1)$$
where $g(X_L)$ is the genus of $X_L$. Then, if $m>1$ we have
$$\dim \SM\left (X_L,\frac{r}{m},d \right ) = \left(\frac{r}{m} \right)^2\left (g(X_L)-1 \right) +1 = \frac{r^2}{m}(g-1) +1 < r^2(g-1)+1=\dim \SM(X,r,d)$$
so there always exist stable bundles which are not fixed by the action unless $m=1$, i.e., unless $L\cong \SO_X$.
\end{proof}

\begin{theorem}
\label{theorem:autoModuli}
Let $(X,D)$ be a smooth projective curve of genus $g\ge 6$ and let $\alpha$ be a full flag generic system of weights over $(X,D)$ of rank $r$. Let $\xi$ be a line bundle over $X$. Then the automorphism group of $\SM(r,\alpha,\xi)$ is the subgroup of $\ST$ consisting on basic transformations $T$ such that
\begin{itemize}
\item $T(\xi)\cong \xi$
\item $T(\alpha)$ is in the same stability chamber as $\alpha$
\item if $r=2$, $T\in \ST^+$.
\end{itemize}
\end{theorem}

\begin{proof}
If we take $(X',D')=(X,D)$ and $\alpha'=\alpha$ in Theorem \ref{theorem:ExtendedTorelli}, we obtain that if $\Phi:\SM(r,\alpha,\xi)\to \SM(r,\alpha,\xi)$ is an automorphism then there must exist a basic transformation $T\in \ST$ such that $\Phi(E,E_\bullet)\cong T(E,E_\bullet)$. Nevertheless, this implies that $\xi'\cong T(\xi)$ and $T(\alpha)$ is in the same stability chamber as $\alpha$.

Clearly, the subset of transformations $T$ preserving $\xi$ and the chamber of $\alpha$ form a subgroup of $\ST$. As the group structure of $\ST$ coincides with the composition of morphisms between moduli spaces of parabolic vector bundles, then this subgroup projects to the group of automorphisms of $\SM(r,\alpha,\xi)$. To prove the theorem it is enough to check that if $T,T'\in \ST$ are different elements in $\ST$ which satisfy the restrictions then the induced automorphisms $T,T'\in \Aut(\SM(r,\alpha,\xi))$ are different. Composing $T'\circ T^{-1}\in \Aut(\SM(r,\alpha,\xi))$, this is equivalent to proving that if $T\ne \id$ and, additionally, $T\in \ST^+$ if $r=2$, then there exists at least a parabolic vector bundle $(E,E_\bullet)$ such that $T(E,E_\bullet) \ne (E,E_\bullet)$. Now we simply apply the previous Lemma.
\end{proof}

\section{Birational geometry}
\label{section:birational}
In this section we will analyze the birational geometry of the moduli space of parabolic vector bundles with fixed determinant and, in particular, in the birational automorphisms of the moduli space. Boden and Yokogawa \cite[Theorem 6.1]{BY} proved that for $g\ge 3$, if $\alpha$ is a full flag system of weights and $\xi$ is any line bundle over $(X,D)$ then $\SM(r,\alpha,\xi)$ is a rational variety of dimension
$$\dim(\SM(r,\alpha,\xi))=(r^2-1)(g-1)+|D| \frac{r^2-r}{2}=m$$
Therefore, we know that for every $(X,D)$ of genus $g$ and $|D|$ parabolic points there is a birational map
$$\SM(X,r,\alpha,\xi) \dashrightarrow \PP^m$$
In particular
$$\Aut_{\Bir}(\SM(X,r,\alpha,\xi)) = \Aut_{\Bir}(\PP^m)$$
It is then clear that two moduli spaces $\SM(X,r,\alpha,\xi)$ and $\SM(X',r',\alpha',\xi')$ are birationally equivalent if and only if their dimension coincide.

In a first approach, this result closes the problem of understanding the rational geometry of the moduli space and blocks the possibility of a ``birational Torelli'' type theorem. However, there is no control ``a priori'' of how far are the birational equivalences that relate two moduli spaces $\SM(X,r,\alpha,\xi)$ and $\SM(X',r',\alpha',\xi')$ from extending to an isomorphism. More precisely, we know that if these moduli spaces have the same dimension, then there exist open subsets $\SU\subset \SM(X,r,\alpha,\xi)$ and $\SU'\subset \SM(X',r',\alpha',\xi')$ and an isomorphism $\Phi:\SU \stackrel{\sim}{\longrightarrow} \SU'$. Nevertheless, $\SU$ and $\SU'$ can be ``small'' open subsets in the sense that their complement can have codimension 1 (and in fact, they are expected to do so). In this section, we will be interested in understanding the birational geometry of the moduli spaces when we restrict the allowed rational maps to those that can be extended to subsets whose complement has codimension at least $3$.

We will start by generalizing some of the core lemmata in section \ref{section:Torelli} so they can work in the $k$-birational setting.

\begin{definition}
\label{def:k-bir}
Let $\SX$ and $\SX'$ be two varieties. We say that $\SX$ and $\SX'$ are $k$-birational if there exist open subsets $\SU\subset \SX$ and $\SU'\subset \SX'$ and an isomorphism $\Phi:\SU\stackrel{\sim}{\longrightarrow} \SU'$ such that
$$\begin{array}{c}
\codim(\SX\backslash \SU)\ge k\\
\codim(\SX'\backslash \SU')\ge k
\end{array}$$
In particular, $\SX$ and $\SX'$ are birational if they are at least $1$-birational. Given a variety $\SX$, we denote by $\Aut_{k-\Bir}(\SX)$ the space of $k$-birational automorphisms of $\SX$.
\end{definition}

The study of $k$-birational maps instead of rational maps is useful in many contexts. For example, some geometric invariants like the Picard group are invariant under $2$-birational maps, but not under $1$-birational ones. Hartog's theorem proves that if $\SX$ and $\SX'$ are $2$-birationally equivalent normal algebraic varieties then $\Gamma(\SX)\cong \Gamma(\SX')$. In the context of the moduli space of vector bundles (and parabolic vector bundles), we know that for $g\ge 4$ the moduli space of (parabolic) Higgs bundles is $3$-birationally equivalent to the cotangent bundle of the moduli space of (parabolic) vector bundles.  The fact that they are $3$-birational and not just $2$-birational was used in Section \ref{section:Torelli} in order to control the geometry of some special fibers of the Hitchin map.

As we cannot distinguish the moduli spaces nor the isomorphisms between them at the $1$-birational level, we will focus on the $k$-birational maps between moduli spaces for $k>1$ and prove that if we restrict to $3$-birational maps we obtain enough information to be able to describe a birational Torelli type theorem and obtain an analogue of Theorem \ref{theorem:ExtendedTorelli} which categorizes all the $3$-birational maps. Although we believe that the presented results will remain true for $2$-birational maps as well and that the classification could be attempted with similar techniques as the ones presented in this work, due to some technical requisites, our proof is restricted to $3$-birational maps.

\begin{corollary}
\label{eq:cotangentCodim2Discr}
Suppose that $g\ge 4$. Let $\SV\subset \SM(r,\alpha,\xi)$ be an open subset whose complement has codimension at least $3$. Then the complement of $T^*\SV\cap H^{-1}(\SD_U)$ inside $H^{-1}(\SD_U)$ has codimension at least $2$.
\end{corollary}

\begin{proof}
Let $\SZ=\SM(r,\alpha,\xi)\backslash \SV$ and let $m=\dim(\SM(r,\alpha,\xi))$. As $g\ge 4$, by Proposition \ref{prop:CotangentCodimension} we know that
$$\dim(\SM_{K(D)}(r,\alpha,\xi)\backslash T^*\SM(r,\alpha,\xi))\le 2m-3$$
Therefore, if we denote $\SE=\SM_{K(D)}(r,\alpha,\xi)\backslash T^*\SM(r,\alpha,\xi)$ then
$$\dim(\SE\cap H^{-1}(\SD_U))\le 2m-3$$
Let us prove that $\dim(T^*\SZ \cap H^{-1}(\SD_U))\le 2m-3$. In that case we would have
\begin{multline*}
\dim(H^{-1}(\SD_U)\backslash (T^*\SV \cap H^{-1}(\SD_U))) = \dim\left ( (\SE\cap H^{-1}(\SD_U) ) \cup (T^*\SZ \cap H^{-1}(\SD_U))  \right) \\
\le 2m-3=\dim(H^{-1}(\SD_U)-2
\end{multline*}
First, assume that $\dim(\SZ)\le m-3$. Then $\dim(T^*\SZ)\le 2m-3$, so $\dim(T^*\SZ\cap H^{-1}(\SD_U))\le 2m-3$.
\end{proof}

\begin{lemma}
\label{lemma:DiscriminantCharacterizationBir}
Let $g\ge 4$ and let $\SV\subset \SM(r,\alpha,\xi)$ be any open subset whose complement has codimension at least $3$. Let $\SR_\SV \subset T^*\SV$ be the union of the complete rational curves in $T^*\SV$. Then $\SD$ is the closure of $H(\SR_\SV)$ in $W$.
\end{lemma}

\begin{proof}
The proof is analogous of Lemma \ref{lemma:DiscriminantCharacterization}. Let $H_\SV:T^*\SV \to W$ be the restriction of the Hitchin map $H$ to $T^*\SV$. If $\PP^1\hookrightarrow T^*\SV$ is a complete rational curve, then it must be contained in a fiber of the Hitchin map. If $s\in W\backslash D$, then $H^{-1}(s)$ is an abelian variety, so $H_\SV^{-1}(s)$ is an open subset of an abelian variety and, therefore, it does not admit any nonconstant morphism from $\PP^1$. Therefore, we only have to prove that for a generic $s$ in every irreducible component of $\SD$ the fiber $H_\SV^{-1}(s)$ contains a complete rational curve. In this case $H_\SV(\SR_\SV)$ is dense in $\SD$ and the lemma holds.

For the components $\SD_x$ for $x\in D$, we can proceed just as in the proof of Proposition \ref{prop:GenericFiberHitchin}, changing the subset $\SU\subset \SM(r,\alpha,\xi)$ parameterizing $(1,0)$-stable parabolic vector bundles $(E,E_\bullet)$ such that $H^0(\ParEnd_0(E,E_\bullet)(x))=0$ with the following open nonempty subset $\SU'$. For every $(E,E_\bullet)\in \SZ=\SM\backslash \SV$ and every $1\le k<r$, let us consider the family of quasi-parabolic vector bundles over $\PP^1$ obtained by changing the $k-$-th step of the filtration of $E|_x$ to all admissible subspaces $E_{x,k}'$ such that
$$E_{x,k+1}\subsetneq E_{x,k}' \subsetneq E_{x,k-1}$$
Consider the union of all the $\alpha$-stable points $(E,E_\bullet')$ in such families. As the codimension of $\SZ$ in $\SM(r,\alpha,\xi)$ is at least $3$ and the families are at most $1$-dimensional, then union of all the families must have positive codimension and therefore, there exists some open nonempty subset $\SW\subset \SM(r,\alpha,\xi)$ whose points are not in the image of any family. Now take $\SU'=\SU\cap \SV\cap \SW$ and repeat the argument in \ref{prop:GenericFiberHitchin}.

For a generic $x\in \SD_U$, $X_s$ has a unique singularity which is a node not lying over a parabolic point. Then $H^{-1}(s)$ is an uniruled variety of dimension $m$. Let $Z=(\SM_{K(D)}(r,\alpha,\xi)\backslash T^*\SV)\cap H^{-1}(\SD_U)$. If $g\ge 4$, by Corollary \ref{eq:cotangentCodim2Discr} the codimension of $Z$ in $H^{-1}(\SD_U)$ is at least $2$. Let $S=H(Z)$. If $\dim(S)<m-1$ then for every $s\in \SD_U\backslash S$, $H^{-1}(s)=H_\SV^{-1}(s)$, so the fiber of the (restricted) Hitchin map contains a complete rational curve.

On the other hand, if $\dim(S)=m-1$, then $H|_Z:Z\mapsto \SD_U$ is dominant and, therefore, the generic fiber has dimension $\dim(Z)-\dim(\SD_U)\le m-2$. Then, for a generic $s\in \SD_U$, $Z\cap H^{-1}(s)$ has codimension at least $2$ in $H^{-1}(s)$. Therefore $H^{-1}(s)\backslash H_{\SV}^{-1}(s)$ has codimension at least $2$ in $H^{-1}(s)$ and $H_\SV^{-1}(s)$ must contain a complete rational curve.
\end{proof}

\begin{proposition}
\label{prop:HitchinGlobalFunctionsBir}
Let $\SV\subset \SM(r,\alpha,\xi)$ be an open subset whose complement has codimension at least $2$. Then the global algebraic functions $\Gamma(T^*\SV)$ produce a map
$$\tilde{h}:T^*\SV \longrightarrow \Spec(\Gamma(T^*\SV))\cong W\cong \CC^m$$
which is the restriction of the parabolic Hitchin map to $T^*\SV$ up to an isomorphism of $\CC^m$, where $m=\dim W$. Moreover, consider the action of $\CC^*$ on $T^*\SV$ given by dilatation on the fibers. Then there is a unique $\CC^*$ action on $W$ such that $\tilde{h}$ is $\CC^*$-equivariant
\end{proposition}

\begin{proof}
For $\SV=\SM(r,\alpha,\xi)$ this was proved in Proposition \ref{prop:HitchinGlobalFunctions}. As $T^*\SV\subset T^*\SM(r,\alpha,\xi)$ is an open subset whose complement has codimension at least $2$ and $T^*\SM(r,\alpha,\xi)$ is smooth then by Hartog's theorem we know that $\Gamma(T^*\SV)=\Gamma(T^*\SM(r,\alpha,\xi))$ and the Proposition follows.
\end{proof}

\begin{theorem}
\label{theorem:TorelliBirational}
Let $(X,D)$ and $(X',D')$ be two smooth projective curves of genus $g\ge 4$ and $g'\ge 4$ respectively with set of marked points $D\subset X$ and $D'\subset X'$. Let $\xi$ and $\xi'$ be line bundles over $X$ and $X'$ respectively, and let $\alpha$ and $\alpha'$ be full flag generic systems of weights over $(X,D)$ and $(X',D')$ respectively. Then if $\SM(X,r,\alpha,\xi)$ is $3$-birational to $\SM(X',r',\alpha',\xi')$ then $r=r'$ and $(X,D)$ is isomorphic to $(X',D')$, i.e., there exists an isomorphism $X\cong X'$ sending the set $D$ to $D'$.
\end{theorem}

\begin{proof}
The proof will be completely analogous to the one given for Theorem \ref{theorem:Torelli}. Let $\SV\subset \SM(X,r,\alpha,\xi)$ and $\SV'\subset \SM(X',r',\alpha',\xi')$ be open subsets whose complement has codimension $3$ and let $\Phi:\SV \to \SV'$ be the $3$-birational morphism between both moduli spaces. By Proposition \ref{prop:HitchinGlobalFunctionsBir} there must exist an algebraic $\CC^*$-equivariant isomorphism $f:W\cong \Spec(\Gamma(T^*\SV)) \stackrel{\sim}{\longrightarrow} \Spec(\Gamma(T^*\SV'))\cong W'$ such that the following diagram commutes
\begin{eqnarray*}
\xymatrixcolsep{3pc}
\xymatrix{
T^*\SV \ar[r]^{d(\Phi^{-1})} \ar[d]_{\tilde{h}} & T^*\SV' \ar[d]^{\tilde{h}}\\
W \ar[r]^{f} & W'
}
\end{eqnarray*}
As $f$ is $\CC^*$-equivariant, it must preserve the filtration by subspaces in terms of the decay and it must send the subspace of maximum decay $|\lambda|^r$ of $W$ to the subspace of maximum decay $|\lambda|^{r'}$ of $W'$. In particular, the number of steps of the filtration must be the same. As the filtrations of $W$ and $W'$ have $r-1$ and $r'-1$ steps respectively, then $r=r'$ and $f(W_r)=W_r'$

As $d(\Phi^{-1})$ is an isomorphism, it maps complete rational curves in $T^*\SV$ to complete rational curves in $T^*\SV'$. By Lemma \ref{lemma:DiscriminantCharacterizationBir}, $f$ sends the locus of singular spectral curves $\SD\subset W$ to the locus of singular spectral curves $\SD'\subset W'$. Moreover, we know that $f(W_r)=W_r'$, so if we let $\SC=\SD\cap W_r$ and $\SC'=\SD'\cap W_r'$ we obtain that $f(\SC)=\SC'$.

By Proposition \ref{prop:recoverDualVariety}, the dual variety of $\PP(\SC_X)$ in $\PP(W_r)$ is $X\subset \PP(W_r^*)$ and, similarly, the dual variety of $\PP(\SC_X')$ in $\PP(W_r')$ is $X'\subset \PP((W_r')^*)$, so $f$ induces an isomorphism $f^\vee:\PP(W_r^*)\to \PP((W_r')^*)$ that sends $X$ to $X'$. Moreover, the dual of the rest of the components $\PP(\SC_x)$ of $\PP(\SC)$ correspond to the divisor $D\subset X \subset \PP(W_r^*)$ and the dual of the components $\PP(\SC_x')$ of $\PP(\SC')$ correspond to the divisor $D'\subset X'\subset \PP((W_r')^*)$, so $f^\vee$ must send $D$ to $D'$. Therefore, $f^\vee$ induces an isomorphism $f^\vee:(X,D)\stackrel{\sim}{\longrightarrow}(X',D')$.
\end{proof}

In contrast with the usual Torelli theorem, where there exist several non-isomorphic moduli spaces of parabolic vector bundles for the same curve $(X,D)$ depending on the stability and topological data of the bundles, in the case of $k$-birational geometry we can state a hard reciprocal of the Torelli theorem

\begin{proposition}
\label{prop:TorelliBirationalReciprocal}
Let $(X,D)$ be a marked smooth projective curve of genus $g\ge 1+\frac{k-1}{r-1}$. Let $\xi$ and $\xi'$ be line bundles over $X$ and let $\alpha$ and $\alpha'$ be full flag generic systems of weights of rank $r$ over $(X,D)$. Then there is a $k$-birational map
$$\SM(r,\alpha,\xi)\dashrightarrow \SM(r,\alpha',\xi')$$
In particular, if $g\ge 3$, $\SM(r,\alpha,\xi)$ and $\SM(r,\alpha',\xi')$ are $3$-birational.
\end{proposition}

\begin{proof}
Let $d=\deg(\xi)$ and $d'=\deg(\xi')$. Let us write $d'-d=rm-k$ for some $0\le k<r$. Let $x\in D$ be any parabolic point. Then
$$\deg(\ST_{\SO_X(mx)}\circ \SH_{kx}(\xi))=\deg(\xi')$$
Therefore, there exists a line bundle $L$ of degree zero such that
$$\xi'=L^r \otimes \left (\ST_{\SO_X(mx)}\circ \SH_{kx}(\xi) \right) = \ST_{L(mx)}\circ \SH_{kx}(\xi)$$
Take $T=(\id,1,L(mx),kx)$. Then $T$ induces an isomorphism
$$T:\SM(r,\alpha,\xi)\longrightarrow \SM(r,T(\alpha),\xi')$$
By Corollary \ref{cor:stabCodim2} there exists an open subset $\SU\subset \SM(r,T(\alpha),\xi')$ whose complement has codimension at least $3$ parameterizing $\alpha'$-stable parabolic vector bundles in $\SM(r,T(\alpha),\xi')$. Similarly, there exists $\SU'\subset \SM(r,\alpha',\xi')$ whose complement have codimension at least 3 parameterizing $T(\alpha)$-stable parabolic vector bundles in $\SM(r,\alpha',\xi')$. Then $\SU$ and $\SU'$ can be canonically identified as the moduli space of parabolic vector bundles of rank $r$ and determinant $\xi$ which are both $T(\alpha)$- stable and $\alpha'$-stable. Finally, $T^{-1}(\SU)\subset \SM(r,\alpha,\xi)$ is an open subset whose complement has codimension at least $k$ and we have an isomorphism
$T^{-1}(\SU)\cong \SU'$ so the moduli spaces are $3$-birational.
\end{proof}

Observe that we obtain analogues of this Proposition in the $k$-birational category by just increasing the genus condition, while the Torelli theorem holds in the $k$-birational category for any $g\ge 4$.

Now let $\Phi:\SM(X,r,\alpha,\xi)\dashrightarrow \SM(X',r',\alpha',\xi')$ be a $3$-birational isomorphism. By the $3$-birational version of the Torelli Theorem we have $r=r'$ and the $3$-birational map $\Phi$ induces an isomorphism $\sigma:X\to X'$ which sends the set $D$ to $D'$. Pulling back by $\sigma$, we obtain a $3$-birational map
$$\Phi'=\Sigma_\sigma \circ \Phi:\SM(X,r,\alpha,\xi) \dashrightarrow \SM(X,r,\sigma^*\alpha,\sigma^*\xi')$$
Let $\alpha''=\sigma^*\alpha$ and $\xi''=\sigma^*\xi'$. Let $\SV\subset\SM(X,r,\alpha,\xi)$ and $\SV''\subset \SM(X,r,\alpha'',\xi'')$ be open subsets whose respective complements have codimension at least $3$ such that $\Phi':\SV\to \SV''$ is an isomorphism. Then the differential induces an isomorphism $d(\Phi^{-1}):T^*\SV \longrightarrow T^*\SV''$. Let $h:T^*\SV\to W$ and $h'':T^*\SV''\to W$ denote the restriction of the Hitchin morphism to $\SV$ and $\SV''$ respectively. Since both moduli spaces are built over the same marked curve $(X,D)$ and with the same rank $r$, the Hitchin space is the same. By Proposition \ref{prop:HitchinGlobalFunctionsBir}, there exists a $\CC^*$-equivariant automorphism $f:W\to W$ such that the following diagram commutes
\begin{eqnarray*}
\xymatrixcolsep{4pc}
\xymatrix{
T^*\SV \ar[r]^{d(\Phi^{-1})} \ar[d]_{h} & T^*\SV'' \ar[d]^{h''} \\
W \ar[r]^{f} & W
}
\end{eqnarray*}
By Lemma \ref{lemma:DiscriminantCharacterizationBir}, $f:W\to W$ must preserve the discriminant locus, i.e., $f(\SD) = \SD$. We know that it is $\CC^*$-equivariant, so using Lemma \ref{lemma:recoverHitchin}, $f$ preserves the decomposition $W=\bigoplus_{k>1} W_k$ and its restrictions $f_k:W_k\to W_k$ are linear. For each $k>1$, let $h_r:T^*\SV\to W_k$ and $h_k'':T^*\SV''\to W_k$ denote the compositions of $h$ and $h''$ with the projection $W\twoheadrightarrow W_k$ respectively. In particular, for each $k>1$ the following diagram commutes
\begin{eqnarray*}
\xymatrixcolsep{4pc}
\xymatrix{
T^*\SV \ar[r]^{d(\Phi^{-1})} \ar[d]_{h_k} & T^*\SV'' \ar[d]^{h_k''} \\
W_k \ar[r]^{f_k} & W_k
}
\end{eqnarray*}

\begin{lemma}
\label{lemma:preserveWrBir}
Let $g\ge 4$. Let $f_r:W_r\to W_r$ be the $\CC^*$-equivariant map on the Hitchin space such that the following diagram commutes
\begin{eqnarray}
\label{eq:frDefinitionBir}
\xymatrixcolsep{4pc}
\xymatrix{
T^*\SV \ar[r]^{d(\Phi^{-1})} \ar[d]_{h_r} & T^*\SV'' \ar[d]^{h_r'} \\
W_r \ar[r]^{f_r} & W_r
}
\end{eqnarray}
Then for every $k>0$ and every $x_0\in X$
$$f_r\left (H^0(K^rD^{r-1}(-kx_0))\right)=H^0(K^rD^{r-1}(-kx_0))$$
\end{lemma}

\begin{proof}
As $d(\Phi^{-1})$ is an isomorphism, it maps complete rational curves on $T^*\SV$ to complete rational curves on $T^*\SV''$. By Lemma \ref{lemma:DiscriminantCharacterizationBir}, the morphism $f$ must preserve $\SC=\SD\cap W_r$. Therefore, the associated map of dual varieties is an automorphism of the marked curve $(X,D)$. Through the previous discussion, we proved that we can assume that the induced automorphism of the curve $X$ is the identity, so we can just proceed as in the proof of Lemma \ref{lemma:preserveWr}.
\end{proof}

Once we have proven the $3$-birational version of the Torelli theorem and the previous Lemma, we automatically obtain that if $(E,E_\bullet)\in \SV$ is a generic parabolic vector bundle and $\Phi(E,E_\bullet)=(E'',E''_\bullet)\in \SV''$ is also generic in the sense of Lemma \ref{lemma:ParEndNoSections} then Lemmas \ref{lemma:recoverSParEndsections}, \ref{lemma:recoverSParEnd} and \ref{lemma:recoverParEnd} hold and we obtain that there is an isomorphism
$$\Phi_{\SParEnd_0}:\ParEnd_0(E,E_\bullet)\cong \ParEnd_0(E'',E''_\bullet)$$
Moreover, we obtain an analogue of Lemma \ref{lemma:recoverNilpotent}

\begin{lemma}
\label{lemma:recoverNilpotentBir}
Suppose that $g\ge 4$. For each $x\in X$, and every $k>1$, the linear subspace
$$H^0(K^kD^{k-1}(-x))\subseteq W_k$$
is preserved by the linear map $f_k:W_k\longrightarrow W_k$.
\end{lemma}

\begin{proof}
Let $\tilde{\SV}\subset \SV$ the open subset of parabolic vector bundles $(E,E_\bullet)\in \SV$ such that both $(E,E_\bullet)$ and $\Phi(E,E_\bullet)$ are generic in the sense of Lemma \ref{lemma:ParEndNoSections}. Applying Corollary \ref{cor:HitchinDominant} to $L=K(D-x)$ and the open subsets $\tilde{\SV}$ and $\tilde{\SV''}=\Phi(\tilde{\SV})$ we obtain that the linear subspace
$$\bigoplus_{k=2}^r H^0(K^kD^{k-1}(-kx))\subseteq W$$
is the space generated by the images $h(H^0(\SParEnd_0(E,E_\bullet)\otimes K(D-x)))$ both when $(E,E_\bullet)$ runs over $\tilde{\SV}$ and when $(E,E_\bullet)$ runs over $\tilde{\SV''}$.

By Lemma \ref{lemma:recoverSParEndsections}, for every $(E,E_\bullet)\in \tilde{\SV}$, if $(E'',E''_\bullet)=\Phi(E,E_\bullet)\in \tilde{\SV''}$, then the image of $H^0(\SParEnd_0(E,E_\bullet)\otimes K(D-x))$ by $d(\Phi^{-1})$ is $H^0(\SParEnd_0(E'',E''_\bullet)\otimes K(D-x))$. Therefore $f$ preserves $\bigoplus_{k=2}^r H^0(K^kD^{k-1}(-kx))$. As it is diagonal, $f_k$ preserves $H^0(K^kD^{k-1}(-kx))$.

For $k>1$ the curve $X$ is embedded in $\PP(W_k^*)$ through the linear system $|K^rD^{r-1}|$. The spaces $H^0(K^kD^{k-1}(-kx))$ for $x\in X$ correspond to the osculating $k$-spaces of $X$ at $x$, $\op{Osc}_k(x)$. As $\PP(f_k)$ preserves $\op{Osc}_k(x)$, by Lemma \ref{lemma:preserveOsculating} it preserves $\op{Osc}_1(x)$ and, therefore, $f_k$ must preserve the hyperplanes
$$H^0(K^kD^{k-1}(-x))\subset H^0(K^kD^{k-1})$$
for every $x\in X$.
\end{proof}

From this result we obtain the following Lemma, whose proof is exactly the same as Lemma \ref{lemma:recoverParEndalgebra}

\begin{lemma}
\label{lemma:recoverParEndalgebraBir}
Suppose that $g\ge 6$. Let $(E,E_\bullet)\in \tilde{\SV}\subset \SV$ and let $(E'',E''_\bullet)=\Phi(E,E_\bullet)$. Then $\ParEnd_0(E,E_\bullet)$ and $\ParEnd_0(E'',E''_\bullet)$ are isomorphic as Lie algebra bundles over $X$.
\end{lemma}

Now we are ready to generalize Theorem \ref{theorem:ExtendedTorelli} to the $3$-birational setting.

\begin{theorem}
\label{theorem:ExtendedTorelliBirational}
Let $(X,D)$ and $(X',D')$ be two smooth projective curves of genus $g\ge 6$ and $g'\ge 6$ respectively with a set of marked points $D\subset X$ and $D'\subset X'$. Let $\xi$ and $\xi'$ be line bundles over $X$ and $X'$ respectively, and let $\alpha$ and $\alpha'$ be full flag generic systems of weights over $(X,D)$ and $(X',D')$ respectively. Let
$$\Phi: \SM(X,r,\alpha,\xi) \dashrightarrow \SM(X',r',\alpha',\xi')$$
be a $3$-birational map. Then
\begin{enumerate}
\item $r=r'$
\item $(X,D)$ is isomorphic to $(X',D')$, i.e., there exists an isomorphism $\sigma:X\stackrel{\sim}{\to} X'$ sending $D$ to $D'$.
\item There exists a basic transformation $T$ such that
\begin{itemize}
\item $\sigma^*\xi'\cong T(\xi)$
\item For every $(E,E_\bullet)\in \SM(r,\alpha,\xi)$, $\sigma^*\Phi(E,E_\bullet) \cong T(E,E_\bullet)$
\end{itemize}
\end{enumerate}
\end{theorem}

\begin{proof}
By the $3$-birational version of the Torelli Theorem (Theorem \ref{theorem:TorelliBirational}) we have $r=r'$ and the $3$-birational map $\Phi$ induces an isomorphism $\sigma:X\to X'$ which sends the set $D$ to $D'$. Pulling back by $\sigma$, we obtain a $3$-birational map
$$\Phi':\SM(X,r,\alpha,\xi) \dashrightarrow \SM(X,r,\sigma^*\alpha,\sigma^*\xi')$$
Let $\alpha''=\sigma^*\alpha$ and $\xi''=\sigma^*\xi'$. Let $\SV\subset\SM(X,r,\alpha,\xi)$ and $\SV''\subset \SM(X,r,\alpha'',\xi'')$ be open subsets whose respective complements have codimension at least $3$ such that $\Phi':\SV\to \SV''$ is an isomorphism. Let $\tilde{\SV}\subset \SV$ be the subset of parabolic vector bundles $(E,E_\bullet)\in \SV$ such that both $(E,E_\bullet)$ and $(E'',E''_\bullet)=\Phi'(E,E_\bullet)$ are generic in the sense of Lemma \ref{lemma:ParEndNoSections}. Then by Lemma \ref{lemma:recoverParEndalgebraBir} for every $(E,E_\bullet)\in \tilde{\SV}$ we have that $\ParEnd_0(E,E_\bullet)$ and $\ParEnd_0(E'',E''_\bullet)$ are isomorphic as Lie algebra bundles over $X$. Then by Lemma \ref{lemma:isoParEnd} there exists a basic transformation $T=(\id,s,L,H)$ such that $(E'',E''_\bullet)\cong T(E,E_\bullet)$.

Up to this point we have proved that for every $(E,E_\bullet)\in \tilde{\SV}$ there exists a basic transformation $T$ such that $\Phi'(E,E_\bullet) = T(E,E_\bullet)$. Repeating the argument given in the proof of Theorem \ref{theorem:ExtendedTorelli} we obtain that there exists some $T\in \ST_{\xi,\xi''}$ such that for every $(E,E_\bullet)\in \tilde{\SV}$, $\Phi'(E,E_\bullet)=T(E,E_\bullet)$. Repeating the argument in Theorem \ref{theorem:ExtendedTorelli}, let $\SW\subset \SV$ be the open subset consisting on parabolic vector bundles $(E,E_\bullet)$ which are both $\alpha$-stable and $T^{-1}(\alpha'')$-stable. By Corollary \ref{cor:stabCodim2}, the complement of $\SW$ has codimension at least $2$ in $\SM(r,\alpha,\xi)$ and, in particular, $\SW\cap \tilde{\SV}$ is dense in $\SW$. Therefore, for every map $\psi:\SW\cap \tilde{\SV} \to \SM(r,\alpha'',\xi'')$ there exist at most a unique extension to $\SW$. By construction of $\SW$, we know that $T$ gives a well defined map $T:\SW \to \SM(r,\alpha'',\xi'')$. Moreover, we know that $\Phi'|_{\SW\cap \tilde{\SV}}=T|_{\SW \cap \tilde{\SV}}$ and $\Phi'|_{\SW}$ is another extension to $\SW$, so $\Phi'|_{\SW}=T|_{\SW}$. Finally, let us prove that $\Phi'$ coincides with $T$ over $\SV$, i.e., that for every $(E,E_\bullet)\in \SM(r,\alpha,\xi)$ such that $\Phi'$ is defined, $\Phi'(E,E_\bullet)=T(E,E_\bullet)$.

As $\alpha''$ is a full flag system of weights, $\SM(r,\alpha'',\xi'')$ is a fine moduli space for every $\xi''$. Therefore, $\Phi'$ is represented by a parabolic vector bundle $(\SE'',\SE''_\bullet)$ over $\SV\times X$ whose fibers are $\alpha''$-stable as parabolic vector bundles over $X$. We have the following commutative diagram
\begin{eqnarray*}
\xymatrixrowsep{0.5pc}
\xymatrix{
\Phi' \ar@{|->}[rrr] \ar@{|->}[dddd] & & & (\SE'',\SE''_\bullet) \ar@{|->}[dddd]\\
& \Hom(\SV,\SM(r,\alpha'',\xi'')) \ar[r]^-{\sim} \ar[dd]^{i^\sharp} & {\underline{\SM(r,\alpha'',\xi'')}}(\SV) \ar[dd]^{i^*} &\\
&&&\\
& \Hom(\SW,\SM(r,\alpha'',\xi'')) \ar[r]^-{\sim} & {\underline{\SM(r,\alpha'',\xi'')}} (\SW) &\\
T \ar@{|->}[rrr] & & & T(\SE,\SE_\bullet)|_{\tilde{\SV}}\\
}
\end{eqnarray*}
where $(\SE,\SE_\bullet)$ is the universal family of the moduli space $\SM(r,\alpha,\xi)$. Therefore, $(\SE'',\SE''_\bullet)$ is an extension of $T(\SE,\SE_\bullet)|_{\SW}$ from $\SW$ to $\SV$. Clearly $T(\SE,\SE_\bullet)|_{\SV}$ is a possible extension as a family of quasi-parabolic vector bundles over $\SV$ and the complement of $\SW$ in $\SV$ has codimension at least $2$, so by Lemma \ref{lemma:uniqueExtensionCodim2} we have $(\SE'',\SE''_\bullet)\cong T(\SE,\SE_\bullet)|_{\SV}$. Taking this isomorphism of families fiberwise we obtain the desired result.
\end{proof}

\begin{corollary}
\label{cor:AutoBir}
Let $(X,D)$ be a smooth projective curve of genus $g\ge 6$ and let $\alpha$ be a full flag generic system of weights over $(X,D)$ of rank $r$. Let $\xi$ be a line bundle over $X$. Then
$$\Aut_{3-\Bir}(\SM(r,\alpha,\xi))=\ST_\xi=\{T\in \ST | T(\xi)\cong \xi\}<\ST$$
if $r>2$ and
$$\Aut_{3-\Bir}(\SM(2,\alpha,\xi))=\ST_\xi^+=\{T\in \ST^+ | T(\xi)\cong \xi\}<\ST^+$$
\end{corollary}

\begin{proof}
Every basic transformation $T\in \ST_\xi$ induce an isomorphism
$$T:\SM(r,\alpha,\xi) \longrightarrow \SM(r,T(\alpha),\xi)$$
By Corollary \ref{cor:stabCodim2}, there exist open subsets $\SU\subset \SM(r,\alpha,\xi)$ and $\SU'\subset \SM(r,T(\alpha),\xi)$ whose complement has codimension at least $3$ parameterizing parabolic vector bundles of rank $r$ and determinant $\xi$ which are both $\alpha$-stable and $T(\alpha)$-stable. Therefore, there is an isomorphism $\Psi:\SU\stackrel{\sim}{\longrightarrow} \SU'$. Composing with $T$, we obtain an isomorphism
$$\Psi^{-1}\circ T:T^{-1}(\SU') \stackrel{\sim}{\longrightarrow} \SU$$
so we obtain a $3$-birational map $\SM(r,\alpha,\xi) \dashrightarrow \SM(r,\alpha,\xi)$.

By the previous Theorem, every $3$-birational automorphism is equivalent to one of the previous ones, so $\Aut_{3-\Bir}(\SM(r,\alpha,\xi))$ is a quotient of $\ST_\xi$. From Lemma \ref{lemma:freeTransforms}, different basic transformations $T,T'\in \ST_\xi$ induce different $3$-birational automorphisms of the moduli space if $r>2$, so we obtain the desired equality for $r>2$.
For $r=2$, by Lemma \ref{lemma:dualrk2}, we know that for every $T\in \ST_\xi$, we can find another transformation $T\in \ST_\xi^+$ whose image in $\Aut_{3-\Bir}(\SM(2,\alpha,\xi)$ is the same and, moreover, by Lemma \ref{lemma:freeTransforms}, two different transformations in $\ST_\xi^+$ induce different $3$-birational automorphisms, so we obtain the remaining equality.
\end{proof}

\section{Concentrated stability chamber}
\label{section:concentrated}

In the analysis of isomorphisms and $k$-birational transformations between moduli spaces of parabolic vector bundles held through the previous sections the systems of weights were allowed to belong to different stability chambers. This flexibility allowed us to describe transformations that transcended the limits of a stability chamber and relate moduli spaces for different choices of the stability and topological data of the bundles.

Nevertheless, by Theorem \ref{theorem:ExtendedTorelli} the possible basic transformations $T\in \ST$ giving rise to automorphisms of a moduli space $\SM(r,\alpha,\xi)$ must satisfy two compatibility conditions.
\begin{itemize}
\item $T(\xi)\cong \xi$
\item $T(\alpha)$ belongs to the same stability chamber as $\alpha$
\end{itemize}
While the first condition is easily computable and relies just on the choice of fixed topological invariants of the bundles, the second one depends on an analysis of the stability chamber where the system of weights $\alpha$ belongs. Therefore, it is possible that depending on the chamber certain basic generators of $\ST$ which preserve the determinant fail to preserve the stability and, therefore, they do not induce an automorphism.

Observe that if $T\in \ST_\xi<\ST$ then by Corollary \ref{cor:AutoBir} $T$ induces a $3$-birational transformation, but $T$
induces an automorphism if and only if $T(\alpha)$ and $\alpha$ share the same stability chamber. Therefore, analyzing the stability chamber of $T(\alpha)$ for each $T\in \ST_\xi$ is the same as studying the set of $3$-birational automorphisms that extend to a regular automorphism of the whole moduli space.

For a general $\alpha$ an explicit analysis may depend greatly on the geometry of the curve, as the geometrical walls in the space of systems of weights may vary with $X$ in low genus. We seek for classification results that do not depend on the choice of the Riemann surface, we will work on two directions. On one hand, we will build invariants that allow us to distinguish stability chambers in a precise way for high genus. This will be done in Section \ref{section:chamberAnalysis}. On the other hand, we will focus on studying some chamber where we can compute the stability conditions explicitly in low genus. In particular, in this section we will classify the automorphisms of the moduli space for a concentrated system of weights $\alpha$.

The chamber of concentrated weights is of particular interest, as its interior corresponds to generic weights for which parabolic stability is roughly equivalent to the stability of the underlying vector bundle in the following sense (see, for example, \cite{AG18TorelliDH})

\begin{lemma}
Let $\alpha$ be a generic concentrated system of weights. Let $(E,E_\bullet)$ be a parabolic vector bundle. Then
\begin{enumerate}
\item If $E$ is stable as a vector bundle then $(E,E_\bullet)$ is $\alpha$-stable as a parabolic vector bundle
\item $(E,E_\bullet)$ is $\alpha$-stable if and only if it is $\alpha$-semistable
\item If $(E,E_\bullet)$ is $\alpha$-semistable then $E$ is semistable as a vector bundle
\end{enumerate}
If moreover the rank and degree of $E$ are coprime then $E$ is semistable if and only if it is stable, so the stability of the parabolic vector bundle $(E,E_\bullet)$ is equivalent to the stability of the underlying vector bundle $E$.
\end{lemma}
The constant system of weights $\alpha_0\equiv 0$ lies in the frontier of the concentrated chamber. A parabolic vector bundle is $\alpha_0$-stable if its underlying vector bundle is stable. If the rank and degree of $E$ are coprime then the numerical wall passing through $\alpha\equiv 0$ cannot be realized in a geometric wall and, therefore, the stability is equivalent of the stability of the underlying vector bundle.

\begin{theorem}
\label{theorem:autoModuliConcentrated}
Let $X$ be an irreducible smooth complex projective curve  of genus $g\ge 6$ and let $D$ be a reduced effective divisor over $X$. Let $r\ge 2$ and let $\alpha$ be a generic concentrated full flag system of weights over $D$ of rank $r$. Let $\xi$ be a line bundle over $X$ such that $\deg(\xi)$ is coprime with $r$. Let $\SM(r,\alpha,\xi)$ be the moduli space of stable parabolic vector bundles of rank $r$ over $(X,D)$ with system of weights $\alpha$ and determinant $\xi$. Let $\Phi:\SM(r,\alpha,\xi)\to \SM(r,\alpha,\xi)$ be an automorphism. Then there exists a basic transformation $T$ of the form $T=(\sigma,s,L,0)$ with $T(\xi)\cong \xi$ such that $\Phi=T$. In fact, if $r>2$, then
$$\Aut(\SM(r,\alpha,\xi))\cong \{T=(\sigma,s,L,0)\in \ST | T(\xi)=\xi)\}<\ST$$
and if $r=2$
$$\Aut(\SM(r,\alpha,\xi))\cong \{T=(\sigma,1,L,0)\in \ST^+ | T(\xi)=\xi)\}<\ST^+$$
\end{theorem}

\begin{proof}
By Theorem \ref{theorem:autoModuli}, for every automorphism $\Phi$ there exists a basic transformation $T\in \ST$ such that $\Phi(E,E_\bullet)=T(E,E_\bullet)$ for all $(E,E_\bullet)\in \SM(r,\alpha,\xi)$ and such that
\begin{itemize}
\item $T(\xi)\cong\xi$
\item $T(\alpha)$ is in the same chamber as $\alpha$
\end{itemize}

Let $T=(\sigma,s,L,0)\in \ST$. The pullback of a concentrated system of weights is concentrated and the dual of a concentrated system of weights is concentrated, so $T(\alpha)$ lies in the concentrated chamber for every concentrated $\alpha$. In particular, this proves that $T$ induces an automorphism whenever $T(\xi)\cong \xi$.

Therefore, it is enough to prove that if $T=(\sigma,s,L,H)\in \ST_\xi$ induces an automorphism of the moduli space then $H=0$. Let $T_0=(\sigma,s,L,0)$. Then $T=T_0\circ \SH_H$. We have
$$T_0^{-1}=(\sigma^{-1},s,\sigma^*L^{-s},0)$$
By the previous discussion we know that $T_0^{-1}(\alpha)$ is concentrated, so it induces an isomorphism
$$T_0^{-1}:\SM(r,\alpha,\xi) \stackrel{\sim}{\longrightarrow} \SM(r,\alpha,T_0^{-1}(\xi))$$
composing with $\Phi$ we obtain an isomorphism
$$T_0^{-1}\circ \Phi = \SH_H:\SM(r,\alpha,\xi)\stackrel{\sim}{\longrightarrow} \SM(r,\alpha,T_0^{-1}(\xi))$$
So for every $(E,E_\bullet)\in \SM(r,\alpha,\xi)$, $\SH_H(E,E_\bullet)$ must be $\alpha$-stable. Let $d=\deg(\xi)$. Tensoring with a suitable line bundle we might assume that $0<d<r$. By hypothesis $T(\xi)\cong \xi$. Computing degrees in the determinant equality yields the following possibilities for $|H|$
\begin{enumerate}
\item If $s=1$, then $|H|$ is a positive multiple of $r$ and, therefore, $|H|\ge r >d$
\item If $s=-1$, then $-(d-|H|+kr)=d$, so $|H|=2d+kr$.
\begin{enumerate}
\item If $k\ge 0$ then $|H|\ge 2d>d$.
\item If $k<0$, then as $d<r$ yields $|H|<2r+kr=(2+k)r$. As we assumed $|H|>0$, then we can only have $k=-1$ and, therefore, $|H|=2d-r>0$.
\end{enumerate}
\end{enumerate}

Nevertheless, applying Lemma \ref{lemma:unstableHecke} in cases (1) and (2a) or Lemma \ref{lemma:unstableHecke2} in case (2b), we deduce that there exists some $(E,E_\bullet)\in \SM(r,\alpha,\xi)$ such that $\SH_H(E,E_\bullet)$ is $\alpha$-unstable if $H\ne 0$. 
\end{proof}

Observe that for every $\sigma:X\to X$ preserving the set $D$, $\deg(\sigma^*\xi)=\deg(\xi)$. Therefore, there exists a line bundle $L_\sigma$ such that
$$L^r\otimes \xi\cong (\sigma^{-1})^*\xi$$
on the other hand, $\deg(\sigma^*\xi^{-1})=-\deg(\xi)$. Therefore, there only exists a line bundle $L$ such that
$$(\sigma,-1,L,0)(\xi)\cong \xi$$
if $r|2d$. Under the hypothesis that $r$ and $d$ are coprime this can only be attained if $r=2$. Moreover, by Lemma \ref{lemma:dualrk2}, for each $T=(\sigma,-1,L,0)$ there exists a line bundle $L'$ such that $T$ and $T'=(\sigma,1,L',0)$ induce the same automorphism of the moduli space. Therefore, for $r\ge 2$ the automorphisms of $\SM(r,\alpha,\xi)$ are the ones generated by pullbacks and tensoring with a line bundle. For every $\sigma:X\to X$ the set of possible line bundles $L$ such that $\sigma^*(L^r\otimes \xi)\cong \xi$ is in bijection with the $r$-torsion points of the Jacobian.

Working in an analogous way to the proof of Proposition \ref{prop:basicTransPresentation} we obtain a short exact sequence of groups
$$1 \longrightarrow J(X)[r] \longrightarrow \Aut(\SM(r,\alpha,\xi)) \longrightarrow \Aut(X,D) \longrightarrow 1$$
Therefore $J(X)[r]$ is normal in $\Aut(\SM(r,\alpha,\xi))$ and, moreover, we have shown that the sequence admits a splitting, so we obtain that
$$\Aut(\SM(r,\alpha,\xi)) \cong J(X)[r] \rtimes \Aut(X,D)$$

This is far less than the order of $\ST_\xi$, as for every $\sigma\in \Aut(X,D)$ and for every $0\le H < (r-1)D$ and $s\in \{1,-1\}$ such that
$$s(d-|H|)\cong d \mod r$$
there exists a line bundle $L$ such that
$$(\sigma,s,L,H)(\xi)\cong \xi$$
where $d=\deg(\xi)$. If $L'$ is another line bundle such that $(\sigma,s,L',H)(\xi)\cong \xi$ then there exists an $r$-torsion point of the Jacobian $S\in J(X)[r]$ such that $L'=L\otimes S$. For any choice of $\sigma$ and $s$, the possible divisors $H$ with $0\le H< (r-1)D$ are isomorphic to the group $(\ZZ/r\ZZ)^{|D|}$. Nevertheless, if we impose the additional constraint
$$|H|\cong (1-s)d \mod r$$
Then solutions for $s=1$ form the subgroup $(\ZZ/r\ZZ)^{|D|-1}$, while any two solutions for $s=-1$ differ by a solution for $s=1$. Then a direct computation using the relations described in Section \ref{section:Hecke} (see Lemma \ref{lem:compositionRules} and Proposition \ref{prop:basicTransPresentation}) yields
$$\Aut_{3-\Bir}(\SM(r,\alpha,\xi)) \cong \ST_\xi \cong  \left( J(X)[r]\rtimes (\ZZ/r\ZZ)^{|D|-1} \right )  \rtimes \left(\ZZ/2\ZZ  \times \Aut(X,D) \right)$$
for $r>2$ and
$$\Aut_{3-\Bir}(\SM(2,\alpha,\xi)) \cong \ST_\xi^+ \cong  \left( J(X)[r]\rtimes (\ZZ/2\ZZ)^{|D|-1} \right )  \rtimes \Aut(X,D) $$
Under the coprimality condition, if $|D|>1$, this group is $2^{|D|-1}$ times bigger than $\Aut(\SM(r,\alpha,\xi))$ for $r=2$ and $2r^{|D|-1}$ times bigger for $r>2$.
This is an example that shows how the combination of the constraint on the topological invariants $T(\xi)\cong\xi$ and the stability constraint stating that $T(\alpha)$ and $\alpha$ share the same stability chamber can be really restrictive and reduce the automorphism group $\SM(r,\alpha,\xi)$ significantly.

In the concentrated chamber, the stability condition eliminates the Hecke transform $\SH_H$ and all its combinations from the possible automorphisms. From the point of view of the restrictions on the topology of the resulting vector bundles, Hecke transformation is the most flexible transformation, in the sense that it is the only one lacking numerical restrictions on the degree of the resulting line bundle. If $\xi$ and $\xi'$ are any two line bundles there exist a line bundle $L$ and a divisor $H$ such that $\ST_L\circ \SH_H(\xi)=\xi'$. On the other hand, dualization can only pass from degree $d$ line bundles to degree $-d$ and $\ST_L$ can only reach line bundles whose degree differs from the original one by a multiple of $r$.

Therefore, once Hecke transformations are discarded, the constraint $T(\xi)\cong \xi$ (or, more precisely, the induced numerical constraint $\deg(T(\xi))=\deg(\xi)$) becomes a really strong condition. This explains the huge difference with respect to $\ST_\xi$. If we allow $2$-rational maps, then Hecke transformations are no longer discarded and, therefore, they are available to be used in combination with dualization and tensorization. This relaxes the restriction $T(\xi)\cong \xi$, leading to more possibilities for the basic transformations $T\in \ST_\xi$.

\section{Stability chamber analysis}
\label{section:chamberAnalysis}

From Theorem \ref{theorem:ExtendedTorelli} we know that every isomorphism between two moduli spaces of parabolic vector bundles is induced by some basic transformation. In particular, in Theorem \ref{theorem:autoModuli} we proved that the automorphism group of $\SM(r,\alpha,\xi)$ is the subgroup of $\ST$ consisting on basic transformations such that
\begin{itemize}
\item $T(\xi)\cong \xi$
\item If $r=2$, $T\in \ST^+$.
\item $T(\alpha)$ belongs to the same stability chamber as $\alpha$
\end{itemize}

As we mentioned in the last section, the two conditions are computable and they just impose certain numerical restrictions on the possible topological invariants of the vector bundles, but the last one is of a different kind. Determining whether two parabolic weights $\alpha$ and $\alpha'$ over the same curve $(X,D)$ belong to the same stability chamber is highly nontrivial and depends greatly on the geometry of the curve $X$. Two systems of weights $\alpha$ and $\alpha'$ belong to different stability chambers if and only if there exists some $\alpha$-stable parabolic vector bundle $(E,E_\bullet)$ which is $\alpha'$-unstable or vice versa, i.e., if there exists some $\alpha'$-stable parabolic vector bundle which is $\alpha$-unstable.

Assume that $(E,E_\bullet)$ is $\alpha$-stable but $\alpha'$-unstable. Then there exists a maximal destabilizing subsheaf $F\subset E$ such that
$$\frac{\pdeg_{\alpha'}(F,F_\bullet)}{\rk(F)} > \frac{\pdeg_{\alpha'}(E,E_\bullet)}{\rk(E)}$$
but, from $\alpha$-stability
$$\frac{\pdeg_{\alpha}(F,F_\bullet)}{\rk(F)}<\frac{\pdeg_{\alpha}(E,E_\bullet)}{\rk(E)}$$
therefore, the existence of a destabilizing subsheaf imposes some numerical conditions on $\alpha$, $\alpha'$ and the topological invariants of $(E,E_\bullet)$ and $(F,F_\bullet)$. If this numerical conditions are not satisfied by $\alpha$ and $\alpha'$ then it is clear that they belong to the same stability chamber. In this case we say that $\alpha$ and $\alpha'$ belong to the same numerical chamber.

Nevertheless, the reciprocal is not always true. Even if $\alpha$ and $\alpha'$ satisfy the numerical conditions which are necessary for the existence of a destabilizing subbundle, finding a parabolic vector bundle $(E,E_\bullet)$ and a subsheaf $F\subset E$ with the needed invariants is not obvious. In fact, there might exist systems of weights $\alpha$ and $\alpha'$ such that the numerical conditions allowed the existence of $\alpha$-stable and $\alpha'$-unstable parabolic vector bundles but such that geometrically there do not exist at all. Therefore, the stability chambers are divided in several numerical chambers whose walls are not realized geometrically by any parabolic vector bundle.

We will start identifying some numerical invariants that will allow us to determine the numerical chambers uniquely.

Let $\{n_1(x),\ldots,n_r(x)\}=\overline{n}$ be any set of nonnegative integers. We say that $\overline{n}$ is admissible if for any $i=1,\ldots,r$ and any $x\in D$, $n_i(x)\in \{0,1\}$ and there exists $0<r'<r$ such that for all $x\in D$ yields $\sum_{i=1}^rn_i(x)=r'$. Let $d=\deg(\xi)$. We define
$$M(r,\alpha,d,\overline{n})=\left \lfloor \frac{r'd+ r'\sum_{x\in D} \sum_{i=1}^r \alpha_i(x)-r\sum_{x\in D} \sum_{i=1}^r n_i(x) \alpha_i(x)}{r}\right \rfloor \in \ZZ$$
Observe that for every $\varepsilon\in \RR^{|D|}$,
$$M(r,\alpha,d,\overline{n})=M(r,\alpha[\varepsilon],d,\overline{n})$$
i.e., $M(r,\alpha,d,\overline{n})$ only depends on the class $\alpha\in \tilde{\Delta}$.

Recall that we say that if a subbundle $F\subsetneq E$ of a parabolic vector bundle $(E,E_\bullet)$ is of type $\overline{n}$ then
$$\owt(F,F_\bullet)=\sum_{x\in D} \sum_{i=1}^r n_i(x) \alpha_i(x)$$

\begin{lemma}
\label{lemma:stability}
Let $(E,E_\bullet)$ be a parabolic vector bundle such that $\deg(E)=d$. Then $(E,E_\bullet)$ is semistable if and only if for every admissible $\overline{n}$ and every subbundle $F\subsetneq E$ of type $\overline{n}$ we have
$$\deg(F)\le M(r,\alpha,d,\overline{n})$$
\end{lemma}

\begin{proof}
The parabolic bundle $(E,E_\bullet)$ is semistable if for every subbundle $F$ with the induced parabolic structure
$$\frac{\deg(F)+\sum_{x\in D}\sum_{i=1}^rn_i(x)\alpha_i(x)}{r'}\le \frac{d+\sum_{x\in D}\sum_{i=1}^r \alpha_i(x)}{r}$$
Equivalently, solving for $\deg(F)$
$$\deg(F)\le \frac{r'd+ r'\sum_{x\in D} \sum_{i=1}^r \alpha_i(x)-r\sum_{x\in D} \sum_{i=1}^r n_i(x) \alpha_i(x)}{r}$$
As $\deg(F)$ is an integer, its value is at most the floor of the right hand side, which is precisely $M(r,\alpha,d,\overline{n})$.
\end{proof}

\begin{corollary}
Let $\alpha$ and $\alpha'$ be rank $r$ systems of weights such that for every admissible $\overline{n}$
$$M(r,\alpha,d,\overline{n})=M(r,\alpha',d,\overline{n})$$
then $\alpha$ and $\alpha'$ belong to the same stability chamber.
\end{corollary}

\begin{proof}
If $(E,E_\bullet)$ is $\alpha$-semistable then for every admissible $\overline{n}$ and every subbundle $F\subsetneq E$ of type $\overline{n}$
$$\deg(F)\le M(r,\alpha,d,\overline{n})=M(r,\alpha',d,\overline{n})$$
so $(E,E_\bullet)$ is $\alpha'$-semistable.
\end{proof}

Let $\SN$ be the set of admissible $\overline{n}$. Let us denote
$$\overline{M}(r,\alpha,d)=\left( M(r,\alpha,d,\overline{n})\right)_{\overline{n}\in \SN}\in \ZZ^{\SN}$$
then we say that $\alpha$ and $\alpha'$ belong to the same numerical stability chamber if and only if $\overline{M}(r,\alpha,d)=\overline{M}(r,\alpha',d)$.

\begin{proposition}
\label{prop:StabiltyChambersFinite}
There is a finite number of stability chambers in $\Delta$.
\end{proposition}

\begin{proof}
For every $\alpha\in \Delta$ and every admissible $\overline{n}$, using that $0\le \alpha_i(x) <1$ and $0\le n_i(x)\le 1$ we obtain the following bounds
\begin{multline*}
M_{\min}(r,d)=\frac{d}{r}-r|D|-1<\left \lfloor \frac{r'd+ r'\sum_{x\in D} \sum_{i=1}^r \alpha_i(x)-r\sum_{x\in D} \sum_{i=1}^r n_i(x) \alpha_i(x)}{r}\right \rfloor\\
 \le \frac{(r-1)d}{r}+(r-1)|D|=M_{\max}(r,d)
\end{multline*}
Therefore $\overline{M}(r,\alpha,d)\in [M_{\min}(r,d),M_{\max}(r,d)]^{\SN}$ for every $\alpha$. In particular this implies that there is a finite number of numerical chambers in $\Delta$. As a numerical chamber is included in exactly one stability chamber we obtain that there is a finite number of stability chambers.
\end{proof}

This proposition has some further implications on the $k$-birational geometry of the moduli space $\SM(r,\alpha,\xi)$.

\begin{corollary}
Let $k>0$. Let $X$ be a genus $g\ge 1+\frac{k-1}{r-1}$ Riemann surface and let $D\subset X$ be a nonempty set of points. Let $\alpha$ be any generic system of weights over $(X,D)$ and let $\xi$ be any line bundle over $X$. Then there exists an open subset $\SM^{\op{us}}(r,\xi)\subset \SM(r,\alpha,\xi)$ whose complement has codimension at least $k$ and such that each parabolic vector bundle $(E,E_\bullet)\in \SM^{\op{us}}(r,\xi)$ is $\alpha'$-stable for every generic $\alpha'\in \Delta$.
\end{corollary}

\begin{proof}
Let $\SC$ denote the set of stability chambers in $\Delta$. By the previous lemma it is a finite set. Let $\alpha_1,\ldots,\alpha_{|\SC|}$ be a set of generic representatives for the stability chambers in $\SC$. Then a parabolic vector bundle is $\alpha'$-stable for all generic $\alpha'\in \Delta$ if and only if it is $\alpha_i$-stable for every $i=1,\ldots,|\SC|$. On the other hand by Corollary\ref{cor:stabCodim2}, for every $\alpha_i$ there exists an open subset $\SU_i\subset \SM(r,\alpha,\xi)$ whose complement has codimension at least $k$ such that every $(E,E_\bullet)\in \SU_i$ is $\alpha$-stable and $\alpha_i$-stable. Take
$$\SM^{\op{us}}(r,\xi)=\bigcap_{i=1}^{|\SC|}\SU_i\subset \SM(r,\alpha,\xi)$$
As $\SC$ is a finite set, $\SM^{\op{us}}(r,\xi)$ is an open subset whose complement has codimension  at least $k$ and such that every $(E,E_\bullet)\in \SM^{\op{us}}(r,\xi)$ is $\alpha_i$-stable for every $i=1,\ldots,|\SC|$.
\end{proof}

One we have classified the space of numerical chambers, our objective is to develop a tool to determine whether some numerical wall separating two numerical chambers is actually realized by a destabilizing subbundle of some parabolic vector bundle, at least for big genus.

\begin{lemma}
\label{lemma:nonemptyBorder}
Let $X$ be a genus $g$ smooth complex projective curve. Suppose that
$$g\ge 1+(r-1)n-\left\lfloor \sum_{x\in D}\alpha_1(x) \right \rfloor$$
Then for every $\overline{n}$ there exist a stable parabolic vector bundle $(E,E_\bullet)\in \SM(r,\alpha,\xi)$ and a subbundle $F\subsetneq E$ of type $\overline{n}$
such that
$$\deg(F)=M(r,\alpha,d,\overline{n})$$
\end{lemma}

\begin{proof}
For every admissible choice of $\overline{n}$
$$\sum_{x\in D} \sum_{i=1}^r n_i(x)\alpha_i(x)\ge \sum_{x\in D}\alpha_1(x)$$
Therefore, the genus condition in \cite[Theorem 1.4.3A]{BB05} hold for every $\overline{n}$ and we obtain that there exists a stable parabolic vector bundle $(E,E_\bullet)$ of rank $r$ and degree $d=\deg(\xi)$ with a subbundle $F\subsetneq E$ satisfying the properties in the Lemma. Now it is enough to tensor it with a suitable degree zero line bundle to obtain another one whose determinant is isomorphic to $\xi$.
\end{proof}

\begin{theorem}
\label{theorem:stabilityChamber}
Let $\alpha$ and $\beta$ be generic full flag systems of weights of rank $r$ over $(X,D)$. Let $\xi$ be a degree $d$ line bundle over $X$ and assume that
$$g\ge 1+(r-1)n-\min\left(\left\lfloor \sum_{x\in D}\alpha_1(x) \right \rfloor,\left\lfloor \sum_{x\in D}\beta_1(x) \right \rfloor\right)$$
Then $\alpha$ and $\beta$ belong to the same stability chamber of the moduli space of rank $r$ determinant $\xi$ full flag parabolic vector bundles if and only if for every admissible $\overline{n}$
$$M(r,\alpha,d,\overline{n})=M(r,\beta,d,\overline{n})$$
\end{theorem}

\begin{proof}
The systems of weights $\alpha$ and $\beta$ belong to different chambers if and only if either there exists an $\alpha$-stable vector bundle $(E,E_\bullet)$ which is not $\beta$-stable or vice versa. Suppose that there exists an $\alpha$-stable, $\beta$-unstable parabolic vector bundle. By Lemma \ref{lemma:stability}, there exist a subbundle $F\subsetneq E$ and integers $\overline{n}$ such that
$$\owt(F,F_\bullet)=\sum_{x\in D} \sum_{i=1}^r n_i(x) \alpha_i(x)$$
and
$$M(r,\beta,d,\overline{n})<\deg(F)\le M(r,\alpha,d,\overline{n})$$
so $M(r,\beta,d,\overline{n})\ne M(r,\alpha,d,\overline{n})$. Reciprocally, suppose that $\overline{M}(r,\alpha,d)\ne \overline{M}(r,\beta,d)$. Then, interchanging $\alpha$ and $\beta$ if necessary, there exists an admissible $\overline{n}$ such that $M(r,\beta,d,\overline{n})<M(r,\alpha,d,\overline{n})$. By Lemma \ref{lemma:nonemptyBorder}, there exist an $\alpha$-stable parabolic vector bundle $(E,E_\bullet)$ and a subbundle $F\subsetneq E$ of type $\overline{n}$ such that
$$\deg(F)= M(r,\alpha,d,\overline{n})>M(r,\beta,d,\overline{n})$$
Therefore, from Lemma \ref{lemma:stability}, $(E,E_\bullet)$ is $\beta$-unstable.
\end{proof}

The genus condition in this Theorem deserves some remarks. First, notice that it is only needed for the ``necessary'' part of the theorem. If $\overline{M}(r,\alpha,d)=\overline{M}(r,\beta,d)$ then $\alpha$ and $\beta$ belong to the same numerical -- and therefore geometrical -- chamber, independently of the genus of the curve.

Second, the genus condition is picked so that it is valid for any couple of systems of weights $\alpha$ and $\beta$. There are stability chambers which are more easily distinguished than others. For some choices of $\alpha$ and $\beta$, the bound for the genus can be really lowered.

\begin{proposition}
\label{prop:stabilityChamberRefined}
Let $\alpha$ and $\beta$ be concentrated systems of weights and let $\overline{n}$ be an admissible array such that
$$M(r,\beta,d,\overline{n})<M(r,\alpha,d,\overline{n})$$
Then $\alpha$ and $\beta$ belong to different stability chambers if
$$g\ge 1+\frac{\left \lfloor \sum_{x\in D} \sum_{i=1}^r (1-\alpha_i(x))(1-n_i(x))\right \rfloor}{r'}$$
\end{proposition}

\begin{proof}
The proof is exactly the same as in the Theorem, but instead of using the genus bound in Lemma \ref{lemma:nonemptyBorder}, we apply the bound in \cite[Theorem 1.4.3A]{BB05}.
\end{proof}

Finally, observe that the genus bounds for the previous results are not well defined for $\alpha,\beta\in \tilde{\Delta}$, rather they depend on the choice of representatives in $\Delta$. We can play this out in our favor and choose suitable $\varepsilon,\delta\in \RR^{|D|}$ such that the genus bound for $\alpha[\varepsilon]$ and $\beta[\delta]$ is as low as possible. The bound for $\alpha[\varepsilon]$ decreases with $\varepsilon$. The maximum possible shift that we can take at each $x\in D$ is $\varepsilon(x)<1-\alpha_r(x)$. Therefore, the previous Proposition holds if for some $\tau>0$
$$g\ge 1+\frac{\left \lfloor \sum_{x\in D} \sum_{i=1}^r (\alpha_r(x)+\tau-\alpha_i(x))(1-n_i(x))\right \rfloor}{r'}$$
In particular, the more concentrated the weights in a numerical chamber are, the lesser genus is needed in order to realize the surrounding numerical walls as geometrical walls. This somehow justifies that our study of the concentrated chamber can be done more explicitly in lower genus.

Finally, we can apply the previous results to obtain the following versions of Theorem \ref{theorem:ExtendedTorelli} and Theorem \ref{theorem:autoModuli}.

\begin{theorem}
\label{theorem:ExtendedTorelliComputable}
Let $(X,D)$ and $(X',D')$ be two smooth projective curves of genus $g\ge \max\{1+(r-1)|D|,6\}$ and $g'\ge 6$ respectively with set of marked points $D\subset X$ and $D'\subset X'$. Let $\xi$ and $\xi'$ be line bundles over $X$ and $X'$ respectively, and let $\alpha$ and $\alpha'$ be full flag generic systems of weights over $(X,D)$ and $(X',D')$ respectively. Let
$$\Phi: \SM(X,r,\alpha,\xi)\stackrel{\sim}{\longrightarrow} \SM(X',r',\alpha',\xi')$$
be an isomorphism. Then
\begin{enumerate}
\item $r=r'$
\item $(X,D)$ is isomorphic to $(X',D')$, i.e., there exists an isomorphism $\sigma:X\stackrel{\sim}{\to} X'$ sending $D$ to $D'$.
\item There exists a basic transformation $T$ such that
\begin{itemize}
\item $\sigma^*\xi'\cong T(\xi)$
\item $\overline{M}(r,\sigma^*\alpha',\deg(\xi'))=\overline{M}(r,T(\alpha),\deg(\xi'))$
\item For every $(E,E_\bullet)\in \SM(r,\alpha,\xi)$, $\sigma^*\Phi(E,E_\bullet) \cong T(E,E_\bullet)$
\end{itemize}
\end{enumerate}
\end{theorem}

\begin{corollary}
\label{cor:autoModuliComputable}
Let $(X,D)$ be a smooth projective curve of genus $g\ge \max\{1+(r-1)|D|,6\}$ and let $\alpha$ be a full flag generic system of weights over $(X,D)$ of rank $r$. Let $\xi$ be a line bundle over $X$. Then the automorphism group of $\SM(r,\alpha,\xi)$ is the subgroup of $\ST$ consisting on basic transformations $T$ such that
\begin{itemize}
\item $T(\xi)\cong \xi$
\item $\overline{M}(r,T(\alpha),\deg(\xi))=\overline{M}(r,\alpha,\deg(\xi))$
\item If $r=2$, $T\in \ST^+$
\end{itemize}

\end{corollary}
Unlike the original results, these versions are fully computable for each specific case, in the sense that for every system of weights $\alpha$ and every line bundle $\xi$ we have an explicit morphism
\begin{eqnarray*}
\xymatrixrowsep{0.05pc}
\xymatrixcolsep{0.3pc}
\xymatrix{
({\det}_\xi,\overline{M}_\alpha)&:&\ST \ar[rrrr] &&&& \Pic(X)\times \ZZ^\SN\\
&& T \ar@{|->}[rrrr] &&&& (T(\xi),\overline{M}(r,T(\alpha),\deg(T(\xi))))
}
\end{eqnarray*}
And for $g\ge 1+(r-1)|D|$ we know that the set of isomorphisms between $\SM(r,\alpha,\xi)$ and $\SM(r,\alpha',\xi')$ is given by
$$({\det}_\xi,\overline{M}_\alpha)^{-1}(\xi',M(r,\alpha',\deg(\xi')))$$
In particular, the moduli spaces $\SM(r,\alpha,\xi)$ and $\SM(r,\alpha',\xi)$ are isomorphic if and only if
$$(\xi',\overline{M}(r,\alpha',\deg(\xi')))\in ({\det}_\xi,\overline{M}_\alpha)(\ST)$$
Moreover, from the description of $\ST$ in terms of the generators $\SD^-$, $\ST_L$ and $\SH_H$ given in Proposition \ref{prop:basicTransPresentation}
$$\ST\cong \langle \ST_L,\ST_H \rangle \rtimes \left( \Aut(X,x)\times \ZZ/2\ZZ\right )$$
for each chamber $\alpha$ and each determinant $\xi$ we can explicitly describe a presentation of
$$\Aut(\SM(r,\alpha,\xi))=({\det}_\xi,\overline{M}_\alpha)^{-1}(\xi,\overline{M}(r,\alpha,\deg(\xi)))<\ST$$
or, if $r=2$,
$$\Aut(\SM(2,\alpha,\xi))=({\det}_\xi,\overline{M}_\alpha)^{-1}(\xi,\overline{M}(r,\alpha,\deg(\xi)))\cap \ST^+<\ST^+$$
just by selecting generators in the right hand side.

\section{Examples}
\label{section:examples}

Let $X$ be a curve with an automorphism $\sigma:X\to X$ such that there exist $x,y\in X$ with $\sigma(x)=y$ and $\sigma(y)=x$. Take $D=\{x,y\}$. Let $0\le \alpha_1<1/2<\alpha_2<1$. Then take the following full flag system of weights of rank $r=2$ at $(X,D)$
\begin{eqnarray*}
\alpha_1(x)&=&\alpha_1\\
\alpha_2(x)&=&\alpha_2\\
\alpha_1(y)&=&\alpha_2-1/2\\
\alpha_2(y)&=&\alpha_1+1/2
\end{eqnarray*}
Then, by construction $\SH_{x+y}(\alpha) \sim \Sigma_\sigma(\alpha)$. Let $L$ be a line bundle of degree $1$ such that $L^2\cong \SO_X(x+y)$
Then we have that
$$(\sigma,1,L,x+y):\SM(r,\alpha,\xi) \longrightarrow \SM(r,\alpha,\xi)$$
is an automorphism. Now let
\begin{eqnarray*}
\Aut^+(X,D)&=& \{\sigma\in \Aut(X) | \sigma(x)=x \, , \, \sigma(y)=y\}\\
\Aut^-(X,D)&=& \{\sigma\in \Aut(X) | \sigma(x)=y \, , \, \sigma(y)=x\}
\end{eqnarray*}

Then the following basic transformations are nontrivial automorphisms of $\SM(r,\alpha,\xi)$
\begin{itemize}
\item $T=(\sigma^-,1,L,x+y)$, where $\sigma^-\in \Aut^-(X,D)$ and $T(\xi)\cong \xi$
\item $T=(\sigma^+,1,L,0)$, where $\sigma^+\in \Aut^+(X,D)$ and $T(\xi)\cong \xi$.
\end{itemize}

Moreover, if $|\delta|$ is small enough and $X$ has genus $g\ge 3$, then the weights $\alpha_i(x)$ are concentrated but the weights $\alpha_i(y)$ are not. Therefore, $\SH_y(\alpha)$ is concentrated, $\alpha$ is concentrated at $x$, $\SH_{x+y}(\alpha)$ is concentrated at $y$ and $\SH_x(\alpha)$ is not concentrated. From the genus condition, it can be proved using Theorem \ref{theorem:stabilityChamber} from the last section, that $\SH_{x+y}(\alpha)$, $\SH_x(\alpha)$ and $\SH_y(\alpha)$ do not belong to the same chamber as $\alpha$. Moreover, taking the pullback by $\sigma^-$ interchange the following (distinct) chambers
\begin{itemize}
\item $\SH_x(\alpha)$ and $\SH_y(\alpha)$
\item $\alpha$ and $\SH_{x+y}(\alpha)$
\end{itemize}
As all the chambers are different, in order for a basic transformation $T=(\sigma,s,L,H)$ to preserve the stability chamber of $\alpha$ we need either
\begin{itemize}
\item $\sigma\in \Aut^+(X,D)$ and $H=0$ or
\item $\sigma\in \Aut^-(X,D)$ and $H=x+y$
\end{itemize}
so, taking into account that for rank $2$ each transformation of the form $T=(\sigma,-1,L,H)$ is equivalent to another one of the form $(\sigma,1,L',H)$ for some $L'$, we obtain that the automorphisms of $\SM(r,\alpha,\xi)$ are precisely the ones described above.

This example proves that there exist curves and systems of weights for which the  Hecke transform induces nontrivial automorphisms when combined with pullbacks by suitable automorphisms of the curve even if the transformation $\SH_H$ alone does not preserve the stability chamber.

As we saw in the last theorem, this cannot happen in the concentrated setting and, in general, it is not expected to happen if the parabolic chamber is stable under transformations $\Sigma_\sigma$ for all $\sigma\in \Aut(X,D)$.

Now let $X$ be any Riemann surface and let $D=x$ for some $x\in X$. Let $0<\varepsilon<1/4$ and let us consider the following rank $3$ system of weights over $(X,D)$
\begin{eqnarray*}
\alpha_1(x) &=&\varepsilon\\
\alpha_2(x) &=&3\varepsilon\\
\alpha_3(x) &=&1-\varepsilon
\end{eqnarray*}

A direct computation shows us that $\SH_x(\alpha) \sim (\varepsilon,1-3\varepsilon,1-\varepsilon)$, so $\SH_x(\alpha)^\vee \sim \alpha$. Let $\xi$ be any degree $-1$ line bundle over $X$.Then
$$\SD^-\circ \SH_x(\xi)=(\xi(-x))^{-1}=\xi^{-1}(x)$$
so $\deg(\SD^-\circ \SH_x(\xi))=1+1=2=\deg(\xi)+3$. Therefore, there exists a line bundle $L$ of degree $1$ such that
$$L^3\otimes \xi(-x) \cong \xi^{-1}$$
Take $T=(\id,-1,L,x)$. As $\ST_L$ does not change the parabolic weights the previous computations shows that
\begin{itemize}
\item $T(\xi)=\xi$
\item $T(\alpha)\sim \alpha$
\end{itemize}
Therefore, we obtain that
$$(\id,-1,L,x):\SM(r,\alpha,\xi) \longrightarrow \SM(r,\alpha,\xi)$$
is an automorphism. Moreover, for any automorphism $\sigma:X\longrightarrow X$ fixing $D=x$ we have that
$$\deg(\SD^-\circ \SH_x(\xi))=2=\deg((\sigma^{-1})^*\xi)+3$$
Therefore, there exists a line bundle $L_\sigma$ of degree $1$ such that
$$L_\sigma^3 \otimes \xi(-x) \cong (\sigma^{-1})^*\xi^{-1}$$
As $\Sigma_\sigma$ fixes the parabolic point then taking $T=(\sigma,-1,L_\sigma,x)$ we obtain  that
\begin{itemize}
\item $T(\xi)=\xi$
\item $T(\alpha)\sim\alpha$
\end{itemize}
Therefore, we obtain that
$$(\sigma,-1,L_\sigma,x):\SM(r,\alpha,\xi) \longrightarrow \SM(r,\alpha,\xi)$$
is an automorphism. Then we have found an example of a marked curve of arbitrary high genus and a system of weights such that the Hecke transformation induces a nontrivial automorphism of the moduli space when combined with the dualization. In contrast with the previous example, where the curved was supposed to have an automorphism interchanging two parabolic points, in this example the existence of an automorphism involving Hecke transformation is achieved even if the curve is  generic and lacks nontrivial automorphisms.

The basic transformation $T=(\id,-1,L,x)$ is particularly interesting. If $g\ge 4$ then from Lemma \ref{lemma:freeTransforms} we know that $T$ acts nontrivially on $\SM(r,\alpha,\xi)$, but a direct computation shows that $T^2=\id_\ST$. Therefore, $T$ is an involution of $\SM(r,\alpha,\xi)$ that does not come from an involution of the Riemann surface $X$.

To complete the example, let us study other kinds of automorphisms that this moduli space admits. Let $T=(\sigma,s,L,H)\in \ST$. By construction $\SD^-(\alpha)\sim \SH_x(\alpha)$. Moreover, if $\epsilon$ is small enough then $\SH_{2x}(\alpha)\sim (1-5\varepsilon,1-3\varepsilon,1-\varepsilon)$ is concentrated. Therefore, so is $\SD^-\circ \SH_{2x}(\alpha)$. On the other hand, $\alpha$ and $\SH_x(\alpha)$ are not concentrated. Using the results of the previous chapter we can prove that if $\varepsilon$ is small enough and $g\ge 3$ then $\alpha\sim \SD^-\circ \SH_x(\alpha)$, $\SH_x(\alpha)\sim \SD^-(\alpha)$ and $\SH_{2x}(\alpha)\sim \SD^-\circ \SH_{2x}(\alpha)$ belong to three different stability chambers.

On the other hand, $\Sigma_\sigma$ and $\ST_L$ do not change the stability chamber, so $T(\alpha)$ is in the same stability chamber as $\alpha$ if and only if either
\begin{itemize}
\item $H=0$ and $s=1$ or
\item $H=x$ and $s=-1$
\end{itemize}

In both cases, for every $\sigma:X\to X$ fixing $D=x$ there exists a line bundle $L$ such that $(\sigma,1,L,0)(\xi)\cong\xi$ or $(\sigma,-1,L,x)(\xi)\cong \xi$ respectively. In every case, such $L$ is unique up to a choice of a $3$-torsion point in $J(X)$. Then
$$\Aut(\SM(r,\alpha,\xi)) \cong J(X)[3] \rtimes \left(\ZZ/2\ZZ \times  \Aut(X,D) \right)$$

An analogous example can be found for any rank. Just take $\alpha$ distributed as $\alpha_r(x)=1-\varepsilon$ and $\alpha_k(x)=(2k-1)\varepsilon$ for $k<r$. Then $\SD^-\circ \SH_{(r-2)x}(\alpha)\sim\alpha$. If we take $\xi$ of degree $-1$ then
$$\deg(\SH_{(r-2)x}(\xi))=\deg(\xi)-r+2=r-1=\deg(\xi^{-1})+r$$
Therefore, there exists a line bundle $L$ of degree $1$ such that if $T=(\id,-1,L,\SH_{(r-2)x})$ then
\begin{itemize}
\item $T(\alpha)\sim \alpha$
\item $T(\xi)=\xi$
\end{itemize}
so $T$ induces an automorphism $T:\SM(r,\alpha,\xi)\longrightarrow \SM(r,\alpha,\xi)$ which is an involution of the moduli space.

\end{document}